\documentclass{amsart}
\usepackage{amsmath,amssymb}
\usepackage{latexsym}

\newtheorem{theorem}{Theorem}[section]
\newtheorem{lemma}[theorem]{Lemma}
\newtheorem{proposition}[theorem]{Proposition}

\theoremstyle{definition}
\newtheorem{definition}[theorem]{Definition}

\theoremstyle{remark}
\newtheorem{remark}[theorem]{Remark}

\numberwithin{equation}{section}

\begin{document}


\title{Invariance of the Gibbs measure for the Benjamin-Ono equation}

\author{Yu Deng}
\address{Department of Mathematics, Princeton University, Princeton, NJ 08544}
\email{yudeng@math.princeton.edu}

\thanks{}


\date{}

\dedicatory{Dedicated to Sekai Saionji}

\begin{abstract}
In this paper we consider the periodic Benjemin-Ono equation. We will establish the invariance of the Gibbs measure associated to this equation, thus answering a question raised in Tzvetkov \cite{Tz10}. As an intermediate step, we also obtain a local well-posedness result in Besov-type spaces rougher than $L^{2}$, extending the $L^{2}$ well-posedness result of Molinet \cite{Mo08}.

\end{abstract}

\maketitle


\setlength{\parskip}{0.8ex}

\section{Introduction}\label{intro}
In this paper we study the periodic Benjamin-Ono equation\begin{equation}\label{bo}u_{t}+Hu_{xx}=uu_{x},\,\,\,\,\,\,(t,x)\in I\times\mathbb{T},\end{equation} where $I$ is a time interval, and $\mathbb{T}=\mathbb{R}/2\pi\mathbb{Z}$. Also the Hilbert transform $H$ is defined by $\widehat{Hu}(n)=-\mathrm{i}\cdot\mathrm{sgn}(n)\widehat{u}(n)$, where we understand that $\mathrm{sgn}(0)=0$. Since equation (\ref{bo}), as well as the truncated version to be introduced below, preserves both the reality and the mean value of $u$, we shall assume throughout this paper that $u$ is real-valued and has mean zero. Under this restriction, (\ref{bo}) is a Hamiltonian PDE with conserved energy\begin{equation}\label{hamiltonian}E[u]=\int_{\mathbb{T}}\frac{1}{2}|\partial_{x}^{1/2}u|^{2}-\frac{1}{6}u^{3}.\end{equation} Being completely integrable, it also has an infinite number of conserved quantities at the level of $H^{\frac{\sigma}{2}}$ for $0\leq \sigma\in \mathbb{Z}$, including the $L^{2}$ mass.

We briefly summarize the relevant previous study of (\ref{bo}) as follows. First, the classical energy method yields local well-posedness in $H^{\sigma}(\mathbb{T})$ for $\sigma>3/2$, see \cite{Io86}. By conservation laws, this implies global well-posedness in (say) $H^{2}$. In \cite{Ta04}, Tao introduces a gauge transform to prove the well-posedness result in $H^{1}$ for the Euclidean counterpart of (\ref{bo}). This approach is then adapted by Molinet-Ribaud \cite{MR09} to prove the $H^{1}$ well-posedness in the periodic case. Then Molinet in \cite{Mo07}, \cite{Mo08} further improves this result to $H^{\frac{1}{2}}$ and then $L^{2}$. For the Euclidean version we now also have well-posedness in $L^{2}$, see Burq-Planchon \cite{BP} and Ionescu-Kenig \cite{IK07}.

Starting from the pioneering work of Lebowitz-Rose-Speer \cite{LRS89} and Bourgain \cite{Bo94}, there has been considerable interest in constructing Gibbs measures for Hamiltonian PDEs and proving their invariance. From the dynamical system point of view, this provides the natural invariant measure for the system (which is the first step in studying this system's long time behavior); from the PDE aspect this is also important since it tells us exactly how rough a space can be so that we still have \emph{strong} solutions for \emph{generic} initial data. In this regard, such results can be viewed as variations on the theme of classical low-regularity well-posedness\footnote[1]{See \cite{BT11} for a general discussion about the notion of well-posedness in probabilistic sense.}. Among the important results in this field we only mention a few: Bourgain \cite{Bo94}, \cite{Bo96}, Burq-Tzvetkov \cite{BT08}, \cite{BT08ii}, Colliander-Oh \cite{CO12}, Nahmod-Oh-Bellet-Staffilani \cite{NORS10}.

The study of (\ref{bo}) along this line is initiated in Tzvetkov \cite{Tz10} where the Gibbs measure is rigorously constructed (see \cite{Tz10} for details; this construction is also reviewed in Section \ref{probab} below). In order to to prove its invariance, one has to construct global flow on its support; since this measure is supported in spaces rougher than $L^{2}$ (namely $L^{2}(\mathbb{T})$ has measure zero), the well-posedness result of Molinet \cite{Mo08} will not apply. Nevertheless, in Section 5 of \cite{Tz10}, the author makes several important observations regarding the behavior of the gauge transform and second Picard iteration for random data, which suggest that global well-posedness and measure invariance may still hold despite the low regularity. 

In the current paper we will solve this problem by establishing the invariance of the Gibbs measure. To be precise, we will construct an almost-surely defined (and unique) global flow for (\ref{bo}) in some Besov-type space $Z_{1}$ rougher than $L^{2}$, and prove that the Gibbs measure is kept invariant by this flow.
\begin{remark}Very recently, Tzvetkov-Visciglia \cite{TV11}, \cite{TV12} have constructed (and proved the invariance of) weighted Gaussian measures associated to the conserved quantities of (\ref{bo}) at the level of $H^{\frac{\sigma}{2}}$ for even $\sigma\geq 6$. From this point of view, the remaining cases are $\sigma\in\{2,4\}$ and $\sigma\geq 3$ odd (since the Gibbs measure corresponds to $\sigma=1$, and we do not expect a positive result for $\sigma=0$). We believe that these cases are easier than the $\sigma=1$ case and can be solved by combining the low regularity analysis here and the probabilistic arguments in \cite{TV11}, \cite{TV12}.
\end{remark} 

\subsection{Notations and preliminaries}\label{parameters}
Throughout this paper, the standard notations, such as $\lesssim$, $\gtrsim$ and $O(*)$, will always be used in terms of absolute values. The Japanese bracket $\langle x \rangle$ will be $(1+|x|^{2})^{\frac{1}{2}}$ and $\mathbb{N}$ will denote the set of nonnegative integers; the characteristic function of a set $E$ is denoted by $\mathbf{1}_{E}$ and if $E$ is finite, its cardinality is denoted by $\#E$. We will use $\mathbb{P}_{*}$ to denote (spatial) frequency projections; for example $\mathbb{P}_{+}$ (or $\mathbb{P}_{\leq 0}$) will be the projection onto strictly positive (or non-positive) frequencies and $\mathbb{P}_{\gtrsim\lambda}$ will be the projection onto frequencies with absolute value $\gtrsim\lambda$. We may use the same (latin or greek) letter in different places, but its meaning will be clear from the context.

Define $\mathcal{V}$ to be the space of distributions on $\mathbb{T}$ that are real-valued and have mean zero; in other words $f\in\mathcal{V}$ if and only if $\widehat{f}(-n)=\overline{\widehat{f}(n)}$ and $\widehat{f}(0)=0$. Let $\mathcal{V}_{N}$ be the subspace of $\mathcal{V}$ containing functions of frequency not exceeding $N$ (so that $\mathcal{V}_{N}$ is identified with $\mathbb{R}^{2N}$), and $\mathcal{V}_{N}^{\perp}$ be its orthogonal complement. Let $\Pi_{N}$ and $\Pi_{N}^{\perp}$ be the projection to the corresponding spaces, we actually have $\Pi_{N}=\mathbb{P}_{\leq N}$ and $\Pi_{N}^{\perp}=\mathbb{P}_{>N}$.

We use a parameter $s>0$ that will be chosen sufficiently small. The large constants $C$ and small constants $c$ may depend on $s$; Any situation in which they are independent of $s$ will be easily recognized. We choose a few other parameters, namely $(p,r,b,\tau,q,\kappa,\gamma, \epsilon)$, as follows:\begin{eqnarray}&&p=\frac{2}{1-2s}+s^{2}, \,\,r=\frac{1}{2}-\frac{1}{p},\,\,b=\frac{1}{2}-s^{\frac{15}{8}},\tau=8-s^{\frac{13}{8}},\nonumber\\
&&q=1+s^{\frac{3}{2}},\,\,\kappa=1-s^{\frac{5}{4}},\,\,\gamma=2-s^{2.5},\,\,\epsilon=s^{\frac{7}{4}}.\nonumber\end{eqnarray} When $s$ is small enough, we have the following hierarchy of smallness factors:\begin{eqnarray}\label{hierarchy}s^{3}&\ll&2-\gamma\ll r-s=\frac{1}{2}-\frac{1}{p}-s\ll \frac{1}{2}-b\ll\\&\ll &\epsilon\ll 8-\tau\ll q-1\ll 1-\kappa\ll s\ll s^{\frac{1}{2}}.\nonumber\end{eqnarray} In (\ref{hierarchy}) each $\ll$ symbol connects two numbers that actually differ in scale by a power of $s$. We will also use $0+$ to denote some small positive number (whether it depends on $s$ will be clear from the context); the meanings of $0-$, and $a+$, $a-$ are then obvious. Finally, using these parameters, we can define the space $Z_{1}$ by\footnote[1]{Here when we write $n\sim 2^{d}$, we are actually including $n=0$ when $d=0$; we will neglect this issue in later discussions.}\begin{equation}\|f\|_{Z_{1}}=\sup_{d\geq 0}\bigg(\sum_{n\sim 2^{d}}2^{rpd}|\widehat{f}(n)|^{p}\bigg)^{\frac{1}{p}}.\end{equation}

In addition to (\ref{bo}), we will introduce finite dimensional truncations of it. Fix a smooth, even cutoff function $\psi$ on $\mathbb{R}$ which equals $1$ on $[-\frac{1}{2},\frac{1}{2}]$ and vanishes outside $[-\frac{3}{4},\frac{3}{4}]$. We also need another function $\psi_{1}$ with the same properties, and let $\psi_{2}=1-\psi_{1}$. For a positive integer $N$, we define the multiplier $S_{N}$ by \begin{equation}\label{truncation}\widehat{S_{N}f}(n)=\psi\bigg(\frac{n}{N}\bigg)\widehat{f}(n).\end{equation} We also allow $N=\infty$, in which case $S_{\infty}=1$. The truncated equations are then\begin{equation}\label{smoothtrunc}u_{t}+Hu_{xx}=S_{N}(S_{N}u\cdot S_{N}u_{x}).\end{equation} Notice that (\ref{smoothtrunc}) conserves the $L^{2}$ mass of $u$; also, if $u$ is a solution of (\ref{smoothtrunc}) whose spatial Fourier transform $\widehat{u}(n)$ is supported in $|n|\leq N$ for one time $t$, then this automatically holds for all time.

\subsection{The main results, and major difficulties} With these preparations, we can now state our main results. The most precise and detailed versions are somewhat technical, and will be postponed to Section \ref{lwp}.
\begin{theorem}[Local well-posedness]\label{main} For any $A>0$, let $T=T(A)=C^{-1}e^{-CA}$; then for the metric space $\mathcal{BO}^{T}$ in Definition \ref{bott} which contains $C^{0}([-T,T]\to Z_{1})$ as a \emph{subset}, we have the following. For any $f$ with $\|f\|_{Z_{1}}\leq A$, there exists a unique function $u\in\mathcal{BO}^{T}$, such that $u(0)=f$, and $u$ verifies (\ref{bo}) on $[-T,T]$ in the sense of distributions (we may define $uu_{x}$ as a distribution for all $u\in\mathcal{BO}^{T}$; for details see Remark \ref{rmk00}). Moreover, if we write $u=\Phi f$, then the map $\Phi$, from the ball $\{f:\|f\|_{Z_{1}}\leq A\}$ to the metric space $\mathcal{BO}^{T}$, will be a Lipschitz extension of the classical solution map for regular data, and its image is bounded away from the zero element in $\mathcal{BO}^{T}$ by $Ce^{CA}$.
\end{theorem}
\begin{theorem}[Measure invariance]\label{main2} Recall the Gibbs measure $\nu$ on $\mathcal{V}$ defined in \cite{Tz10}, which is absolutely continuous with respect to a Wiener measure $\rho$ (see Section \ref{gibbs} for details). There exists a subset $\Sigma$ of $\mathcal{V}$ with full $\rho$ measure such that for each $f\in \Sigma$, the equation (\ref{bo}) has a unique solution $u\in\cap_{T>0}\mathcal{BO}^{T}$ (in the sense described in Remark \ref{rmk00}) with initial data $f$. If we denote $u=\Phi f=(\Phi_{t}f)_{t}$, then for each $t\in\mathbb{R}$ we get a map $f\mapsto \Phi_{t}f$ from $\Sigma$ to itself. These maps form a one parameter group, and each of them keeps invariant the Gibbs measure $\nu$.
\end{theorem}
The analysis of (\ref{bo}) below $L^{2}$ is extremely subtle, even compared to the $L^{2}$ theory. The first step is to use the gauge transform to obtain a more favorable nonlinearity; this already becomes problematic with infinite $L^{2}$ mass. In fact, when we use the gauge $w=\mathbb{P}_{+}\big(ue^{-\frac{\mathrm{i}}{2}\partial_{x}^{-1}u}\big)$ as in \cite{Mo08}, the evolution equation satisfied by $w$ would be
\begin{equation}\label{transgen}(\partial_{t}-\mathrm{i}\partial_{xx})w=\frac{\mathrm{i}}{2}\partial_{x}\mathbb{P}_{+}\big(\partial_{x}^{-1}w\cdot\partial_{x}\mathbb{P}_{-}(\overline{w}\partial_{x}^{-1}w)\big)+\frac{\mathrm{i}}{4}\mathbb{P}_{0}(u^{2})w+GT,\end{equation} where $GT$ represents good terms. Here one can recognize the term $\mathbb{P}_{0}(u^{2})w$ that can be infinite for $u\in Z_{1}$. However, when we further analyze the cubic term above, we find another contribution, namely the ``resonant'' one, which is basically some constant multiple of $\mathbb{P}_{0}(|w|^{2})w$. It then turns out that the coefficients match exactly to give a multiple of $\|w\|_{L^{2}}^{2}-\|\mathbb{P}_{+}u\|_{L^{2}}^{2}$. Since (at least heuristically)\begin{equation}w=\mathbb{P}_{+}\big(ue^{-\frac{\mathrm{i}}{2}\partial_{x}^{-1}u}\big)=\mathbb{P}_{+}u\cdot e^{-\frac{\mathrm{i}}{2}\partial_{x}^{-1}u}+GT,\end{equation} this expression will be finite even if $u$ is only in $Z_{1}$.

The next obstacle to local theory is the failure of standard multilinear $X^{s,b}$ estimates, which play a crucial role in \cite{Mo08}. Recall from (\ref{transgen}) that a typical nonlinearity of the transformed equation looks like \begin{equation}\partial_{x}\mathbb{P}_{+}\big(\partial_{x}^{-1}w\cdot\partial_{x}\mathbb{P}_{-}(\overline{w}\partial_{x}^{-1}w)\big).\end{equation} If the frequency of $\partial_{x}^{-1}w$ appearing in $\overline{w}\partial_{x}^{-1}w$ is low, we may pretend this frequency is zero, obtaining a quadratic nonlinearity which is similar to the KdV equation. In fact, there is a similar failure of bilinear estimates for the KdV equation below $H^{-\frac{1}{2}}$, which is necessary in proving the invariance of white noise. This problem is solved in \cite{Oh10} by considering the second iteration, a strategy already used in \cite{Bo97}. We will use the same method, though the fact that our nonlinearity is only quadratic ``to the first order'' makes the argument a little more involved. 

There is also a special cubic term, omitted in (\ref{transgen}), which involves the function $z=\mathbb{P}_{-}\big(ue^{-\frac{\mathrm{i}}{2}\partial_{x}^{-1}u}\big)$. Recall that it is $w$, not $z$, that satisfies a good evolution equation; therefore $z$ is not supposed to be bounded in any $X^{s,b}$ space where $s$ is close to $0$ and $b$ close to $\frac{1}{2}$ (note $z$ is basically $w$ multiplied by a smoother function, but $X^{s,b}$ spaces are not closed under such multiplications). In \cite{Mo08}, Molinet introduces the space $X^{-1,\frac{7}{8}}$ to accommodate $z$ (he actually considers $u$, but the estimates for $z$ will be the same). In our case, not only do we need (a slightly different version of) this space, but we also have to introduce an atomic space characterizing, roughly speaking, how $z$ is ``shifted'' from $w$; see Section \ref{another} for details.

Passing from local theory to global well-posedness and measure invariance is another challenge. The only known method is to produce finite dimensional truncations such as (\ref{truncation}), exploit the invariance of the (finite dimensional) truncated Gibbs measures, and use a limiting procedure to pass to the original equation. This requires, among other things, uniform estimates for solutions to (\ref{smoothtrunc}). The major difficulty here is that the gauge transform in \cite{Mo08} is now inadequate for eliminating all the bad interactions. To see this, recall that when $v=Mu$ with some function $M$, then
\begin{eqnarray}(\partial_{t}-\mathrm{i}\partial_{xx})v&=&M(\partial_{t}-\mathrm{i}\partial_{xx})u-2\mathrm{i}\partial_{x}M\partial_{x}u+GT\nonumber\\
&=&M\cdot S(Su\cdot Su_{x})-2\mathrm{i}M_{x}u_{x}+GT,\nonumber\end{eqnarray} where we assume $u$ verifies (\ref{smoothtrunc}) with $S=S_{N}$. If $S=1$, then this worst term can be made zero by choosing $M=e^{-\frac{\mathrm{i}}{2}\partial_{x}^{-1}u}$; but it is impossible when $S=S_{N}$ with $N$ finite but large. However, note that we only need to eliminate the ``high-low'' interactions where the factor $Su$ contributes very low frequency and $Su_{x}$ contributes high frequency, and this is indeed possible if we replace multiplication by $M$ with some carefully chosen operator defined from a combination of $S$ (which is a Fourier multiplier) and suitable multiplication operators. See Section \ref{pass} for details. 

Finally, in order for the limiting procedure to work out, we must compare a solution to (\ref{smoothtrunc}) with a solution to (\ref{bo}). Since $\psi(n/N)$ equals $1$ only for $|n|\leq N/2$, the difference will contain some term involving factors like $\mathbb{P}_{\gtrsim N}u$, which does not decay for large $N$ due to the $l^{\infty}$ nature of our norm $Z_{1}$. Nevertheless, these bad terms eventually add up to zero, at least to first order, which is enough for our analysis. Note that the bad terms involve $\psi$ factors which are unique to (\ref{smoothtrunc}) and are not found in (\ref{bo}), this cancellation is really something of a miracle. See Section \ref{gaugetransform2} for details.
\subsection{Plan of this paper} In Section \ref{notation} and Section \ref{linear} we will define the spacetime norms needed in the proof, and prove some linear estimates as well as auxiliary results. In Section \ref{probab} we provide the basic probabilistic arguments. We next introduce the gauge transform for (\ref{smoothtrunc}) and derive the new equations; these will occupy Sections \ref{gaugetransform}-\ref{gaugetransform3}. From Section \ref{begin} to Section \ref{end}, we will prove our main \emph{a priori} estimates. Finally, combing these estimates with the standard probabilistic arguments, we will prove in Section \ref{lwp} (local and almost sure global) well-posedness for (\ref{bo}) and the invariance of the Gibbs measure.

\subsection{Acknowledgements} The author would like to thank Alexandru Ionescu for his encouragement and constructive suggestions regarding the gauge transform; he would also like to thank Tadahiro Oh for introducing the problem and suggesting the treatment used in \cite{Bo97}, and Xuecheng Wang for helpful discussions about the proof of Proposition \ref{initialboot}.

\section{Spacetime norms}\label{notation}
\subsection{The easier norms}\label{spaces}For a function $u$ defined on $\mathbb{T}\times\mathbb{R}$, we define its spacetime Fourier transform $\widehat{u}_{n,\widetilde{\xi}}$ by\begin{equation}u(x,t)=\sum_{n}\int_{\mathbb{R}}\widehat{u}_{n,\widetilde{\xi}}\times e^{\mathrm{i}(nx+\widetilde{\xi}t)}\,\mathrm{d}\widetilde{\xi},\nonumber\end{equation} and denote $\widetilde{u}_{n,\xi}=\widehat{u}_{n,\widetilde{\xi}}:=\widehat{u}_{n,\xi-|n|n}$. Thus we have three ways to represent $u$: $u(t,x)$ as a function of $t$ and $x$, $\widehat{u}_{n,\widetilde{\xi}}$ as a function of $n$ and $\widetilde{\xi}$, and $\widetilde{u}_{n,\xi}$ as a function of $n$ and $\xi$, where the $\xi$ and $\widetilde{\xi}$ are always related by $\widetilde{\xi}=\xi-|n|n$. Since we will be dealing with more than one function, $n$ and $\xi$ may be replaced with other letters possibly with subscipts, say $m_{1}$ or $\beta_{2}$. To simplify the notation, when there is no confusion, we will omit the ``hat'' and ``tilde'' symbols; for example, if we talk about an expression involving $u_{m,\widetilde{\alpha}}$, it will actually mean $\widehat{u}_{m,\widetilde{\alpha}}$. The appearance of functions $f$ defined on $\mathbb{T}$ will not be too frequent, but when they do appear, we will adopt the same convention and write for example $f_{n}$ instead of $\widehat{f}(n)$.

We will need a number of norms in our proof. As a general convention, when we write a norm as $l^{2}L^{1}$, this will mean the $l_{n}^{2}L_{\xi}^{1}$ norm for some $\widetilde{u}$ (which equals the $l_{n}^{2}L_{\widetilde{\xi}}^{1}$ norm for $\widehat{u}$); the meaning of $L^{1}l^{2}$ will thus be clear. The space-time Lebesgue norms will be denoted by $L^{6}L^{6}$ etc. For example, in this notation system the expression $\|u\|_{l_{d\geq 0}^{\infty}l_{\sim 2^{d}}^{p}L^{1}}\nonumber$ actually means\begin{equation}\sup_{d\geq 0}\bigg(\sum_{n\sim 2^{d}}\|\widetilde{u}_{n,\xi}\|_{L_{\xi}^{1}}^{p}\bigg)^{\frac{1}{p}}.\nonumber\end{equation}

Next, observe that up to a constant,\begin{equation}\label{observe}\|u\|_{L^{6}L^{6}}^{6}=\sum_{n\in\mathbb{Z}}\int_{\mathbb{R}}\bigg|\sum_{n_{1}+n_{2}+n_{3}=n}\int_{\widetilde{\xi_{1}}+\widetilde{\xi_{2}}+\widetilde{\xi_{3}}=\widetilde{\xi}}\prod_{i=1}^{3}u_{n_{i},\widetilde{\xi_{i}}}\bigg|^{2}\,\mathrm{d}\widetilde{\xi}.\end{equation}If follows that if $|u_{n,\xi}|\leq v_{n,\xi}$, then $\|u\|_{L^{6}L^{6}}\lesssim\|v\|_{L^{6}L^{6}}$. For any function $u$ we define $\mathfrak{N}u$ by $(\mathfrak{N}u)_{n,\xi}=|u_{n,\xi}|$, then $\|\mathfrak{N}u\|_{L^{6}L^{6}}$ is a norm of $u$. Now we list the norms we will use:
\begin{eqnarray}\label{norm0001}\|u\|_{X_{1}}&=&\big\|\langle n\rangle^{s}\langle\xi\rangle^{b}u\big\|_{l^{p}L^{2}};\\\label{norm0003}\|u\|_{X_{2}}&=&\|\langle n\rangle^{r}u\|_{l_{d\geq 0}^{\infty}l_{\sim 2^{d}}^{p}L^{1}};\\\label{modified}\|u\|_{X_{3}}&=&\|\langle n\rangle^{-\epsilon}\mathfrak{N}u\|_{L^{6}L^{6}};\\\label{norm0004}
\|u\|_{X_{4}}&=&\|\langle n\rangle^{-1}\langle\xi\rangle^{\kappa}u\|_{l^{\gamma}L^{2}};\\\label{norm0005}
\|u\|_{X_{5}}&=&\|u\|_{l_{d\geq 0}^{\infty}L^{q}l_{\sim 2^{d}}^{2}};\\\label{norm0006}
\|u\|_{X_{6}}&=&\|\langle n\rangle^{r}\langle \xi\rangle^{\frac{1}{2}+s^{2}}u\|_{l^{2}L^{2}};\\\label{norm0007}
\|u\|_{X_{7}}&=&\|\langle n\rangle^{r}\langle \xi\rangle^{\frac{1}{8}}u\|_{l_{d\geq 0}^{\infty}l_{\sim 2^{d}}^{p}L^{2}}.\end{eqnarray}
We also recall the norm $Z_{1}$ defined in Section \ref{parameters}, and rewrite it as\begin{equation}\|f\|_{Z_{1}}=\|\langle n\rangle^{r}f\|_{l_{d\geq 0}^{\infty}l_{\sim 2^{d}}^{p}}.
\end{equation}
\subsection{The Another norm}\label{another}
We will need \emph{another} spacetime norm, denoted by $X_{8}$, which is a little tricky to define.

Consider the space of functions $u$ of $(n,\xi)\in\mathbb{Z}\times\mathbb{R}$, normed by \begin{equation}\|u\|_{\Phi}=\|u\|_{L^{q}l^{2}}.\end{equation} The additive group $\mathbb{Z}$ acts on this space by\begin{equation}(\pi_{n_{0}}u)(n,\xi)=u(n+n_{0},\xi+|n+n_{0}|(n+n_{0})-|n|n).\end{equation} If we write\begin{equation}\mathcal{S}:\mathbb{Z}\times\mathbb{R}\to\mathbb{Z}\times\mathbb{R},\,\,(n,\widetilde{\xi})\mapsto(n,\xi)=(n,\widetilde{\xi}+|n|n),\end{equation} then we would have\begin{equation}\pi_{n_{0}}u=u\circ \mathcal{S}\circ T_{n_{0}}\circ \mathcal{S}^{-1},\end{equation} where $T_{n_{0}}:\mathbb{Z}\times\mathbb{R}\to\mathbb{Z}\times\mathbb{R}$ is the translation $(n,\widetilde{\xi})\mapsto(n+n_{0},\widetilde{\xi})$. We then define the atomic $\mathcal{Y}$ norm by\begin{equation}\label{definition}\|u\|_{\mathcal{Y}}=\inf\bigg\{\sum_{i}\langle n_{i}\rangle^{s^{1/2}}|\alpha_{i}|:u=\sum_{i}\alpha_{i}u_{i},\|\pi_{n_{i}}u_{i}\|_{\Phi}\leq 1\bigg\}.\end{equation} The $X_{8}$ norm is thus defined by\begin{equation}\|u\|_{X_{8}}=\sup_{d\geq 0}\|\mathbb{P}_{\sim 2^{d}}u\|_{\mathcal{Y}}:=\sup_{d\geq 0}\|\mathbb{P}_{\sim 2^{d}}\widetilde{u}\|_{\mathcal{Y}},\end{equation} where the last inequality is due to our convention.
\begin{remark}In (\ref{definition}), the convergence takes place in a suitable weighted $L_{n,\xi}^{1}$ space. Therefore, when $v$ is rapidly decaying in $n$ and $\xi$ (for example, $|v|\lesssim (|n|+|\xi|+1)^{-100}$ will suffice), the sum $\sum_{i}\alpha_{i}(u_{i},v)$ will converge absolutely to $(u,v)$, provided the $\sum_{i}\langle n_{i}\rangle^{s^{1/2}}|\alpha_{i}|$ is finite, where $(u,v)$ denotes (up to constant) the standard pairing\begin{equation}\label{pairing}(u,v)=\int_{\mathbb{R}\times\mathbb{T}}u(t,x)\overline{v(t,x)}\,\mathrm{d}t\mathrm{d}x=\sum_{n}\int_{\mathbb{R}}u_{n,\xi}\overline{v_{n,\xi}}\,\mathrm{d}\xi.\end{equation}
\end{remark}
\subsection{The space in which we work}\label{mainspace}
Define \begin{eqnarray}\|u\|_{Y_{1}}&=&\|u\|_{X_{1}}+\|u\|_{X_{2}}+\|u\|_{X_{4}}+\|u\|_{X_{5}}+\|u\|_{X_{7}};\\\|u\|_{Y_{2}}&=&\|u\|_{X_{2}}+\|u\|_{X_{3}}+\|u\|_{X_{4}}+\|u\|_{X_{8}}.\end{eqnarray}
Moreover, for each space $\mathcal{Z}$ (which can be $Y_{1}$, $Y_{2}$ or any other space) we define\footnote[1]{Note that any function in our space $\mathcal{Z}$ will actually be a bounded continuous function of $t$ with value in (say) $H^{-90}(\mathbb{T})$.} \begin{equation}\|u\|_{\mathcal{Z}^{T}}=\inf\big\{\|v\|_{\mathcal{Z}}:v|_{[-T,T]}=u|_{[-T,T]}\big\}.\end{equation} This $[-T,T]$ may also be replaced by any interval $I$.

The main spacetime norms we shall use in the whole bootstrap argument are $Y_{1}^{T}$ and $Y_{2}^{T}$, while other norms may be introduced whenever necessary.

\section{Linear estimates, and more}\label{linear}
Here we shall prove our main linear estimates, as well as some auxiliary results.
\begin{proposition}[Strichartz estimates]\label{stri} For any function $u$, we have \begin{equation}\label{strichartz}\|u\|_{L^{k}L^{k}}\lesssim\|\langle n\rangle^{\sigma}\langle\xi\rangle^{\beta}u\|_{l^{2}L^{2}},\end{equation} provided that the parameters are set as\begin{equation}(k,\sigma,\beta)\in\big\{(2,0,0),\big(4,0,\frac{3}{8}\big),\big(6,s^{5},\frac{1}{2}+s^{5}\big),\big(\infty,\frac{1}{2}+s^{5},\frac{1}{2}+s^{5}\big)\big\}.\end{equation}
\end{proposition}
\begin{proof} When $(k,\sigma,\beta)=(2,0,0)$, the inequality (\ref{strichartz}) is simply Plancherel; when $(k,\sigma,\beta)=(\infty,\frac{1}{2}+s^{5},\frac{1}{2}+s^{5})$, this can also be easily proved by combining Hausdorff-Young and H\"{o}lder. When $(k,\sigma,\beta)=(4,0,\frac{3}{8})$, the inequality reduces, after separating positive and negative frequencies and using time inversion, to the $L^{4}$ Strichartz estimate for the linear Schr\"{o}dinger equation on $\mathbb{T}$ which is well-known; see for example \cite{Ta06}, Proposition 2.13.

Now we assume $k=6$, $\sigma=s^{5}$ and $\beta=\frac{1}{2}+s^{5}$. Again by separating positive and negative frequencies and using time inversion, we only need to consider the case for the Schr\"{o}dinger semigroup, thus \emph{in this proof} our convention will change to $\xi=\widetilde{\xi}+n^{2}$. Now for any function $u$ with the right hand side of (\ref{strichartz}) not exceeding $1$, we write $v_{n,\xi}=\langle n\rangle^{s^{5}}\langle\xi\rangle^{\frac{1}{2}+s^{5}}u_{n,\xi}$ using our (different) convention, and compute up to a constant that
\begin{equation}(u^{3})_{n, \widetilde{\xi}}=\sum_{n_{1}+n_{2}+n_{3}=n}\prod_{i=1}^{3}\langle n_{i}\rangle^{-s^{5}}(f_{n_{1}}\ast f_{n_{2}}\ast f_{n_{3}})_{\widetilde{\xi}+n_{1}^{2}+n_{2}^{2}+n_{3}^{2}},
\end{equation} where\begin{equation}(f_{n})_{\xi}=\langle\xi\rangle^{-\frac{1}{2}-s^{5}}v_{n,\xi}.\end{equation} By our assumption we have $\|\langle\xi\rangle^{\frac{1}{2}+s^{5}}f_{n_{i}}\|_{L^{2}}\lesssim A_{n_{i}}$, where $\{A_{n}\}$ is some sequence satisfying $\|A\|_{l^{2}}\lesssim 1$. By (the Fourier version of) the product estimate for $H^{\sigma}(\mathbb{R})$ spaces, we deduce that \begin{equation}(f_{n_{1}}*f_{n_{2}}*f_{n_{3}})_{\eta}=\langle\eta\rangle^{-\frac{1}{2}-s^{5}}(g_{n_{1}n_{2}n_{3}})_{\eta};\,\,\,\,\,\,\|g_{n_{1}n_{2}n_{3}}\|_{L^{2}}\lesssim A_{n_{1}}A_{n_{2}}A_{n_{3}}.\end{equation} Therefore we can estimate
\begin{eqnarray}|(u^{3})_{n, \widetilde{\xi}}|^{2}&\lesssim &\bigg(\sum_{n_{1}+n_{2}+n_{3}=n}\prod_{i=1}^{3}\langle n_{i}\rangle^{-2s^{5}}\cdot\langle\widetilde{\xi}+n_{1}^{2}+n_{2}^{2}+n_{3}^{2}\rangle^{-1-2s^{5}}\bigg)\times\nonumber\\
&\times&\bigg(\sum_{n_{1}+n_{2}+n_{3}=n}|(g_{n_{1}n_{2}n_{3}})_{\widetilde{\xi}+n_{1}^{2}+n_{2}^{2}+n_{3}^{2}}|^{2}\bigg).\nonumber\end{eqnarray} Now to finish the proof if will suffice to  show\begin{equation}\sum_{n_{1}+n_{2}+n_{3}=n}\prod_{i=1}^{3}\langle n_{i}\rangle^{-2s^{5}}\langle\widetilde{\xi}+n_{1}^{2}+n_{2}^{2}+n_{3}^{2}\rangle^{-1-2s^{5}}\leq C,\end{equation} when $n$ and $\widetilde{\xi}$ are fixed. Now suppose the maximum (in absolute value) of $n_{i}$ and $\Xi=\widetilde{\xi}+n_{1}^{2}+n_{2}^{2}+n_{3}^{2}$ be comparable to $2^{d}$, and $\Xi\sim 2^{d'}$, then the summand is at most $2^{-d'-s^{6}(d+d')}$, so it will suffice to show that there are at most $2^{d'+s^{7}d}$ choices for $(n_{1},n_{2},n_{3})$. Since their can be at most $2^{d'}$ possibilities for $n_{1}^{2}+n_{2}^{2}+n_{3}^{2}$, we only need to show that there are at most $2^{s^{7}d}$ choices for $(n_{1},n_{2},n_{3})$ if we require $|n_{i}|\lesssim 2^{d}$, and fix $n_{1}+n_{2}+n_{3}=n$ and $n_{1}^{2}+n_{2}^{2}+n_{3}^{2}$. But then $m_{i}=3n_{i}-n$ will be integers for $i\in\{1,2\}$, and $m_{1}^{2}+m_{1}m_{2}+m_{2}^{2}$ will be a fixed integer not exceeding $C2^{5d}$. The result then follows from the divisor estimate for the ring $\mathbb{Z}\big[e^{\frac{2\pi\mathrm{i}}{3}}\big]$.
\end{proof}
By Proposition \ref{stri} and interpolation, we get a series of $L^{k}L^{k}$ Strichartz estimates for all $2\leq k\leq\infty$. It is these that we will actually use in the proof; we will not care too much about the exact numerology because there will be enough room whenever we use these estimates.
\begin{proposition}[Relations between norms]\label{relattt} We have the following inequalities:
\begin{equation}\label{relation1}\|u\|_{X_{3}}\lesssim\|u\|_{X_{1}}+\|u\|_{X_{4}},\,\,\|u\|_{X_{8}}\lesssim\|u\|_{X_{5}};\end{equation}
\begin{equation}\label{relation2}\|u\|_{X_{1}}+\|u\|_{X_{2}}+\|u\|_{X_{5}}+\|u\|_{X_{7}}\lesssim\|u\|_{X_{6}}.\end{equation} Note that this in particular implies $\|u\|_{X_{j}}\lesssim \|u\|_{Y_{1}}$ if $1\leq j\leq 8$ and $j\neq 6$.
\end{proposition}
\begin{proof} By Proposition \ref{stri} and hierarchy (\ref{hierarchy}) we know that\begin{equation}\label{loose}\|u\|_{X_{3}}\lesssim\|\langle n\rangle^{-\frac{\epsilon}{2}}\langle\xi\rangle^{\frac{1}{2}+s^{5}}u\|_{l^{2}L^{2}}.\end{equation} Comparing this with the definition of $X_{1}$ and $X_{4}$, noticing that $\gamma<2$ and by (\ref{hierarchy}) and H\"{o}lder,\begin{equation}\|u\|_{X_{1}}\gtrsim\|\langle n\rangle^{-\frac{\epsilon}{4}}\langle \xi\rangle^{b}\|_{l^{2}L^{2}},\nonumber\end{equation} we will be able to prove the first inequality in (\ref{relation1}) provided we can show\begin{equation}\langle n\rangle^{-\frac{\epsilon}{2}}\langle\xi\rangle^{\frac{1}{2}+s^{5}}\lesssim\langle n\rangle^{-1}\langle\xi\rangle^{\kappa}+\langle n\rangle^{-\frac{\epsilon}{4}}\langle\xi\rangle^{b}.\end{equation} But this is clear since by (\ref{hierarchy}), the left hand side is controlled by the first term on the right hand side if $\langle\xi\rangle\geq\langle n\rangle^{100}$, and by the second term if $\langle\xi\rangle<\langle n\rangle^{100}$. The second inequality in (\ref{relation1}) is also easy, since we only need to prove $\|u\|_{\mathcal{Y}}\lesssim\|u\|_{L^{q}l^{2}}$, which is a direct consequence of the definition (\ref{definition}), if we choose to have only one term (with the corresponding $n_{i}=0$) in the proposed atomic decomposition.

Now let us prove (\ref{relation2}). The $X_{1}$ norm is controlled by $X_{6}$ norm because $s<r$, $b<\frac{1}{2}+s^{2}$, and $2<p$. For basically the same reason we can use H\"{o}lder to show $\|u\|_{X_{2}}+\|u\|_{X_{7}}\lesssim\|u\|_{X_{6}}$. Finally, to prove $\|u\|_{X_{5}}\lesssim\|u\|_{X_{6}}$, we only need to show that $\|g_{\xi}\|_{L^{q}}\lesssim\|\langle\xi\rangle^{\frac{1}{2}+s^{2}}g_{\xi}\|_{L^{2}}$, but this again follows from H\"{o}lder since $q>1$.
\end{proof}
Next, we introduce the (cut-off) Duhamel operator $\mathcal{E}$ defined by \begin{equation}\label{duhamel}\mathcal{E}u(t,x)=\chi(t)\int_{0}^{t}\chi(t')(e^{-(t-t')H\partial_{xx}}u(t'))(x)\,\mathrm{d}t',\end{equation} where $\chi(t)$ is a cutoff function (compactly supported and equals $1$ in a neighborhood of $0$) in $t$. Here and below we shall use many such functions, but unless really necessary, we will not distinguish them and will denote them all by $\chi$ (for example, we write $\chi^{2}=\chi$). We shall summarize the required linear estimates for $\mathcal{E}$ in Proposition \ref{linearestimate2} below, but before doing so, we need to introduce two more norms, namely:
\begin{eqnarray}\|u\|_{X_{9}}&=&\|\langle n\rangle^{r}u\|_{l_{d\geq 0}^{\infty}L^{q'}l_{\sim 2^{d}}^{p}},\\
\|u\|_{X_{10}}&=&\|\langle n\rangle^{r}\langle \xi\rangle^{-\frac{1}{8}}u\|_{l_{d\geq 0}^{\infty}L^{\tau}l_{\sim 2^{d}}^{p}}.\end{eqnarray}
\begin{lemma}\label{linearestimate}
Suppose $v(t,x)=\mathcal{E}u(t,x)$, then with constants $c_{j}$,\begin{equation}\label{computation}v_{n,\xi}=c_{1}\big(\widehat{\chi}*(\eta^{-1}(\widehat{\chi}*u_{n,*})_{\eta})\big)_{\xi}+c_{2}\bigg(\int_{\mathbb{R}}\frac{(\widehat{\chi}*u_{n,*})_{\eta}}{\eta}\,\mathrm{d}\eta\bigg)\cdot\widehat{\chi}_{\xi}.\end{equation} Here the $\frac{1}{\eta}$ is to be understood as the principal value distribution. This operator obeys the following basic estimates, valid for all $\sigma,\beta\in\mathbb{R}$ and $1\leq h,k\leq\infty$:
\begin{eqnarray}\label{001}\|\langle n\rangle^{\sigma}\langle\xi\rangle^{\beta}\mathcal{E}u\|_{L^{h}l^{k}}&\lesssim&\|\langle n\rangle^{\sigma}\langle\xi\rangle^{\beta-1}u\|_{L^{h}l^{k}}+\|\langle n\rangle^{\sigma}\langle\xi\rangle^{-1}u\|_{l^{k}L^{1}};\\
\label{002}\|\langle n\rangle^{\sigma}\langle \xi\rangle^{\beta}\mathcal{E}u\|_{l^{k}L^{h}}&\lesssim &\|\langle n\rangle^{\sigma}\langle\xi\rangle^{\beta-1}u\|_{l^{k}L^{h}}+\|\langle n\rangle^{\sigma}\langle\xi\rangle^{-1}u\|_{l^{k}L^{1}}.\end{eqnarray} Note the reversed order of norms in the second term on the right hand side of (\ref{001}). If moreover $\beta>1-\frac{1}{h}$, we can remove the $l^{k}L^{1}$ norms. Finally, by commuting with $\mathbb{P}$ projections, we get similar estimates for norms like $X_{2}$ and $X_{5}$.
\end{lemma}\begin{proof}The computation (\ref{computation}) is basically done in \cite{Bo97}. In our case, noticing that multiplication by $\chi(t)$ corresponds to convolution by $\widehat{\chi}$ on the ``tilde'' side, we only need to express the Fourier transform of $\int_{0}^{t}u(t')\mathrm{d}t'$ (which is exactly the Duhamel operator on the ``tilde'' side) in terms of $u(t)$. We compute\begin{equation}\int_{0}^{t}u(t')\,\mathrm{d}t'=\frac{1}{2}u\ast\mathrm{sgn}(t)+\frac{1}{2}\int_{\mathbb{R}}u(t')\mathrm{sgn}(t')\,\mathrm{d}t'.\end{equation} On the Fourier side, these two terms gives exactly the two terms in (\ref{computation}) after another convolution with $\widehat{\chi}$.

We will only prove (\ref{001}), since the proof of (\ref{002}) will be basically the same; also notice that if $\beta>1-\frac{1}{h}$, then \begin{equation}\|w\|_{l^{k}L^{1}}\lesssim\|w\|_{L^{1}l^{k}}\lesssim\min\big\{\|\langle\xi\rangle^{\beta}w\|_{l^{k}L^{h}},\|\langle\xi\rangle^{\beta}w\|_{L^{h}l^{k}}\big\}\nonumber\end{equation} for $w_{n,\xi}=\langle n\rangle^{\sigma}\langle\xi\rangle^{-1}u_{n,\xi}$, by H\"{o}lder. Now to prove (\ref{001}), we first consider the second term of (\ref{computation}). Due to its structure, we only need to prove for any function $z=z_{\xi}$ that\begin{equation}\bigg|\int_{\mathbb{R}}\frac{(z*\widehat{\chi})_{\eta}}{\eta}\,\mathrm{d}\eta\bigg|\lesssim\|\langle\xi\rangle^{-1}z\|_{L^{1}}.\end{equation} By considering $|\eta|\gtrsim1$ and $|\eta|\lesssim 1$ separately and using the cancelation coming from the $\frac{1}{\eta}$ factor, we can control the left hand side by $\|\langle\eta\rangle^{-1}(z*\widehat{\chi})_{\eta}\|_{L^{1}}$ (which is easily bounded by the right hand side of (\ref{001})), plus another term bounded by $\|\langle\eta\rangle^{-1}\partial_{\eta}(z*\widehat{\chi})\|_{L^{\infty}}$. If we shift the derivative to $\widehat{\chi}$ to get rid of it, we can again bound this expression by the right hand side of (\ref{001}).

Next, we consider the first term of (\ref{computation}). Again we consider the terms with $|\eta|\gtrsim 1$ and $|\eta|\lesssim1$ separately (by introducing a smooth, even cutoff $\phi_{\eta}$, say). The part where $|\eta|\gtrsim 1$ is easy, since convolution by $\widehat{\chi}_{\xi}$ is bounded on any weighted mixed norm Lebesgue space we have here, and $\frac{1}{\eta}$ is comparable to $\langle\eta\rangle^{-1}$ when restricted to the region $|\eta|\gtrsim 1$. Now for the region $|\eta|\lesssim 1$, we can actually prove for $y=y_{\xi}$ and arbitrary $K>0$ that\begin{equation}\bigg|\bigg(\widehat{\chi}*\bigg(\frac{\phi_{\eta}}{\eta}(\widehat{\chi}*y)_{\eta}\bigg)\bigg)_{\tau}\bigg|\lesssim\langle\tau\rangle^{-K}\|\langle\xi\rangle^{-K}y\|_{L^{1}},\end{equation} which easily implies our inequality. To prove this, let $\widehat{\chi}*y=z$, and compute\begin{eqnarray}\bigg(\widehat{\chi}*\bigg(\frac{\phi(\eta)}{\eta}z_{\eta}\bigg)\bigg)_{\tau}&=&\int_{|\eta|\lesssim 1}\widehat{\chi}_{\tau}\frac{\phi_{\eta}z_{\eta}-z_{0}}{\eta}\,\mathrm{d}\eta\nonumber\\
&+&\int_{|\eta|\lesssim 1}\frac{\widehat{\chi}_{\tau-\eta}-\widehat{\chi}_{\tau}}{\eta}\phi_{\eta}z_{\eta}\,\mathrm{d}\eta.\nonumber\end{eqnarray} From this we can readily recognize a decay of $\langle\tau\rangle^{-K}$, and it will suffice to prove that $\sup_{|\eta|\lesssim 1}|z_{\eta}|\lesssim \|\langle\xi\rangle^{-K}y\|_{L^{1}}$, but this will be clear from the definition of $z$.
 \end{proof}
 \begin{proposition}\label{linearestimate2}
 We have the following estimates:
\begin{equation}\label{linnn1}\|\mathcal{E}u\|_{X_{6}}\lesssim\|\langle \xi\rangle^{-1}u\|_{X_{6}},\|\mathcal{E}u\|_{X_{4}}\lesssim\|\langle \xi\rangle^{-1} u\|_{X_{4}};\end{equation}
\begin{equation}\label{linnn2}\|\mathcal{E}u\|_{X_{1}}+\|\mathcal{E}u\|_{X_{2}}\lesssim\|\langle \xi\rangle ^{-1}u\|_{X_{1}}+\|\langle \xi\rangle ^{-1}u\|_{X_{2}}\lesssim\|u\|_{X_{10}};\end{equation}
\begin{equation}\label{linnn3}
\|\mathcal{E}u\|_{X_{7}}\lesssim\|u\|_{X_{10}}\lesssim\|u\|_{X_{9}},\|\mathcal{E}u\|_{X_{5}}\lesssim\|u\|_{X_{10}}.
\end{equation}
Moreover, suppose $u$ is such that $u_{n,\xi}$ is supported in $\{(n,\xi):n\sim 2^{d}, \xi\gtrsim 2^{d}\}$ for some $d$, then \begin{equation}\label{linnn4}\|\mathcal{E}u\|_{X_{5}}+\|\mathcal{E}u\|_{X_{7}}\lesssim\|\langle \xi\rangle ^{-1}u\|_{X_{1}}+\|\langle \xi\rangle ^{-1}u\|_{X_{2}}.\end{equation}
Finally notice that all these estimates naturally imply the dual versions about the boundedness of $\mathcal{E}'$.
 \end{proposition}
 \begin{proof}
 By checking the numerology, we see that (\ref{linnn1}) is a direct consequence of Lemma \ref{linearestimate}. To prove the first inequality in (\ref{linnn2}), we use Lemma \ref{linearestimate} to conclude
 \begin{equation}\|\mathcal{E}u\|_{X_{1}}+\|\mathcal{E}u\|_{X_{2}}\lesssim\|\langle \xi\rangle^{-1} u\|_{X_{1}}+\|\langle \xi\rangle^{-1} u\|_{X_{2}}+\|\langle n\rangle^{s}\langle \xi\rangle^{-1}u\|_{l^{p}L^{1}},\end{equation} and note that the last term can be controlled by $\|\langle \xi\rangle^{-1} u\|_{X_{2}}$ also. To prove that $\|\langle \xi\rangle^{-1} u\|_{X_{2}}\lesssim\|u\|_{X_{10}}$, one first commute with $\mathbb{P}_{\sim 2^{d}}$, then control the $l^{p}L^{1}$ norm by the $L^{1}l^{p}$ norm, then use H\"{o}lder (note the hierarchy (\ref{hierarchy})). To prove that $\|\langle \xi\rangle^{-1} u\|_{X_{1}}\lesssim\|u\|_{X_{10}}$, one first replace the $\|\langle n\rangle^{s}*\|_{l^{p}}$ norm by the larger $\|\langle n\rangle^{r}*\|_{l_{d\geq 0}l_{n\sim 2^{d}}^{p}}$ norm, then commute with $\mathbb{P}_{\sim 2^{d}}$, and control the $l^{p}L^{2}$ norm by the $L^{2}l^{p}$ norm and use H\"{o}lder again. Along the same lines, we have
 \begin{equation}\label{x7norm}\|\mathcal{E}u\|_{X_{7}}\lesssim\|\langle \xi\rangle^{-1} u\|_{X_{2}}+\|\langle \xi\rangle^{-1} u\|_{X_{7}}\end{equation} as well as \begin{equation}\label{x5norm}\|\mathcal{E}u\|_{X_{5}}\lesssim\|\langle\xi\rangle^{-1}u\|_{l_{d\geq 0}^{\infty}l_{\sim 2^{d}}^{2}L^{1}}+\|\langle\xi\rangle^{-1}u\|_{X_{5}},\end{equation}where the first term on the right hand side of (\ref{x5norm}) is bounded by $\|\langle \xi\rangle^{-1} u\|_{X_{2}}$, and the second terms on the right hand side of both (\ref{x7norm}) and (\ref{x5norm}) are bounded by the $X_{10}$ norm, by controlling the $l^{p}L^{2}$ norm by the $L^{2}l^{p}$ norm and using H\"{o}lder. Also we have $\|u\|_{X_{10}}\lesssim\|u\|_{X_{9}}$ by H\"{o}lder. This proves (\ref{linnn3}).
 
 Let us now prove (\ref{linnn4}). For the $X_{7}$ norm we use (\ref{x7norm}), and the support condition will easily allow us to control the second term on the right hand side of (\ref{x7norm}) by $\|\langle \xi\rangle^{-1} u\|_{X_{1}}$. For the $X_{5}$ norm, we only need to bound the second term on the right hand side of (\ref{x5norm}) by $\|\langle \xi\rangle^{-1} u\|_{X_{1}}$. Since we can restrict to $|n|\sim 2^{d}$ and $|\xi|\gtrsim 2^{d}$, we can bound this term by
 \begin{eqnarray}\|\langle\xi\rangle^{-1}u\|_{L^{q}l^{2}}&\lesssim &\|\langle\xi\rangle^{\sigma}u\|_{L^{2}l^{2}}=\|\langle\xi\rangle^{\sigma}u\|_{l^{2}L^{2}}\nonumber\\
 &\lesssim &2^{(\sigma-b+1)d}\|\langle\xi\rangle^{b-1}u\|_{l^{2}L^{2}}\lesssim 2^{(\sigma+\sigma'-b+1)d}\|\langle \xi\rangle^{-1} u\|_{X_{1}},\nonumber
 \end{eqnarray} where $\sigma'=\frac{1}{2}-\frac{1}{p}-s>0$, $\sigma=-\frac{1}{2}-\frac{1}{2q'}$ so that $\sigma+\sigma'-b+1<0$ by (\ref{hierarchy}).
 \end{proof}
 Next we will prove two auxiliary results about our norms $Y_{j}$ and $Y_{j}^{T}$, which are defined in Section \ref{mainspace}. They will be used to validate our main bootstrap argument.
 \begin{proposition}\label{auxi}Suppose $j\in\{1,2\}$, and $u=u(t,x)\in Y_{j}$ is a function that vanishes at $t=0$, then with a time cutoff $\chi$ (recall our convention about such functions) we have, uniformly in $T\lesssim 1$, that \begin{equation}\label{boundd}\|\chi(T^{-1}t)u\|_{Y_{j}}\lesssim \|u\|_{Y_{j}}.\end{equation} If $u$ is smooth, then we also have \begin{equation}\label{limit}\lim_{T\to 0}\|\chi(T^{-1}t)u\|_{Y_{j}}=0.\end{equation}
 \end{proposition}
 \begin{proof}We first assume $u\in Y_{j}$ and $u(0)=0$. We may also assume that $u$ is supported in $|t|\lesssim 1$. Since on the ``hat'' or ``tilde'' side multiplication by $\chi(T^{-1}t)$ is just convolution by $T\widehat{\chi}_{T\xi}$, we need to prove the uniform boundedness of these operators on spaces involved in the definition of $Y_{j}$, as well as the corresponding limit result when $u$ is smooth. The bound in $X_{3}$ is obtained by decomposing this convolution into translations (which preserve the $X_{3}$ norm) and integrating them using the boundedness of $L^{1}$ norm of $T\widehat{\chi}_{T\xi}$. The bound in $X_{8}$ follows from the bound in $\mathcal{Y}$, which is valid because this convolution does not increase the $\Phi$  (or $L^{q}l^{2}$) norm, and commutes with the action described in Section \ref{another}; the bound in $X_{2}$ and $X_{5}$ are shown in the same way. The remaining bounds will follow if we can bound this convolution in weighted norms $\|\langle\xi\rangle^{\sigma}y\|_{L^{2}}$, where $0\leq\sigma<1$, for complex valued functions $y_{\xi}$ such that $\int_{\mathbb{R}}y_{\xi}\mathrm{d}\xi=0$. Namely, we need to prove \begin{equation}\label{more2}\|\langle\eta\rangle^{\sigma}(y*T\widehat{\chi}_{T\xi})_{\eta}\|_{L^{2}}\lesssim\|\langle\xi\rangle^{\sigma}y_{\xi}\|_{L^{2}}.\end{equation} Also, by Proposition \ref{relattt} we can control $Y_{1}$ and $Y_{2}$ norms by $X_{4}$ and $X_{6}$. Thus in order to prove (\ref{limit}), we only need to prove that the left hand side of (\ref{more2}) actually tends to zero when $T\to 0$, for any fixed Schwartz $y$ with integral zero. By taking inverse Fourier transform, the problem an be reduced to proving \begin{equation}\|\chi(T^{-1}t)u\|_{H^{\sigma}}\lesssim\|u\|_{H^{\sigma}},\end{equation} and the limit\begin{equation}\lim_{T\to 0}\|\chi(T^{-1}t)u(t)\|_{H^{\sigma}}=0,\end{equation}for $T\lesssim 1$ and $u\in C_{c}^{\infty}$ such that $u(0)=0$. But these are proved, in a slightly different but equivalent setting, in \cite{De10}, Lemma 2.8.
 \end{proof}
 \begin{proposition}\label{initialboot}Suppose $u=u(t,x)$ is a smooth function defined on $\mathbb{R}\times\mathbb{T}$, then for $j\in\{1,2\}$, the function $T>0\mapsto \mathcal{M}(T)=\|u\|_{Y_{j}^{T}}$ satisfies $\mathcal{M}(T+0)\leq C\mathcal{M}(T-0)$ for all $0<T\lesssim 1$, and also $\mathcal{M}(0+)\leq C\|u(0)\|_{Z_{1}}$.
 \end{proposition}
 \begin{proof}First we prove the estimate about $\mathcal{M}(0+)$. Let $u(0)=f$ and $v(t,x)=u(t,x)-e^{-tH\partial_{xx}}f(x)$. Since $v$ is smooth and $v(0)=0$, we have by Proposition \ref{auxi} that $\lim_{T\to 0}\|v\|_{Y_{j}^{T}}=0$. It then suffices to prove that for some cutoff $\chi(t)$, we have $\|\chi(t)e^{-tH\partial_{xx}f}\|_{Y_{j}}\lesssim \|f\|_{Z_{1}}$. Note that on the ``tilde'' side, the function $\chi(t)e^{-tH\partial_{xx}f}$ simply becomes $\widehat{\chi}_{\xi}f_{n}$; thus this inequality is basically trivial if we take into account that the $Z_{1}$ norm is stronger than the norm $\|\langle n\rangle^{-1}f\|_{L^{\gamma}}$, and the norm $\|f\|_{l_{d\geq 0}^{\infty}l_{\sim 2^{d}}^{2}}$.
 
 Next, we shall prove that $\mathcal{M}(T+0)\lesssim \mathcal{M}(T)$ for $0<T\lesssim 1$. Namely, suppose $u$ is a smooth function, $0<T\lesssim 1$ is such that $\|u\|_{Y_{j}^{T}}\leq 1$, we want to prove for some $T'>T$ that $\|u\|_{Y_{j}^{T'}}\lesssim 1$. Actually we only need to prove $\|u\|_{Y_{j}^{[-T,T']}}\lesssim 1$, since we can use the same argument to move the left point also. Now, due to the presence of $X_{2}$ norm in the definitions of both $Y_{j}$, our assumption implies $\|u(T)\|_{Z_{1}}\lesssim 1$, therefore by what we just proved, $u_{1}=e^{-(t-T)H\partial_{xx}}u(T)$ verifies the estimate $\|u_{1}\|_{Y_{j}^{T'}}\lesssim 1$ for all $T<T'\lesssim 1$. Thus we only need to bound $\|u_{2}\|_{Y_{j}^{[-T,T']}}$ for some $T'>T$ and $u_{2}=u-u_{1}$. Note $u_{2}(T)=0$, by choosing $\delta$ small enough we can produce a function $v$ coinciding with $u_{2}$ on $[T-10\delta,T+10\delta]$ such that $\|v\|_{Y_{j}}\lesssim 1$ by Proposition \ref{auxi}. Also since $\|u_{2}\|_{Y_{j}^{T}}\lesssim 1$, we may choose a function $w$ coinciding with $u_{2}$ on $[-T,T]$ such that $\|w\|_{Y_{j}}\lesssim 1$. Note $v(T)=w(T)=0$. Next, choose a function $\psi_{3}\in C^{\infty}$ supported on $[-9,10]$ that equals $1$ on $[-1,9]$. Define \begin{equation}u_{3}(t)=(1-\psi_{3}(\delta^{-1}(t-T)))w(t)+\psi_{3}(\delta^{-1}(t-T))v(t).\end{equation} Then we can verify that $u_{3}=u_{2}$ on $[-T,T']$ with $T'=T+9\delta$, and by Proposition \ref{auxi} we have $\|u_{3}\|_{Y_{j}}\lesssim 1$, as desired.
 
 Finally let us prove that $\mathcal{M}(T)\lesssim \mathcal{M}(T-0)$ for all $0<T\lesssim 1$. Suppose $T_{k}\uparrow T$, and we can find $u^{k}$ coinciding with $u$ on $[-T_{k},T_{k}]$ such that $\|u^{k}\|_{Y_{j}}\leq 1$. Since $T\lesssim 1$, we may assume $u^{k}$ are supported in $|t|\lesssim 1$. By the uniform boundedness in $X_{4}$ norm, and the fact that on the ``tilde'' side each $u^{k}$ equals itself convolved with some $\widehat{\chi}_{\xi}$, we conclude that $(u^{k})_{n,\xi}$ has second order $\xi$-derivatives bounded by (say) $\langle n\rangle^{10}$. We therefore extract a subsequence so that $\{u^{k}\}$, viewed as a sequence of maps from $\mathbb{R}_{\xi}$ to some weighted $l_{n}^{2}$ space, converges uniformly in any $|\xi|\leq R$. In particular this implies the convergence as spacetime distributions; thus the limit, denoted by $u^{*}$, must coincide with $u$ on $[-T,T]$. It therefore suffices to prove $\|u^{*}\|_{Y_{j}}\lesssim 1$. The bounds for $X_{1}$, $X_{2}$, $X_{4}$ and $X_{7}$ norms immediately follows from distributional convergence; for $X_{3}$, note that the $|(u^{k})_{n,\xi}|$ also converge uniformly to $|u_{n,\xi}|$ in any $|\xi|\leq R$ for any fixed $n$, thus $\mathfrak{N}u^{k}$ (recall Section \ref{spaces} for definition) will converge to $\mathfrak{N}u^{*}$ as spacetime distributions, therefore the $X_{3}$ norm of $u^{*}$ will also be bounded by $O(1)$.
 
 It remains to prove the bound of $X_{8}$ norm for $u^{*}$. By commuting with $\mathbb{P}_{\sim 2^{d}}$, we may assume that $\|u^{k}\|_{\mathcal{Y}}\leq 1$. For any bounded function $v=v_{n,\xi}$ with compact $(n,\xi)$-support, we have $(u^{k},v)\to (u^{*},v)$ with the standard pairing $(u,v)$ as in (\ref{pairing}). By the definition of the $\mathcal{Y}$ norm we can easily see that\begin{equation}|(u^{k},v)|\leq\|u^{k}\|_{\mathcal{Y}}\cdot\sup_{n_{0}\in\mathbb{Z}}\langle n_{0}\rangle^{-s^{2}}\|\pi_{n_{0}}v\|_{L^{q'}l^{2}}\leq\sup_{n_{0}\in\mathbb{Z}}\langle n_{0}\rangle^{-s^{2}}\|\pi_{n_{0}}v\|_{L^{q'}l^{2}}.\end{equation} If we denote the right hand side by $\|v\|_{\mathcal{Z}}$, we then have $|(u^{*},v)|\leq\|v\|_{\mathcal{Z}}$ for $v$ with compact $(n,\xi)$-support. Now consider any $v$ with $\|v\|_{\mathcal{Z}}\leq 1$ (so in particular $v\in L^{q'}l^{2}$). We produce a sequence $v^{R}=v\cdot\mathbf{1}_{\{|v|+|n|+|\xi|\leq R\}}$ so that $\|v^{R}\|_{\mathcal{Z}}\leq1$, and $v^{R}\to v$ in $L^{q'}l^{2}$, thus $(u^{*},v^{R})\to (u^{*},v)$ (notice that $u^{*}\in X_{4}$ and is supported in some $|n|\sim 2^{d}$, thus we have $u^{*}\in L^{q}l^{2}$). This implies $|(u^{*},v)|\leq 1$ for all $v$ such that $\|v\|_{\mathcal{Z}}\leq 1$. Since \emph{a priori} we have $u^{*}\in L^{q}l^{2}\subset\mathcal{Y}$, and it is easily checked that $\mathcal{Y}$ is a Banach space, we may then invoke the Hahn-Banach theorem to conclude $\|u^{*}\|_{\mathcal{Y}}\leq 1$, provided that we can identify the dual space of $\mathcal{Y}$ with $\mathcal{Z}$. Now clearly each element in $\mathcal{Z}$ gives a linear functional on $\mathcal{Y}$ whose norm equals the $\mathcal{Z}$ norm; on the other hand, if we have a (bounded) linear functional on $\mathcal{Y}$, it must be bounded on $L^{q}l^{2}$, thus it is given by pairing with an element of $L^{q'}l^{2}$, and then by considering the action of $\mathbb{Z}$ on this function, we conclude that it is actually in $\mathcal{Z}$.
 \end{proof}
\section{Relevant probabilistic results}\label{probab}
\subsection{Review of the construction of Gibbs measure}\label{gibbs}In this section we briefly review the construction of the Gibbs measure $\nu$ as done in \cite{Tz10}. This measure is defined by adding a weight to some Wiener measure $\rho$, so we first describe the relevant Wiener measure. 

Consider a sequence of independent complex Gaussian random variables $\{g_{n}\}_{n>0}$ living on some ambient probability space $(\Omega,\mathcal{B},\mathbb{P})$ which are normalized so that $\mathbb{E}(|g_{n}|^{2})=1$. By excluding a null set, we also assume\footnote[1]{This assumption is just in order to define the map $\mathbf{f}$ and is irrelevant otherwise.} that $|g_{n}|=O(\langle n\rangle^{10})$ holds everywhere on $\Omega$. Letting $g_{-n}=\overline{g_{n}}$, we define the random series \begin{equation}\mathbf{f}:\Omega\ni\omega\mapsto\sum_{n\neq 0}\frac{g_{n}(\omega)}{2\sqrt{\pi|n|}}e^{\mathrm{i}nx}\in\mathcal{V}\end{equation} as a map from $\Omega$ to $\mathcal{V}$ (recall that $\mathcal{V}$ is the subset of $\mathcal{D}'(\mathbb{T})$ containing real-valued distributions with vanishing mean). This then defines the Wiener measure $\rho$ on $\mathcal{V}$ by $\rho(E)=\mathbb{P}(\mathbf{f}^{-1}(E))$. For each positive integer $N$, if we identify $\mathcal{V}$ with $\mathcal{V}_{N}\times\mathcal{V}_{N}^{\perp}$, then the measure $\mathrm{d}\rho$ is identified with $\mathrm{d}\rho_{N}\times\mathrm{d}\rho_{N}^{\perp}$, with the latter two measures defined by \begin{equation}\rho_{N}(E)=\mathbb{P}((\Pi_{N}\mathbf{f})^{-1}(E)),\,\,\,\,\,\,\,\rho_{N}^{\perp}(E)=\mathbb{P}((\Pi_{N}^{\perp}\mathbf{f})^{-1}(E)).\nonumber\end{equation}

Fix a compactly supported smooth cutoff $\zeta$, $0\leq\zeta\leq 1$, which equals $1$ on some neighborhood of $0$. Consider for each $N$ the functions \begin{equation}\theta_{N}(f)=\zeta\big(\|\Pi_{N}f\|_{L^{2}}^{2}-\alpha_{N}\big)e^{\frac{1}{3}\int_{\mathbb{T}}(S_{N}f)^{3}};\end{equation} \begin{equation}\theta_{N}^{\sharp}(f)=\zeta\big(\|\Pi_{N}f\|_{L^{2}}^{2}-\alpha_{N}\big)e^{\frac{1}{3}\int_{\mathbb{T}}(\Pi_{N}f)^{3}},\end{equation}where we recall $\Pi_{N}=\mathbb{P}_{\leq N}$ as in Section \ref{notation}, $S_{N}$ as in (\ref{truncation}), and \begin{equation}\alpha_{N}=\sum_{n=1}^{N}\frac{1}{n}=\mathbb{E}\big(\|\Pi_{N}\mathbf{f}\|_{L^{2}}^{2}\big).\nonumber\end{equation} Clearly $\theta_{N}$ and $\theta_{N}^{\sharp}$ only depend on $\Pi_{N}f$, thus they can also be understood as functions on $\mathcal{V}_{N}$. Define the measures \begin{equation}\mathrm{d}\nu_{N}=\theta_{N}\mathrm{d}\rho,\,\,\,\,\,\mathrm{d}\nu_{N}^{\circ}=\theta_{N}\mathrm{d}\rho_{N};\,\,\,\,\,\mathrm{d}\nu_{N}^{\sharp}=\theta_{N}^{\sharp}\mathrm{d}\rho,\,\,\,\,\,\mathrm{d}\nu_{N}^{\o}=\theta_{N}^{\sharp}\mathrm{d}\rho_{N}.\nonumber\end{equation} Then we could identify $\mathrm{d}\nu_{N}$ and $\mathrm{d}\nu_{N}^{\sharp}$ with $\mathrm{d}\nu_{N}^{\circ}\times\mathrm{d}\rho_{N}^{\perp}$ and $\mathrm{d}\nu_{N}^{\o}\times\mathrm{d}\rho_{N}^{\perp}$, respectively. Moreover, if we identify $\mathcal{V}_{N}$ with $\mathbb{R}^{2N}$ and thus denote the measure on $\mathcal{V}_{N}$ corresponding to the Lebesgue measure on $\mathbb{R}^{2N}$ by $\mathcal{L}_{N}$, then with some constant $C_{N}$
\begin{equation}\mathrm{d}\nu_{N}^{\circ}=C_{N}\zeta\big(\|f\|_{L^{2}}^{2}-\alpha_{N}\big)e^{-2E_{N}[f]}\mathrm{d}\mathcal{L}_{N};
\end{equation}
\begin{equation}\mathrm{d}\nu_{N}^{\o}=C_{N}\zeta\big(\|f\|_{L^{2}}^{2}-\alpha_{N}\big)e^{-2E[f]}\mathrm{d}\mathcal{L}_{N},
\end{equation} with the $f$ here denoting some element of $\mathcal{V}_{N}$, the Hamiltonian $E$ as in (\ref{hamiltonian}), and the truncated version $E_{N}$ being
\begin{equation}E_{N}[f]=\int_{\mathbb{T}}\frac{1}{2}|\partial_{x}^{1/2}f|^{2}-\frac{1}{6}(S_{N}f)^{3}.
\end{equation}
The main result of \cite{Tz10} now reads as follows.
\begin{proposition}[\cite{Tz10}, Theorem 1]\label{convergence} The sequence $\theta_{N}^{\sharp}$ converges in $L^{r}(\mathrm{d}\rho)$ to some function $\theta$ for all $1\leq r<\infty$, and if we define $\nu$ by $\mathrm{d}\nu=\theta\mathrm{d}\rho$, then $\nu_{N}^{\sharp}$ converges strongly to $\nu$ in the sense that the total variation of their difference tends to zero. This $\nu$ is defined to be the Gibbs measure for (\ref{bo}).
\end{proposition}
\begin{remark}Only weak convergence is claimed in \cite{Tz10}, but an easy elaboration of the arguments there actually gives a much stronger convergence as stated in Proposition \ref{convergence} above.
\end{remark}
\begin{remark}\label{rmk4.3}We note that the measure $\nu$ depends on the choice of $\zeta$. In this regard we have the following easy observation: there exists a countable collection $\{\zeta^{R}\}_{R\in\mathbb{N}}$ with corresponding $\theta^{R}$ such that the union of $\mathcal{A}^{R}=\{f:\theta^{R}(f)>0\}$ has\footnote[1]{Note that $\mathcal{A}^{R}$ is the largest set on which $\rho$ and $\nu^{R}$ are mutually absolutely continuous.} full $\rho$ measure.
\end{remark}
The finite dimensional approximations we will actually use are $\nu_{N}$ instead of $\nu_{N}^{\sharp}$, thus we still need to prove the convergence of $\nu_{N}$. However, the proof is essentially the same as the proof of Proposition \ref{convergence}, so we shall omit it here and only state the result.
\begin{proposition}\label{convergence2}The sequence $\theta_{N}$ converges in $L^{r}(\mathrm{d}\rho)$ to the $\theta$ defined in Proposition \ref{convergence} for all $1\leq r<\infty$, and $\nu_{N}$ converges strongly to the $\nu$ defined in Proposition \ref{convergence} in the sense that the total variation of their difference tends to zero.
\end{proposition}
\subsection{Compatibility with the Besov space} By elementary probabilistic arguments we can see that\begin{equation}\rho\big(f\in\mathcal{V}:\|f\|_{L^{2}}<\infty\big)=0;\end{equation}\begin{equation}\rho\big(f\in\mathcal{V}:\|f\|_{H^{-\delta}}<\infty\big)=1,\end{equation} for all $\delta>0$. Namely, the Wiener measure $\mathrm{d}\rho$ (and hence the Gibbs measure $\mathrm{d}\nu$) is compatible with $H^{-\delta}$ but not $L^{2}$, which is the essential difficulty in establishing the invariance result. In this section we show that this difficulty may be resolved by using the Besov space $Z_{1}$ defined in Section \ref{parameters}. First we prove a lemma.
\begin{lemma}\label{besov} Suppose that $g_{j}(1\leq j\leq N)$ are independent normalized complex Gaussian random variables. Then we have\begin{equation}\label{form}\mathbb{P}\bigg(\sum_{j=1}^{N}|g_{j}|^{4}\geq \alpha N\bigg)\leq 4e^{-\frac{1}{120}\sqrt{\alpha N}},\end{equation} for all $\alpha>1600$ and positive integer $N$.
\end{lemma}
\begin{proof} Let $X=\sum_{j=1}^{N}|g_{j}|^{4}$. Since $\mathbb{E}(|g_{j}|^{4m})=(2m)!$, we can estimate, for each integer $k\geq 1$, the $k$-th moment of $X$ by \begin{eqnarray}\mathbb{E}(X^{k})&=&\sum_{m_{1}+\cdots+m_{N}=k}\frac{k!}{m_{1}!\cdots m_{N}!}\times\mathbb{E}\big(|g_{1}|^{4m_{1}}\cdots |g_{N}|^{4m_{N}}\big)\nonumber\\ &\leq& k!\sum_{m_{1}+\cdots+m_{N}=k}\prod_{j=1}^{N}\frac{(2m_{j})!}{m_{j}!}\nonumber\\&\leq &k!4^{k}\sum_{m_{1}+\cdots+m_{N}=k}\prod_{j=1}^{N}m_{j}!,\nonumber\end{eqnarray} since ${2m \choose m}\leq 4^{m}$. From this, we have that (for $\epsilon>0$)\begin{equation}\mathbb{E}\big(e^{\sqrt{\epsilon X}}\big)\leq 2\mathbb{E}(\cosh\sqrt{\epsilon X})\leq 2+2\sum_{k\geq 1}\frac{\epsilon^{k}}{(2k)!}\mathbb{E}(X^{k})\leq 2+\sum_{k\geq 1}\frac{(8\epsilon)^{k}}{k!}S_{N,k},\nonumber\end{equation} where \begin{equation}S_{N,k}=\sum_{m_{1}+\cdots+m_{N}=k}\prod_{j=1}^{N}m_{j}!,\end{equation} which we shall now estimate. By identifying the nonzero terms in $(m_{1},\cdots,m_{N})$, we can rewrite $S_{N,k}$ as \begin{equation} S_{N,k}=\sum_{1\leq r\leq \min\{N,k\}} {N\choose r}S'_{k,r},\end{equation} where \begin{equation}S'_{k,r}=\sum_{m_{1}+\cdots +m_{r}=k,m_{j}\geq 1}\prod_{j=1}^{r}m_{j}!.\nonumber\end{equation} Clearly the number of choices of $(m_{1},\cdots,m_{r})$ is at most ${k-1\choose r-1}\leq 2^{k}$, and for each choice of $(m_{1},\cdots,m_{r})$, we have\begin{eqnarray}\prod_{j=1}^{r}m_{j}! &\leq  &m_{1}\cdots m_{r}\times\prod_{j=1}^{r}(m_{j}-1)!\nonumber\\
&\leq & (k/r)^{r}\bigg(\sum_{j=1}^{r}(m_{j}-1)\bigg)!\nonumber\\
&\leq & e^{r\frac{k}{r}}(k-r)!
\leq 3^{k}(k-r)!\nonumber.\end{eqnarray} Therefore we know that $S'_{k,r}\leq 6^{k}(k-r)!$. Next, notice that there are at most $k\leq 2^{k}$ choices of $r$, and that ${N\choose r}\leq N^{r}/r!$, we have\begin{equation}\label{est2}S_{N,k}\leq 12^{k}\max_{1\leq r\leq k}\frac{N^{r}(k-r)!}{r!}.\end{equation}

If the maximum in (\ref{est2}) is attained at $r=k$, it will be bounded by $\frac{N^{k}}{k!}$; otherwise it is attained at some $r<k$, from which we know $N\leq (r+1)(k-r)\leq 2r(k-r)$. Therefore the maximum in this case is bounded by\begin{equation}\frac{N^{r}(k-r)!}{r!}\leq\frac{2^{k}r^{r}(k-r)^{r}(k-r)^{k-r}}{r!}\leq (6k)^{k}\leq 18^{k}k!.\nonumber\end{equation} Altogether we have\begin{equation}S_{N,k}\leq216^{k}k!+\frac{(12N)^{k}}{k!},\nonumber\end{equation} and hence\begin{equation}\mathbb{E}\big(e^{\sqrt{\epsilon X}}\big)\leq 2+\sum_{k\geq 1}(1728\epsilon)^{k}+\sum_{k\geq 1}\frac{(384\epsilon N)^{k}}{(2k)!},\end{equation} which is clearly bounded by $4e^{20\sqrt{\epsilon N}}$ if we choose $\epsilon=\frac{1}{3456}$. Now if $\alpha>1600$, we will have\begin{equation}\mathbb{P}(X\geq\alpha N)\leq e^{-\sqrt{\epsilon\alpha N}}\mathbb{E}\big(e^{\sqrt{\epsilon X}}\big)\leq 4e^{-\frac{1}{120}\sqrt{\alpha N}},\nonumber\end{equation} as desired.
\end{proof}
Now we can prove that the Wiener measure $\mathrm{d}\rho$ is compatible with our Besov space $Z_{1}$. Namely, we have
\begin{proposition}\label{compat}With the measure $\rho$ defined in Section \ref{gibbs}, we have $\rho(Z_{1})=1$; more precisely we have\begin{equation}\label{largedeviation}\rho\big(\{f\in\mathcal{V}:\|f\|_{Z_{1}}\leq K\}\big)\geq 1-Ce^{-C^{-1}K^{2}}\end{equation} for all $K>0$.
\end{proposition}
\begin{proof} We only need to prove (\ref{largedeviation}). Setting $C$ large, this inequality will be trivial when $K\leq 100$. When $K>100$, we get from the definition that\begin{equation}\rho\big(\{f\in\mathcal{V}:\|f\|_{Z_{1}}>100K\}\big)\leq\sum_{j\geq 0}\mathbb{P}\bigg(\sum_{0<n\sim 2^{j}}|g_{n}|^{p}\geq K^{p}2^{j}\bigg).\end{equation} By H\"{o}lder, \begin{equation}\sum_{0<n\sim 2^{j}}|g_{n}|^{p}\geq K^{p}2^{j}\nonumber\end{equation} implies \begin{equation}\sum_{0<n\sim 2^{j}}|g_{n}|^{4}\geq K^{4}2^{j}.\nonumber\end{equation} By lemma \ref{besov}, this has probability not exceeding $Ce^{-C^{-1}K^{2}2^{j/2}}$ provided $K>100$. Summing up over $j$, we see that \begin{equation}\rho\big(\{f\in\mathcal{V}:\|f\|_{Z_{1}}>K\}\big)\leq\sum_{j\geq 0}Ce^{-C^{-1}K^{2}2^{j/2}}\leq Ce^{-C^{-1}K^{2}}.\nonumber\end{equation} This completes the proof.
\end{proof}

\section{The gauge transform I: Beating the derivative loss}\label{gaugetransform}
From this section to Section \ref{gaugetransform3}, we will introduce the gauge transform for (\ref{smoothtrunc}), and use it to derive the new equations. We will fix a large positive integer $N$ throughout, and drop the subscript $N$ in $S_{N}$ (we are allowing $N=\infty$, in which case the arguments should be modified slightly but no essential difference occurs). We also fix a smooth solution $u$ to (\ref{smoothtrunc}); note that smooth solutions are automatically global. When $N$ is finite, we also assume that $\widehat{u}$ is supported in $|n|\leq N$ for all time.

The gauge transform we use is defined as a power series, thus in many occasions we will have to deal with summations over sequences of the form $(m_{1},\cdots,m_{\mu})$. To simplify the notation we will define, for such a sequence, the partial sums\begin{equation}m_{ij}=m_{i}+\cdots+m_{j}.\nonumber\end{equation} This notation will also be used for other sequences, say $\mu_{i}$, which will always be nonnegative integers.
\subsection{The definition of $w$}\label{pass}
Let $F$ be the unique mean-zero antiderivative of $u$, namely $F_{n}=\frac{1}{\mathrm{i}n}u_{n}$ for $n\neq 0$ and $F_{0}=0$. Define the operators $Q_{0}:\phi\mapsto (Su)\cdot\phi$ and $P_{0}:\phi\mapsto (SF)\cdot\phi$, as well as $Q=SQ_{0}S$ and $P=SP_{0}S$. Further, define the operator \begin{equation}\label{expo}M=e^{-\frac{\mathrm{i}P}{2}}=\sum_{\mu\geq0}\frac{1}{\mu!}\bigg(-\frac{\mathrm{i}}{2}\bigg)^{\mu}P^{\mu}.\end{equation} The function $w$ will be defined by\begin{equation}w=\mathbb{P}_{+}(Mu).\end{equation} We also define $v=Mu$, so that $w_{n}=v_{n}$ when $n>0$, and $w_{n}=0$ otherwise. The evolution equation satisfied by $w$ can be computed as follows:
\begin{eqnarray}
(\partial_{t}-\mathrm{i}\partial_{xx})w&=&\mathbb{P}_{+}M(\partial_{t}-\mathrm{i}\partial_{xx})u+\mathbb{P}_{+}[\partial_{t},M]u-\mathrm{i}\mathbb{P}_{+}[\partial_{xx},M]u\nonumber\\
&=&-2\mathrm{i}\mathbb{P}_{+}(M\mathbb{P}_{-}u_{xx})+\mathbb{P}_{+}\big([\partial_{t},M]u-\mathrm{i}\big[\partial_{x},[\partial_{x},M]\big]u\big)\nonumber\\
&+&\mathbb{P}_{+}(MS(Su\cdot Su_{x})-2\mathrm{i}[\partial_{x},M]u_{x})\nonumber\\
\label{line1}&=&-2\mathrm{i}\mathbb{P}_{+}\partial_{x}(M\mathbb{P}_{-}u_{x})\\
\label{line2}&-&2\mathrm{i}\mathbb{P}_{+}([\partial_{x},M]+\frac{\mathrm{i}}{2}MQ)u_{x}\\
\label{line3}&+&2\mathrm{i}\mathbb{P}_{+}[\partial_{x},M]\mathbb{P}_{-}u_{x}+\mathbb{P}_{+}\big([\partial_{t},M]-\mathrm{i}\big[\partial_{x},[\partial_{x},M]\big]\big)u.
\end{eqnarray}

\subsection{The term in (\ref{line1})}\label{term1} By expanding $M$ using (\ref{expo}), we can write the term in (\ref{line1}) as\begin{equation}(\ref{line1})=\sum_{\mu_{1}}\frac{(-1)^{\mu_{1}}}{2^{\mu_{1}}\mu_{1}!}\mathcal{K}_{\mu_{1}}^{1},\end{equation} where in Fourier space \begin{equation}\label{k1}(\mathcal{K}_{\mu_{1}}^{1})_{n_{0}}=2\mathrm{i}\sum_{\mathbf{v}\in S_{n_{0},\mu_{1}}^{1}}\Lambda_{\mathbf{v}}^{1\mu_{1}}(u_{m_{1}}\cdots u_{m_{\mu_{1}}})u_{n_{1}}.\end{equation} Here in (\ref{k1}), the spatial frequency set is defined to be
\begin{eqnarray}S_{n_{0},\mu_{1}}^{1}&=&\big\{\mathbf{v}=(m_{1},\cdots,m_{\mu_{1}},n_{1})\in\mathbb{Z}^{\mu_{1}+1}:\nonumber\\
&&m_{i}\neq 0,n_{0}>0,n_{1}<0;
m_{1,\mu_{1}}+n_{1}=n_{0}\big\},\nonumber\end{eqnarray} and the weight is
\begin{eqnarray}
\Lambda_{\mathbf{v}}^{1\mu_{1}}&=&\prod_{i=1}^{\mu_{1}}\frac{1}{m_{i}}\psi\bigg(\frac{m_{i}}{N}\bigg)\times\prod_{i=2}^{\mu_{1}}\psi^{2}\bigg(\frac{m_{i,\mu_{1}}+n_{1}}{N}\bigg)\times\nonumber\\
&\times& n_{0}n_{1}\psi\bigg(\frac{n_{0}}{N}\bigg)\psi\bigg(\frac{n_{1}}{N}\bigg).\nonumber
\end{eqnarray} As the next step, we rewrite part of the weight as\begin{equation}\label{junior}\frac{1}{m_{1}\cdots m_{\mu_{1}}}=\frac{1}{n_{0}-n_{1}}\sum_{i=1}^{\mu_{1}}\frac{1}{m_{1}\cdots m_{i-1}m_{i+1}\cdots m_{\mu_{1}}}.\end{equation} By renaming the variables, we obtain
\begin{equation}\label{k1.1}(\ref{line1})=\sum_{\mu_{1}\geq 1}\sum_{i=1}^{\mu_{1}}\frac{(-1)^{\mu_{1}}}{2^{\mu_{1}-1}\mu_{1}!}\mathcal{K}_{\mu_{1}i}^{1},\end{equation} where in Fourier space\begin{equation}\label{k1.2}(\mathcal{K}_{\mu_{1}i}^{1})_{n_{0}}=\mathrm{i}\sum_{\mathbf{v}\in S_{n_{0},\mu_{1}}^{1.1}}\Lambda_{\mathbf{v}}^{1\mu_{1}i}(u_{m_{1}}\cdots u_{m_{\mu_{1}-1}})u_{n_{1}}u_{n_{2}}.\end{equation} The frequency set here is\begin{eqnarray}\label{1-2}S_{n_{0},\mu_{1}}^{1.1}&=&\big\{\mathbf{v}=(m_{1},\cdots,m_{\mu_{1}-1},n_{1},n_{2}):\\
&&\mathbf{v}\in\mathbb{Z}^{\mu_{1}+1},m_{j}\neq 0,n_{0}>0;\nonumber\\&&
n_{1}<0,m_{1,\mu_{1}-1}+n_{1}+n_{2}=n_{0}\big\},\nonumber\end{eqnarray}
and the weight is
\begin{eqnarray}\label{1-3}
\Lambda_{\mathbf{v}}^{1\mu_{1}i}&=&
\prod_{j=1}^{\mu_{1}-1}\frac{1}{m_{j}}\psi\bigg(\frac{m_{j}}{N}\bigg)\prod_{j=i+1}^{\mu_{1}}\psi^{2}\bigg(\frac{m_{j-1,\mu_{1}-1}+n_{1}}{N}\bigg)\times\\
&\times&\prod_{j=2}^{i}\psi^{2}\bigg(\frac{m_{j,\mu_{1}-1}+n_{1}+n_{2}}{N}\bigg)\frac{n_{0}n_{1}}{|n_{0}|+|n_{1}|}\prod_{j=0}^{2}\psi\bigg(\frac{n_{j}}{N}\bigg).\nonumber
\end{eqnarray}

\subsection{The term in (\ref{line2})}\label{term2}Since\begin{equation}[\partial_{x},P]=Q,\end{equation} we may compute
\begin{eqnarray}
[\partial_{x},M]&=&\sum_{\mu_{1}}\frac{1}{\mu_{1}!}\bigg(-\frac{\mathrm{i}}{2}\bigg)^{\mu_{1}}[\partial_{x},P^{\mu_{1}}]\nonumber\\
&=&\sum_{\mu_{1},\mu_{2}}\frac{1}{(\mu_{12}+1)!}\bigg(-\frac{\mathrm{i}}{2}\bigg)^{\mu_{12}+1}P^{\mu_{1}}QP^{\mu_{2}}\nonumber\\
\label{line4}&=&-\frac{\mathrm{i}}{2}MQ-\frac{\mathrm{i}}{2}\sum_{\mu_{1},\mu_{2}}\frac{1}{(\mu_{12}+1)!}\bigg(-\frac{\mathrm{i}}{2}\bigg)^{\mu_{12}}P^{\mu_{1}}[Q,P^{\mu_{2}}].
\end{eqnarray} By expanding the commutator in (\ref{line4}), we can write the term in (\ref{line2}) as
\begin{equation}\label{line5}(\ref{line2})=-\sum_{\mu_{1},\mu_{2}}\frac{\mu_{1}+1}{(\mu_{12}+2)!}\bigg(-\frac{\mathrm{i}}{2}\bigg)^{\mu_{12}+1}\mathbb{P}_{+}P^{\mu_{1}}[Q,P]P^{\mu_{2}}u_{x}.\end{equation} Notice that \begin{equation}\label{line6}[Q,P]=S(Q_{0}S^{2}P_{0}-P_{0}S^{2}Q_{0})S,\end{equation} we can thus write\begin{equation}(\ref{line2})=\sum_{\mu_{1},\mu_{2}}\frac{(-1)^{\mu_{12}}(\mu_{1}+1)}{2^{\mu_{12}+1}(\mu_{12}+2)!}\mathcal{K}_{\mu_{1}\mu_{2}}^{2},\end{equation} where in the Fourier space \begin{equation}(\mathcal{K}_{\mu_{1}\mu_{2}}^{2})_{n_{0}}=\mathrm{i}\sum_{\mathbf{v}\in S_{n_{0},\mu_{1}\mu_{2}}^{2}}\lambda_{\mathbf{v}}^{2\mu_{1}\mu_{2}}(u_{m_{1}}\cdots u_{m_{\mu_{12}}})u_{n_{1}}u_{n_{2}} u_{n_{3}}.\end{equation} Here the frequency set is\begin{eqnarray}\label{freqset}S_{n_{0},\mu_{1}\mu_{2}}^{2}&=&\big\{\mathbf{v}=(m_{1},\cdots,m_{\mu_{12}},n_{1},n_{2},n_{3}):\\
&&\mathbf{v}\in(\mathbb{Z}^{*})^{\mu_{12}+3},m_{i}\neq 0,n_{1}n_{2}n_{3}\neq 0,n_{0}>0;\nonumber\\
&&m_{1,\mu_{12}}+n_{1}+n_{2}+n_{3}=n_{0}\big\},\nonumber\end{eqnarray} and the weight is\begin{eqnarray}\lambda_{\mathbf{v}}^{2\mu_{1}\mu_{2}}&=&\frac{n_{3}}{n_{2}}\prod_{i=0}^{3}\psi\bigg(\frac{n_{i}}{N}\bigg)\prod_{i=1}^{\mu_{12}}\frac{1}{m_{i}}\psi\bigg(\frac{m_{i}}{N}\bigg)\prod_{i=1}^{\mu_{2}}\psi^{2}\bigg(\frac{n_{3}+m_{\mu_{1}+i,\mu_{12}}}{N}\bigg)\times\nonumber\\
&\times &\prod_{i=2}^{\mu_{1}+1}\psi^{2}\bigg(\frac{n_{1}+n_{2}+n_{3}+m_{i,\mu_{12}}}{N}\bigg)\times\nonumber\\
&\times&\bigg[\psi^{2}\bigg(\frac{n_{2}+n_{3}+m_{\mu_{1}+1,\mu_{12}}}{N}\bigg)-\psi^{2}\bigg(\frac{n_{1}+n_{3}+m_{\mu_{1}+1,\mu_{12}}}{N}\bigg)\bigg].\nonumber\end{eqnarray}Note that $S_{n_{0},\mu_{1}\mu_{2}}^{2}$ is symmetric with respect to $n_{1}$ and $n_{3}$, we can swap these two variables and rearrange the terms to obtain \begin{equation}\label{2-1}(\mathcal{K}_{\mu_{1}\mu_{2}}^{2})_{n_{0}}=\mathrm{i}\sum_{\mathbf{v}\in S_{n_{0},\mu_{1}\mu_{2}}^{2}}\Lambda_{\mathbf{v}}^{2\mu_{1}\mu_{2}}(u_{m_{1}}\cdots u_{m_{\mu_{12}}})u_{n_{1}}u_{n_{2}} u_{n_{3}},\end{equation} where the frequency set $S_{n_{0},\mu_{1}\mu_{2}}^{2}$ is as in (\ref{freqset}), and the weight is \begin{eqnarray}\label{2-3}
\Lambda_{\mathbf{v}}^{2\mu_{1}\mu_{2}}&=&\frac{1}{2n_{2}}\prod_{i=0}^{3}\psi\bigg(\frac{n_{i}}{N}\bigg)\prod_{i=1}^{\mu_{12}}\frac{1}{m_{i}}\psi\bigg(\frac{m_{i}}{N}\bigg)\times\\
&\times&\prod_{i=2}^{\mu_{1}+2}\psi^{2}\bigg(\frac{n_{1}+n_{2}+n_{3}+m_{i,\mu_{12}}}{N}\bigg)\times\nonumber\\
&\times&\bigg[n_{3}\psi^{2}\bigg(\frac{n_{2}+n_{3}+m_{\mu_{1}+1,\mu_{12}}}{N}\bigg)\prod_{i=1}^{\mu_{2}}\psi^{2}\bigg(\frac{n_{3}+m_{\mu_{1}+i,\mu_{12}}}{N}\bigg)\nonumber\\
&& -n_{3}\psi^{2}\bigg(\frac{n_{1}+n_{3}+m_{\mu_{1}+1,\mu_{12}}}{N}\bigg)\prod_{i=1}^{\mu_{2}}\psi^{2}\bigg(\frac{n_{3}+m_{\mu_{1}+i,\mu_{12}}}{N}\bigg)\nonumber\\
&&+n_{1}\psi^{2}\bigg(\frac{n_{1}+n_{2}+m_{\mu_{1}+1,\mu_{12}}}{N}\bigg)\prod_{i=1}^{\mu_{2}}\psi^{2}\bigg(\frac{n_{1}+m_{\mu_{1}+i,\mu_{12}}}{N}\bigg)\nonumber\\
&&-n_{1}\psi^{2}\bigg(\frac{n_{1}+n_{3}+m_{\mu_{1}+1,\mu_{12}}}{N}\bigg)\prod_{i=1}^{\mu_{2}}\psi^{2}\bigg(\frac{n_{1}+m_{\mu_{1}+i,\mu_{12}}}{N}\bigg)\bigg]\nonumber.
\end{eqnarray}
\subsection{The term in (\ref{line3})}\label{term3}Clearly we have\begin{equation}\label{line7}[\partial_{t},M]=\sum_{\mu_{1},\mu_{2}}\frac{1}{(\mu_{12}+1)!}\bigg(-\frac{\mathrm{i}}{2}\bigg)^{\mu_{12}+1}P^{\mu_{1}}[\partial_{t},P]P^{\mu_{2}},\end{equation} where\begin{equation}\label{line8}[\partial_{t},P]:\psi\mapsto S(SF_{t}\cdot S\psi);\end{equation} also we may compute\begin{eqnarray}\big[\partial_{x},[\partial_{x},M]\big]&=&\sum_{\mu_{1},\mu_{2}}\frac{1}{(\mu_{12}+1)!}\bigg(-\frac{\mathrm{i}}{2}\bigg)^{\mu_{12}+1}[\partial_{x},P^{\mu_{1}}QP^{\mu_{2}}]\nonumber\\
\label{line9}&=&\sum_{\mu_{1},\mu_{2}}\frac{1}{(\mu_{12}+1)!}\bigg(-\frac{\mathrm{i}}{2}\bigg)^{\mu_{1}+\mu_{2}+1}P^{\mu_{1}}[\partial_{x},Q]P^{\mu_{2}}\nonumber\\
&+&2\sum_{\mu_{1},\mu_{2},\mu_{3}}\frac{1}{(\mu_{13}+2)!}\bigg(-\frac{\mathrm{i}}{2}\bigg)^{\mu_{13}+2}\times\nonumber\\
\label{line10}&\times&P^{\mu_{1}}QP^{\mu_{2}}QP^{\mu_{3}}.\nonumber
\end{eqnarray} Using the fact that\begin{equation}[\partial_{t},P]-\mathrm{i}[\partial_{x},Q]:\psi\mapsto S(SG\cdot S\psi)\end{equation} where\begin{equation}G=F_{t}-\mathrm{i}F_{xx}=-2\mathrm{i}\mathbb{P}_{-}u_{x}+\frac{1}{2}\bigg(S((Su)^{2})-\mathbb{P}_{0}((Su)^{2})\bigg),\end{equation} we may write
\begin{eqnarray}
(\ref{line3})&=&\sum_{\mu_{1},\mu_{2}}\frac{(-1)^{\mu_{12}}}{2^{\mu_{12}}(\mu_{12}+1)!}\big(\mathcal{K}_{\mu_{1}\mu_{2}}^{3}+\mathcal{K}_{\mu_{1}\mu_{2}}^{4}\big)\\
&+&\sum_{\mu_{1},\mu_{2},\mu_{3}}\frac{(-1)^{\mu_{13}+1}}{2^{\mu_{13}+2}(\mu_{13}+2)!}\mathcal{K}_{\mu_{1}\mu_{2}\mu_{3}}^{5}.\nonumber
\end{eqnarray} Here, $\mathcal{K}_{\mu_{1}\mu_{2}}^{3}$ is defined in the Fourier space as
\begin{equation}(\mathcal{K}_{\mu_{1}\mu_{2}}^{3})_{n_{0}}=\mathrm{i}\sum_{\mathbf{v}\in S_{n_{0},\mu_{1}\mu_{2}}^{3}}\Lambda_{\mathbf{v}}^{3\mu_{1}\mu_{2}}(u_{m_{1}}\cdots u_{m_{\mu_{12}}})u_{n_{1}}u_{n_{2}},\end{equation} with the frequency set
\begin{eqnarray}\label{3-2}S_{n_{0},\mu_{1}\mu_{2}}^{3} &=&\big\{\mathbf{v}=(m_{1},\cdots,m_{\mu_{12}},n_{1},n_{2}):\\
&&\mathbf{v}\in\mathbb{Z}^{\mu_{12}+2},m_{i}\neq 0, n_{1}\neq 0,n_{0}>0;\nonumber\\
&&n_{2}<0,m_{1,\mu_{12}}+n_{1}+n_{2}=n_{0}\big\}\nonumber
\end{eqnarray}
and the weight
\begin{eqnarray}\Lambda_{\mathbf{v}}^{3\mu_{1}\mu_{2}}&=&n_{2}\prod_{i=1}^{\mu_{12}}\frac{1}{m_{i}}\psi\bigg(\frac{m_{i}}{N}\bigg)\prod_{i=0}^{2}\psi\bigg(\frac{n_{i}}{N}\bigg)\prod_{i=2}^{\mu_{1}+1}\psi^{2}\bigg(\frac{n_{1}+n_{2}+m_{i,\mu_{12}}}{N}\bigg)\times\nonumber\\
&\times&\bigg[\prod_{i=1}^{\mu_{2}}\psi^{2}\bigg(\frac{n_{2}+m_{\mu_{1}+i,\mu_{12}}}{N}\bigg)-\prod_{i=1}^{\mu_{2}}\psi^{2}\bigg(\frac{n_{1}+m_{\mu_{1}+i,\mu_{12}}}{N}\bigg)\bigg].\nonumber
\end{eqnarray}
The term $\mathcal{K}_{\mu_{1}\mu_{2}}^{4}$ is defined as
\begin{equation}\label{4-1}
(\mathcal{K}_{\mu_{1}\mu_{2}}^{4})_{n_{0}}=\frac{\mathrm{i}}{4}\mathbb{P}_{0}((Su)^{2})\sum_{\mathbf{v}\in S_{n_{0},\mu_{1}\mu_{2}}^{4}}\Lambda_{\mathbf{v}}^{4\mu_{1}\mu_{2}}(u_{m_{1}}\cdots u_{m_{\mu_{12}}})u_{n_{1}},
\end{equation} with frequency set
\begin{eqnarray}\label{4-2}
S_{n_{0},\mu_{1}\mu_{2}}^{4}&=&\big\{\mathbf{v}=(m_{1},\cdots,m_{\mu_{12}},n_{1}):\\
&&\mathbf{v}\in\mathbb{Z}^{\mu_{12}+1},m_{i}\neq 0, n_{1}\neq 0,n_{0}>0;\nonumber\\
&&m_{1,\mu_{12}}+n_{1}=n_{0}\big\};\nonumber
\end{eqnarray} and weight \begin{eqnarray}\label{4-3}
\Lambda_{\mathbf{v}}^{4\mu_{1}\mu_{2}}&=&\psi\bigg(\frac{n_{0}}{N}\bigg)\psi\bigg(\frac{n_{1}}{N}\bigg)\prod_{i=1}^{\mu_{12}}\frac{1}{m_{i}}\psi\bigg(\frac{m_{i}}{N}\bigg)\times\\
&\times&\psi^{2}\bigg(\frac{n_{1}+m_{\mu_{1}+1,\mu_{12}}}{N}\bigg)\prod_{i=2}^{\mu_{12}}\psi^{2}\bigg(\frac{n_{1}+m_{i,\mu_{12}}}{N}\bigg).\nonumber
\end{eqnarray}
The term $\mathcal{K}_{\mu_{1}\mu_{2}\mu_{3}}^{5}$ is defined as
\begin{equation}\label{5-1}(\mathcal{K}_{\mu_{1}\mu_{2}\mu_{3}}^{5})_{n_{0}}=\mathrm{i}\sum_{\mathbf{v}\in S_{n_{0},\mu_{1}\mu_{2}\mu_{3}}^{5}}\Lambda_{\mathbf{v}}^{5\mu_{1}\mu_{2}\mu_{3}}(u_{m_{1}}\cdots u_{m_{\mu_{13}}})u_{n_{1}}u_{n_{2}}u_{n_{3}},\end{equation} with frequency set

\begin{eqnarray}\label{5-2}S_{n_{0},\mu_{1}\mu_{2}\mu_{3}}^{5}&=&\big\{\mathbf{v}=(m_{1},\cdots,m_{\mu_{13}},n_{1},n_{2},n_{3}):\\
&&\mathbf{v}\in\mathbb{Z}^{\mu_{13}+3},m_{i}\neq 0, n_{1}n_{2}n_{3}\neq 0,n_{0}>0;\nonumber\\
&&m_{1,\mu_{13}}+n_{1}+n_{2}+n_{3}=n_{0}\big\}\nonumber
\end{eqnarray} and weight
\begin{eqnarray}\label{5-3}
\Lambda_{\mathbf{v}}^{5\mu_{1}\mu_{2}\mu_{3}}&=&\prod_{i=2}^{\mu_{1}+1}\psi^{2}\bigg(\frac{n_{1}+n_{2}+n_{3}+m_{i,\mu_{13}}}{N}\bigg)\times\\
&\times&\prod_{i=1}^{\mu_{13}}\frac{1}{m_{i}}\psi\bigg(\frac{m_{i}}{N}\bigg)\prod_{i=0}^{3}\psi\bigg(\frac{n_{i}}{N}\bigg)\prod_{i=1}^{\mu_{3}}\psi^{2}\bigg(\frac{n_{3}+m_{\mu_{12}+i,\mu_{13}}}{N}\bigg)\times\nonumber\\
&\times&\bigg[\psi^{2}\bigg(\frac{n_{1}+n_{2}}{N}\bigg)\prod_{i=2}^{\mu_{2}+1}\psi^{2}\bigg(\frac{n_{1}+n_{2}+n_{3}+m_{\mu_{1}+i,\mu_{13}}}{N}\bigg)\nonumber\\
&&+\psi^{2}\bigg(\frac{n_{1}+n_{2}}{N}\bigg)\prod_{i=1}^{\mu_{2}}\psi^{2}\bigg(\frac{n_{3}+m_{\mu_{1}+i,\mu_{13}}}{N}\bigg)\nonumber\\
&&-2\prod_{i=1}^{\mu_{2}+1}\psi^{2}\bigg(\frac{n_{2}+n_{3}+m_{\mu_{1}+i,\mu_{13}}}{N}\bigg)\bigg].
\end{eqnarray}
Next, we shall rewrite a part of the weight $\Lambda_{\mathbf{v}}^{3\mu_{1}\mu_{2}}$ as \begin{eqnarray}\frac{1}{m_{1}\cdots m_{\mu_{12}}}&=&\frac{1}{m_{1}\cdots m_{12}}\psi_{2}\bigg(\frac{n_{1}}{|n_{0}|+|n_{2}|}\bigg)\nonumber\\
&+&\frac{1}{|n_{0}|+|n_{2}|-n_{1}}\psi_{1}\bigg(\frac{n_{1}}{|n_{0}|+|n_{2}|}\bigg)\sum_{i=1}^{\mu_{12}}\frac{1}{m_{1}\cdots m_{i-1}m_{i+1}\cdots m_{\mu_{12}}},\nonumber\end{eqnarray} then rename the variables (separating the cases $i\leq \mu_{1}$ and $i>\mu_{1}$) to obtain 
\begin{eqnarray}(\ref{line3})&=&\sum_{\mu_{1},\mu_{2}}\frac{(-1)^{\mu_{12}}}{2^{\mu_{12}}(\mu_{12}+1)!}\big(\mathcal{K}_{\mu_{1}\mu_{2}0}^{3}+\mathcal{K}_{\mu_{1}\mu_{2}}^{4}\big)\\
&+&\sum_{\mu_{1}+\mu_{2}\geq 1}\sum_{i=1}^{\mu_{12}}\frac{(-1)^{\mu_{12}}}{2^{\mu_{12}}(\mu_{12}+1)!}\mathcal{K}_{\mu_{1}\mu_{2}i}^{3}\nonumber\\
&+&\sum_{\mu_{1},\mu_{2},\mu_{3}}\frac{(-1)^{\mu_{13}+1}}{2^{\mu_{13}+2}(\mu_{13}+2)!}\mathcal{K}_{\mu_{1}\mu_{2}\mu_{3}}^{5},\nonumber
\end{eqnarray}where in the Fourier space
\begin{equation}\label{3-1}(\mathcal{K}_{\mu_{1}\mu_{2}0}^{3})_{n_{0}}=\mathrm{i}\sum_{\mathbf{v}\in S_{n_{0},\mu_{1}\mu_{2}}^{3}}\Lambda_{\mathbf{v}}^{3\mu_{1}\mu_{2}0}(u_{m_{1}}\cdots u_{m_{\mu_{12}}})u_{n_{1}}u_{n_{2}},\end{equation} with frequency set $S_{n_{0},\mu_{1}\mu_{2}}^{3}$ is as in (\ref{3-2}), and new weight
\begin{eqnarray}\label{3-3}\Lambda_{\mathbf{v}}^{3\mu_{1}\mu_{2}0}&=&\prod_{i=2}^{\mu_{1}+1}\psi^{2}\bigg(\frac{n_{1}+n_{2}+m_{i,\mu_{12}}}{N}\bigg)\times\\
&\times&n_{2}\psi_{2}\bigg(\frac{n_{1}}{|n_{0}|+|n_{2}|}\bigg)\prod_{i=1}^{\mu_{12}}\frac{1}{m_{i}}\psi\bigg(\frac{m_{i}}{N}\bigg)\prod_{i=0}^{2}\psi\bigg(\frac{n_{i}}{N}\bigg)\times\nonumber\\
&\times&\bigg[\prod_{i=1}^{\mu_{2}}\psi^{2}\bigg(\frac{n_{2}+m_{\mu_{1}+i,\mu_{12}}}{N}\bigg)-\prod_{i=1}^{\mu_{2}}\psi^{2}\bigg(\frac{n_{1}+m_{\mu_{1}+i,\mu_{12}}}{N}\bigg)\bigg];\nonumber
\end{eqnarray} the other term will be
\begin{equation}\label{3-4}(\mathcal{K}_{\mu_{1}\mu_{2}i}^{3})_{n_{0}}=\mathrm{i}\sum_{\mathbf{v}\in S_{n_{0},\mu_{1}\mu_{2}}^{3.1}}\Lambda_{\mathbf{v}}^{3\mu_{1}\mu_{2}i}(u_{m_{1}}\cdots u_{m_{\mu_{12}-1}})u_{n_{1}}u_{n_{2}}u_{n_{3}},\end{equation} where the new frequency set is
\begin{eqnarray}\label{3-5}
S_{n_{0},\mu_{1}\mu_{2}}^{3.1}&=&\big\{\mathbf{v}=(m_{1},\cdots,m_{\mu_{12}-1},n_{1},n_{2},n_{3}):\\
&&\mathbf{v}\in\mathbb{Z}^{\mu_{12}+2},m_{i}\neq 0, n_{1}n_{2}\neq 0,n_{0}>0;\nonumber\\
&&n_{3}<0,m_{1,\mu_{12}-1}+n_{1}+n_{2}+n_{3}=n_{0}\big\};\nonumber
\end{eqnarray} the new weight is
\begin{eqnarray}\label{3-6}\Lambda_{\mathbf{v}}^{3\mu_{1}\mu_{2}i}&=&\frac{n_{3}}{|n_{0}|+|n_{3}|-n_{1}}\psi_{1}\bigg(\frac{n_{1}}{|n_{0}|+|n_{3}|}\bigg)\prod_{j=1}^{\mu_{12}-1}\frac{1}{m_{j}}\psi\bigg(\frac{m_{j}}{N}\bigg)\times\\
&\times&\prod_{j=0}^{3}\psi\bigg(\frac{n_{j}}{N}\bigg)\prod_{j=2}^{i}\psi^{2}\bigg(\frac{n_{1}+n_{2}+n_{3}+m_{j,\mu_{12}-1}}{N}\bigg)\times\nonumber\\
&\times&\prod_{j=i+1}^{\mu_{1}+1}\psi^{2}\bigg(\frac{n_{1}+n_{3}+m_{j-1,\mu_{12}-1}}{N}\bigg)\times\nonumber\\
&\times&\bigg[\prod_{j=1}^{\mu_{2}}\psi^{2}\bigg(\frac{n_{3}+m_{\mu_{1}+j-1,\mu_{12}-1}}{N}\bigg)\nonumber\\
&&-\prod_{j=1}^{\mu_{2}}\psi^{2}\bigg(\frac{n_{1}+m_{\mu_{1}+j-1,\mu_{12}-1}}{N}\bigg)\bigg]\nonumber
\end{eqnarray} for $1\leq i\leq \mu_{1}$, and 
\begin{eqnarray}\label{3-62}\Lambda_{\mathbf{v}}^{3\mu_{1}\mu_{2}i}&=&\frac{n_{3}}{|n_{0}|+|n_{3}|-n_{1}}\psi_{1}\bigg(\frac{n_{1}}{|n_{0}|+|n_{3}|}\bigg)\prod_{j=1}^{\mu_{12}-1}\frac{1}{m_{j}}\psi\bigg(\frac{m_{j}}{N}\bigg)\times\\&\times&\prod_{j=0}^{3}\psi\bigg(\frac{n_{j}}{N}\bigg)\prod_{j=2}^{\mu_{1}+1}\psi^{2}\bigg(\frac{n_{1}+n_{2}+n_{3}+m_{j,\mu_{12}-1}}{N}\bigg)\times\nonumber\\
&\times&\bigg[\prod_{j=1}^{i-\mu_{1}}\psi^{2}\bigg(\frac{n_{2}+n_{3}+m_{\mu_{1}+j,\mu_{12}-1}}{N}\bigg)\times\nonumber\\
&&\times\prod_{j=i-\mu_{1}+1}^{\mu_{2}}\psi^{2}\bigg(\frac{n_{3}+m_{\mu_{1}+j-1,\mu_{12}-1}}{N}\bigg)\nonumber\\
&&-\prod_{j=1}^{i-\mu_{1}}\psi^{2}\bigg(\frac{n_{1}+n_{2}+m_{\mu_{1}+j,\mu_{12}-1}}{N}\bigg)\times\nonumber\\
&&\times\prod_{j=i-\mu_{1}+1}^{\mu_{2}}\psi^{2}\bigg(\frac{n_{1}+m_{\mu_{1}+j-1,\mu_{12}-1}}{N}\bigg)\bigg]\nonumber
\end{eqnarray} for $\mu_{1}+1\leq i\leq \mu_{12}$.

\subsection{Summary} Now we have obtained a first version of the equation satisfied by $w$, namely
\begin{eqnarray}\label{000}
(\partial_{t}-\mathrm{i}\partial_{xx})w&=&\sum_{\mu_{1}\geq 1}\sum_{i=1}^{\mu_{1}}\frac{(-1)^{\mu_{1}}}{2^{\mu_{1}-1}\mu_{1}!}\mathcal{K}_{\mu_{1}i}^{1}\\
&+&\sum_{\mu_{1},\mu_{2}}\frac{(-1)^{\mu_{12}}(\mu_{1}+1)}{2^{\mu_{12}+1}(\mu_{12}+2)!}\mathcal{K}_{\mu_{1}\mu_{2}}^{2}\nonumber\\
&+&\sum_{\mu_{1},\mu_{2}}\frac{(-1)^{\mu_{12}}}{2^{\mu_{12}}(\mu_{12}+1)!}\big(\mathcal{K}_{\mu_{1}\mu_{2}0}^{3}+\mathcal{K}_{\mu_{1}\mu_{2}}^{4}\big)\nonumber\\
&+&\sum_{\mu_{1}+\mu_{2}\geq 1}\sum_{i=1}^{\mu_{12}}\frac{(-1)^{\mu_{12}}}{2^{\mu_{12}}(\mu_{12}+1)!}\mathcal{K}_{\mu_{1}\mu_{2}i}^{3}\nonumber\\
&+&\sum_{\mu_{1},\mu_{2},\mu_{3}}\frac{(-1)^{\mu_{13}+1}}{2^{\mu_{13}+2}(\mu_{13}+2)!}\mathcal{K}_{\mu_{1}\mu_{2}\mu_{3}}^{5},\nonumber
\end{eqnarray} where:

the term $\mathcal{K}_{\mu_{1}i}^{1}$ is defined in (\ref{k1.2}), (\ref{1-2}) and (\ref{1-3});

the term $\mathcal{K}_{\mu_{1}\mu_{2}}^{2}$ is defined in (\ref{freqset}), (\ref{2-1}), and (\ref{2-3});

the term $\mathcal{K}_{\mu_{1}\mu_{2}0}^{3}$ is defined in (\ref{3-2}), (\ref{3-1}), and (\ref{3-3});

the term $\mathcal{K}_{\mu_{1}\mu_{2}i}^{3}$ is defined in (\ref{3-4}), (\ref{3-5}), (\ref{3-6}), and (\ref{3-62});

the term $\mathcal{K}_{\mu_{1}\mu_{2}}^{4}$ is defined in (\ref{4-1}), (\ref{4-2}), and (\ref{4-3});

the term $\mathcal{K}_{\mu_{1}\mu_{2}\mu_{3}}^{5}$ is defined in (\ref{5-1}), (\ref{5-2}), and (\ref{5-3}).

 In the next section we will further examine the structure of these terms.
\section{The gauge transform II: A miraculous cancellation}\label{gaugetransform2}
In this section we identify the bad resonant terms coming from each $\mathcal{K}^{j}$ term in (\ref{000}). Our computation will show that these bad terms will eventually add up to zero, leaving only the better-behaved ones. Throughout this section we will use a variable $k$, and define $\theta=\psi(k/N)$, $\eta=\psi'(k/N)$.
\subsection{The resonant terms in $\mathcal{K}^{1}$} In the expression (\ref{k1.2}), let $n_{1}+n_{2}=0$. Notice that $n_{1}<0$, we get a sum
\begin{equation}
-\mathrm{i}n_{0}\sum_{k>0}|u_{k}|^{2}\sum_{m_{1}+\cdots+m_{\mu_{1}-1}=n_{0}}\Delta\cdot \frac{u_{m_{1}}\cdots u_{m_{\mu_{1}-1}}}{m_{1}\cdots m_{\mu_{1}-1}},
\end{equation} where we always assume $m_{i}\neq 0$, and the factor
\begin{eqnarray}
\Delta&=&\prod_{j=1}^{\mu_{1}-1}\psi\bigg(\frac{m_{j}}{N}\bigg)\prod_{j=i+1}^{\mu_{1}}\psi^{2}\bigg(\frac{k-m_{j-1,\mu_{1}-1}}{N}\bigg)\times\nonumber\\
&\times&\prod_{j=2}^{i}\psi^{2}\bigg(\frac{m_{j,\mu_{1}-1}}{N}\bigg)\frac{k}{|n_{0}|+|k|}\psi\bigg(\frac{n_{0}}{N}\bigg)\psi^{2}\bigg(\frac{k}{N}\bigg).\nonumber
\end{eqnarray} We then replace each variable in this expression, except $k$, by zero\footnote[1]{Strictly speaking, we should replace $n$ by $m_{1,\mu_{1}-1}$ and cancel each $m_{j}$ in the numerator before this process, but the results will be the same and no estimate is affected.}, and get a term which reads \begin{equation}
-\mathrm{i}n_{0}\sum_{k>0}|u_{k}|^{2}\sum_{m_{1}+\cdots+m_{\mu_{1}-1}=n_{0}}\theta^{2(\mu_{1}-i+1)}\prod_{i=1}^{\mu_{1}-1}\frac{u_{m_{i}}}{m_{i}}.
\end{equation}Note that the summation over the $m_{i}$'s gives exactly $((\mathrm{i}F)^{\mu_{1}-1})_{n_{0}}$, we can then sum over $\mu_{1}$ and $i$ to get
\begin{eqnarray}(\mathcal{R}^{1})_{n_{0}}&=&-\mathrm{i}\sum_{k>0}|u_{k}|^{2}\sum_{\mu_{1}\geq 1}\frac{(-1)^{\mu_{1}}}{2^{\mu_{1}-1}\mu_{1}!}n_{0}((\mathrm{i}F)^{\mu_{1}-1})_{n_{0}}\sum_{i=1}^{\mu_{1}}\theta^{2(\mu_{1}-i+1)}\\
&=&-\sum_{k>0}|u_{k}|^{2}\sum_{\mu_{1}\geq 1}\frac{(-1)^{\mu_{1}}}{2^{\mu_{1}-1}\mu_{1}!}(\partial_{x}(\mathrm{i}F)^{\mu_{1}-1})_{n_{0}}\sum_{i=1}^{\mu_{1}}\theta^{2(\mu_{1}-i+1)}\nonumber\\
&=&\sum_{k>0}\sum_{\mu\geq 0}\frac{\mathrm{i}|u_{k}|^{2}}{(\mu+2)!}\bigg(u\bigg(-\frac{\mathrm{i}F}{2}\bigg)^{\mu}\bigg)_{n_{0}}\cdot \mathcal{C}^{1},\nonumber
\end{eqnarray} where \begin{equation}\mathcal{C}^{1}=-\frac{\mu+1}{2}\big(\theta^{2}+\theta^{4}+\cdots +\theta^{2\mu+4}\big),\end{equation} and we have dropped the dependence of $\mathcal{C}^{1}$ on $k$ and $\mu$ for simplicity.
\subsection{The resonant terms in $\mathcal{K}^{2}$}
In the expression (\ref{2-1}), let $n_{2}+n_{3}=0$ to obtain a term
\begin{equation}
\frac{\mathrm{i}}{2}\sum_{k\neq 0}|u_{k}|^{2}\sum_{n_{1}+m_{1}+\cdots+m_{\mu_{12}}=n_{0}}\Delta\cdot\frac{u_{m_{1}}\cdots u_{m_{\mu_{12}}}}{m_{1}\cdots m_{\mu_{12}}}u_{n_{1}},
\end{equation}where the factor
\begin{eqnarray}\label{factor2.1}
\Delta &=&\psi^{2}\bigg(\frac{k}{N}\bigg)\prod_{i=0}^{1}\psi\bigg(\frac{n_{i}}{N}\bigg)\prod_{i=1}^{\mu_{12}}\psi\bigg(\frac{m_{i}}{N}\bigg)\prod_{i=2}^{\mu_{1}+2}\psi^{2}\bigg(\frac{n_{1}+m_{i,\mu_{12}}}{N}\bigg)\times\\
&\times&\bigg[-\psi^{2}\bigg(\frac{m_{\mu_{1}+1,\mu_{12}}}{N}\bigg)\prod_{i=1}^{\mu_{2}}\psi^{2}\bigg(\frac{k+m_{\mu_{1}+i,\mu_{12}}}{N}\bigg)\nonumber\\
&& +\psi^{2}\bigg(\frac{k+n_{1}+m_{\mu_{1}+1,\mu_{12}}}{N}\bigg)\prod_{i=1}^{\mu_{2}}\psi^{2}\bigg(\frac{k+m_{\mu_{1}+i,\mu_{12}}}{N}\bigg)\nonumber\\
&&-\frac{n_{1}}{k}\psi^{2}\bigg(\frac{n_{1}+m_{\mu_{1}+1,\mu_{12}}-k}{N}\bigg)\prod_{i=1}^{\mu_{2}}\psi^{2}\bigg(\frac{n_{1}+m_{\mu_{1}+i,\mu_{12}}}{N}\bigg)\nonumber\\
&&+\frac{n_{1}}{k}\psi^{2}\bigg(\frac{n_{1}+m_{\mu_{1}+1,\mu_{12}}+k}{N}\bigg)\prod_{i=1}^{\mu_{2}}\psi^{2}\bigg(\frac{n_{1}+m_{\mu_{1}+i,\mu_{12}}}{N}\bigg)\bigg]\nonumber.
\end{eqnarray} We then discard the last two summands in the bracket, and in what remains replace each variable except $k$ by zero to get
\begin{equation}
\frac{\mathrm{i}}{2}\sum_{k\neq 0}|u_{k}|^{2}\sum_{n_{1}+m_{1}+\cdots+m_{\mu_{12}}=n_{0}}(\theta^{2\mu_{2}+4}-\theta^{2\mu_{2}+2})\cdot\frac{u_{m_{1}}\cdots u_{m_{\mu_{12}}}}{m_{1}\cdots m_{\mu_{12}}}u_{n_{1}}.
\end{equation} Since the summation over $m_{i}$ and $n_{1}$ gives exactly $(u\cdot (iF)^{\mu_{12}})_{n_{0}}$, we can then sum over $\mu_{1}$ and $\mu_{2}$ to obtain an expression which involves a sum over all $k\neq 0$. We may include a factor of $2$ and restrict to $k>0$ (since $\theta$ is even in $k$), and then take into account the symmetry with respect to $n_{1}$ and $n_{3}$ (namely, we are considering also the term where $n_{1}+n_{2}=0$) to include another factor of $2$, and the final expression will be 
\begin{eqnarray}(\mathcal{R}^{2.1})_{n_{0}}&=&2\mathrm{i}\sum_{k>0}|u_{k}|^{2}\sum_{\mu\geq 0}\frac{(-1)^{\mu}}{2^{\mu+1}(\mu+2)!}(u(\mathrm{i}F)^{\mu})_{n_{0}}\times\\
&\times&\sum_{\mu_{2}=0}^{\mu}(\mu-\mu_{2}+1)(\theta^{2\mu_{2}+4}-\theta^{2\mu_{2}+2})\nonumber\\
&=&\sum_{k>0}\sum_{\mu\geq 0}\frac{\mathrm{i}|u_{k}|^{2}}{(\mu+2)!}\bigg(u\bigg(-\frac{\mathrm{i}F}{2}\bigg)^{\mu}\bigg)_{n_{0}}\cdot \mathcal{C}^{2},\nonumber
\end{eqnarray} where
\begin{equation}
\mathcal{C}^{2}=-(\mu+1)\theta^{2}+(\theta^{4}+\cdots+\theta^{2\mu+4}).
\end{equation}
The other possibility is when $n_{1}+n_{3}=0$. In this case we rename $n_{2}$ by $n_{1}$ and get
\begin{equation}\label{0-2-3}
\frac{\mathrm{i}}{2}\sum_{k\neq0}|u_{k}|^{2}\sum_{n_{1}+m_{1}+\cdots+m_{\mu_{12}}=n_{0}}\Delta\cdot \frac{u_{m_{1}}\cdots u_{m_{\mu_{12}}}u_{n_{1}}}{m_{1}\cdots m_{\mu_{1}-1}n_{1}},
\end{equation} where the factor
\begin{eqnarray}\label{1-2-3}
\Delta&=&\psi^{2}\bigg(\frac{k}{N}\bigg)\prod_{i=0}^{1}\psi\bigg(\frac{n_{i}}{N}\bigg)\prod_{i=1}^{\mu_{12}}\psi\bigg(\frac{m_{i}}{N}\bigg)\prod_{i=2}^{\mu_{1}+2}\psi^{2}\bigg(\frac{n_{1}+m_{i,\mu_{12}}}{N}\bigg)\times\\
&\times&k\bigg[\psi^{2}\bigg(\frac{k+n_{1}+m_{\mu_{1}+1,\mu_{12}}}{N}\bigg)\prod_{i=1}^{\mu_{2}}\psi^{2}\bigg(\frac{k+m_{\mu_{1}+i,\mu_{12}}}{N}\bigg)\nonumber\\
&& -\psi^{2}\bigg(\frac{k-n_{1}-m_{\mu_{1}+1,\mu_{12}}}{N}\bigg)\prod_{i=1}^{\mu_{2}}\psi^{2}\bigg(\frac{k-m_{\mu_{1}+i,\mu_{12}}}{N}\bigg)\nonumber\\
&&-\psi^{2}\bigg(\frac{m_{\mu_{1}+1,\mu_{12}}}{N}\bigg)\prod_{i=1}^{\mu_{2}}\psi^{2}\bigg(\frac{k+m_{\mu_{1}+i,\mu_{12}}}{N}\bigg)\nonumber\\
&&+\psi^{2}\bigg(\frac{m_{\mu_{1}+1,\mu_{12}}}{N}\bigg)\prod_{i=1}^{\mu_{2}}\psi^{2}\bigg(\frac{k-m_{\mu_{1}+i,\mu_{12}}}{N}\bigg)\bigg]\nonumber.
\end{eqnarray} Next, we examine the terms in the bracket, which basically can be written, for some $\sigma_{j}$ which are linear combinations of $n_{1}$ and $m_{i}$, as $\prod_{j}\psi^{2}((k+\sigma_{j})/N)-\prod_{j}\psi^{2}((k-\sigma_{j})/N)$. We then replace this expression by $4\theta^{2\mu-1}\eta\sum_{j}\frac{\sigma_{j}}{N}$, where $\mu$ is the number of factors. If we plug into (\ref{1-2-3}) this and the expression of each $\sigma_{j}$, cancel each $n_{1}$ or $m_{i}$ factor with the corresponding denominator in (\ref{0-2-3}), and finally replace each variable other than $k$ by zero, we will get a term which, up to a rearrangement of variables, reads as
\begin{eqnarray}&&2\mathrm{i}\sum_{k\neq 0}\frac{k}{N}|u_{k}|^{2}\sum_{n_{1}+m_{1}+\cdots+m_{\mu_{12}}=n_{0}}\cdot\frac{u_{m_{1}}\cdots u_{m_{\mu_{12}}}}{m_{1}\cdots m_{\mu_{12}}}u_{n_{1}}\times\\
&\times&\bigg(\frac{(\mu_{2}+1)(\mu_{2}+2)}{2}\theta^{2\mu_{2}+3}\eta-\frac{\mu_{2}(\mu_{2}+1)}{2}\theta^{2\mu_{2}+1}\eta\bigg).\nonumber
\end{eqnarray} We may restrict to $k>0$ since $\eta$ is odd, and then sum over $\mu_{1}$ and $\mu_{2}$ to obtain \begin{eqnarray}(\mathcal{R}^{2.2})_{n_{0}}&=&4\mathrm{i}\sum_{k>0}\frac{k\eta}{N}|u_{k}|^{2}\sum_{\mu\geq 0}\sum_{\mu_{2}=0}^{\mu}\frac{(-1)^{\mu}}{2^{\mu+1}(\mu+2)!}\times\\
&\times&(u(\mathrm{i}F)^{\mu})_{n_{0}}(\mu-\mu_{2}+1)\times\nonumber\\
&\times&\bigg(\frac{(\mu_{2}+1)(\mu_{2}+2)}{2}\theta^{2\mu_{2}+3}-\frac{\mu_{2}(\mu_{2}+1)}{2}\theta^{2\mu_{2}+1}\bigg)\nonumber\\
&=&\sum_{k>0}\sum_{\mu\geq 0}\frac{\mathrm{i}k\eta|u_{k}|^{2}}{N(\mu+2)!}\bigg(u\bigg(-\frac{\mathrm{i}F}{2}\bigg)^{\mu}\bigg)_{n_{0}}\cdot \mathcal{D}^{2},\nonumber
\end{eqnarray} where
\begin{equation}
\mathcal{D}^{2}=2\theta^{3}+\cdots+\mu(\mu+1)\theta^{2\mu+1}+(\mu+1)(\mu+2)\theta^{2\mu+3}.
\end{equation}
\subsection{The resonant terms in $\mathcal{K}^{3}$} In the expression (\ref{3-1}), let $n_{1}+n_{2}=0$. Note that $n_{2}<0$, we obtain a term
\begin{equation}\mathrm{i}\sum_{k>0}|u_{k}|^{2}\sum_{m_{1}+\cdots +m_{\mu_{12}}=n_{0}}\Delta\cdot\frac{u_{m_{1}}\cdots u_{m_{\mu_{12}}}}{m_{1}\cdots m_{\mu_{12}}},\end{equation} where the factor
\begin{eqnarray}
\Delta &=&-k\psi_{2}\bigg(\frac{k}{|n_{0}|+|k|}\bigg)\psi^{2}\bigg(\frac{k}{N}\bigg)\prod_{i=2}^{\mu_{1}+1}\psi^{2}\bigg(\frac{m_{i,\mu_{12}}}{N}\bigg)\prod_{i=1}^{\mu_{12}}\psi\bigg(\frac{m_{i}}{N}\bigg)\times\\
&\times&\psi\bigg(\frac{n_{0}}{N}\bigg)\bigg[\prod_{i=1}^{\mu_{2}}\psi^{2}\bigg(\frac{k-m_{\mu_{1}+i,\mu_{12}}}{N}\bigg)-\prod_{i=1}^{\mu_{2}}\psi^{2}\bigg(\frac{k+m_{\mu_{1}+i,\mu_{12}}}{N}\bigg)\bigg].\nonumber
\end{eqnarray} Then we replace the term in the bracket by $-4\theta^{2\mu_{2}-1}\eta\sum_{i}(m_{\mu_{1}+i,\mu_{12}}/N)$, cancel the corresponding $m_{j}$ factor in the denominator, and replace all the variables except $k$ by zero to obtain, after a rearrangement of variables, the sum
\begin{equation}4\mathrm{i}\sum_{k>0}\frac{k}{N}|u_{k}|^{2}\frac{\mu_{2}(\mu_{2}+1)}{2}\theta^{2\mu_{2}+1}\eta\sum_{n_{1}+m_{1}+\cdots+m_{\mu_{12}-1}=n_{0}}\cdot\frac{u_{m_{1}}\cdots u_{m_{\mu_{12}-1}}}{m_{1}\cdots m_{\mu_{12}-1}}u_{n_{1}}.\end{equation} Then we sum over $\mu_{1}$ and $\mu_{2}$ to obtain
\begin{eqnarray}(\mathcal{R}^{3.1})_{n_{0}}&=&4\mathrm{i}\sum_{k>0}\frac{k\eta}{N}|u_{k}|^{2}\sum_{\mu\geq 0}\frac{(-1)^{\mu+1}}{2^{\mu+1}(\mu+2)!}\times\\
&\times&(u(\mathrm{i}F)^{\mu})_{n_{0}}\bigg(\sum_{\mu_{2}=0}^{\mu}\frac{\mu_{2}(\mu_{2}+1)}{2}\theta^{2\mu_{2}+1}\bigg)\nonumber\\
&=&\sum_{k>0}\sum_{\mu\geq 0}\frac{\mathrm{i}k\eta|u_{k}|^{2}}{N(\mu+2)!}\bigg(u\bigg(-\frac{\mathrm{i}F}{2}\bigg)^{\mu}\bigg)_{n_{0}}\cdot \mathcal{D}^{3},\nonumber
\end{eqnarray} where
\begin{equation}
\mathcal{D}^{3}=-\big(2\theta^{3}+\cdots+\mu(\mu+1)\theta^{2\mu+1}+(\mu+1)(\mu+2)\theta^{2\mu+3}\big).
\end{equation}
Next, in the expression (\ref{3-4}), let $n_{2}+n_{3}=0$, note that $n_{3}<0$, we get a term
\begin{equation}\mathrm{i}\sum_{k>0}|u_{k}|^{2}\sum_{n_{1}+m_{1}+\cdots +m_{\mu_{12}-1}=n_{0}}\Delta\cdot\frac{u_{m_{1}}\cdots u_{m_{\mu_{12}-1}}}{m_{1}\cdots m_{\mu_{12}-1}}u_{n_{1}},\end{equation} where the factor
\begin{eqnarray}
\Delta &=&\frac{-k}{|k|+|n_{0}|-n_{1}}\psi_{1}\bigg(\frac{n_{1}}{|k|+|n_{0}|}\bigg)\psi^{2}\bigg(\frac{k}{N}\bigg)\prod_{j=1}^{\mu_{12}-1}\psi\bigg(\frac{m_{j}}{N}\bigg)\prod_{j=0}^{1}\psi\bigg(\frac{n_{j}}{N}\bigg)\times\nonumber\\
&\times&\prod_{j=2}^{i}\psi^{2}\bigg(\frac{n_{1}+m_{j,\mu_{12}-1}}{N}\bigg)\prod_{j=i+1}^{\mu_{1}+1}\psi^{2}\bigg(\frac{k-n_{1}-m_{j-1,\mu_{12}-1}}{N}\bigg)\times\nonumber\\
&\times&\bigg[\prod_{j=1}^{\mu_{2}}\psi^{2}\bigg(\frac{k-m_{\mu_{1}+j-1,\mu_{12}-1}}{N}\bigg)-\prod_{j=1}^{\mu_{2}}\psi^{2}\bigg(\frac{n_{1}+m_{\mu_{1}+j-1,\mu_{12}-1}}{N}\bigg)\bigg]\nonumber
\end{eqnarray} for $1\leq i\leq \mu_{1}$, and 
\begin{eqnarray}
\Delta &=&\frac{-k}{|k|+|n_{0}|-n_{1}}\psi_{1}\bigg(\frac{n_{1}}{|k|+|n_{0}|}\bigg)\psi^{2}\bigg(\frac{k}{N}\bigg)\prod_{j=1}^{\mu_{12}-1}\psi\bigg(\frac{m_{j}}{N}\bigg)\prod_{j=2}^{\mu_{1}+1}\psi^{2}\bigg(\frac{n_{1}+m_{i,\mu_{12}-1}}{N}\bigg)\times\nonumber\\
&\times &\prod_{j=0}^{1}\psi\bigg(\frac{n_{j}}{N}\bigg)\bigg[\prod_{j=1}^{i-\mu_{1}}\psi^{2}\bigg(\frac{m_{\mu_{1}+j,\mu_{12}-1}}{N}\bigg)\prod_{j=i-\mu_{1}+1}^{\mu_{2}}\psi^{2}\bigg(\frac{k-m_{\mu_{1}+j-1,\mu_{12}-1}}{N}\bigg)\nonumber\\
&&-\prod_{j=1}^{i-\mu_{1}}\psi^{2}\bigg(\frac{k+n_{1}+m_{\mu_{1}+j,\mu_{12}-1}}{N}\bigg)\prod_{j=i-\mu_{1}+1}^{\mu_{2}}\psi^{2}\bigg(\frac{n_{1}+m_{\mu_{1}+j-1,\mu_{12}-1}}{N}\bigg)\bigg]\nonumber
\end{eqnarray} for $\mu_{1}+1\leq i\leq \mu_{1}+\mu_{2}$. Then we replace every variable other than $k$ by zero, and sum over $i$ to obtain
\begin{eqnarray}
&&-\mathrm{i}\sum_{k>0}|u_{k}|^{2}\sum_{n_{1}+m_{1}+\cdots +m_{\mu_{12}-1}=n_{0}}\frac{u_{m_{1}}\cdots u_{m_{\mu_{12}-1}}}{m_{1}\cdots m_{\mu_{12}-1}}u_{n_{1}}\times\nonumber\\
&\times &\bigg(\sum_{i=1}^{\mu_{1}}(\theta^{2\mu_{12}-2i+4}-\theta^{2\mu_{1}-2i+4})+\sum_{i=\mu_{1}+1}^{\mu_{12}}(\theta^{2\mu_{12}-2i+2}-\theta^{2i-2\mu_{1}+2})\bigg).\nonumber
\end{eqnarray} We then sum over $\mu_{1}$ and $\mu_{2}$ to get
\begin{eqnarray}(\mathcal{R}^{3.2})_{n_{0}}&=&-\mathrm{i}\sum_{k>0}|u_{k}|^{2}\sum_{\mu\geq 0}\frac{(-1)^{\mu+1}}{2^{\mu+1}(\mu+2)!}(u(\mathrm{i}F)^{\mu})_{n_{0}}\times\\
&\times& \bigg(\sum_{\mu_{2}=0}^{\mu+1}\sum_{i=1}^{\mu+1-\mu_{2}}(\theta^{2\mu-2i+6}-\theta^{2\mu-2i-2\mu_{2}+6})\nonumber\\
&+&\sum_{\mu_{2}=0}^{\mu+1}\sum_{i=\mu+2-\mu_{2}}^{\mu+1}(\theta^{2\mu-2i+4}-\theta^{2i+2\mu_{2}-2\mu})\bigg)\nonumber\\
&=&\sum_{k>0}\sum_{\mu\geq 0}\frac{\mathrm{i}|u_{k}|^{2}}{(\mu+2)!}\bigg(u\bigg(-\frac{\mathrm{i}F}{2}\bigg)^{\mu}\bigg)_{n_{0}}\cdot \mathcal{C}^{3},\nonumber
\end{eqnarray} where
\begin{equation}
\mathcal{C}^{3}=\frac{1}{2}\big((\mu+1)\theta^{2}+(-\mu-1)\theta^{4}+(-\mu+1)\theta^{6}+\cdots +(\mu-1)\theta^{2\mu+4}\big).
\end{equation}
\subsection{The resonant terms in $\mathcal{K}^{4}$} The whole term $\mathcal{K}^{4}$ should be viewed as resonant. Here we simply expand $\mathbb{P}_{0}((Su)^{2})=2\sum_{k>0}\theta^{2}|u_{k}|^{2}$, and replace every variable in (\ref{4-3}) by zero (after extracting the $\prod_{i}m_{i}^{-1}$ factor, as we have done before) to obtain
\begin{eqnarray}(\mathcal{R}^{4})_{n_{0}}&=&\frac{\mathrm{i}}{2}\sum_{k>0}\theta^{2}|u_{k}|^{2}\sum_{\mu\geq 0}\frac{(-1)^{\mu}}{2^{\mu}(\mu+1)!}(u(\mathrm{i}F)^{\mu})_{n_{0}}\times\sum_{\mu_{2}=0}^{\mu}1\\
&=&\sum_{k>0}\sum_{\mu\geq 0}\frac{\mathrm{i}|u_{k}|^{2}}{(\mu+2)!}\bigg(u\bigg(-\frac{\mathrm{i}F}{2}\bigg)^{\mu}\bigg)_{n_{0}}\cdot \mathcal{C}^{4},\nonumber
\end{eqnarray} where
\begin{equation}
\mathcal{C}^{4}=\frac{(\mu+1)(\mu+2)}{2}\theta^{2}.
\end{equation}
\subsection{The resonant terms in $\mathcal{K}^{5}$} In the expression (\ref{5-1}), consider the contribution where $n_{1}+n_{2}=0$, $n_{2}+n_{3}=0$, or where $n_{1}+n_{3}=0$. For each of these cases, we perform the same operation as in the above sections, and collect all the resulting terms (and rearrange the variables) to obtain
\begin{equation}
\mathrm{i}\sum_{k\neq 0}|u_{k}|^{2}\sum_{m_{1}+\cdots +m_{\mu_{13}}+n_{1}=n_{0}}\Delta\cdot\frac{u_{m_{1}}\cdots u_{m_{\mu_{13}}}}{m_{1}\cdots m_{\mu_{13}}}u_{n_{1}},
\end{equation} where the net factor
\begin{eqnarray}
\Delta&=&\psi^{2}\bigg(\frac{k}{N}\bigg)\prod_{i=2}^{\mu_{1}+1}\psi^{2}\bigg(\frac{n_{1}+m_{i,\mu_{13}}}{N}\bigg)\prod_{i=1}^{\mu_{13}}\psi\bigg(\frac{m_{i}}{N}\bigg)\prod_{i=0}^{1}\psi\bigg(\frac{n_{i}}{N}\bigg)\times\nonumber\\
&\times&\bigg[2\psi^{2}\bigg(\frac{k-n_{1}}{N}\bigg)\prod_{i=1}^{\mu_{3}}\psi^{2}\bigg(\frac{k+m_{\mu_{12}+i,\mu_{13}}}{N}\bigg)\prod_{i=2}^{\mu_{2}+1}\psi^{2}\bigg(\frac{n_{1}+m_{\mu_{1}+i,\mu_{13}}}{N}\bigg)\nonumber\\
&&+2\psi^{2}\bigg(\frac{k-n_{1}}{N}\bigg)\prod_{i=1}^{\mu_{3}}\psi^{2}\bigg(\frac{k+m_{\mu_{12}+i,\mu_{13}}}{N}\bigg)\prod_{i=1}^{\mu_{2}}\psi^{2}\bigg(\frac{k+m_{\mu_{1}+i,\mu_{13}}}{N}\bigg)\nonumber\\
&&-2\prod_{i=1}^{\mu_{3}}\psi^{2}\bigg(\frac{k+m_{\mu_{12}+i,\mu_{13}}}{N}\bigg)\prod_{i=1}^{\mu_{2}+1}\psi^{2}\bigg(\frac{m_{\mu_{1}+i,\mu_{13}}}{N}\bigg)\nonumber\\
&&+\prod_{i=1}^{\mu_{3}}\psi^{2}\bigg(\frac{n_{1}+m_{\mu_{12}+i,\mu_{13}}}{N}\bigg)\prod_{i=2}^{\mu_{2}+1}\psi^{2}\bigg(\frac{n_{1}+m_{\mu_{1}+i,\mu_{13}}}{N}\bigg)\nonumber\\
&&+\prod_{i=1}^{\mu_{3}}\psi^{2}\bigg(\frac{n_{1}+m_{\mu_{12}+i,\mu_{13}}}{N}\bigg)\prod_{i=1}^{\mu_{2}}\psi^{2}\bigg(\frac{n_{1}+m_{\mu_{1}+i,\mu_{13}}}{N}\bigg)\nonumber\\
&&-2\prod_{i=1}^{\mu_{3}}\psi^{2}\bigg(\frac{n_{1}+m_{\mu_{12}+i,\mu_{13}}}{N}\bigg)\prod_{i=1}^{\mu_{2}+1}\psi^{2}\bigg(\frac{k+n_{1}+m_{\mu_{1}+i,\mu_{13}}}{N}\bigg)\nonumber\\
&&-2\prod_{i=1}^{\mu_{3}}\psi^{2}\bigg(\frac{k+m_{\mu_{12}+i,\mu_{13}}}{N}\bigg)\prod_{i=1}^{\mu_{2}+1}\psi^{2}\bigg(\frac{k+n_{1}+m_{\mu_{1}+i,\mu_{13}}}{N}\bigg)\bigg].\nonumber
\end{eqnarray} Then we replace each variable other than $k$ by zero, obtaining
\begin{eqnarray}
&&\mathrm{i}\sum_{k\neq 0}|u_{k}|^{2}\sum_{m_{1}+\cdots +m_{\mu_{13}}+n_{1}=n_{0}}\frac{u_{m_{1}}\cdots u_{m_{\mu_{13}}}}{m_{1}\cdots m_{\mu_{13}}}u_{n_{1}}\times\\
&\times&2\theta^{2}(\theta^{2\mu_{3}+2}-\theta^{2\mu_{2}+2}+1-\theta^{2\mu_{3}}).\nonumber
\end{eqnarray} Again we restrict to $k>0$ and sum over $\mu_{1},\mu_{2},\mu_{3}$ to get
\begin{eqnarray}(\mathcal{R}^{5})_{n_{0}}&=&4\mathrm{i}\sum_{k>0}|u_{k}|^{2}\sum_{\mu\geq 0}\frac{(-1)^{\mu+1}}{2^{\mu+2}(\mu+2)!}(u(iF)^{\mu})_{n_{0}}\times\nonumber\\
&\times &\sum_{\mu_{1}+\mu_{2}+\mu_{3}=\mu}(\theta^{2\mu_{3}+4}-\theta^{2\mu_{2}+4}+\theta^{2}-\theta^{2\mu_{3}+2})\nonumber\\
&=&\sum_{k>0}\sum_{\mu\geq 0}\frac{\mathrm{i}|u_{k}|^{2}}{(\mu+2)!}\bigg(u\bigg(-\frac{\mathrm{i}F}{2}\bigg)^{\mu}\bigg)_{n_{0}}\cdot \mathcal{C}^{5},\nonumber
\end{eqnarray} where 
\begin{equation}
\mathcal{C}^{5}=-\frac{\mu(\mu+1)}{2}\theta^{2}+\mu\theta^{4}+(\mu-1)\theta^{6}+\cdots +2\theta^{2\mu}+\theta^{2\mu+2}.
\end{equation}
\subsection{When put together...} Now we can directly verify from the above computations that
\begin{equation}\mathcal{C}^{1}+\mathcal{C}^{2}+\mathcal{C}^{3}+\mathcal{C}^{4}+\mathcal{C}^{5}=0;\end{equation}
\begin{equation}\mathcal{D}^{2}+\mathcal{D}^{3}=0,\end{equation} which then implies
\begin{equation}\mathcal{R}^{1}+\mathcal{R}^{2.1}+\mathcal{R}^{2.2}+\mathcal{R}^{3.1}+\mathcal{R}^{3.2}+\mathcal{R}^{4}+\mathcal{R}^{5}=0.\end{equation}
\subsection{What remains?} Here we analyze what remains after we subtract from each $\mathcal{K}^{j}$ term the resonant contribution, and deduce a second version of the equation satisfied by $w$. To simplify the argument, we need to introduce a few more notions.
\begin{definition}We say a function $f:\mathbb{Z}\to \mathbb{R}$ is \emph{slowly varying of type $1$}, or $f\in SV_{1}$, if we have $|f(n)|\leq C$ and 
\begin{equation}|f(n+1)-f(n)|\leq C\langle n\rangle^{-1}
\end{equation} for some constant $C$. We say $f$ is \emph{slowly varying of type $2$}, or $f\in SV_{2}$, if we have \begin{equation}|f(n+1)-f(n)|\leq C\langle n\rangle^{-1}(|f(n)|+|f(n+1)|)
\end{equation} for some constant $C$. For a function $f:\mathbb{Z}^{\mu}\to \mathbb{R}$, we say it is slowly varying of type $1$ or $2$ if it verifies the above inequalities for each single variable when the other variables are fixed, with uniformly bounded constants.
\begin{proposition}\label{gene} The following functions are in $SV_{1}$:

(1) function of the form $\phi(f_{1},\cdots,f_{k})$, where $f_{j}\in SV_{1}$, $\phi:\mathbb{R}^{k}\to\mathbb{R}$ is Lipschitz; 

(2) function of the form $\phi(f_{1},\cdots,f_{k})$, where $f_{j}\in SV_{2}$, $\phi$ is smooth and is constant outside some compact set.

The following functions are in $SV_{2}$:

(3) any monomial (say $n_{1}^{2}$ or $n_{2}n_{3}$), or characteristic function of any set generated by $\{n_{j}>0\}$ and $\{n_{j}<0\}$;

(4) product or reciprocal of functions in $SV_{2}$ (with $1/f$ defined to be $1$ at points where $f=0$); $\max(f,g)$, $\min(f,g)$ or $f+g$ for nonnegative $f,g\in SV_{2}$;

(5) function of the form $|f|$, $\langle f\rangle$ or $(\max(f,0))^{\lambda}$, where $f\in SV_{2}$ and $\lambda>0$.
\end{proposition}
\begin{proof} Omitted.
\end{proof}
\end{definition}
\begin{proposition}\label{propneweqn} We have \begin{equation}\label{neweqn1}(\partial_{t}-\mathrm{i}\partial_{xx})w=\mathcal{H}=\sum_{\mu}C_{\mu}\mathcal{H}_{\mu},\end{equation}where $|C_{\mu}|\leq C^{\mu}/\mu!$ with some absolute constant $C$, and $\mathcal{H}_{\mu}=\mathcal{H}_{\mu}^{2}+\mathcal{H}_{\mu}^{3}+\mathcal{H}_{\mu}^{4}$. The $\mathcal{H}^{j}$ terms can be written as
\begin{equation}\label{neweqn3}(\mathcal{H}_{\mu}^{2})_{n_{0}}=\mathrm{i}\sum_{n_{1}+n_{2}+m_{1}+\cdots +m_{\mu}=n_{0}}\min\{\langle n_{0}\rangle,\langle n_{1}\rangle,\langle n_{2}\rangle\}\cdot\Theta_{\mu}^{2}\prod_{l=1}^{2}u_{n_{l}}\prod_{i=1}^{\mu}\frac{u_{m_{i}}}{m_{i}};\end{equation}
\begin{equation}\label{neweqn4}(\mathcal{H}_{\mu}^{j})_{n_{0}}=\mathrm{i}\sum_{n_{1}+\cdots +n_{j}+m_{1}+\cdots +m_{\mu}=n_{0}}\Theta_{\mu}^{j}\prod_{l=1}^{j}u_{n_{l}}\prod_{i=1}^{\mu}\frac{u_{m_{i}}}{m_{i}},\,\,\,\,\,\,\,\,j\in\{3,4\},\end{equation}for positive $n_{0}$. For each $(\mu,j)$, the function \begin{equation}\Theta_{\mu}^{j}=\Theta_{\mu}^{j}(n_{0},n_{1},\cdots,n_{j},m_{1},\cdots,m_{\mu}),\,\,\,\,\,\,\,\,j\in\{2,3,4\},\nonumber\end{equation} is a linear combination of products $\mathbf{1}_{E}\cdot\Theta$, where $E$ is some set generated by the sets $\{n_{h}+n_{l}=0\}, 1\leq h<l\leq j$, and $\Theta$ is\footnote[1]{Later we may slightly abuse the notation and use the term ``$\Theta$ factor'' or ``$\Theta^{j}$ factor'' to refer to both the $\Theta_{\mu}^{j}$ and the $\Theta$ here. Moreover, in the $SV_{1}$ bound and other estimates (see for example (\ref{neglect}) below) we may pick up factors depending on $\mu$; but they are clearly at most $O(1)^{\mu}$ so can be safely absorbed into the $C_{\mu}$ factor.} \emph{slowly varying of type $1$}; note in particular they are \emph{real valued}. Moreover we have the following:

(i) When $j=2$, $\Theta$ is nonzero only when \begin{equation}\label{neglect}\max_{i}\langle m_{i}\rangle\ll (\mu+1)^{-2}\min\big\{\langle n_{0}\rangle,\langle n_{1}\rangle,\langle n_{2}\rangle\big\}.\end{equation}

(ii) When $j=3$, if $E$ is contained in $\{n_{1}+n_{2}=0\}$ \emph{but not} $\{|n_{1}|=|n_{2}|=|n_{3}|\}$, we must have \begin{equation}|\Theta|\lesssim\min\bigg\{1,\frac{\langle n_{0}\rangle+\langle n_{3}\rangle}{\langle n_{1}\rangle}\bigg\}.\end{equation} The same holds for other permutations of $(1,2,3)$.

(iii) When $j=4$, we have
\begin{equation}|\Theta|\lesssim\big(\max_{0\leq l\leq 4}\langle n_{l}\rangle\big)^{-1}.\end{equation}
\end{proposition}
\begin{proof}The estimate on the coefficients $C_{\mu}$, whose choice will be clear from the expressions we have, is elementary based on the factorial decay we have, and the simple observation that
\begin{equation}\frac{(\mu_{1}+\cdots +\mu_{k})!}{\mu_{1}!\cdots \mu_{k}!}\leq k^{\mu_{1}+\cdots +\mu_{k}},\end{equation} where in practice we always have (say) $k\leq30$. Next we shall examine the terms left after the subtraction of resonant ones, and define the $\Theta$ factors. We will first prove the boundedness of $\Theta$ and properties (i), (ii), (iii), and then show that $\Theta\in SV_{1}$.

Before proceeding, let us make one useful observation. If we have a term (temporarily called term of type $R$ for convenience) of type (\ref{neweqn3}) in which the $\Theta$ factor is bounded and is accompanied by some $\mathbf{1}_{E}$ with $E\subset\{n_{1}+n_{2}\neq 0\}$, then we can use a smooth cutoff (similar to $\psi_{1}$ or $\psi_{2}$) to separate the part where (\ref{neglect}) holds, and the part where $\langle m_{i}\rangle\gtrsim(\mu+1)^{-2}\min\{\langle n_{0}\rangle,\langle n_{1}\rangle,\langle n_{2}\rangle\}$ for some $1\leq i\leq \mu$; in the former case we have $\mathcal{H}_{\mu}^{2}$, and in the latter case we promote $m_{i}$ and rename it $n_{3}$ to obtain $\mathcal{H}_{\mu-1}^{3}$ (since here the $\Theta^{3}$ factor is bounded, $n_{1}+n_{2}\neq 0$, and if $n_{l}+n_{3}=0$ for $l\in\{1,2\}$, then the $\Theta^{3}$ factor will have an $n_{3}$ on the denominator, and at most $\langle n_{0}\rangle$ on the numerator). Also we may assume that all the $n_{l}$ and $m_{i}$ variables are nonzero and $\lesssim N$.

The first contribution we need to consider is when none of the equalities we proposed in obtaining the $\mathcal{R}^{j}$ terms hold; these include the contribution from each $\mathcal{K}^{j}$ which we discuss separately.

For the part in $\mathcal{K}_{\mu_{1}i}^{1}$ we have $n_{1}+n_{2}\neq 0$. If $\langle n_{2}\rangle\gtrsim\langle n_{0}\rangle+\langle n_{1}\rangle$, then we have a term of type $R$ and obtain either $\mathcal{H}_{\mu_{1}-1}^{2}$ or $\mathcal{H}_{\mu_{1}-2}^{3}$. Now if $\langle n_{2}\rangle \ll\langle n_{0}\rangle+\langle n_{1}\rangle$, then $\langle m_{j}\rangle \gtrsim\langle n_{0}\rangle +\langle n_{1}\rangle$ for at least one $j$, so we can promote that $m_{j}$ and rename it $n_{3}$ to obtain $\mathcal{H}_{\mu_{1}-2}^{3}$, due to a similar argument as above and the restriction $n_{1}+n_{2}\neq 0$.

For the part of $\mathcal{K}_{\mu_{1}\mu_{2}}^{2}$, no $n_{h}+n_{l}=0$ happens. In the expression (\ref{2-3}), first assume $\langle n_{3}\rangle$ (or, by symmetry, $\langle n_{1}\rangle$) is $\lesssim\min\{\langle n_{0}\rangle ,\langle n_{1}\rangle,\langle n_{3}\rangle\big\} $, then the first two terms in the bracket on the right hand side of (\ref{2-3}) contributes at most $O(\langle n_{3}\rangle)$, so for this term we may relegate $n_{2}$ (rename it by some $m_{i}$) to obtain a term of type $R$. For the last two terms in the bracket, the contribution is at most $N^{-1}\langle n_{1}\rangle(\langle n_{2}\rangle+\langle n_{3}\rangle)$, which is a sum of two terms. One of them is at most $\langle n_{3}\rangle$ and can be treated as above; the other can be cancelled by the $n_{2}^{-1}$ factor and we get $\mathcal{H}_{\mu_{12}}^{3}$ (since we have pre-assumed that no $n_{h}+n_{l}$ can be zero). 

Next suppose (say) $\langle n_{0}\rangle\ll\langle n_{3}\rangle\ll\langle n_{1}\rangle$. In this case the first two terms in bracket on the right hand side of (\ref{2-3}) contributes at most $\langle n_{3}\rangle$, and at least one of $\langle m_{j}\rangle$ or $\langle n_{2}\rangle$ must be $\gtrsim\langle n_{1}\rangle$ here, so we get $\mathcal{H}_{\mu_{12}-1}^{3}$ after making appropriate promotion or relegations; the last two terms contribute at most $N^{-1}\langle n_{1}\rangle(\langle n_{2}\rangle+\langle n_{3}\rangle)$, which is bounded either by $\langle n_{3}\rangle$ (which can be treated the same way as above), or $N^{-1}\langle n_{1}\rangle\langle n_{2}\rangle$ (which is cancelled by the $n_{2}^{-1}$ to obtain $\mathcal{H}_{\mu_{12}-1}^{3}$).

The only remaining possibility is $\langle n_{0}\rangle\ll \langle n_{1}\rangle\sim\langle n_{3}\rangle$. we may write
\begin{equation}\tau_{1}(n)=\prod_{i=1}^{\mu_{2}}\psi^{2}\bigg(\frac{n+m_{\mu_{1}+i,\mu_{12}}}{N}\bigg)\end{equation} and \begin{equation}\tau_{2}(n)=\psi^{2}\bigg(\frac{n+n_{2}+m_{\mu_{1}+1,\mu_{12}}}{N}\bigg)\tau_{1}(n),\end{equation} so the net contribution in the bracket will be 
\begin{equation}(n_{3}\tau_{2}(n_{3})+n_{1}\tau_{2}(n_{1}))-\psi^{2}(n_{3}\tau_{1}(n_{3})+n_{1}\tau_{1}(n_{1}))\nonumber\end{equation} with some factor $\psi$. Since we can write 
\begin{equation}\label{n0small}n_{3}\tau_{j}(n_{3})+n_{1}\tau_{j}(n_{1})=(n_{1}+n_{3})\tau_{j}(n_{3})+n_{1}(\tau_{j}(n_{1})-\tau_{j}(n_{3})),\end{equation} and $n_{1}+n_{3}$ is a linear combination of $n_{0}$, $n_{2}$ and $m_{i}$, the first term on the right hand side of (\ref{n0small}) will be bounded either by $\langle n_{0}\rangle$ (in which case we have a term of type $R$),or by $\langle n_{2}\rangle$ (in which case we obtain $\mathcal{H}_{\mu_{12}-1}^{3}$), or by some $\langle m_{j}\rangle$ (in which case we relegate $n_{2}$ and promote $m_{j}$ to obtain $\mathcal{H}_{\mu_{12}-1}^{3}$ under the restriction $\langle n_{1}\rangle\sim\langle n_{3}\rangle$). The contribution of the second term will be bounded by $N^{-1}\langle n_{1}\rangle$ times either $\langle n_{0}\rangle$ (in which case we have a term of type $R$), $\langle n_{2}\rangle$ (in which case we have a part of $\mathcal{H}_{\mu_{12}-1}^{3}$), or some $\langle m_{j}\rangle$ (in which case we relegate $n_{2}$ and promote $m_{j}$ to get $\mathcal{H}_{\mu_{12}-1}^{3}$).

For the part of $\mathcal{K}_{\mu_{1}\mu_{2}0}^{3}$ we have $n_{1}+n_{2}\neq 0$. By the assumptions about this term, if $\langle n_{0}\rangle\gtrsim\langle n_{2}\rangle$, we will have a term of type $R$. Now assume $\langle n_{0}\rangle\ll\langle n_{2}\rangle$, we can extract from the bracket in (\ref{3-3}) a factor of $n_{0}/N$ or $m_{i}/N$. If we have an $n_{0}/N$ factor then the net $\Theta$ factor will be $\lesssim\langle n_{0}\rangle$ and we again have a term of type $R$; if we have an $m_{i}/N$ factor then we may cancel this with the $1/m_{i}$ factor, promote this $m_{i}$ and rename it $n_{3}$, to obtain $\mathcal{H}_{\mu_{12}-2}^{3}$. Notice that in this case the $\Theta$ factor is bounded by $\langle n_{2}\rangle/N\lesssim 1$, $n_{1}+n_{2}\neq 0$, and if $n_{2}+n_{3}=0$, we must have $\langle n_{1}\rangle \gtrsim\langle n_{2}\rangle$.

For the part of $\mathcal{K}_{\mu_{1}\mu_{2}i}^{3}$ we have $n_{2}+n_{3}\neq 0$. We claim that this part is $\mathcal{H}_{\mu_{12}-1}^{3}$. In fact, this will be the case if both $n_{1}+n_{3}$ and $n_{1}+n_{2}$ are nonzero since the $\Theta$ factor is bounded; if $n_{1}+n_{3}=0$, then from the assumptions about the $\mathcal{K}_{\mu_{1}\mu_{2}i}^{3}$ term we have $\langle n_{0}\rangle\gtrsim\langle n_{3}\rangle$, so we also have $\mathcal{H}_{\mu_{12}-1}^{3}$; if $n_{1}+n_{2}=0$, then either $\langle n_{0}\rangle$ or $\langle n_{3}\rangle$ must be $\gtrsim\langle n_{1}\rangle$, so we still have $\mathcal{H}_{\mu_{12}-1}^{3}$.
 
For the part of $\mathcal{K}_{\mu_{1}\mu_{2}\mu_{3}}^{5}$, no $n_{h}+n_{l}=0$ happens. In this case the $\Theta$ factor is clearly bounded, thus we obtain $\mathcal{H}_{\mu_{13}}^{3}$.

Next, we have the ``error term'' which is some resonant contribution in $\mathcal{K}^{j}$ (for example, the contribution in $\mathcal{K}_{\mu_{1}i}^{1}$ where $n_{1}+n_{2}=0$) minus the corresponding $\mathcal{R}^{j}$. In this term we may specify some $k$ (for example, in the term corresponding to $\mathcal{K}_{\mu_{1}i}^{1}$ we will have $n_{1}=-k$ and $n_{2}=k$). From the computations made before, we can see that the corresponding terms may be written in an appropriate form so that the $\Theta$ factor is bounded even without subtracting $\mathcal{R}^{j}$. Note that here we may need to promote some $m_{i}$ so that we can include $m_{i}^{-1}$ in $\Theta$ to cancel certain factors (for example when dealing with $\mathcal{K}_{\mu_{1}i}^{1}$). Therefore, \emph{before subtracting the $\mathcal{R}^{j}$ terms}, the resonant contributions can be written in the form of (\ref{neweqn4}), with $j=3$, the $\Theta$ factor bounded, and (say) $n_{1}=-k$, $n_{2}=k$. In particular, if $\langle n_{0}\rangle +\langle n_{3}\rangle\gtrsim \langle k\rangle$, we will obtain $\mathcal{H}^{3}$ and subtraction of $\mathcal{R}^{j}$ will not affect this. Now we assume $\langle n_{0}\rangle +\langle n_{3}\rangle\ll \langle k\rangle$.

After the subtraction of the $\mathcal{R}^{j}$ factors, the $\Theta$ will remain bounded; moreover, it can be checked case-by-case that in the remaining term, we gain an additional factor of \begin{equation}\label{factor00}\min\bigg\{1,\frac{1}{\langle k\rangle}\bigg(\langle n_{0}\rangle+\langle n_{3}\rangle+\sum_{i=1}^{\mu}\langle m_{i}\rangle\bigg)\bigg\},\end{equation} if $n_{1}=-k$ and $n_{2}=k$. For example, say we are replacing $\prod_{j}\psi^{2}((k+\sigma_{j})/N)-\prod_{j}\psi^{2}((k-\sigma_{j})/N)$ by $4\theta^{2\mu-1}\eta\sum_{j}\frac{\sigma_{j}}{N}$, then the error term we introduce is at most $O(N^{-2}\langle\sigma_{j}\rangle^{2})$, which is then at most $O(N^{-2}\langle n_{l}\rangle^{2})$ or $O(N^{-2}\langle m_{i}\rangle^{2})$ for some $i$ and $l\in\{0,3\}$. Since this contribution can be cancelled by other factors to produce a bounded $\Theta$ even if we replace the power of $2$ by $1$ (which will be the case if we do not subtract the $\mathcal{R}^{j}$), we will have in the error term an additional factor as in (\ref{factor00}). The other factors are treated in the same way, provided that in some cases we replace the $N$ on the denominator by something larger than $\langle k\rangle$. This guarantees that either we obtain $\mathcal{H}^{3}$, or we may promote some $m_{i}$ to obtain $\mathcal{H}^{4}$.

Next, notice that in obtaining $\mathcal{R}^{2.1}$, we have discarded the last two terms in the bracket on the right hand side of (\ref{factor2.1}). However, they add up to produce a factor of at most $N^{-1}\langle n_{1}\rangle$, thus they can be included in $\mathcal{H}^{3}$. Finally, there are terms where at least \emph{two} of the proposed equalities hold (these term appear due to inclusion-exclusion principle), for example we have the term where $n_{1}+n_{2}=n_{2}+n_{3}=0$ in $\mathcal{K}_{\mu_{1}\mu_{2}\mu_{3}}^{5}$, but by the discussion above, the corresponding $\Theta$ factor will be bounded, thus they can also be included in $\mathcal{H}^{3}$.

Now we only need to show $\Theta\in SV_{1}$. This will follow from Proposition \ref{gene}, since it can be checked that all the $\Theta$ factors are formed using rules (1) through (5) in that proposition, with rule (2) used at least once (in particular, all the cut-off factors we introduce will be in $SV_{1}$).
\end{proof}
\section{The gauge transform III: The final substitution}\label{gaugetransform3} Starting from equations (\ref{neweqn1}) and (\ref{neweqn3}), we need to make further substitutions before we can state and prove the main estimates. Here we introduce one more notation, namely when we write $g^{\omega}$ for a function $g$, where $\omega\in\{-1,1\}$, this will mean $g$ if $\omega=-1$, and $\overline{g}$ if $\omega=1$. Also in the following, we will use the letter $\upsilon$ to represent a function that can be either $u$ or $v$.
\subsection{From $u$ to $w$} Recall that $v=Mu$ and $w=\mathbb{P}_{+}v$, we have \begin{equation}u=\sum_{\mu}\frac{(-\mathrm{i})^{\mu}}{2^{\mu}\mu!}P^{\mu}v,\nonumber\end{equation} which then implies, for $n>0$,
\begin{equation}\label{possub}u_{n}=\sum_{\mu}\frac{(-1)^{\mu}}{2^{\mu}\mu!}\sum_{n_{1}+m_{1}+\cdots +m_{\mu}=n}\Psi_{\mu}\cdot v_{n_{1}}\prod_{i=1}^{\mu}\frac{u_{m_{i}}}{m_{i}},\end{equation} where \begin{equation}\Psi_{\mu}=\Psi_{\mu}(n,n_{1},m_{1},\cdots,m_{\mu})\nonumber\end{equation} is a product of $\psi$ factors. When $n<0$, since $u_{n}=\overline{u_{-n}}$, we have instead
\begin{equation}\label{negsub}u_{n}=\sum_{\mu}\frac{1}{2^{\mu}\mu!}\sum_{n_{1}+m_{1}+\cdots +m_{\mu}=n}\Psi_{\mu}\cdot (\overline{v})_{n_{1}}\prod_{i=1}^{\mu}\frac{u_{m_{i}}}{m_{i}},\end{equation} where we note $(\overline{v})_{n}=\overline{v_{-n}}$. By replacing each $u_{n_{l}}$ in (\ref{neweqn3}) with one of the above expressions, we can prove
\begin{proposition}\label{intereqn} We have \begin{equation}\label{intereqn0}(\partial_{t}-\mathrm{i}\partial_{xx})w=\mathcal{J}=\sum_{\mu}C_{\mu}\mathcal{J}_{\mu},\end{equation}where $|C_{\mu}|\lesssim C^{\mu}/\mu!$, the nonlinearity is written as\begin{equation}\mathcal{J}_{\mu}=\sum_{j\in\{2,3,3.5,4,4.5\}}\sum_{\omega\in\{-1,1\}^{[j]}}\mathcal{J}_{\mu}^{\omega j}.\end{equation} The terms are then
\begin{equation}\label{3term}(\mathcal{J}_{\mu}^{\omega j})_{n_{0}}=\mathrm{i}\sum_{n_{1}+\cdots +n_{j}+m_{1}+\cdots +m_{\mu}=n_{0}}\phi_{\mu}^{j}\prod_{l=1}^{j}(w^{\omega_{l}})_{n_{l}}\prod_{i=1}^{\mu}\frac{u_{m_{i}}}{m_{i}}\end{equation} for $j\in\{2,3\}$;
\begin{equation}\label{4.5term}(\mathcal{J}_{\mu}^{\omega j})_{n_{0}}=\mathrm{i}\sum_{n_{1}+\cdots+n_{[j]}+m_{1}+\cdots +m_{\mu}=n_{0}}\phi_{\mu}^{j}\prod_{l=1}^{[j]}(\upsilon^{\omega_{l}})_{n_{l}}\prod_{i=1}^{\mu}\frac{u_{m_{i}}}{m_{i}}\end{equation}for $j\in\{3.5,4,4.5\}$. Here the real valued weights
\begin{equation}\phi_{\mu}^{j}=\phi_{\mu}^{j}(n_{0},n_{1},\cdots,n_{[j]},m_{1},\cdots,m_{\mu}),
\end{equation} where $j\in\{2,3,3.5,4,4.5\}$, verify the following.

(i) When $j=2$, we have \begin{equation}|\phi_{\mu}^{2}|\lesssim \min\{\langle n_{0}\rangle,\langle n_{1}\rangle,\langle n_{2}\rangle\},\nonumber\end{equation} also $\phi_{\mu}^{2}$ is nonzero only when \begin{equation}\min\{\langle n_{0}\rangle,\langle n_{1}\rangle,\langle n_{2}\rangle\}\gg(\mu+1)^{2}\max_{i}\langle m_{i}\rangle.\nonumber\end{equation}

(ii) When $j=3$, we have $|\phi_{\mu}^{3}|\lesssim 1$; also when $n_{1}+n_{2}=0$, and neither $n_{1}$ or $n_{2}$ is related to $n_{3}$ by $m$ (here and after, we say two $n$ variables are ``related by $m$'', if their sum of difference belongs to some fixed, finite set of linear combinations of the $m$ variables), we will have 
\begin{equation}\label{bound3.0}\big|\phi_{\mu}^{3}\big|\lesssim \min\bigg\{1,\frac{\langle n_{0}\rangle+\langle n_{3}\rangle}{\langle n_{1}\rangle}\bigg\},\nonumber\end{equation} and the estimate also holds for other permutations of $(1,2,3)$. Also, when all three of $(n_{1},n_{2},n_{3})$ are related by $m$, we are allowed to have $(\upsilon^{\omega_{l}})_{n_{l}}$ instead of $(w^{\upsilon_{l}})_{n_{l}}$ in (\ref{3term}) for $j=3$.

(iii) When $j=3.5$, we have\begin{equation}\label {bound3.5}|\phi_{\mu}^{3.5}|\lesssim\frac{\min\{\langle n_{0}\rangle,\langle n_{1}\rangle,\langle n_{2}+n_{3}\rangle\}}{\max\{\langle n_{2}\rangle,\langle n_{3}\rangle\}}.\end{equation} Moreover, we can replace the $\upsilon$ in $(\upsilon^{\omega_{1}})_{n_{1}}$ in (\ref{4.5term}) for $j=3.5$ by $w$; also, if \begin{equation}(\max_{0\leq l\leq 3}\langle n_{l}\rangle)^{\frac{1}{2}}\ll\min_{0\leq l\leq 3}\langle n_{l}\rangle,\nonumber\end{equation} then $n_{2}$ and $n_{3}$ must have opposite sign.

(iv) When $j=4$, we have \begin{equation}|\phi_{\mu}^{4}|\lesssim \big(\max_{0\leq l\leq 4}\langle n_{l}\rangle^{\frac{1}{20}}+\min_{1\leq l\leq 4}\langle n_{l}\rangle\big)^{-1}.\nonumber\end{equation}

(v) When $j=4.5$, we have $n_{1}+n_{2}\neq 0$, and\begin{equation}|\phi_{\mu}^{4.5}|\lesssim \big(\langle n_{3}\rangle+\langle n_{4}\rangle\big)^{-1},\nonumber\end{equation} and this factor is nonzero only if \begin{equation}\max\{\langle n_{3}\rangle,\langle n_{4}\rangle\}+\max_{i}\langle m_{i}\rangle\ll(\max\{\langle n_{0}\rangle,\langle n_{1}\rangle,\langle n_{2}\rangle\})^{\frac{1}{10}};\nonumber\end{equation} also, whenever \begin{equation}\langle n_{l}\rangle\gtrsim(\max\{\langle n_{0}\rangle,\langle n_{1}\rangle,\langle n_{2}\rangle\})^{\frac{1}{10}}\nonumber\end{equation}for some $l\in\{1,2\}$, we can replace the $\upsilon$ in $(\upsilon^{\omega_{l}})_{n_{l}}$ in (\ref{4.5term}) for $j=4.5$ by $w$.

(vi) When $j=3$, suppose $n_{0}=n_{1}=n, -n_{2}=n_{3}=k$, and $\langle k\rangle \ll(\mu+3)^{-11}\langle n\rangle$, then the $\phi_{\mu}^{3}$ factor will be a function of $n$, $k$ and other variables. This function can then be divided into two parts, with the first part satisfying
\begin{equation}\label{ysx}|\phi_{\mu}^{3}(n,k,m_{1},\cdots,m_{\mu})|\lesssim\frac{\min\{\langle n\rangle,\langle k\rangle\}}{\max\{\langle n\rangle,\langle k\rangle\}},\end{equation} and the second part satisfying \begin{equation}\label{sbx}|\phi_{\mu}^{3}(n,k,m_{1},\cdots,m_{\mu})-\phi_{\mu}^{3}(n+1,k,m_{1},\cdots,m_{\mu})|\lesssim\langle n\rangle^{-1}.\end{equation}
\end{proposition}
\begin{proof} We will first prove (i) through (v) as well as (\ref{ysx}); the proof of (\ref{sbx}) will be left to the end. Since each $\mathcal{H}^{4}$ term is also a $\mathcal{J}^{4}$ term, we only need to consider the expressions (\ref{neweqn3}) and (\ref{neweqn4}) with $j\in\{2,3\}$. We replace each $u_{n_{l}}$, where $1\leq l\leq j$, by either (\ref{possub}) or (\ref{negsub}), depending on whether $n_{l}$ is positive or negative, to obtain
\begin{equation}\mathcal{H}_{\mu_{0}}^{j}=\sum_{\omega;\mu_{1},\cdots,\mu_{j}}C_{\mu_{0}}\frac{\omega_{1}^{\mu_{1}}\cdots\omega_{j}^{\mu_{j}}}{2^{\mu_{1j}}\mu_{1}!\cdots \mu_{j}!}\mathcal{H}_{\mu_{0}\cdots\mu_{j}}^{\omega j}\end{equation} for all $\mu_{0}$ and $j\in\{2,3\}$, where $\omega=(\omega_{1},\cdots,\omega_{j})\in\{-1,1\}^{j}$, and
\begin{equation}\label{neweqn5}(\mathcal{H}_{\mu_{0}\cdots\mu_{j}}^{\omega j})_{n_{0}}=\sum_{\mathbf{w}\in V_{n_{0},\mu_{0}\cdots\mu_{j}}^{\omega j}}\Theta_{\mathbf{w}}^{\mu_{0}\cdots\mu_{j}j}\prod_{l=1}^{j}(v^{\omega_{l}})_{n_{l}'}\prod_{l=0}^{j}\prod_{i=1}^{\mu_{l}}\frac{u_{(m^{l})_{i}}}{(m^{l})_{i}}.\end{equation} Here the frequency set
\begin{eqnarray}\label{weighttt}
V_{n_{0},\mu_{0}\cdots\mu_{j}}^{\omega j} &=&\big\{\mathbf{w}=\big((n_{l},n_{l}')_{1\leq l \leq j},((m^{l})_{i})_{1\leq i\leq \mu_{l};1\leq l\leq j}\big):\\
&&n_{l}=n_{l}'+(m^{l})_{1\mu_{l}},\omega_{l}n_{l}<0;\nonumber\\
&&n_{1}+\cdots +n_{j}+(m^{0})_{1\mu_{0}}=n_{0}\big\}.\nonumber
\end{eqnarray} note that the free variables are $n_{l}'$ and $(m^{l})_{i}$, and they satisfy a constraint
\begin{equation}\sum_{l=1}^{j}n_{l}'+\sum_{l=0}^{j}\sum_{i=1}^{\mu_{l}}(m^{l})_{i}=n_{0}\nonumber\end{equation} as well as several inequalities. Also the weight is
\begin{eqnarray}\label{coef}\Theta_{\mathbf{w}}^{\mu_{0}\mu_{1}\mu_{2}2}&=&\Theta_{\mu_{0}}^{2}(n_{0},n_{1},n_{2},(m^{0})_{1},\cdots,(m^{0})_{\mu_{0}})\times\\
&\times&\min_{0\leq l\leq 2}\langle n_{l}\rangle\cdot\prod_{l=1}^{j}\Psi_{\mu_{l}}(n_{l},n_{l}',(m^{l})_{1},\cdots,(m^{l})_{\mu_{l}});\nonumber\end{eqnarray}
\begin{eqnarray}\label{coef2}\Theta_{\mathbf{w}}^{\mu_{0}\cdots\mu_{3}3}&=&\Theta_{\mu_{0}}^{3}(n_{0},\cdots,n_{3},(m^{0})_{1},\cdots,(m^{0})_{\mu_{0}})\times\\
&\times&\prod_{l=1}^{j}\Psi_{\mu_{l}}(n_{l},n_{l}',(m^{l})_{1},\cdots,(m^{l})_{\mu_{l}}).\nonumber\end{eqnarray} 
Our argument will be an enumerative examination of all the possible terms, and this can be greatly simplified with the following lemma, which we will assume for now, and prove after the proof of this main proposition.
\begin{lemma}\label{onlyclaim}
We say a term has type $A$, if it has the form (\ref{neweqn5}), with some factor $\Theta'$ in place of $\Theta_{\mathbf{w}}^{\mu_{0}\cdots\mu_{j}j}$, which is bounded by \begin{equation}\label{add00}|\Theta'|\lesssim\min_{0\leq j\leq 2}\langle n_{j}\rangle\cdot\min\bigg\{1,\frac{\langle (m^{l})_{i}\rangle}{\langle n_{l}\rangle+\langle n_{l}'\rangle}\bigg\},\,\,\,\,\,\,\,\,j=2;\end{equation}\begin{equation}\label{add00}|\Theta'|\lesssim\min\bigg\{1,\frac{\langle (m^{l})_{i}\rangle}{\langle n_{l}\rangle+\langle n_{l}'\rangle}\bigg\},\,\,\,\,\,\,\,\,j=3,\end{equation} for some $l\geq 1$ and $1\leq i\leq \mu_{l}$. Moreover we assume that (1) either there is some $h\neq l$ such that $n_{l}'+n_{h}'=0$, or no $n_{j}'+n_{k}'=0$ regardless whether $j$ or $k$ is equal to $l$; (2) either $(m^{l})_{i}n_{l}'<0$, or the $v$ in $(v^{\omega_{l}})_{n_{l}'}$ is replaced by $w$. Then this term will be $\mathcal{J}^{b}$ for some $b\in\{3,3.5,4,4.5\}$.
\end{lemma}

We now start to analyze the sum (\ref{neweqn5}). Note that the $\Theta_{\mu_{0}}^{2}$ in (\ref{coef}) and the $\Theta_{\mu_{0}}^{3}$ in (\ref{coef2}) are fixed linear combinations of products $\mathbf{1}_{E}\cdot\Theta$ (recall Proposition \ref{propneweqn}), so we only need to consider one product of this type.

First, we collect the terms in (\ref{neweqn5}) where for some $1\leq h\neq l\leq j$ we have $n_{h}'+n_{l}'=0$. We fix such pair $(h,l)$ and fix a $k>0$ (the case $k=0$ being trivial) so that $n_{h}'=k$ and $n_{l}'=-k$, then we fix $\omega$ and all the $\mu$'s except for $\mu_{h}$ and $\mu_{l}$, and fix all the variables except for $(m^{h})_{i}$ and $(m^{l})_{i}$. There are then two possibilities.

(1) If $(\omega_{h},\omega_{l})\neq (-1,1)$, say $\omega_{l}=-1$, then from (\ref{weighttt}) we have $(m^{l})_{1\mu_{l}}+k<0$, which implies $\langle k\rangle\lesssim\langle (m^{i})_{i}\rangle$ for some $i$, and we may assume that $(m^{l})_{i}$ has opposite sign with $k$. Therefore we get a term of type $A$ and reduce to Lemma \ref{onlyclaim}.

(2) If $(\omega_{h},\omega_{l})=(-1,1)$, then in particular we may replace the $v$ in $(v^{\omega_{h}})_{n_{h}'}$ and $(v^{\omega_{l}})_{n_{l}'}$ by $w$ in (\ref{neweqn5}). Now we make the restriction that $\langle(m^{h})_{i}\rangle\ll (\mu+1)^{-2}\langle k\rangle$ for all $1\leq i\leq \mu_{h}$ and the same for $l$, where $\mu$ is the sum of all $\mu_{j}$, including $\mu_{h}$ and $\mu_{l}$. It is important to notice that this restriction depends only on $\mu_{h}+\mu_{l}$; also, the remaining part is of type $A$ and can be treated using Lemma \ref{onlyclaim}.

Next, assume $\mu_{h}+\mu_{l}>0$; we will replace the $\Psi_{\mu_{h}}$ factor in (\ref{coef}) and (\ref{coef2}) by $\psi^{2\mu_{h}}(n_{h}'/N)$ and the same for $l$; thus the modified version of $\Psi_{\mu_{h}}\Psi_{\mu_{l}}$ will depend only on $\mu_{h}+\mu_{l}$. Also we may replace the $n_{h}$ and $n_{l}$ appearing in $\Theta$ factors in $\Theta_{\mu_{0}}^{j}$ functions, as well as $\min_{0\leq j\leq 2}\langle n_{j}\rangle$, by $n_{h}'(=k)$ and $n_{l}'(=-k)$; note that we are \emph{not} doing this for the $\mathbf{1}_{E}$ factor. Now, since the $\Theta$ factors and the $\Psi$ factors are in $SV_{1}$, $\min_{0\leq j\leq 2}\langle n_{j}\rangle$ is in $SV_{2}$, and we already have $\langle n_{h}\rangle\sim\langle n_{h}'\rangle$ and the same for $l$, we can easily show that the error introduced in this way will be of type $A$.

Now, apart from the $\mathbf{1}_{E}$ factors, we have replaced $\Theta_{\mathbf{w}}^{\mu_{0}\cdots\mu_{j}j}$ with some $\Theta'$ independent of the $(m^{h})_{i}$ and the $(m^{l})_{i}$ variables. Regarding the $\mathbf{1}_{E}$ factor, let us consider the case $E=\{n_{l'}+n_{h'}=0\}$. If $\{l',h'\}=\{l,h\}$, this factor will again be independent of the chosen $m$ variables (since it only depends on $(m^{h})_{1\mu_{h}}+(m^{l})_{1\mu_{l}}$ which is fixed). Therefore, up to an error term which only involves the summation where $j=3$, the weight $\Theta_{\mathbf{w}}^{\mu_{0}\cdots\mu_{3}3}$ factor is bounded, and all three of $(n_{1}',n_{2}',n_{3}')$ are related by $m$ (thus it will be $\mathcal{J}^{3}$), we may assume that $\Theta'$ is completely independent of the $(m^{h})_{i}$ and $(m^{l})_{i}$ variables.

Next, we will fix $\mu_{h}$ and $\mu_{l}$, so that we are summing over $(m^{h})_{i}$ and $(m^{l})_{i}$, the restriction being
\begin{equation}\label{restrict}(m^{h})_{1\mu_{h}}+(m^{l})_{1\mu_{l}}=\mathrm{cst},\,\,\,\,\,\max\{\langle (m^{h})_{i}\rangle,\langle(m^{l})_{i'}\rangle\}\ll (\mu+1)^{-2}\langle k\rangle,\end{equation} where the constant depends on the other fixed variables, and the summand will be
\begin{equation}\prod_{i=1}^{\mu_{h}}\frac{u_{(m^{h})_{i}}}{(m^{h})_{i}}\prod_{i=1}^{\mu_{l}}\frac{u_{(m^{l})_{i}}}{(m^{l})_{i}}.\end{equation} Note that when each $(m^{h})_{i}$ and $(m^{l})_{i'}$ is small, the restriction
\begin{equation}(m^{h})_{1\mu_{h}}>-k,(m^{l})_{1\mu_{l}}<k,\end{equation} which comes from (\ref{weighttt}), \emph{will be void}. Now we can see that this sum actually depends \emph{only on} $\mu_{h}+\mu_{l}$, thus when we sum over $\mu_{h}$ fixing $\mu_{h}+\mu_{l}$, we will get zero, since
\begin{equation}\sum_{\mu_{1}+\mu_{2}=\mu>0}\frac{(-1)^{\mu_{1}}}{\mu_{1}!\mu_{2}!}=0.\end{equation} Therefore all the terms in this case can be treated using Lemma \ref{onlyclaim}.

We still need to consider when $\mu_{h}=\mu_{l}=0$. In this case we have $n_{h}=k$ and $n_{l}=-k$. Note in particular we must have $j=3$ due to the restriction (i) in proposition \ref{propneweqn}; we may assume $h=1$ and $l=2$, so the $\Theta_{\mathbf{w}}^{\mu_{0}\cdots\mu_{3}3}$ factor is bounded by 
\begin{equation}\min\bigg\{1,\frac{\langle n_{0}\rangle +\langle n_{3}\rangle}{\langle k\rangle}\bigg\}\end{equation} provided $n_{3}\neq\pm k$, which we may assume since otherwise all three of $(n_{1}',n_{2}',n_{3}')$ will be related by $m$ and we will get $\mathcal{J}^{3}$. Now if $\langle (m^{3})_{i}\rangle\ll(\mu_{3}+1)^{-2}\langle n_{3}\rangle$ for all $i$, then $\langle n_{3}\rangle\sim\langle n_{3}'\rangle$ and the $v$ in $(v^{\omega_{3}})_{n_{3}'}$ may be replaced by $w$, thus we will have $\mathcal{J}^{3}$; otherwise we may promote some $(m^{3})_{i}$ and rename it $n_{4}$, and it can be easily checked that this part will be $\mathcal{J}^{4}$.

Now we collect the terms where no two $n_{l}'$ add to zero. Among these, we will first take out the part where $\langle (m^{l})_{i}\rangle\gtrsim (\mu_{0l}+1)^{-2}\langle n_{i}\rangle$ for at least one $1\leq l\leq j$ and at least one $1\leq i\leq \mu_{l}$, since this again will be of type $A$. In what remains, we will have $\langle n_{l}'\rangle\sim \langle n_{l}\rangle$, and that $\omega_{l}n_{l}'<0$, and we may replace the $v$ in $(v^{\omega_{l}})_{n_{l}'}$ by $w$. Now when $j=3$, we already obtain a part of $\mathcal{J}^{3}$. Finally, when $j=2$ we separate the cases where \begin{equation}\label{finalclass}\max_{l,i}\langle (m^{l})_{i}\rangle\ll (\mu_{0j}+1)^{-2}\min\{\langle n_{0}\rangle,\langle n_{1}\rangle,\langle n_{2}\rangle\}\end{equation} or otherwise, again by inserting smooth cutoffs. If (\ref{finalclass}) holds we get a part of $\mathcal{J}^{2}$; if (\ref{finalclass}) fails, we can promote some $(m^{l})_{i}$ and call it $n_{3}$ so that the new $\Theta$ factor is bounded, and then replace $u_{n_{3}}$ by (\ref{possub}) or (\ref{negsub}), introducing the $n_{3}'$ and $(m^{3})_{i}$ variables. Now, if $\langle (m^{3})_{i}\rangle\ll\langle n_{3}\rangle$ for all $i$, so that $\langle n_{3}'\rangle\sim\langle n_{3}\rangle$ and the $v$ in $(v^{\omega_{3}})_{n_{3}'}$ may be replaced by $w$, so we get $\mathcal{J}^{3}$ due to the same argument as in the proof of Proposition \ref{propneweqn}; otherwise we could promote some $(m^{3})_{i}$ to be $n_{4}$. We then obtain $\mathcal{J}^{4}$ if one of $n_{3}$, $n_{3}'$, $n_{4}$ or the remaining $m$ variables is $\gtrsim(\max\{\langle n_{0}\rangle,\langle n_{1}\rangle,\langle n_{2}\rangle \})^{\frac{1}{10}}$, and obtain a part of $\mathcal{J}^{4.5}$ otherwise.

Finally, to prove part (vi), first notice that in (\ref{sbx}) we may assume each $\langle m_{i}\rangle\ll n$ also, since otherwise we will have $\mathcal{J}^{4}$. It can then be checked that \emph{in this proof}, whenever we have such a term in $\mathcal{J}^{3}$, all the $\Theta$ and $\Psi$ factors involved in the weight will be in $SV_{1}$; there will be another part of $\mathcal{J}^{3}$ coming from type $A$ terms, and we will show below that they verify (\ref{ysx}). The possible $\mathbf{1}_{E}$ is not in $SV_{1}$ but \emph{under our assumptions}, they will not change if we increase or decrease $n$ by $1$. Thus (\ref{sbx}) will be a direct consequence of the definition of $SV_{1}$ and Proposition \ref{gene}.
\end{proof}
\begin{proof}[Proof of Lemma \ref{onlyclaim}]Fix the $l$ and the $i$ in (\ref{add00}), and first suppose $j=2$. We may assume $l=2$ and promote the $(m^{2})_{i}$ by calling it $n_{3}$, then the new $\Theta$ factor will be bounded by\begin{equation}\frac{\min\{\langle n_{0}\rangle,\langle n_{1}\rangle,\langle n_{2}\rangle\}}{\max\{\langle n_{2}\rangle,\langle n_{3}\rangle,\langle n_{2}'\rangle\}}.\nonumber\end{equation} Notice that\begin{equation}\langle n_{1}\rangle\lesssim\langle n_{1}'\rangle+\sum_{i}\langle (m^{1})_{i}\rangle;\,\,\,\,\,\,\,\,\langle n_{2}\rangle\lesssim\langle n_{2}'+n_{3}\rangle+\sum_{i}\langle (m^{2})_{i}\rangle,\nonumber\end{equation} thus the $\Theta$ factor will be bounded either by
\begin{equation}\label{xxxyyy}\frac{\min\{\langle n_{0}\rangle,\langle n_{1}'\rangle,\langle n_{2}'+n_{3}\rangle\}}{\max\{\langle n_{2}'\rangle,\langle n_{3}\rangle\}}\end{equation} or by some \begin{equation}\min\bigg\{1,\frac{\langle (m^{l})_{i}\rangle}{\max\{\langle n_{2}\rangle,\langle n_{3}\rangle,\langle n_{2}'\rangle\}}\bigg\},\,\,\,\,\,\,\,\,\,\,l\in\{1,2\},1\leq i\leq \mu_{l}.\nonumber\end{equation} 
In the former case the bound (\ref{bound3.5}) is already verified, and we will have a part of $\mathcal{J}^{3.5}$ if $\omega_{1}n_{1}'<0$ and $n_{2}'$ has different sign with $n_{3}$. If $\omega_{1}n_{1}'\geq0$, we have for some $1\leq i\leq \mu_{1}$ that $\langle n_{1}\rangle+\langle n_{1}'\rangle\lesssim \langle(m^{1})_{i}\rangle$ and we are reduced to the latter case above. Now in the latter case, we promote the $(m^{l})_{i}$ and call it $n_{4}$, so that we get an expression of the form
\begin{equation}\label{neweqn10}\sum_{n_{1}+\cdots+n_{4}+m_{1}+\cdots +m_{\mu}=n_{0}}\Phi\cdot\prod_{l=1}^{4}(\upsilon^{\omega_{l}})_{n_{l}}\prod_{i=1}^{\mu}\frac{u_{m_{i}}}{m_{i}},\end{equation} where we may assume $\omega_{l}(n_{l}+\lambda_{l})<0$, where $\lambda_{l}$ is some linear combination of the $m$ variables, and the $\upsilon$ in $(\upsilon^{\omega_{l}})_{n_{l}}$ can be replaced by $v$ for $l\in\{1,2\}$; also the $\Phi$ factor is bounded by $(\langle n_{2}\rangle+\langle n_{3}\rangle+\langle n_{4}\rangle)^{-1}$. Now if one of $n_{l}(2\leq l\leq 4)$ or $m_{i}$ is $\gtrsim \max\{\langle n_{0}\rangle,\langle n_{1}\rangle\}^{\frac{1}{20}}$ we will obtain a part of $\mathcal{J}^{4}$; otherwise we must have $\omega_{1}n_{1}<0$ and thus we are in $\mathcal{J}^{4.5}$.

Now, if we have some term similar to $\mathcal{J}^{3.5}$ (i.e. with $\Theta$ factor bounded by (\ref{xxxyyy})), with $\omega_{1}n_{1}<0$ but $n_{2}$ and $n_{3}$ have the same sign (note the $n_{j}$ here was $n_{j}'$ before we rename it), then from the definition of type $A$ terms, we can replace the $\upsilon$ in $(\upsilon^{\omega_{j}})_{n_{j}'}$ by $w$ for $j\in\{1,2\}$. Next, we replace $u_{n_{3}}$ by (\ref{possub}) or (\ref{negsub}) according to the sign of $n_{3}$, introducing the $n_{3}'$ and $(m^{3})_{i}$ variables. Under the assumption $\langle n_{3}\rangle\gg\max_{0\leq l\leq 3}\langle n_{l}\rangle^{\frac{1}{2}}$, we may assume $\langle (m^{3})_{i}\rangle\ll \langle n_{3}\rangle^{\frac{1}{4}}$, otherwise we will have $\mathcal{J}^{4}$. In particular we have $(w^{\omega_{3}})_{n_{3}'}$ and the weight will be bounded by $\frac{\min\{\langle n_{0}\rangle,\langle n_{1}\rangle\}}{\max\{\langle n_{2}\rangle,\langle n_{3}'\rangle\}}$. Since $n_{2}$ will have the same sign with $n_{3}'$, we \emph{cannot} have $n_{0}=n_{1}$ or $n_{2}+n_{3}'=0$; then we can check that this term will be $\mathcal{J}^{3}$, and that it verifies (\ref{ysx}).

Now assume $j=3$. We may assume $l=3$, and by a similar argument we will obtain an expression of form (\ref{neweqn10}), but with $\Phi$ factor bounded only by $(\langle n_{3}\rangle+\langle n_{4}\rangle)^{-1}$. If we can assume that some other $n_{j}$ (say $n_{2}$) are related to $n_{3}+n_{4}$ by $m$, then we can reduce to the case just studied. Otherwise we must have $n_{1}+n_{2}\neq 0$. Now we may assume that $n_{3}$, $n_{4}$ and all the $m$ variables are $\ll\langle (\max\{\langle n_{0}\rangle,\langle n_{1}\rangle,\langle n_{2}\rangle\})^{\frac{1}{20}}$ or we are in $\mathcal{J}^{4}$; also, if for some $l\in\{1,2\}$ we have $\langle n_{l}\rangle\gtrsim(\max\{\langle n_{0}\rangle,\langle n_{1}\rangle,\langle n_{2}\rangle\})^{\frac{1}{10}}$ (we make this restriction by inserting a smooth cutoff), then $\omega_{l}n_{l}<0$ and we can replace $\upsilon$ in $(\upsilon^{\omega_{l}})_{n_{l}}$ by $w$. Therefore we will have $\mathcal{J}^{4.5}$.
\end{proof}
\subsection{From $w$ to $w^{*}$}
We still need to remove from the right hand side of (\ref{intereqn0}) the part that cannot be controlled directly, by means of a substitution which will be described in the following proposition.
\begin{proposition}\label{finalver}
We can define $w^{*}$, for each fixed time, by
\begin{equation}(w^{*})_{n}=e^{-\mathrm{i}\Delta_{n}}w_{n},\end{equation} where the $\Delta$ factors are
\begin{equation}\label{factor0}\Delta_{n}(t)=\int_{0}^{t}\delta_{n}(t')\,\mathrm{d}t',
\end{equation} and the $\delta$ factors are\footnote[1]{Note that we may replace the $w_{k}$ by $(w^{*})_{k}$ in this expression.}
\begin{equation}\label{factor1}\delta_{n}=\sum_{\mu}C_{\mu}\sum_{k>0}\sum_{m_{1}+\cdots+m_{\mu}=0}\Gamma\cdot|w_{k}|^{2}\prod_{i=1}^{\mu}\frac{u_{m_{i}}}{m_{i}}
\end{equation} for $n>0$.
Here we have $|C_{\mu}|\leq C^{\mu}/\mu!$, the factor $\Gamma=\Gamma(n,k,m_{1},\cdots,m_{\mu})$ is nonzero only when $\langle k\rangle+\max_{i}\langle m_{i}\rangle\leq (\mu+3)^{-12}\langle n\rangle$, and it verifies the estimates
\begin{equation}\label{factor2}|\Gamma|\lesssim1;\,\,\,\,\,\,\,\,|\Gamma(n,k,m_{1},\cdots,m_{\mu})-\Gamma(n+1,k,m_{1},\cdots,m_{\mu})|\lesssim\langle n\rangle^{-1}.
\end{equation}
Moreover, we extend $\delta_{n}$ and $\Delta_{n}$ to $n\leq 0$ so that they are odd in $n$, then define $u^{*}$ and $v^{*}$ by
\begin{equation}(\upsilon^{*})_{n}=e^{-\mathrm{i}\Delta_{n}}\upsilon_{n},\,\,\,\,\,\,\,\upsilon\in\{u,v\}.\nonumber\end{equation} With these definitions, we have
\begin{equation}\label{neweqnfin}(\partial_{t}-\mathrm{i}\partial_{xx})w^{*}=\mathcal{N}^{2}(w^{*},w^{*})+\sum_{j\in\{3,2.5,4,4.5\}}\mathcal{N}^{j},
\end{equation} where $\mathcal{N}^{j}=\sum_{\mu}\sum_{\omega\in\{-1,1\}^{[j]}}C_{\mu}^{\omega j}\mathcal{N}_{\mu}^{\omega j}$ for each $j$ with $|C_{\mu}^{\omega j}|\leq C^{\mu}/\mu!$. The nonlinearities are 
\begin{equation}(\mathcal{N}_{\mu}^{\omega2}(f,g))_{n_{0}}=\sum_{n_{1}+n_{2}+m_{1}+\cdots+m_{\mu}=n_{0}}\Phi_{\mu}^{2}\cdot e^{\mathrm{i}(\Delta_{n_{1}}+\Delta_{n_{2}}-\Delta_{n_{0}})}(f^{\omega_{1}})_{n_{1}}(g^{\omega_{2}})_{n_{2}}\prod_{i=1}^{\mu}\frac{u_{m_{i}}}{m_{i}};\nonumber
\end{equation}
\begin{equation}(\mathcal{N}_{\mu}^{\omega3}(f,g))_{n_{0}}=\sum_{n_{1}+n_{2}+n_{3}+m_{1}+\cdots+m_{\mu}=n_{0}}\Phi_{\mu}^{3}\cdot e^{-\mathrm{i}\Delta_{n_{0}}}\prod_{l=1}^{3}(w^{\omega_{l}})_{n_{l}}\prod_{i=1}^{\mu}\frac{u_{m_{i}}}{m_{i}};\nonumber
\end{equation} and
\begin{equation}(\mathcal{N}_{\mu}^{\omega j})_{n_{0}}=\sum_{n_{1}+\cdots+n_{[j]}+m_{1}+\cdots+m_{\mu}=n_{0}}\Phi_{\mu}^{j}\cdot e^{-\mathrm{i}\Delta_{n_{0}}}\prod_{l=1}^{[j]}(\upsilon^{\omega_{l}})_{n_{l}}\prod_{i=1}^{\mu}\frac{u_{m_{i}}}{m_{i}}\nonumber
\end{equation} for $n_{0}>0$ and $j\in\{3.5,4,4.5\}$. Here $\Phi_{\mu}^{j}=\Phi_{\mu}^{j}(n_{1},\cdots,n_{[j]},m_{1},\cdots,m_{\mu})$, and these factors (and the corresponding terms they come from) satisfy the requirements in parts (i), (iii), (iv), (v) in Proposition \ref{propneweqn} for $j=2,3.5,4,4.5$, respectively.

Finally when $j=3$ and we only consider the case where $\Phi_{\mu}^{3}\neq 0$, we have one of the following: (a) either three of the four variables $(-n_{0},n_{1},n_{2},n_{3})$ are related by $m$ (in which case we are allowed to have $\upsilon$ instead of $w$), or no two of them add up to zero, and $|\Phi_{\mu}^{3}|\lesssim 1$; (b) up to some permutation, $n_{1}+n_{2}=0\neq n_{0}-n_{3}$ and $|\Phi_{\mu}^{3}|\lesssim\min\{1,(\langle n_{0}\rangle+\langle n_{3}\rangle)/\langle n_{1}\rangle\}$; (c) up to some permutation, $n_{0}=n_{1}$, and either $n_{2}+n_{3}\neq 0$ and $|\Phi_{\mu}^{3}|\lesssim 1$, or $n_{2}=-n_{3}$ and $|\Phi_{\mu}^{3}|\lesssim\frac{\min\{\langle n_{0}\rangle,\langle n_{2}\rangle\}}{\max\{\langle n_{0}\rangle,\langle n_{2}\rangle\}}$.

\end{proposition}
\begin{proof} Note that in $\mathcal{J}^{3}$ we may assume $\langle m_{i}\rangle\ll(\mu+3)^{-13}\max_{0\leq l\leq 3}\langle n_{l}\rangle$, since otherwise we would obtain a part of $\mathcal{J}^{4}$. Now we collect the terms in $\mathcal{J}^{3}$ where $n_{0}=n_{1}$, $n_{2}+n_{3}=0$, $\langle n_{2}\rangle\ll (\mu+3)^{-12}\langle n_{1}\rangle$, as well as the terms corresponding to other permutations, \emph{that satisfy} (\ref{sbx}). It is clear that the sum of these terms can be written in the form of
\begin{equation}\label{subterm}\mathrm{i}\cdot w_{n}\sum_{k>0}\sum_{m_{1}+\cdots+m_{\mu}=0}\Gamma\cdot|w_{k}|^{2}\prod_{i=1}^{\mu}\frac{u_{m_{i}}}{m_{i}},
\end{equation} with the $\Gamma$ factor satisfying the requirements (the inequality about the difference is guaranteed by part (vi) of Proposition \ref{intereqn}, since the new factors we introduce will always be in $SV_{1}$). Now if we define the $\delta_{n}$ and $\Delta_{n}$ factors accordingly and make the substitution, we will be able to get rid of the term in (\ref{subterm}). The terms $\mathcal{J}^{j}$ with $j\neq 3$ are transformed into $\mathcal{N}^{j}$ without any change; as for the remaining terms in $\mathcal{J}^{3}$, we can see by an easy enumeration that the coefficient $\Phi_{\mu}^{3}$ will meet our requirements.
\end{proof}
\section{The \emph{a priori} estimate I: The general setting}\label{begin}
In this section we state our main estimate that works for a single solution. Its proof will occupy Sections \ref{mid1}-\ref{mid3}. There will be another version concerning the difference of two solutions, which will be stated and proved in Section \ref{end}.
\subsection{The bootstrap}
Let us fix a smooth solution $u$, defined on $\mathbb{R}\times\mathbb{T}$, to the equation (\ref{smoothtrunc}), with the parameter $1\ll N\leq\infty$. In what follows we will assume $N<\infty$, since the case $N=\infty$ will follow from a similar (and simpler) argument. The main estimate can then be stated as follows.
\begin{proposition}\label{uniformest}There exists an absolute constant $C$ such that the following holds. Suppose $\|u(0)\|_{Z_{1}}\leq A$ for some large $A$, then within a short time $T=C^{-1}e^{-CA}$, for the functions $v$ and $w$ defined in Section \ref{gaugetransform}, and the functions $u^{*}$, $v^{*}$ and $w^{*}$ defined in Section \ref{gaugetransform3}, we have
\begin{equation}\label{output}\|w^{*}\|_{Y_{1}^{T}}+\|v^{*}\|_{Y_{2}^{T}}+\|u^{*}\|_{Y_{2}^{T}}\leq Ce^{CA};\end{equation}
\begin{equation}\label{output2}\|\langle \partial_{x}\rangle^{-s^{3}}u\|_{(X_{2}\cap X_{3}\cap X_{4})^{T}}\leq CA.\end{equation} Here the space $X_{2}\cap X_{3}\cap X_{4}$ is normed by $\|\cdot\|_{X_{2}}+\|\cdot\|_{X_{3}}\|+\|\cdot\|_{X_{4}}$, for which we can easily show that Proposition \ref{initialboot} still holds.
\end{proposition}
\begin{remark}The constant $C$ will depend on the constants in the inequalities in earlier sections, such as Proposition \ref{initialboot} and Proposition \ref{finalver}. To make this clear, we will now use $C_{0}$ to denote any (large) constant that can be bounded by the constants appearing in those inequalities.
\end{remark}
In the proof of Proposition \ref{uniformest} we will use a bootstrap argument. The starting point is
\begin{proposition}\label{inii}
The estimates (\ref{output}) and (\ref{output2}) are true, with $C$ replaced by $C_{0}$, when $T>0$ is sufficiently small.
\end{proposition}
\begin{proof} Note that $u^{*}(0)=u(0)$ and the same holds for $v^{*}$ and $w^{*}$, and also $w(0)=\mathbb{P}_{+}v(0)$, by invoking Proposition \ref{initialboot}, we only need to prove that $\|u(0)\|_{Z_{1}}\leq C_{0}A$ and $\|v(0)\|_{Z_{1}}\leq C_{0}e^{C_{0}A}$. The first inequality follows from our assumption, so we only need to prove that $\|Mu\|_{Z_{1}}\lesssim C_{0}e^{C_{0}\|u\|_{Z_{1}}}$. By the definition of $M$, we only need to prove that \begin{equation}\label{sep0}\|P^{\mu}u\|_{Z_{1}}\leq C_{0}^{\mu}\|u\|_{Z_{1}}^{\mu+1}\end{equation} for all $\mu$. Now we clearly have
\begin{equation}|(P^{\mu}u)_{n_{0}}|\lesssim\sum_{n_{1}}|u_{n_{1}}|\cdot|z_{n_{0}-n_{1}}|,\end{equation} where \begin{equation}z_{m}=\sum_{m_{1}+\cdots+m_{\mu}=m}\prod_{i=1}^{\mu}\frac{|u_{m_{i}}|}{\langle m_{i}\rangle}.\end{equation} Since when $m=m_{1\mu}$ we have $\langle m\rangle\leq C_{0}^{\mu}\langle m_{1}\rangle\cdots \langle m_{\mu}\rangle$, we conclude that
\begin{equation}\label{separate}\sum_{m}\langle m\rangle^{\frac{1}{4}}|z_{m}|\lesssim C_{0}^{\mu}\prod_{i=1}^{\mu}\sum_{m_{i}}\frac{|u_{m_{i}}|}{\langle m_{i}\rangle^{\frac{3}{4}}}\lesssim (C_{0}\|u\|_{Z_{1}})^{\mu},\end{equation} where the last inequality is because
\begin{eqnarray}\sum_{m}\frac{|u_{m}|}{\langle m\rangle^{\frac{3}{4}}}&\lesssim&\sum_{d}2^{-\frac{3d}{4}}\sum_{m\sim 2^{d}}|u_{m}|\nonumber\\
&\lesssim &\sum_{d}2^{(-\frac{3}{4}+1-\frac{1}{p}-r)d}\|\langle m\rangle^{r}u_{m}\|_{l_{m\sim 2^{d}}^{p}}\nonumber\\
&\lesssim &\sum_{d}2^{-\frac{d}{4}}\|u\|_{Z_{1}}\lesssim \|u\|_{Z_{1}}.\nonumber\end{eqnarray} Now using (\ref{separate}), we will be able to prove (\ref{sep0}) once we can prove
\begin{equation}\|(u_{n+m})_{n\in \mathbb{Z}}\|_{Z_{1}}\lesssim\langle m\rangle^{\frac{1}{4}}\|u\|_{Z_{1}}.\end{equation} To prove this, by definition we need to control $\|\langle n\rangle^{r}u_{n+m}\|_{l_{n\sim 2^{d}}^{p}}$ for each $d$. If $m\ll 2^{d}$ this is easy, since we then have $n+m\sim 2^{d}$ and also $\langle n\rangle^{r}\lesssim\langle m\rangle^{\frac{1}{4}}\langle n+m\rangle^{r}$. Now if $m\sim 2^{d'}\gtrsim 2^{d}$, we can use $\langle n\rangle^{r}\lesssim\langle m\rangle^{\frac{1}{8}}\langle n+m\rangle^{r}$ and
\begin{equation}\|\langle n+m\rangle^{r}u_{n+m}\|_{l_{n\sim 2^{d}}^{p}}\lesssim\|\langle n\rangle^{r}u_{n}\|_{l_{n\lesssim 2^{d'}}^{p}}\lesssim (d'+1)\|u\|_{Z_{1}}\lesssim\langle m\rangle^{\frac{1}{8}}\|u\|_{Z_{1}}\end{equation} to complete the proof.
\end{proof}
Starting from Proposition \ref{inii} and with the help of Proposition \ref{initialboot}, it is easily seen that we only need to prove the following
\begin{proposition}\label{boot} Suppose $C_{j}$ is large enough depending on $C_{j-1}$ for $1\leq j\leq 2$, and $0<T\leq C_{2}^{-1}e^{-C_{2}A}$. Then if the inequalities
\begin{equation}\label{output3}\|w^{*}\|_{Y_{1}^{T}}+\|v^{*}\|_{Y_{2}^{T}}+\|u^{*}\|_{Y_{2}^{T}}\leq C_{1}e^{C_{1}A},\end{equation}
\begin{equation}\label{output4}\|\langle \partial_{x}\rangle^{-s^{3}}u\|_{(X_{2}\cap X_{3}\cap X_{4})^{T}}\leq C_{1}A\end{equation} hold, these inequalities must hold with $C_{1}$ replaced by $C_{0}$.
\end{proposition}
The rest of this section, as well as Sections \ref{mid1} and \ref{mid2}, is devoted to the proof of the estimate for $w^{*}$ in Proposition \ref{boot}; in Section \ref{mid3} we consider the other three functions. During the whole proof, the inequalities (\ref{output3}) and (\ref{output4}) will be assumed.
\subsection{The extensions}\label{theextensions}
By the definition of $Y_{j}^{T}$ norms, we have globally defined functions $u''$, $v''$, $w''$ and $u'''$ which agree with $u^{*}$, $v^{*}$, $w^{*}$ and $u$ on $[-T,T]$ respectively\footnote[1]{Note that all the extensions we construct are in $X_{2}$, thus in particular we can talk about their restrictions to $[-T,T]$.}, and verify the inequalities (\ref{output3}) and (\ref{output4}) with the superscript $T$ in the norms removed. By inserting a time cutoff $\chi(t)$, we may assume that they are all supported in $|t|\leq 1$. We then define the factors $\delta_{n}$ and $\Delta_{n}$ for all time as in (\ref{factor0}) and (\ref{factor1}), with $w^{*}$ and $u$ replaced by $w''$ and $u'''$ respectively. We may also define functions $u'$ by $(u')_{n}=e^{\mathrm{i}\Delta_{n}}(u'')_{n}$; the functions $v'$ and $w'$ are defined similarly.

Now we could interpret the bilinear form $\mathcal{N}^{2}$ and terms $\mathcal{N}^{j}$ on the right hand side of (\ref{neweqnfin}), by replacing each $u$ with $u'''$, each $w$ with $w'$, each $\upsilon$ with $\upsilon'$ (note $\upsilon$ is either $u$ or $v$), each $\delta_{n}$ and $\Delta_{n}$ with what we defined above. If we then choose some $0<\mathcal{T}\leq T$ and define the function $z$ by $z(t)=w''(t)$ for $t\in[-\mathcal{T},\mathcal{T}]$ and $(\partial_{t}-\mathrm{i}\partial_{xx})z(t)=0$ on both $(-\infty,-\mathcal{T}]$ and $[\mathcal{T},+\infty)$, then we can check that this function $z$ verifies the equation
\begin{equation}\label{global1}
(\partial_{t}-\mathrm{i}\partial_{xx})z=\mathbf{1}_{[-\mathcal{T},\mathcal{T}]}(t)\mathcal{N}^{2}(z,z)+\mathbf{1}_{[-\mathcal{T},\mathcal{T}]}(t)\sum_{j\in\{3,3.5,4,4.5\}}\mathcal{N}^{j},
\end{equation} with initial data $z(0)=w(0)$. Using the time cutoff $\chi(t)$, we can define $y(t)=\chi(t)z(t)$. From (\ref{global1}) we conclude that
\begin{equation}\label{global2}
y=\chi(t)e^{\mathrm{i}t\partial_{xx}}w(0)+\mathcal{E}\big(\mathbf{1}_{[-\mathcal{T},\mathcal{T}]}\cdot\mathcal{N}^{2}(y,y)\big)+\sum_{j\in\{3,3.5,4,4.5\}}\mathcal{E}\big(\mathbf{1}_{[-\mathcal{T},\mathcal{T}]}\cdot\mathcal{N}^{j}\big).
\end{equation}Since $w''$ is smooth on the interval $[-T,T]$, we may conclude that $\mathcal{T}\mapsto y$ is a continuous map from $(0,T]$ to $Y_{1}$; also it is clear that when $\mathcal{T}$ is sufficiently small we have $\|y\|_{Y_{1}}\leq C_{0}e^{C_{0}A}$. Thus in order to prove the estimate for $w^{*}$, we only need to prove the following
\begin{proposition}\label{finalreduct}
Suppose $y\in Y_{1}$ is a function verifying (\ref{global2}) with $0<\mathcal{T}\leq C_{2}^{-1}e^{-C_{2}A}$, and $\|y\|_{Y_{1}}\leq C_{1}e^{C_{1}A}$, then we must have $\|y\|_{Y_{1}}\leq C_{0}e^{C_{0}A}$.
\end{proposition}
 In what follows, we will use $T$ instead of $\mathcal{T}$ for simplicity; note that $T\leq C_{2}^{-1}e^{-C_{2}A}$. Before proceeding, let us prove a few results concerning the exponential factors $e^{\pm\mathrm{i}\Delta_{n}(t)}$. The first lemma is a general feature.
\begin{lemma}\label{general}Suppose $h_{j}=h_{j}(t)$, $j\in\{0,1\}$ are two functions of $t$, and define $J_{j}(t)=\chi(t)e^{\mathrm{i}H_{j}(t)}$, where $H_{j}(t)=\int_{0}^{t}h_{j}(t')\mathrm{d}t'$, then we have the estimate
\begin{equation}\label{exppp}\|\langle \xi\rangle(J_{1}-J_{0})^{\wedge}(\xi)\|_{L^{k}}\lesssim\|(h_{1}-h_{0})^{\wedge}\|_{L^{1}}(1+\|\widehat{h_{1}}\|_{L^{1}}+\|\widehat{h_{0}}\|_{L^{1}})^{2}\end{equation} for all $1\leq k\leq\infty$.
\end{lemma}
\begin{proof} Recall from Section \ref{linear} that $\chi=\chi(t)$ is some time cutoff that may vary from place to place. Thanks to this factor, we only need to prove (\ref{exppp}) for $k=1$. Next, notice that
\begin{equation}J_{1}-J_{0}=\mathrm{i}\chi\cdot(H_{1}-H_{0})\int_{0}^{1}e^{\mathrm{i}(\theta H_{1}+(1-\theta)H_{0})}\,\mathrm{d}\theta,\end{equation} we only need to prove (\ref{exppp}) for a fixed $\theta$. Let $h=h_{1}-h_{2}$ and $h_{\theta}=\theta h_{1}+(1-\theta)h_{0}$ and $H,H_{\theta}$ defined accordingly, and define $\chi\cdot He^{\mathrm{i}H_{\theta}}=\Phi$. We then have
\begin{equation}\partial_{x}\Phi=(\chi'\cdot H+\chi\cdot h+\mathrm{i}\chi\cdot Hh_{\theta})e^{\mathrm{i}H_{\theta}},
\end{equation}
which then implies
\begin{eqnarray}\|\langle \xi\rangle\widehat{\Phi}(\xi)\|_{L^{1}}&\lesssim & \|\widehat{\Phi}\|_{L^{1}}+\|\widehat{\partial_{x}\Phi}\|_{L^{1}}\nonumber\\
&\lesssim &\|(\chi\cdot e^{\mathrm{i}H_{\theta}})^{\wedge}\|_{L^{1}}\cdot\big(\|\widehat{\chi h}\|_{L^{1}}+\|\widehat{\chi H}\|_{L^{1}}+\|\widehat{\chi H}\|_{L^{1}}\|\widehat{\chi h_{\theta}}\|_{L^{1}}\big)\nonumber\\
&\lesssim &\|\chi\cdot e^{\mathrm{i}H_{\theta}}\|_{H^{1}}\cdot\big(\|\widehat{h}\|_{L^{1}}+\|\chi H\|_{H^{1}}+\|\chi H\|_{H^{1}}(\|\widehat{h}_{1}\|_{L^{1}}+\|\widehat{h_{0}}\|_{L^{1}})\big)\nonumber\\
&\lesssim&\|\widehat{h}\|_{L^{1}}(1+\|\widehat{h_{0}}\|_{L^{1}}+\|\widehat{h_{1}}\|_{L^{1}})^{2},\nonumber
\end{eqnarray} where $H^{1}$ is the standard Sobolev norm.
\end{proof}
\begin{proposition}\label{factt}We have \begin{equation}\label{factt1}\|\widehat{\delta_{n}}\|_{L^{1}}\leq C_{0}C_{1}e^{C_{0}C_{1}A}\log(2+|n|),\end{equation}as well as\begin{equation}\label{factt2}\|(\delta_{n+1}-\delta_{n})^{\wedge}\|_{L^{1}}\leq C_{0}C_{1}e^{C_{0}C_{1}A}\langle n\rangle^{-\frac{1}{2}}.\end{equation}
\end{proposition}
\begin{proof}Recall from Proposition \ref{finalver} that
\begin{equation}\delta_{n}=\sum_{\mu}C_{\mu}\sum_{k;m_{1}+\cdots+m_{\mu}=0}\Gamma\cdot|(w'')_{k}|^{2}\prod_{i=1}^{\mu}\frac{(u''')_{m_{i}}}{m_{i}},\end{equation} where $|C_{\mu}|\lesssim C_{0}^{\mu}/\mu!$, and the factor $\Gamma$, as in Proposition \ref{finalver}, verifies (\ref{factor2}). Now, using the fact that $\|\widehat{fg}\|_{L^{1}}\leq \|\widehat{f}\|_{L^{1}}\|\widehat{g}\|_{L^{1}}$, we obtain
\begin{eqnarray}
\|\widehat{\delta_{n}}\|_{L^{1}}&\lesssim&\sum_{\mu}\frac{C_{0}^{\mu}}{\mu!}\sum_{k\lesssim \langle n\rangle}\|\widehat{(w'')_{k}}\|_{L^{1}}\|\widehat{(\overline{w''})_{-k}}\|_{L^{1}}\cdot\sum_{m_{1},\cdots,m_{\mu}}\prod_{i=1}^{\mu}\frac{1}{\langle m_{i}\rangle}\|\widehat{(u''')_{m_{i}}}\|_{L^{1}}\nonumber\\
&\lesssim &\sum_{\mu}\frac{C_{0}^{\mu}}{\mu!}\|\langle m\rangle ^{-1}u'''\|_{l^{1}L^{1}}^{\mu}\cdot \|w''\|_{l_{k\lesssim\langle n\rangle}^{2}L^{1}}\|\overline{w''}\|_{l_{k\lesssim\langle n\rangle}^{2}L^{1}}\nonumber\\
&\lesssim &C_{0}C_{1}e^{C_{0}C_{1}A}\log(2+|n|)\nonumber.
\end{eqnarray} Here we have used the fact that \begin{equation}\|\langle m\rangle^{-1}u'''\|_{l^{1}L^{1}}\lesssim\|\langle \partial_{x}\rangle^{-s^{3}}u'''\|_{X_{2}}\leq C_{1}A,\end{equation} as well as 
\begin{equation}\|w''\|_{l_{k\lesssim\langle n\rangle}^{2}L^{1}}\lesssim\log(2+|n|)\cdot\|w''\|_{l_{d\geq 0}^{\infty}l_{k\sim 2^{d}}^{2}L^{1}}\lesssim\log(2+|n|)\|w''\|_{X_{2}},\end{equation} and the same estimate for $\overline{w''}$.

The estimate for the difference is proved in the same way, by using the second inequality in (\ref{factor2}). In fact we get a power $\langle n\rangle^{-1}\log(2+|n|)$ which is better than $\langle n\rangle^{-\frac{1}{2}}$.
\end{proof}
\begin{remark}Note that all our norms are invariant under complex conjugation. Occasionally we will make restrictions such as $n_{l}>0$ which breaks this symmetry, but such information is only used in controlling the weights and the non-resonance factors, thus in terms of norm estimates for a single function, we will basically view $w$ and $\overline{w}$ as the same function.
\end{remark}
\begin{proposition}\label{weaken} for any function $h$, let $h'$ be defined by $(h')_{n}=\chi(t)e^{\pm\mathrm{i}\Delta_{n}}h_{n}$ for each fixed time. We then have \begin{equation}\|\langle\partial_{x}\rangle^{-s^{3}}h'\|_{X_{j}}\leq O_{C_{1}}(1)e^{C_{0}C_{1}A}\|h\|_{X_{j}}\end{equation} for $1\leq j\leq 7$.
\end{proposition}
\begin{proof} Apart from $X_{3}$, all the other norms we are considering are (some Besov versions of) $\|\langle n\rangle^{\sigma}\langle\xi\rangle^{\beta}u\|_{l^{k}L^{h}}$ or $\|\langle n\rangle^{\sigma}\langle\xi\rangle^{\beta}u\|_{L^{h}l^{k}}$ with $\beta<1$, and in the latter case we have $\sigma=\beta=0$. Since the map $h\mapsto h'$ commutes with $\mathbb{P}$ projections, we only need to consider these kinds of norms. Notice that on the $\widetilde{u}$ side, this map is just a convolution with the Fourier transform of $\chi(t)e^{\pm\mathrm{i}\Delta_{n}(t)}$ for each $\widetilde{u}_{n}$. Thus to prove the result for $l^{k}L^{h}$ norm, we only need to prove that convolution by this function is bounded with respect to the weighted norm $\|\langle \xi\rangle^{\beta}\cdot\|_{L^{h}}$ by $O_{C_{1}}(1)e^{C_{0}C_{1}A}\langle n\rangle^{s^{3}}$. An elementary argument yields that this bound does not exceed the norm \begin{equation}\big\|\langle \xi\rangle(\chi(t)e^{\pm\mathrm{i}\Delta_{n}(t)})^{\wedge}(\xi)\big\|_{L^{1}},\nonumber\end{equation} which is bounded by $C_{0}C_{1}^{3}e^{C_{0}C_{1}A}(\log(2+|n|))^{3}$, thanks to Lemma \ref{general} and Proposition \ref{factt}.

Now let us consider the $L^{h}l^{k}$ norm and the $X_{3}$ norm. Let $I_{n}(\xi)$ be the Fourier transform of $\chi(t)e^{\pm\mathrm{i}\Delta_{n}(t)}$, then we conclude that
\begin{eqnarray}\|\langle n\rangle^{-s^{3}}(h')_{n,\xi}\|_{L^{h}l^{k}}&\lesssim &\int_{\mathbb{R}}\|\langle n\rangle^{-s^{3}}h_{n,\xi-\eta}I_{n}(\eta)\|_{L^{h}l^{k}}\,\mathrm{d}\eta\nonumber\\
&\lesssim &\int_{\mathbb{R}}\sup_{n}\langle n\rangle^{-s^{3}}|I_{n}(\eta)|\cdot\|h\|_{L^{h}l^{k}}\,\mathrm{d}\eta,
\end{eqnarray} note the same argument also works for $X_{3}$. Therefore we need to bound the expression
\begin{equation}\int_{\mathbb{R}}\sup_{n}\langle n\rangle^{-s^{3}}|I_{n}(\xi)|\,\mathrm{d}\xi\nonumber\end{equation} by $O_{C_{1}}(1)e^{C_{0}C_{1}A}$. By performing a dyadic summation in $n$, we only need to bound \begin{equation}\int_{\mathbb{R}}\max_{n\sim 2^{d}}|I_{n}(\xi)|\,\mathrm{d}\xi\end{equation} by $O_{C_{1}}(1)e^{C_{0}C_{1}A}(d+2)^{O(1)}$. Now suppose $|\xi|\lesssim 2^{10d}$, then we simply use Proposition \ref{factt} as well as the $L^{\infty}$ estimate of Lemma \ref{general} to bound this contribution by $O_{C_{1}}(1)e^{C_{0}C_{1}A}$ times $(d+2)^{O(1)}\int_{|\xi|\lesssim 2^{20d}}\langle \xi\rangle^{-1}\mathrm{d}\xi=(d+2)^{O(1)}$. If $|\xi|\gg 2^{10d}$, we may replace the ``maximum'' in this expression by summation (during which we lose a power $2^{d}$), then use the $L^{1}$ estimate of Lemma \ref{general} and the largeness of $\xi$ to gain a power $2^{10d}$. Thus in any case we obtain the desired estimate.
\end{proof}
\section{The \emph{a priori} estimate II: Quadratic and cubic terms}\label{mid1}
We now begin the proof of Proposition \ref{finalreduct}, the starting point being (\ref{global2}). The linear term is clearly bounded in $Y_{1}$ by $C_{0}e^{C_{0}A}$, so we only need to bound the $\mathcal{N}^{j}$ terms. There will be a large number of cases, and they are ordered according to the difficulty level. In this section we will be able to treat every term except $\mathcal{N}^{3.5}$.
\begin{proposition}\label{easiest} For each $j\in\{2,3,3.5,4,4.5\}$, define \begin{equation}
\mathcal{M}^{j}=\mathcal{E}\big(\mathbf{1}_{[-T,T]}\mathcal{N}^{j}\big),
\end{equation}where we may write $\mathcal{N}^{2}$ or $\mathcal{N}^{2}(y,y)$ depending on the context. We then have \begin{equation}\label{x4est}\|\mathcal{M}^{2}\|_{X_{4}}\leq O_{C_{1}}(1)e^{C_{0}C_{1}A}T^{0+},\end{equation} as well as
\begin{equation}\label{auxest}\sum_{j\in\{3,3.5,4,4.5\}}\|\langle n\rangle^{-\frac{1}{20}}\langle \xi\rangle^{\kappa}(\mathcal{M}^{j})_{n,\xi}\|_{l^{2}L^{2}}\leq O_{C_{1}}(1)e^{C_{0}C_{1}A}T^{0+}.\end{equation}
\end{proposition}
\begin{remark} Since we have
\begin{equation}\|\langle n\rangle^{-1}\langle\xi\rangle^{\kappa}u\|_{l^{\gamma}L^{2}}\leq C_{0}\|\langle n\rangle^{-\frac{1}{20}}\langle\xi\rangle^{\kappa}u\|_{l^{2}L^{2}}\end{equation} by H\"{o}lder, the inequalities (\ref{x4est}) and (\ref{auxest}) will imply $\|y\|_{X_{4}}\leq C_{0}e^{C_{0}A}$, due to the restriction $T\leq C_{2}^{-1}e^{-C_{2}A}$.
\end{remark}
\begin{proof} In this proof, as well as the following ones, we will use the $\lesssim$ and $\gtrsim$ symbols, with the convention that all the implicit constants are $\leq O_{C_{1}}(1)e^{C_{0}C_{1}A}$. Note that in the estimate for any possible multilinear term, the total number of appearances of all functions other than $u'''$ is bounded by $10$, thus as long as we only use the norm $\|\langle \partial_{x}\rangle^{-s^{3}}u'''\|_{X_{2}\cap X_{3}\cap X_{4}}$ (which is bounded by $C_{1}A$) for the function $u'''$, the implicit constants will be bounded by
\begin{equation}(O_{C_{1}}(1)e^{C_{0}C_{1}A})^{C_{0}}\sum_{\mu}\frac{C_{0}^{\mu}}{\mu!}(C_{0}C_{1}A)^{\mu}\leq O_{C_{1}}(1)e^{C_{0}C_{1}A}\end{equation} and is thus under control. We also need to be careful with the \emph{sharp} cutoff $\mathbf{1}_{[-T,T]}$. Denote by $\phi_{\xi}=\frac{e^{\mathrm{i}T\xi}-e^{-\mathrm{i}T\xi}}{\mathrm{i}\xi}$ the Fourier transform of $\mathbf{1}_{[-T,T]}$; note that $|\phi_{\xi}|\lesssim\min(T,\frac{1}{\langle \xi\rangle})$, and that $\|\phi\|_{L^{1+}(\{|\xi|\geq K\})}\lesssim T^{0+}\langle K\rangle^{0-}$.

First let us prove $\|\mathcal{M}^{2}\|_{X_{4}}\lesssim T^{0+}$. As argued above, we may fix $\mu\geq 0$ and $\omega\in\{-1,1\}^{2}$ (though we will not add any sub- or superscript for simplicity). Choose a function $g$ such that $\|g\|_{X_{4}'}\leq 1$ and define $f=\mathcal{E}'g$. Also define $f'$ by\footnote[1]{Since $f$ has compact time support, we may insert $\chi(t)$ in the definition of $f'$, so that we can use the arguments in the proof of Proposition \ref{weaken}. The same comment applies also for later discussions.} $(f')_{n}=e^{\mathrm{i}\Delta_{n}}f_{n}$ and $y'$ similarly; these notations will be standard throughout the proof.

From the bound $\|g\|_{X_{4}'}\leq 1$ we obtain by Proposition \ref{linearestimate2} that \begin{equation}\|\langle n_{0}\rangle\langle \alpha_{0}\rangle^{1-\kappa}f_{n_{0},\alpha_{0}}\|_{l^{\gamma'}L^{2}}\lesssim 1,\nonumber\end{equation} which then implies, thanks to (H\"{o}lder and) an argument similar to the proof of Proposition \ref{weaken}, that \begin{equation}\label{weaker0}\|\langle n_{0}\rangle^{1-O(s^{2.5})}\langle \alpha_{0}\rangle^{1-\kappa}(f')_{n_{0},\alpha_{0}}\|_{l^{2}L^{2}}\lesssim 1.\end{equation}

Now we only need to bound the expression\footnote[1]{Here in order to pass from $f$ to $f'$ we have used the identity\begin{equation}\int_{\mathbb{R}}\overline{f_{n_{0},\alpha_{0}}}\mathcal{N}_{n_{0},\alpha_{0}}\,\mathrm{d}\alpha_{0}=\int_{\mathbb{R}}\overline{(e^{\mathrm{i}\Delta_{n_{0}}}f_{n})^{\wedge}(\alpha_{0}-|n_{0}|n_{0})}(e^{\mathrm{i}\Delta_{n_{0}}}\mathcal{N}_{n})^{\wedge}(\alpha_{0}-|n_{0}|n_{0})\,\mathrm{d}\alpha_{0}\nonumber,\end{equation} which is a consequence of Plancherel.}
\begin{eqnarray}\mathcal{S}&=&\sum_{n_{0}}\int_{\mathbb{R}}\overline{f_{n_{0},\alpha_{0}}}\cdot(\mathbf{1}_{[-T,T]}\mathcal{N}^{2})_{n_{0},\alpha_{0}}\,\mathrm{d}\alpha_{0}\nonumber\\
&=&\sum_{n_{0}=n_{1}+n_{2}+m_{1}+\cdots+m_{\mu}}\int_{\mathbb{R}}\Phi^{2}\cdot\overline{f_{n_{0},\alpha_{0}}}\times\nonumber\\
&\times&\bigg(\mathbf{1}_{[-T,T]}e^{\mathrm{i}(\Delta_{n_{1}}+\Delta_{n_{2}}-\Delta_{n_{0}})}\prod_{l=1}^{2}(y^{\omega_{l}})_{n_{l}}\prod_{i=1}^{\mu}\frac{(u''')_{m_{i}}}{m_{i}}\bigg)^{\wedge}(\alpha_{0}-|n_{0}|n_{0})\,\mathrm{d}\alpha_{0}\nonumber\\
&=&\sum_{n_{0}=n_{1}+n_{2}+m_{1}+\cdots+m_{\mu}}\int_{\mathbb{R}}\Phi^{2}\cdot\overline{(f')_{n_{0},\alpha_{0}}}\times\nonumber\\
&\times&\bigg(\mathbf{1}_{[-T,T]}\prod_{l=1}^{2}((y')^{\omega_{l}})_{n_{l}}\prod_{i=1}^{\mu}\frac{(u''')_{m_{i}}}{m_{i}}\bigg)^{\wedge}(\alpha_{0}-|n_{0}|n_{0})\,\mathrm{d}\alpha_{0}\nonumber\\
&=&\sum_{n_{0}=n_{1}+n_{2}+\cdots +m_{\mu}}\int_{(T)}\Phi^{2}\cdot\overline{(f')_{n_{0},\alpha_{0}}}\prod_{l=1}^{2}((y')^{\omega_{l}})_{n_{l},\alpha_{l}}\cdot\phi_{\alpha_{3}}\prod_{i=1}^{\mu}\frac{(u''')_{m_{i},\beta_{i}}}{m_{i}}.\nonumber
\end{eqnarray} Here the integration $(T)$ is interpreted as the integration over the set
\begin{equation}
\big\{(\alpha_{0},\cdots,\alpha_{3},\beta_{1},\cdots,\beta_{\mu}):\alpha_{0}=\alpha_{13}+\beta_{1\mu}+\Xi\big\},\nonumber
\end{equation} which is a hyperplane in $\mathbb{R}^{\mu+4}$ (recall the notation that $\alpha_{13}=\alpha_{1}+\alpha_{2}+\alpha_{3}$), with respect to the standard measure \begin{equation}\prod_{l=1}^{3}\mathrm{d}\alpha_{l}\cdot\prod_{i=1}^{\mu}\mathrm{d}\beta_{i},\nonumber\end{equation} where the non-resonance (NR) factor 
\begin{equation}\label{nlfac}\Xi=|n_{0}|n_{0}-\sum_{l=1}^{2}|n_{l}|n_{l}-\sum_{i=1}^{\mu}|m_{i}|m_{i}.\end{equation} Note that we are using the convention that $u_{n,\alpha}$ stands for $\widetilde{u}_{n,\alpha}$; also we may always restrict to $n_{0}>0$.

Notice that the $m$ variables are all $\ll \min_{0\leq l\leq 2}\langle n_{l}\rangle$ (again here we may have harmless polynomial factors in $\mu$), we can check from (\ref{nlfac}) that \begin{equation}|\Xi|\sim\min_{0\leq l\leq 2}\langle n_{l}\rangle\cdot\max_{0\leq l\leq 2}\langle n_{l}\rangle.\end{equation} We will first take the summation-integration over the set where $\sum_{l=0}^{2}\langle n_{l}\rangle\sim 2^{d}$, and then sum over $d$. In this case, at least one of the $\alpha$ and $\beta$ variables must be $\gtrsim 2^{d}$. Now, with a loss of $2^{O(s^{2.5})d}$, we can replace the $1-O(s^{2.5})$ exponent in (\ref{weaker0}) by $1$. Also notice that $|\Phi^{2}|\lesssim \langle n_{0}\rangle$, we may further (upon taking absolute values) remove this $\Phi$ factor and the $\langle n\rangle$ factor in (\ref{weaker0}) simultaneously.

With these reductions, we then proceed to the estimate of $\mathcal{S}$. First assume $\langle \alpha_{0}\rangle\gtrsim 2^{d}$, thus we gain from the bound (\ref{weaker0}) a power $2^{(1-\kappa)d}$, while after exploiting this, we still have\footnote[1]{Actually it is the modified version of $f'$ that is bounded in $l^{2}L^{2}$ (namely, the $f''$ is defined by $(f'')_{n_{0},\alpha_{0}}=2^{d}\mathbf{1}_{\langle \alpha_{0}\rangle\gtrsim 2^{d}}(f')_{n_{0},\alpha_{0}}$). The same comment applies also for every other reduction we make. Our wording, though imprecise, should not cause any confusion.} $\|f'\|_{l^{2}L^{2}}\lesssim 1$. In the same way, we can use the $X_{1}$ and $X_{4}$ bounds for $y$ to deduce some bound for $y'$ (see Proposition \ref{weaken}), and strengthen the bound to $\|\langle n_{l}\rangle^{s^{2}}\langle \alpha_{l}\rangle^{\frac{1}{2}+s^{2}}(y')_{n_{l},\alpha_{l}}\|_{l^{2}L^{2}}\lesssim 1$ at a price of at most $2^{O(\frac{1}{2}-b)d}$.

We then fix all the $m$ and $\beta$ variables to get a sub-summation-integration that is bounded by (with $C$ being irrelevant constants)
\begin{eqnarray}\mathcal{S}_{sub}&\lesssim&\sum_{n_{0}=n_{1}+n_{2}+C}\int_{\widetilde{\alpha_{0}}=\widetilde{\alpha_{1}}+\widetilde{\alpha_{2}}+\widetilde{\alpha_{3}}+C}|(f')_{n_{0},\widetilde{\alpha_{0}}}||\phi_{0,\widetilde{\alpha_{3}}}|\cdot\prod_{l=1}^{2}|((y')^{\omega_{l}})_{n_{l},\widetilde{\alpha_{l}}}|\nonumber\\
&\lesssim &\big\|\big|\widehat{\overline{f'}}\big|*\big|\widehat{(y')^{\omega_{1}}}\big|*\big|\widehat{(y')^{\omega_{2}}}\big|*\big|\widehat{\phi}\big|\big\|_{l^{\infty}L^{\infty}}\nonumber\\
\label{crucial}&\lesssim&\|f'\|_{l^{2}L^{2}}\|\mathfrak{N}(y')^{\omega_{1}}\|_{L^{6+}L^{6+}}\|\mathfrak{N}(y')^{\omega_{2}}\|_{L^{3}L^{3}}\|\widehat{\phi}\|_{l^{1+}L^{1+}}.
\end{eqnarray} where $\widetilde{\alpha_{l}}=\alpha_{l}-|n_{l}|n_{l}$, and $\phi$ is viewed as a function of $(t,x)$ that is supported at $n=0$ (so that $\widetilde{\alpha_{3}}=\alpha_{3}$); also recall the $\mathfrak{N}$ notation defined in Section \ref{spaces}. The right hand side will be bounded by $T^{0+}$ by our (reduced) assumptions and Strichartz estimates, provided we choose the $6+$ to be $6+cs^{2}$ with some small $c$, and choose $1+$ accordingly.

Now we sum over $m_{i}$ and integrate $\beta_{i}$, exploiting the bound $\|\langle m_{i}\rangle^{-1}u'''\|_{l^{1}L^{1}}\leq C_{1}A$, to bound the whole summation-integration for a single $d$; taking into account the gains and losses from the reductions made before and exploiting (\ref{hierarchy}), we conclude that the part of $\mathcal{S}$ considered above is bounded by $T^{0+}2^{(0-)d}$, which allows us to sum over $d$.

Next, assume that $\langle \alpha_{1}\rangle\gtrsim 2^{d}$ (the case for $\alpha_{2}$ will follow by symmetry). In this case we do not gain from the bound (\ref{weaker0}), so that we still have $\|\langle \xi\rangle^{1-\kappa}f'\|_{l^{2}L^{2}}\lesssim 1$, which, via Strichartz, allows us to control $\|\mathfrak{N}f'\|_{L^{2+}L^{2+}}$, where this $2+$ is some $2+c(1-\kappa)$. Instead, we gain from the bound
\begin{equation}\|\langle n_{1}\rangle^{s^{2}}\langle \alpha_{1}\rangle^{\frac{1}{2}+s^{2}}(y')_{n_{1},\alpha_{1}}\|_{l^{2}L^{2}}\lesssim 1\nonumber\end{equation} as above (with a loss of $2^{O(\frac{1}{2}-b)d}$) and change the exponent $\langle\alpha_{1}\rangle^{\frac{1}{2}+s^{2}}$ to $\langle\alpha_{1}\rangle^{\frac{1}{2}-c(1-\kappa)}$ to gain the power $2^{c(1-\kappa)d}$, and the reduced bound will allow us to control $\mathfrak{N}y'$ (in the form of $\mathfrak{N}(y')^{\omega_{1}}$) in $L^{6-}L^{6-}$ with the $6-$ here being $6-c(1-\kappa)$. Choosing the constants $c$ appropriately, we can then proceed as in (\ref{crucial}), with the $f'$ factor estimated in $L^{2+}L^{2+}$, two $y'$ factors estimated in $L^{6+}L^{6+}$ and $L^{3}L^{3}$ respectively and the $\phi$ factor in $l^{1+}L^{1+}$, to get the desired bound. In the same spirit, if $\langle \alpha_{3}\rangle\gtrsim 2^{d}$, we will use the $L^{2+}L^{2+}$ bound for $\mathfrak{N}f'$ (with $2+$ being $2+c(1-\kappa)$), $L^{6+}L^{6+}$ and $L^{3}L^{3}$ bound for $\mathfrak{N}y'$ (with $6+$ being $6+cs^{2}$) and $l^{1+}L^{1+}$ bound for $\phi$ (with $1+$ being $1+c(1-\kappa)$; note that we gain a power $2^{c(1-\kappa)d}$ here due to the largeness of $\alpha_{3}$) to conclude. Again we gain at least $2^{c(1-\kappa)d}$ and lose at most $2^{O(\frac{1}{2}-b)d}$ so we have enough room for summation in $d$.

Next, assume that $\langle \beta_{i}\rangle\gtrsim 2^{d}$ for some $i$. If for this $i$ we also have $\langle m_{i}\rangle\gtrsim 2^{\frac{d}{30}}$, then we would bound $|m_{i}|^{-1}\lesssim 2^{-\frac{d}{90}}\langle m_{i}\rangle^{-\frac{2}{3}}$ to gain a power of $2^{cd}$ and proceed as above, since we still have 
\begin{equation}\|\langle m_{i}\rangle^{-\frac{2}{3}}(u''')_{m_{i},\beta_{i}}\|_{l^{1}L^{1}}\lesssim \|\langle \partial_{x}\rangle^{-s^{3}}u'''\|_{X_{2}}\leq C_{0}C_{1}A\end{equation} which allows us to sum over $m_{i}$ and integrate over $\beta_{i}$. If instead $\langle m_{i}\rangle\lesssim 2^{\frac{d}{30}}$, we could use the $X_{4}$ bound of $\langle \partial_{x}\rangle^{-s^{3}}u'''$ and Proposition \ref{weaken} to bound \begin{equation}\|\langle m_{i}\rangle^{-\frac{3}{2}}\langle \beta_{i}\rangle^{\frac{9}{10}}(y')_{m_{i},\beta_{i}}\|_{l^{2}L^{2}}\lesssim 1,\nonumber\end{equation} and exploit the largeness of $\beta_{i}$ to gain a power $2^{\frac{d}{20}}$ and reduce the above bound to $\|\langle m_{i}\rangle^{2}\langle \beta_{i}\rangle^{\frac{3}{5}}(y')_{m_{i},\beta_{i}}\|_{l^{2}L^{2}}\lesssim 1$ which would imply $\|y'\|_{l^{1}L^{1}}\lesssim 1$ so that we can still apply the argument above, sum over $m_{i}$ and integrate over $\beta_{i}$. This concludes the proof of (\ref{x4est}).

Now let us prove (\ref{auxest}). Let $g$, $f$ and $f'$ be as before, but with the new bound $\|\langle n_{0}\rangle^{\frac{1}{30}}\langle \alpha_{0}\rangle^{1-\kappa}f'\|_{l^{2}L^{2}}\lesssim 1$. Note that the estimate for $f'$ is again easily deduced from the estimate for $g$ and the same type of arguments as in the proof of Propositions \ref{linearestimate} and \ref{weaken}. To bound $\mathcal{M}^{3}$ and $\mathcal{M}^{3.5}$, we need to bound
\begin{eqnarray}\label{s3expre}\mathcal{S}&=&\sum_{n_{0}=n_{1}+n_{2}+n_{3}+\cdots +m_{\mu}}\int_{(T)}\Phi^{j}\cdot\overline{(f')_{n_{0},\alpha_{0}}}\times\\
&\times&((w')^{\omega_{1}})_{n_{1},\alpha_{1}}\prod_{l=2}^{3}(z^{l})_{n_{l},\alpha_{l}}\cdot\phi_{\alpha_{4}}\prod_{i=1}^{\mu}\frac{(u''')_{m_{i},\beta_{i}}}{m_{i}},\nonumber
\end{eqnarray} with the integration $(T)$ interpreted as the integral over the set 
\begin{equation}\big\{(\alpha_{0},\cdots,\alpha_{4},\beta_{1},\cdots,\beta_{\mu}):\alpha_{0}=\alpha_{14}+\beta_{1\mu}+\Xi\big\},\end{equation} with respect to the standard measure, where the NR factor
\begin{equation}\label{nlfac2}\Xi=|n_{0}|n_{0}-\sum_{l=1}^{3}|n_{l}|n_{l}-\sum_{i=1}^{\mu}|m_{i}|m_{i}.\end{equation}Here $z^{l}$ or $\overline{z^{l}}$ equals $u'$, $v'$ or $w'$ for each $l$. Again we assume $\sum_{l=0}^{3}\langle n_{l}\rangle\sim 2^{d}$. By losing at most $2^{O(\epsilon)d}$, we may assume that $w'$ verifies the same bound as $y'$ before, and $\mathfrak{N}z^{l}$ is bounded in $X_{4}$ and $L^{6}L^{6}$. Also note that $|\Phi^{j}|\lesssim 1$ in any situation.

If $\langle n_{0}\rangle+\langle \alpha_{0}\rangle\gtrsim 2^{\frac{d}{90}}$, we may gain a power $2^{c(1-\kappa)d}$ (note our loss is at most $2^{O(\epsilon)d}$) from the bound of $f'$, and reduce this bound to $\|f'\|_{l^{2}L^{2}}\lesssim 1$. Then we can first fix the $m_{i}$ and $\beta_{i}$ variables and obtain the $\mathcal{S}_{sub}$, estimate in the same was as in (\ref{crucial}), then sum over $m_{i}$ and integrate over $\beta_{i}$. The only difference with (\ref{crucial}) is that now $\mathcal{S}_{sub}$ contains five functions instead of four; however, here we may estimate the $f'$ factor in $L^{2}L^{2}$, the $\mathfrak{N}w'$ factor in $L^{6+}L^{6+}$ with the $6+$ being $6+cs^{2}$, the $\mathfrak{N}z^{l}$ factors in $L^{6}L^{6}$ and the $\phi$ factor in $l^{1+}L^{1+}$ so that we can still close the argument.

If $\langle \alpha_{1}\rangle\gtrsim 2^{\frac{d}{90}}$, we may perform the same reduction as in the estimate of $\|\mathcal{M}^{2}\|_{X_{4}}$ before, gain a power of $2^{c(1-\kappa)d}$ and use Strichartz and the reduced bound to control $\|\mathfrak{N}w'\|_{L^{6-}L^{6-}}$, where the $6-$ is $6-c(1-\kappa)$. We may now control $\mathfrak{N}f'$ in $L^{2+}L^{2+}$ with the $2+$ being $2+c(1-\kappa)$, then control $\mathfrak{N}z^{l}$ in $L^{6}L^{6}$ and $\phi$ in some $l^{1+}L^{1+}$. The exponents will match if we choose the constants $c$ appropriately.

If $\langle \alpha_{2}\rangle\gtrsim 2^{\frac{d}{4}}$ (the $\alpha_{3}$ case being identical), we have two possibilities. If $j=3$ then $z^{2}$ is also taken from $\{w',\overline{w'}\}$ so that we are in the same situation as above. If $j=3.5$ then either $\langle n_{2}\rangle\gtrsim 2^{\frac{d}{89}}$ and we gain a power $2^{cd}$ from the $\Phi$ factor thanks to (\ref{bound3.5}) and the assumption that $\langle n_{0}\rangle\ll 2^{\frac{d}{90}}$, or $\langle n_{2}\rangle\lesssim 2^{\frac{d}{89}}$ and we can exploit the $X_{4}$ bound of $z^{l}$, gain a power $2^{cd}$, and use the reduced estimate to bound $\|\mathfrak{N}z^{2}\|_{L^{6}L^{6}}$ (again, as we already did in the $X_{4}$ estimate before). In any case we gain a power $2^{c(1-\kappa)d}$, lose at most $2^{O(\epsilon)d}$, and can control the reduced $\mathcal{S}_{sub}$ expression.

If $\langle \alpha_{4}\rangle\gtrsim 2^{\frac{d}{90}}$, we can again control $\mathfrak{N}f'$ in $L^{2+}L^{2+}$ with the $2+$ being $2+c(1-\kappa)$, then control the $\mathfrak{N}w'$ in $L^{6+}L^{6+}$ (with $6+$ being $6+cs^{2}$), $\mathfrak{N}z^{l}$ factors in $L^{6}L^{6}$ and $\phi$ in $l^{1+}L^{1+}$ with the $1+$ being $1+c(1-\kappa)$, with the $c$ chosen appropriately. Note that since $\alpha_{4}$ is large, we will gain a power $2^{c(1-\kappa)d}$ from the $l^{1+}L^{1+}$ bound of $\phi$. Moreover, if $\langle m_{i}\rangle\gtrsim 2^{\frac{d}{90}}$ for some $i$, we can repeat the argument made before to gain a (small) $2^{cd}$ power from this factor alone while keeping the ability to sum over $m_{i}$ and integrate over $\beta_{i}$, and reduce to the above cases.

Finally, if none of the above holds, we must have $\langle n_{0}\rangle\ll 2^{\frac{d}{90}}$ and $|\Xi|\ll 2^{\frac{d}{4}}$. We may also assume $\langle m_{i}\rangle\ll 2^{\frac{d}{90}}$ or we are reduced to one of the cases above. Thus from (\ref{nlfac2}) we deduce \begin{equation}\label{smallness}\big||n_{1}|n_{1}+|n_{2}|n_{2}+|n_{3}|n_{3}\big|\ll 2^{\frac{d}{4}}.\end{equation} Note that we may assume $j=3$, since when $j=3.5$, one of $\langle n_{2}\rangle$ and $\langle n_{3}\rangle$ must be $\gtrsim 2^{d}$ and we gain a power $2^{cd}$ from the weight $\Phi$ so that we can proceed as above. Now if the minimum of $\langle n_{l}\rangle$ for $1\leq l\leq 3$ is at least $\gtrsim 2^{\frac{d}{9}}$, then we will be in the same situation as in (\ref{nlfac}) and the expression in (\ref{smallness}) has to be $\gtrsim 2^{d}$. Therefore we may further assume $\langle n_{3}\rangle\ll 2^{\frac{d}{9}}$, and it will be clear that the NR factor can be small only if $n_{1}+n_{2}=0$. However, in this case we gain from the factor $\Phi$ a positive power $2^{cd}$, due to parts (b) and (c) in the requirements for $\mathcal{N}^{3}$ in Proposition \ref{finalver}. This allows us to complete the estimate in the same way as above.

Notice that in estimating $\mathcal{M}^{3}$ above, we have ignored the term where three of $(-n_{0},n_{1},n_{2},n_{3})$ are related by $m$ and we are allowed to have $\upsilon$ instead of $w$ (in the discussion here, they will be $\upsilon'$ and $w'$ respectively). To handle this term, simply fix the $m$ and $\beta$ variables and bound the $\Phi$ factor by $1$ (we may assume $\langle m_{i}\rangle\ll2^{\frac{d}{90}}$ or we gain a power $2^{cd}$ and can proceed as above). We can bound the resulting $\mathcal{S}_{sub}$ (note that we are restricting to $n_{l}=c_{l}\pm n_{0}\sim 2^{d}$)
\begin{eqnarray}\mathcal{S}_{sub}&\lesssim&\sum_{n_{0}}\int_{\alpha_{0}=\alpha_{1}+\cdots+\alpha_{4}+c_{4}(n_{0})}\big|(f')_{n_{0},\alpha_{0}}\big|\cdot\prod_{l=1}^{3}\big|(z^{l})_{c_{l}\pm n_{0},\alpha_{l}}\big|\cdot|\phi_{\alpha_{4}}|\nonumber\\
&\lesssim&T^{0+}\sum_{n_{0}}\big\|\widehat{(f')_{n_{0}}}\big\|_{L^{2}}\prod_{l=1}^{3}\big\|\widehat{(z^{l})_{c_{l}\pm n_{0}}}\big\|_{L^{1}}\nonumber\\
&\lesssim &T^{0+}\big\|f'\big\|_{l_{\sim 2^{d}}^{4}L^{2}}\prod_{l=1}^{3}\big\|z^{l}\big\|_{l_{\sim 2^{d}}^{4}L^{1}}\lesssim T^{0+}2^{-cd},\nonumber\end{eqnarray} where $c_{j}$ are constants (or functions of $n_{0}$). Thus this term is also acceptable.

Now let us bound $\mathcal{M}^{4}$ and $\mathcal{M}^{4.5}$. The quantity we need to control is now
\begin{eqnarray}\label{s4expre}\mathcal{S}&=&\sum_{n_{0}=n_{1}+\cdots+n_{4}+\cdots +m_{\mu}}\int_{(T)}\Phi^{j}\cdot\overline{(f')_{n_{0},\alpha_{0}}}\times\\
&\times&\prod_{l=1}^{4}(z^{l})_{n_{l},\alpha_{l}}\cdot\phi_{\alpha_{5}}\prod_{i=1}^{\mu}\frac{(u''')_{m_{i},\beta_{i}}}{m_{i}},\nonumber
\end{eqnarray} with the integration $(T)$ interpreted as the integral over the set 
\begin{equation}\big\{(\alpha_{0},\cdots,\alpha_{5},\beta_{1},\cdots,\beta_{\mu}):\alpha_{0}=\alpha_{15}+\beta_{1\mu}+\Xi\big\},\end{equation} with respect to the standard measure, where the NR factor
\begin{equation}\label{nlfac3}\Xi=|n_{0}|n_{0}-\sum_{l=1}^{4}|n_{l}|n_{l}-\sum_{i=1}^{\mu}|m_{i}|m_{i}.\end{equation}Here $z^{l}$ or $\overline{z^{l}}$ equals $u'$, $v'$ or $w'$ for each $l$. We assume the maximum of the $n$ variables is $\sim 2^{d}$, and with a loss of at most $2^{O(\epsilon)d}$, we may assume that $\mathfrak{N}w'$ verifies the same estimate as appeared before, and $\mathfrak{N}\upsilon'$ is bounded in $X_{4}$ and $L^{6}L^{6}$ (again, it is the modified versions of $w'$ and $\upsilon'$ that verifies the estimates).

If $j=4$ we may assume (up to a permutation) that $|\Phi|\lesssim 2^{-\frac{d}{90}}\langle n_{3}\rangle^{-\frac{2}{3}}$. Due to the presence of the $\langle n_{3}\rangle^{-\frac{2}{3}}$ factor, we may also fix $n_{3}$ and $\alpha_{3}$ when we fix the $m_{i}$ and $\beta_{i}$ variables. Once these variables are fixed, we only need to control the resulting $\mathcal{S}_{sub}$. But since we gain a power $2^{\frac{d}{90}}$ from the $\Phi$ factor, the estimate for $\mathcal{S}_{sub}$ will be easy; we simply bound $\mathfrak{N}f'$ in $L^{2+}L^{2+}$, bound $\mathfrak{N}z^{l}$ in $L^{6+}L^{6+}$ for $l\in\{1,2,4\}$, and bound $\phi$ in $l^{1+}L^{1+}$. This proves (\ref{auxest}) for $j=4$.

If $j=4.5$, then all the $m_{i}$ variables, as well as $n_{3}$ and $n_{4}$, are $\ll 2^{\frac{d}{10}}$. This in particular implies that either $\langle n_{0}\rangle\gtrsim 2^{\frac{d}{10}}$ so that we gain from the $\langle n_{0}\rangle^{\frac{1}{30}}$ factor in the bound for $f'$, or the NR factor $|\Xi|\gtrsim 2^{d}$ (note that $n_{1}+n_{2}\neq 0$) so we can gain from one of the $\alpha$ or $\beta$ variables. Note that $|\Phi|\lesssim (\langle n_{3}\rangle+\langle n_{4}\rangle)^{-1}$, thus when we fix the $m_{i}$ and $\beta_{i}$ variables, we always has the choice of also fixing $(n_{3},\alpha_{3})$ or $(n_{4},\alpha_{4})$. This means that though the $\mathcal{S}_{sub}$ expression seems to involve six factors, in practice we will always use only five of them. The rest will be basically the same as before. If we gain at least $2^{c(1-\kappa)d}$ from $n_{0}$, $\alpha_{0}$, $\alpha_{3}$ (or similarly $\alpha_{4}$) or some $\beta_{i}$, then we will fix (and then sum and integrate over) $(m_{j},\beta_{j})$ and $(n_{3},\alpha_{3})$ to produce $\mathcal{S}_{sub}$, then control $\mathfrak{N}f'$ in $L^{2+}L^{2+}$, $\mathfrak{N}z^{l}(l\in\{1,2,4\})$ in $L^{6}L^{6}$, $\phi$ in $l^{1+}L^{1+}$ with the $2+$ being $2+c(1-\kappa)$ and $1+$ defined accordingly. If we gain from $\alpha_{l}$ for $l\in\{1,2\}$ (say $l=1$), we will again fix $(m_{j},\beta_{j})$ and $(n_{3},\alpha_{3})$. To estimate $\mathcal{S}_{sub}$, we control $\mathfrak{N}f'$ in $L^{2+}L^{2+}$ with $2+$ being $2+c(1-\kappa)$, $\mathfrak{N}z^{2}$ and $\mathfrak{N}z^{4}$ in $L^{6}L^{6}$, $\phi$ in $l^{1+}L^{1+}$, and $\mathfrak{N}z^{1}$ in $L^{6-}L^{6-}$ with $6-$ being $6-c(1-\kappa)$. Here, if $\langle n_{1}\rangle\gtrsim 2^{\frac{d}{10}}$, $z^{1}$ will be either $w'$ or $\overline{w'}$ so that we can get the $L^{6-}L^{6-}$ bound from the same arguments made before; otherwise $\langle n_{1}\rangle\lesssim 2^{\frac{d}{10}}$, and we can use the $X_{4}$ bound for $\mathfrak{N}z^{1}$ to deduce the $L^{6-}L^{6-}$ bound with a gain of $2^{cd}$. This concludes the proof of Proposition \ref{easiest}.
\end{proof}
\begin{proposition}\label{easier} We have \begin{equation}\label{cubicest}
\sum_{j\in\{3,4,4.5\}}\sum_{j'\in\{1,2,5,7\}}\|\mathcal{M}^{j}\|_{X_{j'}}\lesssim T^{0+}.
\end{equation}
\end{proposition}
\begin{proof} Note that $\mathcal{M}_{j}$ involves a sum over the $n_{l}$ and $m_{i}$ variables. We shall first prove the bound for the terms where $j=4$, or $j=3$ and $n_{0}\not\in\{n_{1},n_{2},n_{3}\}$, or $j=4.5$ and $n_{0}\not\in\{n_{1},n_{2}\}$. By (\ref{relation2}), we only need to bound this part of $\mathcal{M}^{j}$ in $X_{6}$.

Let the functions $g$ and $f$, $f'$ be as usual, with $\|g\|_{X_{6}'}\leq 1$. This would imply \begin{equation}\|\langle n_{0}\rangle^{-s-O(s^{3})}\langle \alpha_{0}\rangle^{\frac{1}{2}-O(s^{2})}(f')_{n_{0},\alpha_{0}}\|_{l^{2}L^{2}}\lesssim 1\nonumber\end{equation} (we have done this kind of reduction many times before). What we need to control is the same quantity $\mathcal{S}$ with $j\in\{3,4,4.5\}$ as in (\ref{s3expre}) and (\ref{s4expre}), and we assume the maximal $n_{l}$ variable is $\sim 2^{d}$ as usual. With a loss of $2^{O(\epsilon)d}$, we may assume that $w'$ and $\upsilon'$ verifies the good bounds appearing in the proof of Proposition \ref{easiest}. Using Strichartz, we can deduce from the bound for $f'$ as above the $L^{2}L^{2}\cap L^{4}L^{4}$ bound for $\mathfrak{N}f'$ with a loss of $2^{O(s)d}$. 

Now we will be able to bound the $\mathcal{S}$ expression in (\ref{s4expre}) easily. In fact, if we gain from anything except $\alpha_{0}$, we can repeat the argument in the proof of (\ref{auxest}), but with the $c(1-\kappa)$ involved in various $2+$ or $6-$ replaced by $c$ (since we now have the $L^{2+c}L^{2+c}$ control for $\mathfrak{N}f'$), and check that in these cases we always gain a power $2^{cd}$, which will be enough to cover the loss $2^{O(s)d}$. If we gain from $\alpha_{0}$, this gain will be $2^{cd}$, with a loss of at most $2^{O(s)d}$, and we can bound the reduced $\mathfrak{N}f'$ factor in $L^{2+c}L^{2+c}$, so this contribution will be acceptable. On the other hand, if we do not gain anything from any of the variables or weights, it must be the case that $j=4.5$ and $|\Xi|\ll2^{\frac{d}{4}}$. Since all the $m$ variables as well as $n_{3}$ and $n_{4}$ are assumed to be $\ll 2^{\frac{d}{10}}$, we then deduce that
\begin{equation}\big||n_{0}|n_{0}-|n_{1}|n_{1}-|n_{2}|n_{2}\big|\ll 2^{\frac{d}{4}}.\nonumber\end{equation} By repeating the argument in the proof of Proposition \ref{easiest}, we see that this can happen only if $n_{1}+n_{2}=0$, or $n_{0}=n_{1}$, or $n_{0}=n_{2}$, but all these possibilities contradict our assumptions.

Next, assume $j=3$. Recall that $\sum_{l=0}^{3}\langle n_{l}\rangle\sim 2^{d}$, and that the $\mathcal{S}$ we need to estimate is bounded by
\begin{eqnarray}|\mathcal{S}|&\lesssim &\sum_{n_{0}=n_{1}+n_{2}+n_{3}+m_{1}+\cdots+m_{\mu}}\int_{(T)}\big|\overline{(f')_{n_{0},\alpha_{0}}}\big|\times\nonumber\\
&\times&\prod_{l=1}^{3}\big|((w')^{\omega_{l}})_{n_{l},\alpha_{l}}\big|\cdot|\phi_{\alpha_{4}}|\prod_{i=1}^{\mu}\bigg|\frac{(u''')_{m_{i},\beta_{i}}}{m_{i}}\bigg|.\nonumber
\end{eqnarray} First assume that some of the $\alpha$ or $\beta$ variables is at least $2^{\frac{d}{90}}$. Then, by the same argument we made before (notice that the $n_{l}$ variables for $1\leq l\leq 3$ correspond to the function $w'$ or $\overline{w'}$, which, up to a loss of $2^{O(s)d}$, verifies the estimate $\|\langle n_{l}\rangle^{s^{2}}\langle \alpha_{l}\rangle^{\frac{1}{2}+s^{2}}w'\|_{l^{2}L^{2}}\lesssim 1$), we can gain a power $2^{cd}$ from the corresponding factor, then fix the $m_{i}$ and $\beta_{i}$ variables (and sum and integrate over them afterwards), produce the $\mathcal{S}_{sub}$ term, and estimate it by controlling $\mathfrak{N}f'$ in $L^{2+}L^{2+}$, $\mathcal{N}(w')^{\omega_{l}}$ in $L^{6-}L^{6-}$ with $2+$ and $6-$ being $2+c$ and $6-c$ respectively, and finally control $\phi$ in $l^{1+}L^{1+}$.

Now let us assume that all $\alpha$ and $\beta$ variables are $\ll 2^{\frac{d}{90}}$; we may assume that all $m_{i}$ variables are $\ll 2^{\frac{d}{90}}$ also. Thus, the variables $(-n_{0},n_{1},n_{2},n_{3})$ will verify the conditions in the following lemma.
\begin{lemma}\label{divisor}Suppose four numbers $n_{0},\cdots,n_{3}$ satisfy
\begin{equation}n_{0}+n_{1}+n_{2}+n_{3}=K_{1};\,\,\,\,|n_{0}|n_{0}+|n_{1}|n_{1}+|n_{2}|n_{2}+|n_{3}|n_{3}=K_{2},\nonumber\end{equation} where $K_{j}$ are constants such that
\begin{equation}|K_{1}|+|K_{2}|\ll 2^{\frac{d}{40}},\,\,\,\,\,\,\,\,\max_{0\leq l\leq 3}\langle n_{l}\rangle\sim 2^{d},\nonumber\end{equation} then one of the followings must hold.

(i) Up to some permutation, we have $n_{0}+n_{1}=n_{2}+n_{3}=0$. In particular, this can happen only if $K_{1}=K_{2}=0$.

(ii) Up to some permutation, we have $n_{0}+n_{1}=0$, $\langle n_{0}\rangle\sim 2^{d}$, and that $\langle n_{2}\rangle+\langle n_{3}\rangle\ll 2^{\frac{d}{40}}$. Note that it is possible that (say) $n_{1}+n_{2}=0$ and $n_{0},n_{3}$ are small.

(iii) No two of $n_{l}$ add to zero. \emph{Under this restriction} we must have $\langle n_{l}\rangle\gtrsim 2^{0.9d}$ for each $l$; moreover, if we fix $K_{1}$, $K_{2}$ and any single $n_{l}$, there will be at most $\lesssim 2^{s^{3}d}$ choices for the quadruple $(n_{0},n_{1},n_{2},n_{3})$.
\end{lemma}
\begin{proof}[Proof of Lemma \ref{divisor}] Suppose some $\langle n_{l}\rangle\ll 2^{0.9d}$ (say $l=0$), then one of $\langle n_{l}\rangle$ for $1\leq l\leq 3$ must also be $\ll 2^{0.9d}$, since otherwise we would have \begin{equation}\big||n_{1}|n_{1}+|n_{2}|n_{2}+|n_{3}|n_{3}\big|\gtrsim\max_{1\leq l\leq 3}\langle n_{l}\rangle\cdot\min_{1\leq l\leq 3}\langle n_{l}\rangle\gtrsim 2^{1.9d}\nonumber\end{equation} while $|n_{0}|^{2}\lesssim 2^{1.8d}$, which is impossible. Now assume that $\langle n_{1}\rangle\ll 2^{0.9d}$, then in particular $\langle n_{2}+n_{3}\rangle\ll 2^{d}$, thus $n_{2}n_{3}<0$ as well as $|n_{2}-n_{3}|\sim 2^{d}$. Suppose $n_{0}+n_{1}=k$ and $n_{2}+n_{3}=l$, we will have $|k+l|\leq c2^{\frac{d}{40}}$ and\begin{equation}2^{d}|l|\leq 2^{0.9d}\langle k\rangle+2^{\frac{d}{40}}.\nonumber\end{equation} Now if $l\neq 0$, this inequality cannot hold, since it would require\footnote[1]{Here we may assume that $2^{d}$ is larger than some constant (which is polynomial in $\mu$) such that $2^{0.9d}\ll 2^{d}$, since the summation for small values of $2^{d}$ will be trivial. This comment also applies to some of the arguments below.} $|k|\gg |l|$, which implies $\langle k\rangle\lesssim 2^{\frac{d}{40}}$, so that the right hand side will be at most $2^{(0.9+1/40)d}$ and the left hand side is at least $2^{d}$. Therefore we must have $n_{2}+n_{3}=0$. If also $n_{0}+n_{1}=0$, we will be in case (i); otherwise $k\neq 0$, so that we always have $\big||n_{0}|n_{0}+|n_{1}|n_{1}\big|\gtrsim |n_{0}|+|n_{1}|$, which then implies that $\langle n_{0}\rangle+\langle n_{1}\rangle\ll 2^{\frac{d}{40}}$ and we will be in case (ii).

Now assume that $\langle n_{l}\rangle\gtrsim 2^{0.9d}$ for each $l$. By the discussion above, we cannot have any $n_{h}+n_{l}=0$ (unless we are in case (i)), so we will be in case (iii). Finally, suppose we fix $K_{1}$, $K_{2}$ and $n_{0}$. The requirement $n_{h}+n_{l}\neq 0$ implies that each $\langle n_{l}\rangle\gtrsim 2^{0.9d}$, so without loss of generality we may assume $n_{0}>0>n_{1}$. Now $n_{2}$ and $n_{3}$ cannot have the same sign since $|K_{1}|\lesssim 2^{\frac{d}{40}}$, thus we may assume $n_{2}>0$ and $n_{3}<0$. Therefore we will have
\begin{equation}n_{0}+n_{1}+n_{2}+n_{3}=K_{1},\,\,\,\,\,\,n_{0}^{2}-n_{1}^{2}+n_{2}^{2}-n_{3}^{2}=K_{2},\nonumber\end{equation} which implies
\begin{equation}(n_{2}+n_{1})(n_{2}+n_{3})=\frac{1}{2}\big(K_{1}^{2}-2K_{1}n_{0}+K_{2}\big).\nonumber\end{equation} By our assumptions, the right hand side is a nonzero constant whose absolute value does not exceed $2^{2d}$. The result now follows from the standard divisor estimate, since knowing $n_{2}+n_{1}$ and $n_{2}+n_{3}$ will allow us to recover the whole quadraple.
\end{proof}Proceeding to the estimate of the $\mathcal{M}^{3}$ term, we can see that the only possibility is case (iii) in Lemma \ref{divisor} (since we have required $n_{0}\not\in\{n_{1},n_{2},n_{3}\}$; also if $n_{1}+n_{2}=0$ and $n_{0},n_{3}$ are small, we will gain a power $2^{cd}$ from the weight $\Phi$ so we can argue as above to close the estimate). In this case we will use a completely different argument.

Recall that up to a loss of $2^{O(s^{3})d}$ we may assume that with small $c$,
\begin{equation}\label{estforf'}\big\|\langle n_{0}\rangle^{-s}\langle\alpha_{0}\rangle^{\frac{1}{2}-c}f'\big\|_{l^{2}L^{2}}\lesssim 1;\end{equation} also by invoking the $X_{1}$ norm of $w$ we obtain the estimate
\begin{equation}\label{estforw'}\big\|\langle n_{l}\rangle^{s}\langle\alpha_{l}\rangle^{\frac{1}{2}-c}w'\big\|_{l^{p}L^{2}}\lesssim 1\end{equation} with a loss of at most $2^{O(s^{3})d}$. Since we now have $2^{0.9d}\lesssim n_{l}\lesssim 2^{d}$, we may remove the $\langle n_{0}\rangle^{-s}$ and $\langle n_{l}\rangle^{s}$ factors in (\ref{estforf'}) and (\ref{estforw'}), and through this process we gain at least $2^{1.7sd}$. Therefore, by fixing $m_{i}$ and $\beta_{i}$ first, we will be able to get the desired result if we can prove the following inequality:
\begin{eqnarray}\label{nonsharp}
\mathcal{S}_{sub}&=&\sum_{n_{0}=n_{1}+n_{2}+n_{3}+K_{1}}\int_{(T)}\prod_{l=0}^{3}\big|(A^{l})_{n_{l},\alpha_{l}}\big|\cdot\min\bigg\{T,\frac{1}{\langle \alpha_{4}\rangle}\bigg\}\\
&\lesssim &2^{O(s^{3})d}T^{0+}\prod_{l=0}^{3}\big\|\langle \alpha_{l}\rangle^{\frac{1}{2}-c}A^{l}\big\|_{l^{2+c}L^{2}},\nonumber
\end{eqnarray} provided $c$ is a small absolute constant, where the integral $(T)$ is taken over the set
\begin{equation}\bigg\{(\alpha_{0},\cdots,\alpha_{4}):\alpha_{0}=\alpha_{14}+|n_{0}|n_{0}-\sum_{l=1}^{3}|n_{l}|n_{l}+K_{2}\bigg\},\end{equation} and we restrict to the region where no two of $(-n_{0},n_{1}.n_{2},n_{3})$ add to zero, $\sum_{l}\langle n_{l}\rangle\sim 2^{d}$, the NR factor $\big||n_{0}|n_{0}-\sum_{l=1}^{3}|n_{l}|n_{l}\big|\ll 2^{\frac{d}{40}}$ and $|K_{1}|+|K_{2}|\ll 2^{\frac{d}{40}}$.

We will use an interpolation argument to prove (\ref{nonsharp}); in fact, if will suffice to prove the estimate when we replace the parameter set $(\frac{1}{2}-c,2+c)$ with $(\frac{2}{5},2)$ or $(3,4)$. When we have $(\frac{2}{5},2)$ we will be able to control $\mathfrak{N}A^{l}$ in $L^{4+}L^{4+}$ for each $l$, so that we can control the $\alpha_{4}$ factor in $l^{1+}L^{1+}$, and invoke the argument used many times before to conclude. When we have $(3,4)$, assuming the norm of each $A^{l}$ is one, we will get that
\begin{equation}\big\|\langle\widetilde{\alpha_{l}}+|n_{l}|n_{l}\rangle ((A^{l})_{n_{l}})^{\wedge}(\widetilde{\alpha_{l}})\big\|_{L^{1}}\lesssim\big\|\langle\widetilde{\alpha_{l}}+|n_{l}|n_{l}\rangle^{3} ((A^{l})_{n_{l}})^{\wedge}(\widetilde{\alpha_{l}})\big\|_{L^{2}}:=A_{n_{l}}^{l},\nonumber\end{equation} with \begin{equation}\|A_{n_{l}}^{l}\|_{l^{4}}\lesssim\|\langle\widetilde{\alpha_{l}}+|n_{l}|n_{l}\rangle^{3}A^{l}\|_{l^{4}\widetilde{L}^{2}}\lesssim 1.\end{equation} Therefore when we fix $(n_{0},\cdots,n_{3})$ and integrate over $(\alpha_{0},\cdots,\alpha_{4})$, we get
\begin{eqnarray}\mathcal{S}_{sub}'&\lesssim&T^{0+}\int_{\mathbb{R}^{4}}\bigg\langle\widetilde{\alpha_{0}}-\sum_{l=1}^{3}\widetilde{\alpha_{l}}-K_{2}\bigg\rangle^{-1+s^{3}}\prod_{l=0}^{3}\langle \widetilde{\alpha_{l}}+|n_{l}|n_{l}\rangle^{-1}\times\nonumber\\
&\times& \prod_{l=0}^{3}\langle\widetilde{\alpha_{l}}+|n_{l}|n_{l}\rangle \big|((A^{l})_{n_{l}})^{\wedge}(\widetilde{\alpha_{l}})\big|\cdot\prod_{l=1}^{4}\mathrm{d}\widetilde{\alpha_{l}}\nonumber\\
&\lesssim &T^{0+}\bigg\langle|n_{0}|n_{0}-\sum_{l=1}^{3}|n_{l}|n_{l}+K_{2}\bigg\rangle^{-1+s^{3}}\prod_{l=0}^{3}A_{n_{l}}^{l}.\nonumber\end{eqnarray} We then sum this over $(n_{l})$; by H\"{o}lder, we only need to bound the sum
\begin{equation}\sum_{(n_{0},\cdots,n_{3})}\bigg\langle|n_{0}|n_{0}-\sum_{l=1}^{3}|n_{l}|n_{l}+K_{2}\bigg\rangle^{-1+s^{4}}(A_{n_{0}}^{0})^{4}.\nonumber\end{equation} If we fix $|n_{0}|n_{0}-\sum_{l=1}^{3}|n_{l}|n_{l}=K_{3}$ with $|K_{3}|\ll 2^{\frac{d}{40}}$, the corresponding sum will be $\lesssim2^{O(s^{3})d}$, since each $n_{0}$ appears at most this many times due to Lemma \ref{divisor}; also the sum over $K_{3}$ will contribute at most $\sum_{|K_{3}|\lesssim 2^{\frac{d}{40}}}\langle K_{3}-K_{2}\rangle^{-1+s^{3}}=2^{O(s^{3})d}$. This completes the proof for $\mathcal{M}^{3}$.

What remains is when $j\in\{3,4.5\}$ and (say) $n_{0}=n_{1}$. Note that the case when three of $n_{l}$ are related by $m$ will be treated at the end of the proof. In both cases we will use the expressions (\ref{s3expre}) and (\ref{s4expre}), but with $f'$  with $f$, and $(w')^{\omega_{1}}$ replaced by $(w'')^{\omega_{1}}$ (if $j=3$), $z^{1}$ replaced by $y^{1}$ (if $j=4.5$). This is easily justified by definition and the fact that $n_{0}=n_{1}$. We will assume $\sum_{l\geq 2}\langle n_{l}\rangle\sim 2^{d'}$, then fix $d$ and $d'$. Here we will use a new bound for $f$. Recall from Proposition \ref{linearestimate2} that $\|g\|_{X_{j}'}\lesssim 1$ for some $j\in\{1,2,5,7\}$ implies $\|f\|_{X_{9}'}\lesssim 1$, or equivalently
\begin{equation}\label{festt}\|f\|_{L^{q}l_{\sim 2^{d}}^{p'}}\lesssim 2^{rd}T_{d},
\end{equation} 
where the $T_{d}$ is such that
\begin{equation}\sum_{d\geq 0}T_{d}\lesssim 1.\end{equation}

In the easier case $j=4.5$, we will be able to fix $m_{i}$ and $\beta_{i}$, then estimate $\mathcal{S}_{sub}$ by (note we have all the restrictions made above, say $\langle n_{0}\rangle\sim 2^{d}$)
\begin{eqnarray}\mathcal{S}_{sub}&\lesssim&\sum_{n_{0};n_{2}+n_{3}+n_{4}=c_{1}}\big(\langle n_{3}\rangle+\langle n_{4}\rangle\big)^{-1}\int_{(T)}\big|f_{n_{0},\alpha_{0}}\big|\times\nonumber\\
&\times&\big|(y^{1})_{n_{0},\alpha_{1}}\big|\prod_{l=2}^{4}\big|(z^{l})_{n_{l},\alpha_{l}}\big|\cdot|\phi_{\alpha_{5}}|\nonumber\\
&\lesssim&T^{0+}\sum_{n_{0};n_{2}+n_{3}+n_{4}=c_{1}}\big(\langle n_{3}\rangle+\langle n_{4}\rangle\big)^{-1}\times\nonumber\\
&\times&\big\|\widehat{f_{n_{0}}}\big\|_{L^{q}}\big\|\widehat{(y^{1})_{n_{0}}}\big\|_{L^{1}}\prod_{l=2}^{4}\big\|\widehat{(z^{l})_{n_{l}}}\big\|_{L^{1}}\nonumber\\
\label{luckystar}&\lesssim&T^{0+}2^{rd}T_{d}\cdot 2^{-rd}\cdot\big\|z^{2}\big\|_{l^{3}L^{1}}\prod_{l=3}^{4}\big\|\langle n_{l}\rangle^{-\frac{1}{2}}z^{l}\big\|_{l^{\frac{6}{5}}L^{1}}\\
&\lesssim& T^{0+}2^{-csd'}T_{d},\nonumber\end{eqnarray} using (\ref{festt}) for $f$, $X_{2}$ bound for $y^{1}$, and slightly weaker bounds for $z^{l}$ that follows from Proposition \ref{weaken}. Here $c_{j}$ are constants, and the integral $(T)$ is taken over the set
\begin{equation}\bigg\{(\alpha_{0},\cdots,\alpha_{5}):\alpha_{0}=\alpha_{15}-\sum_{l=2}^{4}|n_{l}|n_{l}+c_{2}\bigg\}.\end{equation} The reason we can gain $2^{csd'}$ is that, in (\ref{luckystar}) we can restrict some $n_{l}$, where $2\leq l\leq 4$, to be $\sim 2^{d'}$ before using the corresponding control for $z^{l}$ (for example, when $n_{2}\sim 2^{d}$ we will have $\|z^{2}\|_{l_{n_{2}\sim 2^{d'}}^{3}L^{1}}\lesssim 2^{-csd'}$). If we then sum over $m_{i}$, integrate over $\beta_{i}$, and sum over $d,d'$, we will get the desired estimate.

In the harder case $j=3$, we will assume $\langle m_{i}\rangle+\langle \beta_{i}\rangle\ll 2^{\frac{d'}{90}}$. In fact, if this does not hold, we will gain a power $2^{cd'}$ from this term and estimate the $\mathcal{S}_{sub}$ as above, except that we estimate $(w')^{\omega_{2}}$ and $(w')^{\omega_{3}}$ in $l^{2}L^{1}$ with a loss of $2^{O(s)d'}$ (note in particular we estimate $f$ and $(w'')^{\omega_{1}}$ exactly as above, so we do not gain or lose any power of $2^{d}$), to conclude. In the same way, we may assume $\langle \alpha_{4}\rangle\ll 2^{\frac{d'}{90}}$ in (\ref{s3expre}). Now if $n_{2}+n_{3}=0$, we must have $|\Phi|\lesssim 2^{-|d-d'|}$. Also\footnote[1]{Since $j=3$, we are actually replacing $(w')^{\omega_{l}}$ with $(w'')^{\omega_{l}}$ for $l\in\{2,3\}$; the fact that $z^{l}$ comes from $w''$ will be used later in the estimates.} we may replace $z^{2}$ and $z^{3}$ in (\ref{s3expre}) with $y^{2}$ and $y^{3}$ (in the same way we replace $f'$ and $(w')^{\omega_{1}}$ with $f$ and $(w'')^{\omega_{1}}$; note that we have not made any restrictions for $\alpha_{2}$ and $\alpha_{3}$). Then we may fix $m_{i}$ and $\beta_{i}$ (here the $m$ variables satisfy some linear relation which we ignore) and bound $\mathcal{S}_{sub}$ by
\begin{eqnarray}
\mathcal{S}_{sub}&\lesssim&2^{-|d-d'|}\sum_{n_{0},n_{2}}\int_{\alpha_{0}=\alpha_{1}+\cdots+\alpha_{4}+c_{2}}\big|f_{n_{0},\alpha_{0}}\big|\times\nonumber\\
&\times&\big|((w'')^{\omega_{1}})_{n_{0},\alpha_{1}}\big|\cdot\big|(y^{2})_{n_{2},\alpha_{2}}\big|\cdot\big|(y^{3})_{-n_{2},\alpha_{3}}\big|\cdot|\phi_{\alpha_{4}}|\nonumber\\
&\lesssim &T^{0+}2^{-|d-d'|}\sum_{n_{0},n_{2}}\big\|\widehat{f_{n_{0}}}\big\|_{L^{q}}\times\nonumber\\
&\times&\big\|(((w^{''})^{\omega_{1}})_{n_{0}})^{\wedge}\big\|_{L^{1}}\big\|\widehat{(y^{2})_{n_{2}}}\big\|_{L^{1}}\big\|\widehat{(y^{3})_{-n_{2}}}\big\|_{L^{1}}\nonumber\\
&\lesssim&T^{0+}2^{-|d-d'|}\big\|f\big\|_{l_{\sim 2^{d}}^{p'}L^{q}}\big\|w''\big\|_{l_{\sim 2^{d}}^{p}L^{1}}\times\nonumber\\
&\times&\big\|y^{2}\big\|_{l_{\sim 2^{d'}}^{2}L^{1}}\big\|y^{3}\big\|_{l_{\sim 2^{d'}}^{2}L^{1}}\nonumber\\
&\lesssim & T^{0+}2^{-|d-d'|}T_{d},\nonumber
\end{eqnarray} where $c_{j}$ are constants, and we are restricting to $n_{0}\sim 2^{d}$, $n_{2}\sim 2^{d'}$. Then we may sum and integrate over $(m_{i},\beta_{i})$, and sum over $d,d'$ to bound this part by $T^{0+}$.

Now assume $j=3$, $n_{2}+n_{3}\neq 0$, and all the restrictions made before hold. In particular we have $n_{2}\sim n_{3}\sim 2^{d'}$ and $|\Xi'|\gtrsim 2^{d'}$ where $\Xi'=|n_{2}|n_{2}+|n_{3}|n_{3}$ (again we may assume $2^{d'}$ is large, otherwise we proceed as before). Fixing $m_{i}$ and $\beta_{i}$, we then need to bound
\begin{equation}
\mathcal{S}_{sub}=\sum_{n_{0},n_{2}}\int_{(T)}\big|f_{n_{0},\alpha_{0}}\big|\cdot\big|((w'')^{\omega_{1}})_{n_{0},\alpha_{1}}\big|\cdot\big|(z^{2})_{n_{2},\alpha_{2}}\big|\cdot\big|(z^{3})_{c_{1}-n_{2},\alpha_{3}}\big|\cdot|\phi_{\alpha_{4}}|,\nonumber
\end{equation} where $c_{1}-n_{2}=n_{3}$, the integral $(T)$ is over the set\begin{equation}\big\{(\alpha_{0},\cdots,\alpha_{4}):\alpha_{0}=\alpha_{14}-\Xi'+c_{2}\big\},\end{equation} and $c_{j}\ll 2^{\frac{d'}{10}}$ are constants. Also each $z^{l}$ \emph{here} is either $w'$ or $\overline{w'}$. Now, by Proposition \ref{linearestimate2}, we can show that $\|g\|_{X_{j}'}\leq 1$ for some $j\in\{1,2,5,7\}$ implies \begin{equation}\|f\|_{X_{10}'}=\big\|\langle n_{0}\rangle^{-r}\langle\alpha_{0}\rangle^{\frac{1}{8}}f\big\|_{l_{d\geq 0}^{1}L^{\tau'}l_{n_{0}\sim 2^{d}}^{p'}}\lesssim 1.\end{equation} For the $w''$ we will use the $X_{7}$ bound, and for $z^{l}$ we will simply use the $X_{1}$ bound. Now, since at least one $\alpha_{l}$ must be $\gtrsim 2^{d'}$, we will gain some $2^{cd'}$ from the $\langle\alpha_{l}\rangle$ weight in one of the above bounds. If $l=0$, we can then estimate $f$ in $l^{p'}L^{\tau'}$ by $2^{rd}T_{d}$, $w''$ in $l^{p}L^{1}$ by $2^{-rd}$ (recall we are restricting to $n_{0}\sim 2^{d}$ and $n_{2}\sim 2^{d'}$), and $z^{2,3}$ in $l^{2}L^{1}$ with a loss of $2^{O(s)d'}$, so that we can use H\"{o}lder to conclude. If $l=1$, we simply replace the $l^{p}L^{1}$ bound by the $l^{p}L^{2}$ bound and argue as in the case $l=0$. If $l\in\{2,3\}$, we may replace the $l^{2}L^{1}$ bound for $z^{l}$ by the $l^{2}L^{2}$ bound and argue as in the case $l=0$. If $l=4$ we simply gain from the $\phi$ factor. This completes the proof for the $n_{0}=n_{1}$ case.

Finally, assume that $j=3$, and three of $n_{l}$ are related by $m$. We may assume that $\langle m_{i}\rangle \ll2^{\frac{d}{90}}$, so that $n_{l}\sim 2^{d}$ for each $l$. Then we fix $m_{i}$ and $\beta_{i}$, so that $n_{l}$ are uniquely determined by $n_{0}$. The corresponding $\mathcal{S}_{sub}$ will be bounded by\begin{equation}T^{0+}\|f'\|_{l_{\sim 2^{d}}^{4}L^{q}}\|z^{1}\|_{l_{\sim 2^{d}}^{4}L^{1}}\|z^{2}\|_{l_{\sim 2^{d}}^{4}L^{1}}\|z^{3}\|_{l_{\sim 2^{d}}^{4}L^{1}}\lesssim T^{0+}2^{-sd},\nonumber\end{equation} due to a similar computation as in the proof of Proposition \ref{easiest}.
\end{proof}
Now we start to consider the $\mathcal{M}^{2}$ term. Fixing the functions $g,f,f'$ and the scale $d$ as before, we need to bound the expression
\begin{eqnarray}\label{m2control}\mathcal{S}&=&\sum_{n_{0}=n_{1}+n_{2}+m_{1}+\cdots+m_{\mu}}\int_{(T)}\Phi^{2}\cdot\overline{f_{n_{0},\alpha_{0}}}\prod_{l=1}^{2}(y^{\omega_{l}})_{n_{l},\alpha_{l}}\cdot\phi_{\alpha_{3}}\times\\
&\times&\big(\chi e^{\mathrm{i}(\Delta_{n_{1}}+\Delta_{n_{2}}-\Delta_{n_{0}})}\big)^{\wedge}(\alpha_{4})\prod_{i=1}^{\mu}\frac{(u''')_{m_{i},\beta_{i}}}{m_{i}}.\nonumber\end{eqnarray} Here the integration $(T)$ is over the set
\begin{equation}\big\{(\alpha_{0},\cdots,\alpha_{4},\beta_{1},\cdots,\beta_{\mu}):\alpha_{0}=\alpha_{14}+\beta_{1\mu}+\Xi\big\},\nonumber\end{equation} where the NR factor
\begin{equation}\label{nrm2}\Xi=|n_{0}|n_{0}-\sum_{l=1}^{2}|n_{l}|n_{l}-\sum_{i=1}^{\mu}|m_{i}|m_{i}.\end{equation} Note that we may insert $\chi$ since $f$ has compact time support. Suppose the minimum of $\langle n_{l}\rangle$ is $\sim 2^{h}$ and also fix $h$, then we have $\langle m_{i}\rangle \ll 2^{h}$, so that $|\Xi|\gtrsim 2^{d+h}$; also we have $h\leq d+O(1)$ and $|\Phi^{2}|\lesssim 2^{h}$. Note that one of $\alpha$ or $\beta$ variables must be $\gtrsim 2^{d+h}$; we first treat the easy cases, which we collect in the following proposition.
\begin{proposition}\label{easter} Let $\mathcal{S}$ be defined in (\ref{m2control}), where all the restrictions made above are assumed. Then if we have $h<0.9d$, or
\begin{equation}\langle\alpha_{0}\rangle+\langle\alpha_{3}\rangle+\langle\alpha_{4}\rangle+\sum_{i=1}^{\mu}\big(\langle m_{i}\rangle+\langle\beta_{i}\rangle\big)\gtrsim 2^{\frac{d}{90}},\end{equation} the corresponding contribution will be bounded by $T^{0+}2^{(0-)d}$.
\end{proposition}
\begin{proof}
First assume $h\geq 0.9d$. If $\beta_{i}\gtrsim 2^{d+h}$ for some $i$, we may use the $X_{4}$ bound for $\langle \partial_{x}\rangle^{-s^{3}}u'''$ to gain a power $2^{0.99(d+h)}$ and then estimate this $(u''')_{m_{i},\beta_{i}}$ factor in $L^{2}L^{2}$. Next we may bound 
\begin{equation}\label{expcon}\big|\big(\chi e^{\mathrm{i}(\Delta_{n_{1}}+\Delta_{n_{2}}-\Delta_{n_{0}})}\big)^{\wedge}(\alpha_{4})\big|\lesssim 2^{s^{3}d}\langle\alpha_{4}\rangle^{-1}\end{equation} by Lemma \ref{general} and Proposition \ref{factt}, estimate the right hand side (again, viewed as a function of space-time supported at $n=0$) as well as the $\phi_{\alpha_{3}}$ factor in $l^{1+}L^{1+}$. We then fix $(m_{j},\beta_{j})$ for $j\neq i$ to produce an $\mathcal{S}_{sub}$ involving $(u''')_{m_{i},\beta_{i}}$, which we estimate by controlling $\mathfrak{N}f$ in $L^{6-}L^{6-}$, $\mathfrak{N}y^{\omega_{l}}$ in $L^{6+}L^{6+}$ (using the $X_{1}$ bound for $y$ and the norm for $f$ deduced from the $X_{6'}$ bound for $g$; here the $6-$ and $6+$ are $6+O(s)$). In this process we lose at most $2^{O(s)d}$, but the gain $2^{0.99(d+h)}$ (even after canceling the $2^{h}$ loss coming from the $\Phi^{2}$ weight) will allow us to cancel the $\Phi^{2}$ factor and still gain $2^{cd}$.

Next, suppose $\langle\alpha_{3}\rangle\gtrsim 2^{d+h}$. By using (\ref{expcon}) and losing a harmless $2^{O(s)d}$ factor, the argument for $\alpha_{4}$ can be done in the same way. Let $c_{j}$ be constants (or functions of $n_{l}$), and recall we are restricting to $\sum_{l}\langle n_{l}\rangle\sim 2^{d}$, we may fix $m_{i}$ and $\beta_{i}$, and bound the $\mathcal{S}_{sub}$ term by
\begin{eqnarray}\mathcal{S}_{sub}&\lesssim&T^{0+}\sum_{n_{0}=n_{1}+n_{2}+c_{1}}\int_{\alpha_{0}=\alpha_{1}+\cdots+\alpha_{4}+c_{2}(n_{0},\cdots,n_{2})}2^{h}\times\nonumber\\
&\times&\big|f_{n_{0},\alpha_{0}}\big|\cdot\prod_{l=1}^{2}\big|(y^{\omega_{l}})_{n_{l},\alpha_{1}}\big|\cdot\frac{\mathbf{1}_{\{\alpha_{3}\gtrsim 2^{d+h}\}}}{\langle\alpha_{3}\rangle^{0.9}\langle\alpha_{4}\rangle}\nonumber\\
&\lesssim &T^{0+}\sum_{n_{0}=n_{1}+n_{2}+c_{1}}2^{-0.62d}\big\|\widehat{f_{n_{0}}}\big\|_{L^{q}}\big\|\widehat{(y^{\omega_{1}})_{n_{1}}}\big\|_{L^{1}}\big\|\widehat{(y^{\omega_{2}})_{n_{2}}}\big\|_{L^{1}}\nonumber\\
&\lesssim &T^{0+}2^{-cd}\big\|\langle n_{0}\rangle^{-0.2}f\big\|_{l^{\frac{3}{2}}L^{q}}\big\|\langle n_{1}\rangle^{-0.2}y\big\|_{l^{\frac{3}{2}}L^{1}}\big\|\langle n_{2}\rangle^{-0.2}y\big\|_{l^{\frac{3}{2}}L^{1}}\nonumber\\
&\lesssim& T^{0+}2^{-cd}.\nonumber\end{eqnarray} Thus this term is also acceptable.

Next, assume that $\langle\alpha_{1}\rangle\gtrsim 2^{d+h}$ (the $\alpha_{2}$ case is proved in the same way), and that one of $\alpha_{0}$, $\alpha_{3}$, $\alpha_{4}$, $m_{i}$ or $\beta_{i}$ is $\gtrsim 2^{\frac{d}{90}}$. We then use (\ref{expcon}) to bound the exponential factor and fix $(m_{i},\beta_{i})$. To estimate the resulting $\mathcal{S}_{sub}$, we use the $\langle\alpha_{1}\rangle^{b}$ factor in the $X_{1}$ bound for $y$ to cancel the $\Phi^{2}$ factor which is at most $2^{h}$ and bound the resulting $y^{\omega_{1}}$ factor in $L^{2}L^{2}$, then bound $\mathfrak{N}f$ and $\mathfrak{N}y^{\omega_{2}}$ in $L^{4+}L^{4+}$, and bound the factors involving $\alpha_{3}$ and $\alpha_{4}$ in $l^{1+}L^{1+}$, where $4+$ some $4+c$. In this process we may lose $2^{O(s)d}$, but since another $\alpha_{l}$ or $(m_{i},\beta_{i})$ is $\gtrsim 2^{\frac{d}{90}}$, we will be able to gain $2^{cd}$ from this factor (since the $L^{4+}L^{4+}$ Strichartz estimate allows for some room), we will find this term acceptable.

The only remaining case is when $\langle\alpha_{0}\rangle\gtrsim 2^{d+h}$. By basically the same argument as above, we may assume that the other $\alpha_{l}$ and  $(m_{i},\beta_{i})$ are all $\ll 2^{\frac{d}{90}}$. Also recall that two of $n_{l}(0\leq l\leq 2)$ are $\sim 2^{d}$ and the third is $\sim 2^{h}$. Now we may use the bound (\ref{expcon}), then fix $(\alpha_{3},\alpha_{4})$ and all $(m_{i},\beta_{i})$ to produce \begin{equation}
|\mathcal{S}_{sub}|\lesssim2^{(b-s)h-(1-b)d}\sum_{n_{0}=n_{1}+n_{2}+c_{1}}\int_{\alpha_{0}=\alpha_{1}+\alpha_{2}+\Xi'+c_{2}}A_{n_{0},\alpha_{0}}B_{n_{1},\alpha_{1}}C_{n_{2},\alpha_{2}},
\end{equation} where $c_{j}\ll 2^{\frac{d}{10}}$ are constants, the factor
\begin{equation}\Xi'=|n_{0}|n_{0}-|n_{1}|n_{1}-|n_{2}|n_{2},\end{equation} and the relevant functions are defined by
\begin{equation}A_{n_{0},\alpha_{0}}=\langle n_{0}\rangle^{-s}\langle\alpha_{0}\rangle^{1-b}|f_{n_{0},\alpha_{0}}|;\nonumber\end{equation} 
\begin{equation}B_{n_{1},\alpha_{1}}=\langle n_{1}\rangle^{s}|(y^{\omega_{1}})_{n_{1},\alpha_{1}}|;\,\,\,\,C_{n_{2},\alpha_{2}}=\langle n_{2}\rangle^{s}|(y^{\omega_{2}})_{n_{2},\alpha_{2}}|.\nonumber\end{equation}Also note that when we sum over $m_{i}$, integrate over $\beta_{i}$ and $(\alpha_{3},\alpha_{4})$, we will gain $T^{0+}$ and lose at most $2^{O(s^{3})d}$.

Now we estimate $\mathcal{S}_{sub}$. If $\|g\|_{X_{m}'}\leq 1$ for some $m\in\{1,2\}$, by using Proposition \ref{linearestimate2}, we may assume\footnote[1]{Note that in proving $\|\mathcal{E}u\|_{X_{1}}\lesssim\|u\|_{W_{1}}+\|u\|_{W_{2}}$ with $\|u\|_{W_{j}}=\|\langle\xi\rangle^{-1}u\|_{X_{j}}$, we actually break $\mathcal{E}$ into two linear operators that are bounded from each $W_{j}$ to $X_{1}$ separately. Therefore, when $g\in X_{1}$, we can write $\mathcal{E}'g=f$ as a sum of two functions, each bounded in one $W_{j}$. The same applies for $X_{2}$, $X_{5}$ and $X_{7}$ norms.} that $\|\langle\alpha_{0}\rangle f\|_{X_{j}'}\lesssim 1$ for some $j\in\{1,2\}$. If $\|g\|_{X_{m}'}\lesssim 1$ for some $m\in\{5,7\}$, since we may insert a $\mathbf{1}_{E}$ factor to $\overline{f_{n_{0},\alpha_{0}}}$ with $E=\{n_{0}\sim 2^{d'},\alpha_{0}\gtrsim 2^{d'}\}$ with $d'\in\{d,h\}$, we can use (\ref{linnn4}) and again assume $\|\langle\alpha_{0}\rangle f\|_{X_{j}'}\lesssim 1$ for some $j\in\{1,2\}$. Next, notice that $|\alpha_{0}-\Xi'|\lesssim 2^{\frac{d}{10}}$, so $\alpha_{0}$ is also restricted to some set of measure $O(2^{1.1d})$ for each fixed $n_{0}$. Since $\alpha_{0}$ is restricted to be $\gtrsim 2^{1.9d}$ and $n_{0}\lesssim 2^{d}$, we will have
\begin{eqnarray}\|\langle\alpha_{0}\rangle f\|_{X_{1}'}&\lesssim& 2^{O(s)d}\|\langle\alpha_{0}\rangle^{0.6}f\|_{l^{2}L^{2}}\lesssim 2^{(O(s)+0.55)d}\|\langle\alpha_{0}\rangle^{0.6}f\|_{l^{2}L^{\infty}}\nonumber\\
&\lesssim&2^{(0.6-0.4\times1.9)d}\|\langle\alpha_{0}\rangle f\|_{l^{2}L^{\infty}}\lesssim 2^{-cd}\|\langle\alpha_{0}\rangle f\|_{X_{2}'},\nonumber\end{eqnarray} thus we may furthermore assume $j=1$.

Now, using this bound for $f$ and the $X_{1}$ bound for $y$, we deduce that
\begin{equation}\|A\|_{l^{p'}L^{2}}+\|\langle\alpha_{1}\rangle^{b}B\|_{l^{p}L^{2}}+\|\langle \alpha_{2}\rangle^{b}C\|_{l^{p}L^{2}}\lesssim 2^{O(s^{2})d}.\nonumber\end{equation} Let us define
\begin{equation}\mathbf{B}_{n_{1}}=\big\|\langle\widetilde{\alpha_{1}}+|n_{1}|n_{1}\rangle^{b}\widehat{B_{n_{1}}}(\widetilde{\alpha_{1}})\big\|_{L^{2}}\nonumber\end{equation} and $\mathbf{C}_{n_{3}}$ similarly, so that $\|\mathbf{B}\|_{l^{p}}+\|\mathbf{C}\|_{l^{p}}\lesssim 2^{O(s^{2})d}$, then we will have the estimate
\begin{equation}\big\|\langle\alpha_{0}-\Xi'-c_{2}\rangle^{2b-\frac{1}{2}}(\widehat{B_{n_{1}}}*\widehat{C_{n_{2}}})(\alpha_{0}-|n_{0}|n_{0}-c_{2})\big\|_{L^{2}}\lesssim\mathbf{B}_{n_{1}}\mathbf{C}_{n_{2}},\nonumber\end{equation} which, after taking Fourier transform, follows from the standard one dimensional inequality $\|fg\|_{H^{2b-\frac{1}{2}}}\lesssim\|f\|_{H^{b}}\|g\|_{H^{b}}$. Now we will be able to control $\mathcal{S}_{sub}$ by
\begin{equation}\mathcal{S}_{sub} \lesssim2^{\lambda}\bigg(\sum_{n_{0}}\bigg\|\sum_{n_{1}+n_{2}=n_{0}-c_{1}}\big(\widehat{B_{n_{1}}}*\widehat{C_{n_{2}}}\big)(\alpha_{0}-|n_{0}|n_{0}-c_{2})\bigg\|_{L_{\alpha_{0}}^{2}}^{p}\bigg)^{\frac{1}{p}},\nonumber
\end{equation} where $\lambda=(b-s)h+(b-1+O(s^{2}))d$, and the square of the inner $L^{2}$ norm is bounded by
\begin{eqnarray}
\mathcal{J}^{2}&=&\int_{\mathbb{R}}\bigg|\sum_{n_{1}+n_{2}=n_{0}-c_{1}}\big(\widehat{B_{n_{1}}}*\widehat{C_{n_{2}}}\big)(\alpha_{0}-|n_{0}|n_{0}-c_{2})\bigg|^{2}\,\mathrm{d}\alpha_{0}\nonumber\\
&\lesssim&\int_{\mathbb{R}}\bigg(\sum_{n_{1}+n_{2}=n_{0}-c_{1}}\langle \alpha_{0}-\Xi'-c_{2}\rangle^{1-4b}\bigg)\mathrm{d}\alpha_{0}\times\nonumber\\
&\times&\bigg(\sum_{n_{1}+n_{2}=n_{0}-c_{1}}\langle \alpha_{0}-\Xi'-c_{2}\rangle^{4b-1}\big|\big(\widehat{B_{n_{1}}}*\widehat{C_{n_{2}}}\big)(\alpha_{0}-|n_{0}|n_{0}-c_{2})\big|^{2}\bigg)\nonumber\\&\lesssim&\sup_{\alpha_{0}}\bigg(\sum_{n_{1}+n_{2}=n_{0}-c_{1}}\langle \alpha_{0}-\Xi'-c_{2}\rangle^{1-4b}\bigg)\times\nonumber\\
&\times&\sum_{n_{1}+n_{2}=n_{0}-c_{1}}\int_{\mathbb{R}}\langle \alpha_{0}-\Xi'-c_{2}\rangle^{4b-1}\big|\big(\widehat{B_{n_{1}}}*\widehat{C_{n_{2}}}\big)(\alpha_{0}-|n_{0}|n_{0}-c_{2})\big|^{2}\mathrm{d}\alpha_{0}\nonumber\\
&\lesssim&\sup_{\alpha_{0}}\bigg(\sum_{n_{1}+n_{2}=n_{0}-c_{1}}\langle \alpha_{0}-\Xi'-c_{2}\rangle^{4b-1}\bigg)\cdot\sum_{n_{1}+n_{2}=n_{0}-c_{1}}\mathbf{B}_{n_{1}}^{2}\mathbf{C}_{n_{2}}^{2}.\nonumber
\end{eqnarray} Next we claim that for fixed $n_{0}$ and $\alpha_{0}$ we have \begin{equation}\label{sparsesum}\sum_{n_{1}+n_{2}=n_{0}-c_{1}}\langle \alpha_{0}-\Xi'-c_{2}\rangle^{-\frac{3}{4}}\lesssim 1.\end{equation} In fact, if $n_{1}n_{2}<0$, then $\alpha_{0}-\Xi'-c_{2}$ is a linear expression in $n_{1}$ with leading coefficient $k=\pm(n_{0}-c_{1})/2\gtrsim 2^{0.9d}$ (we assume $d$ is large enough), so any two summands in (\ref{sparsesum}) differ by at least $k$, while there are $\lesssim 2^{d}$ summands. The sum is thus bounded by
\begin{equation}1+\sum_{h=1}^{2^{d}}(kh)^{-\frac{3}{4}}\lesssim 1+k^{-\frac{3}{4}}2^{\frac{d}{4}}\lesssim 1.\end{equation} If $n_{1}n_{2}>0$, then $\alpha_{0}-\Xi'-c_{2}$ equals $\pm\frac{1}{2}(n_{1}-n_{2})^{2}$ plus a constant, so similarly we only need to prove
\begin{equation}\sum_{k\in\mathbb{Z}}\langle \alpha- k^{2}\rangle^{-\frac{3}{4}}\lesssim 1\nonumber\end{equation} for each $\alpha$, but this is again easily proved by separating the cases $\langle k\rangle^{2}\lesssim\langle\alpha\rangle$ and otherwise, and applying elementary inequalities.

Now we are able to bound
\begin{eqnarray}\mathcal{S}_{sub}&\lesssim & 2^{\lambda}\bigg(\sum_{n_{0}}\bigg(\sum_{n_{1}+n_{2}=n_{0}-c_{1}}\mathbf{B}_{n_{1}}^{2}\mathbf{C}_{n_{2}}^{2}\bigg)^{\frac{p}{2}}\bigg)^{\frac{1}{p}}\nonumber\\
&\lesssim &2^{\lambda+(\frac{1}{2}-\frac{1}{p})d}\bigg(\sum_{n_{0}}\sum_{n_{1}+n_{2}=n_{0}-c_{1}}\mathbf{B}_{n_{1}}^{p}\mathbf{C}_{n_{2}}^{p}\bigg)^{\frac{1}{p}}\nonumber\\
&\lesssim & 2^{\lambda+(\frac{1}{2}-\frac{1}{p})d}\|\mathbf{B}\|_{l^{p}}\|\mathbf{C}\|_{l^{p}},\nonumber
\end{eqnarray} where we notice that
\begin{eqnarray}\lambda+\bigg(\frac{1}{2}-\frac{1}{p}\bigg)d&=&(b-s)h+\bigg(b-\frac{1}{2}-\frac{1}{p}+O(s^{2})\bigg)d\nonumber\\
&\lesssim &(2b-1)d+\bigg(\frac{1}{2}-s-\frac{1}{p}+O(s^{2})\bigg)d,\nonumber\end{eqnarray} and this is $\leq-c(1/2-b)d$ by (\ref{hierarchy}). We may then sum and integrate over the previously fixed variables to get a desirable estimate for $\mathcal{S}$.

Finally, suppose $h<0.9d$. Since at least one $\alpha_{l}$ or $\beta_{i}$ will be $\gtrsim 2^{d+h}$, we may repeat the arguments above; using the inequality $2^{b(d+h)}\gtrsim 2^{cd+h}$ that holds for $h<0.9d$, we will be able to gain an additional power of $2^{cd}$ after canceling the $\Phi^{2}$ weight, which will allow us to close the estimate as above. This completes the proof.
\end{proof}

What remains to be bounded, denoted by $\mathcal{S}^{E}$, is actually the same summation-integration as $\mathcal{S}$, but restricted to the region $h\geq 0.9d$ and with the additional factor $\mathbf{1}_{E}$, where
\begin{equation}E=\big\{\langle\alpha_{0}\rangle\vee\langle\alpha_{3}\rangle\vee\langle\alpha_{4}\rangle\vee\big\langle m_{i}\rangle\vee\langle\beta_{i}\rangle\ll 2^{\frac{d}{90}},\,\forall i\big\},\nonumber\end{equation} with $a\vee b$ meaning $\max\{a,b\}.$ Now let $E_{l}=\big\{\langle\alpha_{l}\rangle\ll 2^{\frac{d}{90}}\big\}$ for $l\in\{1,2\}$, we have
\begin{equation}\mathbf{1}_{E}=\mathbf{1}_{E\cap E_{1}}+\mathbf{1}_{E\cap E_{2}}+\mathbf{1}_{E-(E_{1}\cup E_{2})}.\end{equation} By symmetry, we need to bound $\mathcal{S}^{E\cap E_{1}}$ and $\mathcal{S}^{E-(E_{1}\cup E_{2})}$ (whose meaning is obvious). In the latter case, we may assume that $\alpha_{1}\gtrsim 2^{d+h}$, and also $\alpha_{2}\gtrsim 2^{\frac{d}{90}}$, so we can estimate this part in the same was as in the proof of Proposition \ref{easter}.

It remains to bound $\mathcal{S}^{E\cap E_{1}}$. Let $E\cap E_{1}=F$, using (\ref{computation}) and (\ref{global2}) we may compute
\begin{eqnarray}(y^{\omega_{2}})_{n_{2},\alpha_{2}}&=&\big(\chi(t)e^{-H\partial_{xx}}w^{\omega_{2}}(0)\big)_{n_{2},\alpha_{2}}+\big(\mathcal{E}(\mathbf{1}_{[-T,T]}\cdot\mathcal{N}^{2}(y,y))^{\omega_{2}}\big)_{n_{2},\alpha_{2}}\nonumber\\
&+&\sum_{j\in\{3,3.5,4,4.5\}}(\mathcal{E}(\mathbf{1}_{[-T,T]}\mathcal{N}^{j})^{\omega_{2}})_{n_{2},\alpha_{2}}\nonumber\\
\label{secondit}&=&\sum_{j\in\{0,3,3.5,4,4.5\}}((\mathcal{M}^{j})^{\omega_{2}})_{n_{2},\alpha_{2}}+(\mathcal{L}^{1})_{n_{2},\alpha_{2}}+(\mathcal{L}^{2})_{n_{2},\alpha_{2}}.
\end{eqnarray} Here we denote $\mathcal{M}^{0}=\chi(t)e^{\mathrm{i}\partial_{xx}}w(0)$, and
\begin{eqnarray}(\mathcal{L}^{1})_{n_{2},\alpha_{2}}&=&c_{1}\int_{\mathbb{R}^{2}}\frac{\widehat{\chi}(\alpha_{2}-\gamma_{2})\widehat{\chi}(\gamma_{2}-\gamma_{1})}{\gamma_{2}}\mathcal{I}_{n_{2},\gamma_{1}}\,\mathrm{d}\gamma_{1}\mathrm{d}\gamma_{2};\nonumber\\
(\mathcal{L}^{2})_{n_{2},\alpha_{2}}&=&c_{2}\widehat{\chi}(\alpha_{2})\cdot\int_{\mathbb{R}^{2}}\frac{\widehat{\chi}(\gamma_{2}-\gamma_{1})}{\gamma_{2}}\mathcal{I}_{n_{2},\gamma_{1}}\,\mathrm{d}\gamma_{1}\mathrm{d}\gamma_{2},\nonumber\end{eqnarray} where $\mathcal{I}=(\mathbf{1}_{[-T,T]}\mathcal{N}^{2}(y,y))^{\omega_{2}}$. Interpreting the singular integral as a principal value, we may compute that\begin{equation}\bigg|\int_{\mathbb{R}}\frac{\widehat{\chi}(\gamma_{2}-\gamma_{1})}{\gamma_{2}}\,\mathrm{d}\gamma_{2}\bigg|\lesssim\frac{1}{\langle\gamma_{1}\rangle};\nonumber\end{equation}\begin{equation}\bigg|\int_{\mathbb{R}}\frac{\widehat{\chi}(\alpha_{2}-\gamma_{2})\widehat{\chi}(\gamma_{2}-\gamma_{1})}{\gamma_{2}}\,\mathrm{d}\gamma_{2}\bigg|\lesssim\frac{1}{(\langle\alpha_{2}\rangle+\langle\gamma_{1}\rangle)\langle\alpha_{2}-\gamma_{1}\rangle^{\frac{1}{s}}};\nonumber\end{equation}\begin{equation}\bigg|\nabla_{\alpha_{2},\gamma_{1}}\int_{\mathbb{R}}\frac{\widehat{\chi}(\alpha_{2}-\gamma_{2})\widehat{\chi}(\gamma_{2}-\gamma_{1})}{\gamma_{2}}\,\mathrm{d}\gamma_{2}\bigg|\lesssim\frac{1}{(\langle\alpha_{2}\rangle+\langle\gamma_{1}\rangle)^{2}\langle\alpha_{2}-\gamma_{1}\rangle^{\frac{1}{s}}},\nonumber\end{equation} where the third inequality can be proved by integrating by parts in $\gamma_{2}$. Now, to treat the first three terms in (\ref{secondit}), we may use Proposition \ref{easiest}, the easy observation that \begin{equation}\big\|\langle n_{2}\rangle^{-\frac{1}{20}}\langle\alpha_{2}\rangle^{\kappa}(\mathcal{M}^{0})_{n_{2},\alpha_{2}}\big\|_{l^{2}L^{2}}\lesssim\big\|\langle n\rangle^{-\frac{1}{20}}(w(0))_{n}\big\|_{l^{2}}\lesssim 1,\nonumber\end{equation} together with the following
\begin{proposition}If we consider the sum (\ref{m2control}) with the factor $\mathbf{1}_{F}$, and $y^{\omega_{2}}$ replaced by some function $\zeta$ verifying
\begin{equation}\big\|\langle n_{2}\rangle^{-\frac{1}{20}}\langle\alpha_{2}\rangle^{\kappa}\zeta_{n_{2},\alpha_{2}}\big\|_{l^{2}L^{2}}\lesssim 1,\end{equation} then this contribution can be bounded by $T^{0+}$.
\end{proposition}
\begin{proof}Since in $F$ we will have $\langle\alpha_{2}\rangle\gtrsim 2^{d+h}$, we can gain a power $2^{0.999(d+h)}$ from the $\langle\alpha_{2}\rangle^{\kappa}$ factor in the bound for $\zeta$. After exploiting this, we may then estimate $\zeta$ in $L^{2}L^{2}$ with a loss $2^{(\frac{1}{20}+O(s))d}$. Then we fix $(m_{i},\beta_{i})$ as usual, and use the inequality (\ref{expcon}) to bound the factor involving $\alpha_{4}$. To bound the resulting $\mathcal{S}_{sub}$ term, we estimate $\zeta$ in $L^{2}L^{2}$, $\mathfrak{N}f$ and $\mathfrak{N}y$ in $L^{4+}L^{4+}$ (where $4+$ equals $4+c$) with a loss of $2^{O(s)d}$, the $\alpha_{3}$ and $\alpha_{4}$ factors in $l^{1+}L^{1+}$. Note that here we will gain a power $T^{0+}$, and the total power of $2^{d}$ we may lose is at most $2^{(1.1+O(s))d}$, which is smaller than the gain $2^{0.999(d+h)}$. Then we sum over $m_{i}$ and integrate over $\beta_{i}$ to conclude.
\end{proof}
Next consider the contribution of $\mathcal{L}^{2}$. Since we are in $F$ (thus $\alpha_{2}\gtrsim 2^{d}$), the gain from $\widehat{\chi}(\alpha_{2})$ will overwhelm any possible loss in terms of $2^{d}$. Therefore we may even fix all the $n$, $m$ and $\beta$ variables and estimate the integral in $\alpha$ variables and $\gamma_{1}$ only; but we can easily estimate this integral by controlling all the factors except $\langle\gamma_{1}\rangle^{-1}|\mathcal{I}_{n_{2},\gamma_{1}}|$  in $L^{1+}$ (since the expression now has a convolution structure in the $\alpha$ variables), and estimate the $\langle\gamma_{1}\rangle^{-1}|\mathcal{I}_{n_{2},\gamma_{1}}|$ factor in $L^{1}$. This last estimate is due to (the proof of) Proposition \ref{easiest}, which implies
\begin{equation}\|\langle\gamma_{1}\rangle^{-1}\mathcal{I}_{n_{2},\gamma_{1}}\|_{L^{1}}\lesssim \|\langle\gamma_{1}\rangle^{\kappa-1}\mathcal{I}_{n_{2},\gamma_{1}}\|_{L^{2}}\lesssim 2^{O(1)d}.\nonumber\end{equation}

It then remains to bound the $\mathcal{L}^{1}$ contribution. After integrating over $\gamma_{2}$, we may rename the variable $\alpha_{2}-\gamma_{1}$ as $\gamma_{2}$, and reduce to estimating (up to a constant)
\begin{eqnarray}\mathcal{S}^{F}&=&\sum_{n_{0}=n_{1}+n_{2}+m_{1}+\cdots+m_{\mu}}\int_{(T)}\mathbf{1}_{F}\cdot\Phi^{2}\cdot\overline{f_{n_{0},\alpha_{0}}}\cdot(y^{\omega_{1}})_{n_{1},\alpha_{1}}\cdot\phi_{\alpha_{3}}\times\nonumber\\
&\times&\big(\chi e^{\mathrm{i}(\Delta_{n_{1}}+\Delta_{n_{2}}-\Delta_{n_{0}})}\big)^{\wedge}(\alpha_{4})\cdot\prod_{i=1}^{\mu}\frac{(u''')_{m_{i},\beta_{i}}}{m_{i}}\cdot\eta(\gamma_{1},\gamma_{2})\cdot\mathcal{I}_{n_{2},\gamma_{1}},\nonumber\end{eqnarray} where $\eta$ is some function bounded by\begin{equation}|\eta(\gamma_{1},\gamma_{2})|\lesssim\frac{1}{\langle\gamma_{1}\rangle\langle \gamma_{2}\rangle^{\frac{1}{s}}};\,\,\,\,|\partial_{\gamma_{1}}\eta(\gamma_{1},\gamma_{2})|\lesssim\frac{1}{\langle \gamma_{1}\rangle^{2}\langle\gamma_{2}\rangle^{\frac{1}{s}}}.\nonumber\end{equation} and the integral $(T)$ is taken over the set
\begin{equation}\big\{(\alpha_{0},\alpha_{1},\alpha_{3},\alpha_{4},\beta_{1},\cdots,\beta_{\mu},\gamma_{1},\gamma_{2}):\alpha_{0}=\alpha_{1}+\alpha_{34}+\beta_{1\mu}+\gamma_{12}+\Xi\big\},\nonumber\end{equation} with the NR factor is as in (\ref{nrm2}). Clearly we may also assume $\langle\gamma_{2}\rangle\ll 2^{\frac{d}{90}}$ and add this restriction into $F$ (or we simply gain a large power of $2^{d}$ and proceed as above); after doing this we will have $F\subset\{\langle\gamma_{1}\rangle\gtrsim 2^{d+h}\}$.

Next, note that $\overline{\mathcal{N}^{2}(y,y)}=\overline{\mathcal{N}^{2}}(\overline{y},\overline{y})$, where $\mathcal{N}^{2}$ is another bilinear form that differ from $\mathcal{N}^{2}$ only in the $\Phi^{2}$ weights; moreover, the $\Phi$ weight for $\overline{\mathcal{N}^{2}}$ will verify all the bounds we have for the $\Phi$ weight for $\mathcal{N}^{2}$. Thus we only need to bound the above expression with $\mathcal{I}_{n_{2},\gamma_{1}}$ replaced by $(\mathbf{1}_{[-T,T]}\mathcal{N}^{2}(y^{\omega_{2}},y^{\omega_{2}}))_{n_{2},\gamma_{1}}$. Clearly we may also fix the parameters $\mu'$ and $\omega'$ in $\mathcal{N}_{\mu'}^{\omega'2}$ and reduce to estimating
\begin{eqnarray}
\mathcal{S}'&=&\sum_{n_{0}=n_{1}+n_{5}+n_{6}+m_{1}+\cdots+m_{\nu}}\int_{(T)}\Phi^{2}(\Phi^{2})'\cdot \overline{f_{n_{0},\alpha_{0}}}\times\nonumber\\
&\times&\prod_{l\in\{1,5,6\}}(y^{\omega_{l}})_{n_{l},\alpha_{l}}\cdot\phi_{\alpha_{3}}\phi_{\alpha_{7}}\cdot\prod_{i=1}^{\nu}\frac{(u''')_{m_{i},\beta_{i}}}{m_{i}}\times\nonumber\\
&\times&\int_{\alpha_{4}+\alpha_{8}=\alpha_{9}}\mathbf{1}_{F}\eta(\gamma_{1},\gamma_{2})\cdot\big(\chi e^{\mathrm{i}(\Delta_{n_{1}}+\Delta_{n_{2}}-\Delta_{n_{0}})}\big)^{\wedge}(\alpha_{4})\times\nonumber\\
&\times&\big(\chi e^{\mathrm{i}(\Delta_{n_{5}}+\Delta_{n_{6}}-\Delta_{n_{2}})}\big)^{\wedge}(\alpha_{8}),\nonumber
\end{eqnarray} where $\nu=\mu+\mu'$, the integral $(T)$ is taken over the set
\begin{eqnarray}&&\big\{(\alpha_{0},\alpha_{1},\alpha_{3},\alpha_{5},\alpha_{6},\alpha_{7},\alpha_{9},\beta_{1},\cdots,\beta_{\nu},\gamma_{2}):\nonumber\\
&&\alpha_{0}=\alpha_{1}+\alpha_{3}+\alpha_{57}+\alpha_{9}+\beta_{1\nu}+\gamma_{2}+\Xi'\big\},\nonumber\end{eqnarray} with the new NR factor
\begin{equation}\Xi'=|n_{0}|n_{0}-|n_{1}|n_{1}-|n_{5}|n_{5}-|n_{6}|n_{6}-\sum_{i=1}^{\nu}|m_{i}|m_{i}.\nonumber\end{equation} The $\Phi^{2}$ and $(\Phi^{2})'$ are functions of the $n$ and $m$ variables that are bounded by $\min_{l\in\{0,1,2\}}\langle n_{l}\rangle$ and $\min_{l\in\{2,5,6\}}\langle n_{l}\rangle$ respectively. The other implicit variables are $n_{2}=n_{5}+n_{6}+m_{\mu+1,\nu}$ and
\begin{equation}\gamma_{1}=\alpha_{0}-\alpha_{1}-\alpha_{34}-\beta_{1\mu}-\gamma_{2}-\Xi=\alpha_{58}+\beta_{\mu+1,\nu}+(\Xi'-\Xi),\nonumber\end{equation} where $\Xi$ is the same as in (\ref{nrm2}). Also recall from the definition of $\mathcal{N}^{2}$ that $n_{0}\neq n_{1}$ and $n_{5}+n_{6}\neq 0$.

Next, let $\max\{\langle n_{2}\rangle,\langle n_{5}\rangle,\langle n_{6}\rangle\}\sim 2^{d'}$ so that $d'\geq h\geq 0.9d$, and fix $d'$ also. In the expression for $\mathcal{S}'$, we may assume
\begin{equation}\label{notona4}\langle m_{i}\rangle+\langle \beta_{i}\rangle+\langle\alpha_{j}\rangle\ll2^{\frac{d'}{70}}\end{equation} for all $\mu+1\leq i\leq \nu$ and $j\in\{5,6,7,9\}$ (note we already have this for $1\leq i\leq \mu$ and $j\in\{0,1,3\}$ due to the factor $\mathbf{1}_{F}$). In fact, if any one of these does not hold, we may bound $|\eta|\lesssim 2^{-(d+h)}\langle\gamma_{2}\rangle^{-10}$ and $|\Phi^{2}(\Phi^{2})'|\lesssim 2^{d+h}$ (so that the weight is cancelled by the part of the $\eta$ factor), then use (\ref{expcon}) to bound the $\alpha_{4}$ and $\alpha_{8}$ factors by $\langle\alpha_{4}\rangle^{-1}$ and $\langle\alpha_{8}\rangle^{-1}$ respectively with a loss $2^{O(s^{3})d'}$. Then we fix $(m_{i},\beta_{i})$ to produce $\mathcal{S}_{sub}'$, and estimate it by bounding the $\gamma_{2}$ and $\alpha_{l}$ factors for $l\in\{3,4,7,8\}$ in $l^{1+}L^{1+}$, and bounding the $\mathfrak{N}f$ and $\mathfrak{N}y$ factors in $L^{4+}L^{4+}$ (with $4+$ being $4+c$). Note that in the whole process we lose at most $2^{O(s)d'}$; but by our assumptions at least one $(m_{i},\beta_{i})$ or $\alpha_{l}$ must be $\gtrsim 2^{\frac{d'}{70}}$, so we will be able to gain some $2^{cd'}$ power from the corresponding factor (again using the room available for $L^{4+}L^{4+}$ Strichartz estimate) to complete the estimate.

Now we may assume all the variables mentioned above are small. This in particular implies that $|\Xi'|\ll 2^{\frac{d'}{70}}$. By Lemma \ref{divisor} (combined with the restrictions made above, such as $h\geq 0.9d$), we can conclude that either (i) $n_{0}=n_{5}$ and $n_{1}+n_{6}=0$ (or with $5$ and $6$ switched); or (ii) no two of $(-n_{0},n_{1},n_{5},n_{6})$ add to zero, and $\langle n_{l}\rangle\gtrsim 2^{0.9d'}$ for $l\in\{0,1,5,6\}$. In case (ii), we have in particular
\begin{equation}\langle n_{1}\rangle^{s}\langle n_{5}\rangle^{s}\langle n_{6}\rangle^{s}\gtrsim 2^{1.5sd}\langle n_{0}\rangle^{r},\end{equation} thus we may gain a power $2^{csd}$ from the $\langle n_{l}\rangle$ weights (after canceling $\Phi^{2}(\Phi^{2})'$ by the $\eta$ factor) if we use the $X_{2}$ bound for $y$ and the bound for $f$ deduced from the $X_{6}'$ bound for $g$. Then we simply bound the $\alpha_{4}$ and $\alpha_{8}$ factors using (\ref{expcon}) with $2^{O(s^{2})d}$ loss, take absolute value of everything, then fix $m_{i}$ and $\beta_{i}$ to produce a term $\mathcal{S}_{sub}$ that has basically the same form as the left hand side of (\ref{nonsharp}), with possibly some additional loss of $2^{O(s^{2})d}$, and with the $\min\{T,\langle\alpha_{4}\rangle^{-1}\}$ factor in (\ref{nonsharp}) replaced by $T^{0+}\langle\alpha_{4}\rangle^{-1+s^{4}}$, which is due to the estimate
\begin{equation}\int_{\alpha_{3}+\alpha_{4}+\alpha_{7}+\alpha_{8}+\gamma_{2}=\alpha_{10}}\min\big\{T,\langle\alpha_{3}\rangle^{-1}\big\}\cdot\langle\gamma_{2}\rangle^{-10}\prod_{l\in\{4,7,8\}}\langle\alpha_{l}\rangle^{-1}\lesssim T^{0+}\langle\alpha_{10}\rangle^{-1+s^{4}}.\nonumber\end{equation} We can then repeat the proof of (\ref{nonsharp}) to conclude (notice that every variable is now $\lesssim 2^{O(1)d'}$).

Now we consider case (i), so that $d'=d$. We will first replace the $\eta(\gamma_{1},\gamma_{2})$ factor appearing in the expression of $\mathcal{S}'$ by $\eta(\gamma_{1}',\gamma_{2})$, where $\gamma_{1}'=\gamma_{1}-\alpha_{8}$. Note that $\gamma_{1}'$ depends on $\alpha_{4}$ only through $\alpha_{9}=\alpha_{4}+\alpha_{8}$. When we estimate the difference caused by this substitution, since we still have the restriction $\mathbf{1}_{F}$, we will have $\gamma_{1}\sim 2^{d+h}$, so we will gain a power $2^{2(d+h)-\frac{d'}{70}}$, which is more than enough to cancel $\Phi^{2}(\Phi^{2})'$, thus this part will be acceptable. We also note that the assumption (\ref{notona4}) allows us to insert another characteristic function which depends on $\alpha_{4}$ only through $\alpha_{9}$; the presence of this function (as well as the part of $\mathbf{1}_{F}$ independent of $\alpha_{4}$) will allow us to conclude $\langle\gamma_{1}'\rangle\sim 2^{d+h}$. Therefore, if we \emph{remove} the part in $\mathbf{1}_{F}$ depending on $\alpha_{4}$, the error we create will be a summation-integration of the type $\mathcal{S}'$, but restricted to some set on which we have $|\eta|\lesssim 2^{-(d+h)}\langle \gamma_{2}\rangle^{-10}$ (note that here we already have $\eta(\gamma_{1}',\gamma_{2})$ instead of $\eta(\gamma_{1},\gamma_{2})$), \emph{as well as} $\langle\alpha_{4}\rangle \gtrsim2^{\frac{d'}{90}}$. Then we will be able to take absolute values, cancel $\Phi^{2}(\Phi^{2})'$ by the $\eta$ factor, and gain a power $2^{cd'}$ from the assumption about $\alpha_{4}$, and proceed exactly as above.

After we have made the above substitutions, the integral with respect to $\alpha_{4}$ (or $\alpha_{8}$) will be \emph{exactly}\begin{equation}\int_{\mathbb{R}}\big(\chi e^{\mathrm{i}(\Delta_{n_{1}}+\Delta_{n_{2}}-\Delta_{n_{0}})}\big)^{\wedge}(\alpha_{4})\big(\chi e^{\mathrm{i}(\Delta_{n_{0}}-\Delta_{n_{1}}-\Delta_{n_{2}})}\big)^{\wedge}(\alpha_{9}-\alpha_{4})\,\mathrm{d}\alpha_{4}=\widehat{\chi^{2}}(\alpha_{9}).\nonumber\end{equation} Then we will get rid of this integration, then take absolute values, fix $(m_{i},\beta_{i})$ (again we ignore the restriction that the $m_{i}$ must add to zero) to obtain an expression
\begin{eqnarray}
\mathcal{S}_{sub}&\lesssim&\sum_{n_{0},n_{1}}\int_{(T)}2^{2h}\big|f_{n_{0},\alpha_{0}}\big|\cdot\prod_{l\in\{1,5,6\}}\big|(y^{\omega_{l}})_{n_{l},\alpha_{l}}\big|\times\nonumber\\
&\times&\prod_{l\in\{3,7\}}\min\bigg\{T,\frac{1}{\langle\alpha_{l}\rangle}\bigg\}\cdot 2^{-d-h}\langle \alpha_{9}\rangle^{-10}\langle\gamma_{2}\rangle^{-10}\nonumber\\
&\lesssim& 2^{-|d-h|}T^{0+}\sum_{n_{0},n_{1}}\big\|\widehat{f_{n_{0}}}\big\|_{L^{q}}\prod_{l\in\{1,5,6\}}\big\|\widehat{(y^{\omega_{l}})_{n_{l}}}\big\|_{L^{1}}.\nonumber
\end{eqnarray} where $c_{j}$ are constants, $n_{5}=n_{0}$, $n_{6}=-n_{1}$, the summation is restricted to the set
\begin{equation}\big\{(n_{0},n_{1}):\max\{\langle n_{0}\rangle,\langle n_{1}\rangle\sim 2^{d},\,\,\min\{\langle n_{0}\rangle,\langle n_{1}\rangle,\langle n_{0}-n_{1}\rangle\}\sim 2^{h}\big\},\nonumber\end{equation} and the integration $(T)$ is taken over the set
\begin{equation}\big\{(\alpha_{0},\alpha_{1},\alpha_{3},\alpha_{5},\alpha_{6},\alpha_{7},\alpha_{9},\gamma_{2}):\alpha_{0}=\alpha_{1}+\alpha_{3}+\alpha_{57}+\alpha_{9}+\gamma_{2}+c_{2}\big\};\nonumber\end{equation}note that the restriction we make here is enough to guarantee that \begin{equation}|\eta|\lesssim 2^{-(d+h)}\langle\gamma_{2}\rangle^{-10}.\nonumber\end{equation}

Now, if we restrict to $n_{0}\sim 2^{d''}$ and $n_{1}\sim 2^{d'''}$, then up to an additive constant they are between $h$ and $d$, and the restricted $\mathcal{S}_{sub}$ is bounded by $2^{-|d-h|}T_{d''}$ due to (\ref{festt}). We may sum over $d$ and $h$ for fixed $d''$ and $d'''$ to obtain a bound $2^{-|d''-d'''|}T_{d''}$, then sum over $d''$ and $d'''$ to conclude.

\section{The \emph{a priori} estimate III: A special term}\label{mid2}
In this section we prove the following proposition, with which we will be able to close the proof of Proposition \ref{finalreduct}.
\begin{proposition}\label{easy} We have
\begin{equation}\label{n3.5est}\sum_{j\in\{1,2,5,7\}}\|\mathcal{M}^{3.5}\|_{X_{j}}\lesssim T^{0+}.\end{equation}
\end{proposition}
\begin{proof}Define the functions $g$, $f$ and $f'$, and fix the scale $2^{d}$ as usual. Note in particular that $\|g\|_{X_{6}'}\leq 1$, so that \begin{equation}\label{weakest}\big\|\langle n_{0}\rangle^{-s}\langle\alpha_{0}\rangle^{\frac{1}{2}-O(s^{2})}f'\big\|_{l^{2}L^{2}}\lesssim 1.\end{equation}Now, according to a computation similar to those did before (for example, in proof of Propositions \ref{easiest} and \ref{easier}), we can write the expression $\mathcal{S}$ we need to bound in two ways\footnote[1]{Note that from the arguments made before, we can switch between these two expressions even if we fix the $n_{l}$, $m_{i}$, $\beta_{i}$ or $\alpha_{4}$ variables. Also we may insert $\chi$ since $f$ has compact time support.}:
\begin{eqnarray}\label{3.5e1}\mathcal{S}&=&\sum_{n_{0}=n_{1}+n_{2}+n_{3}+m_{1}+\cdots+m_{\mu}}\int_{(T)}\Phi^{3.5}\cdot\overline{(f')_{n_{0},\alpha_{0}}}\times\\
&\times&((w')^{\omega_{1}})_{n_{1},\alpha_{1}}\prod_{l=2}^{3}(z^{l})_{n_{l},\alpha_{l}}\cdot\phi_{\alpha_{4}}\prod_{i=1}^{\mu}\frac{(u''')_{m_{i},\beta_{i}}}{m_{i}};\nonumber
\end{eqnarray}
\begin{eqnarray}\label{3.5e2}\mathcal{S}&=&\sum_{n_{0}=n_{1}+n_{2}+n_{3}+m_{1}+\cdots+m_{\mu}}\int_{(T)}\Phi^{3.5}\cdot\overline{f_{n_{0},\alpha_{0}}}((w'')^{\omega_{1}})_{n_{1},\alpha_{1}}\times\\
&\times&\prod_{l=2}^{3}(y^{l})_{n_{l},\alpha_{l}}\cdot\phi_{\alpha_{4}}\big(\chi e^{\mathrm{i}(\Delta_{n_{1}}+\Delta_{n_{2}}+\Delta_{n_{3}}-\Delta_{n_{0}})}\big)^{\wedge}(\alpha_{5})\prod_{i=1}^{\mu}\frac{(u''')_{m_{i},\beta_{i}}}{m_{i}}.\nonumber
\end{eqnarray} Here the integration $(T)$ in (\ref{3.5e1}) is the integration over the set
\begin{equation}\label{hyperplanee}\big\{(\alpha_{0},\cdots,\alpha_{4},\beta_{1},\cdots,\beta_{\mu}):\alpha_{0}=\alpha_{14}+\beta_{1\mu}+\Xi\big\},\end{equation} while the integration $(T)$ in (\ref{3.5e2}) is over the set
\begin{equation}\label{hyperplanee2}\big\{(\alpha_{0},\cdots,\alpha_{5},\beta_{1},\cdots,\beta_{\mu}):\alpha_{0}=\alpha_{15}+\beta_{1\mu}+\Xi\big\},\end{equation} where the NR factor
\begin{equation}\label{nrfac}\Xi=|n_{0}|n_{0}-\sum_{l=1}^{3}|n_{l}|n_{l}-\sum_{i=1}^{\mu}|m_{i}|m_{i}.\end{equation} Also each $z^{l}$ or $\overline{z^{l}}$ equals $u'$, $v'$ or $w'$, and $y^{l}$ or $\overline{y^{l}}$ equals $u''$, $v''$ or $w''$.

First we treat the case when \begin{equation}\label{notdeviate}\min_{0\leq l\leq 3}\langle n_{l}\rangle\gtrsim 2^{\frac{2d}{3}}.\end{equation}In this situation, $n_{2}$ and $n_{3}$ must have opposite sign (note that here we are again assuming $2^{d}$ is large enough). By symmetry, we may assume $n_{2}>0$ and $n_{3}<0$; also note that $n_{0}>0$. 

Next, we may assume that $\langle m_{i}\rangle+\langle\beta_{i}\rangle+\langle\alpha_{4}\rangle\ll 2^{\frac{d}{90}}$ for all $i$, since otherwise we will be able to gain a power $2^{cd}$ from the corresponding factor alone, and estimate the expression (\ref{3.5e1}) by controlling $\mathfrak{N}f'$ in $L^{2+}L^{2+}$, $\mathfrak{N}w'$ in $L^{6+}L^{6+}$, $\mathfrak{N}z^{l}$ in $L^{6}L^{6}$ and $\phi$ in $l^{1+}L^{1+}$ with a loss of at most $2^{(s+O(\epsilon))d}$. Notice that the loss from the $\langle n_{0}\rangle^{-s}$ factor in (\ref{weakest}) is at most $2^{sd}$, while the loss from other places is at most $2^{O(\epsilon)d}$. In the same way, we will also be done if $\langle m_{i}\rangle\gtrsim 2^{1.2sd}$ for some $i$, or when $|\Xi|\gtrsim 2^{(1+1.01s)d}$. In fact, in the former case we invoke the $X_{3}$ norm for $\langle\partial_{x}\rangle^{-s^{3}}u'''$ to gain a power of $2^{(1+c)sd}$ to cancel the $2^{sd}$ loss, then fix $m_{j}$ and $\beta_{j}$ for $j\neq i$ to produce $\mathcal{S}_{sub}$, which is estimated by controlling $\mathfrak{N}f'$ in $L^{4}L^{4}$, $\mathfrak{N}w'$, $\mathfrak{N}z^{l}$ and $\mathfrak{N}u'''$ in $L^{6}L^{6}$, $\phi$ in some $l^{1+}L^{1+}$ with a loss of at most $2^{O(\epsilon)d}$. In the latter case at least one of $\alpha_{l}$ must be $\gtrsim 2^{(1+1.01s)d}$. If $l\in\{0,1\}$, we could gain $2^{cd}$ from the corresponding factor and proceed as above (since $2^{cd}$ gain will overwhelm any loss). If $l\in\{2,3\}$ (say $l=2$), we invoke the $X_{4}$ norm for $z^{2}$; notice that $1-\kappa=s^{\frac{5}{4}}$, we will gain at least $2^{1.001sd}$ from $z^{2}$ and estimate the reduced function in $l^{2}L^{2}$. This will cancel the $2^{sd}$ loss from $f'$ and we can fix all $m_{i}$ and $\beta_{i}$, then bound $\mathcal{S}_{sub}$by controlling $\mathfrak{N}f'$ in $L^{6-}L^{6-}$, $\mathfrak{N}w'$ in $L^{6+}L^{6+}$ (where $6-$ and $6+$ differ from $6$ by $cs^{2}$ with appropriately chosen $c$), $\mathfrak{N}z^{2}$ in $L^{2}L^{2}$, $\mathfrak{N}z^{3}$ in $L^{6}L^{6}$, $\phi$ in $l^{1+}L^{1+}$ with a loss of at most $2^{O(\epsilon)d}$.

Now, we have $\langle m_{i}\rangle\ll 2^{1.2sd}$ and $|\Xi|\ll 2^{(1+1.01s)d}$. Since $n_{0},n_{2}>0>n_{3}$ and $\langle n_{1}\rangle\gtrsim 2^{\frac{2d}{3}}$, we can easily see that $n_{1}>0$, which implies
\begin{equation}\label{resss}\big|n_{0}^{2}-n_{1}^{2}-n_{2}^{2}+n_{3}^{2}\big|\ll 2^{(1+1.01s)d}.\end{equation} Note that $|n_{2}+n_{3}|\gg 2^{\frac{d}{2}}$ (otherwise we gain $2^{cd}$ from the weight and everything will again be easy; also this will imply $n_{0}\neq n_{1}$), we write $n_{2}+n_{3}=k$ and $n_{0}-n_{1}=l$ so that $l-k=O(2^{1.2sd})$. We deduce from (\ref{resss}) and elementary algebra that
\begin{equation}\label{resss2}|k|\cdot|n_{0}+n_{1}-n_{2}+n_{3}|\lesssim 2^{(1+1.2s)d},\end{equation} which implies that $\max\{\langle n_{2}\rangle,\langle n_{3}\rangle\}\sim 2^{d}$. Since we will be done we we gain $2^{(1+c)sd}$ from the weight, we may then assume $|n_{2}+n_{3}|\gtrsim 2^{(1-1.01s)d}$.

Next, we claim that we may assume $|n_{0}-n_{2}|+|n_{1}+n_{3}|\ll 2^{1.9sd}$. In fact, the difference between $n_{0}-n_{2}$ and $n_{1}+n_{3}$ is already $O(2^{1.2sd})$, so if one of them is $\gtrsim 2^{1.9sd}$, the factor $n_{0}+n_{1}-n_{2}+n_{3}$ in (\ref{resss2}) will be at least $2^{1.9sd}$ also. This would force $k$ to be $\ll 2^{(1-0.7s)d}$. Note that $\max\{\langle n_{2}\rangle,\langle n_{3}\rangle\}\sim 2^{d}$, we will gain $2^{0.7sd}$ from the weight $\Phi^{3.5}$. Therefore, we will still be able to close the estimate if we can gain more than $2^{0.3sd}$ elsewhere, for example, when $|\Xi|\gtrsim 2^{(1+0.31s)d}$ or when $\langle m_{i}\rangle\gtrsim 2^{0.4sd}$ for some $i$. If we assume further that $|\Xi|\ll 2^{(1+0.31s)d}$ and $\langle m_{i}\rangle\ll 2^{0.4sd}$, then (\ref{resss2}) will hold with the right hand side replaced by $2^{(1+0.4s)d}$. This would then force $|k|\lesssim 2^{(1-1.5s)d}$, which is impossible since we have already had $|k|\gtrsim 2^{(1-1.01s)d}$.

Note that all the restrictions made above concerns only the $n_{l}$, $m_{i}$, $\beta_{i}$ and $\alpha_{4}$ variables, so to this point we still have the freedom of choosing (\ref{3.5e1}) or (\ref{3.5e2}). After making these restrictions, we will now choose (\ref{3.5e2}) and analyze the exponential factor first. Note that
\begin{equation}\|(\delta_{n_{1}}+\delta_{n_{2}}+\delta_{n_{3}}-\delta_{n_{0}})^{\wedge}\|_{L^{1}}\lesssim 2^{-\frac{d}{4}}\end{equation} by Proposition \ref{factt}, we deduce from Lemma \ref{general} that
\begin{equation}\big\|\langle \alpha_{5}\rangle J_{(n)}(\alpha_{5})\big\|_{L^{\mu}}\lesssim 2^{-\frac{d}{8}}\end{equation} for all $1\leq\mu\leq\infty$ with \begin{equation}\label{J(n)}J_{(n)}(\alpha_{5})=\big(\chi(t)\cdot(e^{\mathrm{i}(\Delta_{n_{1}}+\Delta_{n_{2}}+\Delta_{n_{3}}-\Delta_{n_{0}})}-1\big))^{\wedge}(\alpha_{5}).\end{equation} This implies, by a similar argument as in the proof of Proposition \ref{weaken}, we deduce that (where, of course, the supreme is taken over $(n)$ such that $\sum_{l}\langle n_{l}\rangle\sim 2^{d}$)
\begin{equation}\int_{\mathbb{R}}\sup_{n_{0},\cdots,n_{3}}|J_{(n)}(\alpha_{5})|\,\mathrm{d}\alpha_{5}\lesssim 2^{-\frac{d}{9}}.\end{equation}Therefore, if we replace in (\ref{3.5e2}) the exponential factor by $J_{(n)}$, we will be able to first fix $\alpha_{5}$ and then integrate over it, and gain a power $2^{cd}$ from this process. Once $\alpha_{5}$ is fixed and the $J_{(n)}$ factor is removed with a $2^{cd}$ gain, we will be in the same situation as considered before. We can then fix $m_{i}$ and $\beta_{i}$ to produce $\mathcal{S}_{sub}$, and estimate it by controlling $\mathfrak{N}f$ in $L^{2+}L^{2+}$, $\mathfrak{N}w''$ and $\mathfrak{N}y^{l}$ in $L^{6}L^{6}$, $\phi$ in $l^{1+}L^{1+}$ with a loss $2^{O(s)d}$.

Now we may replace the exponential factor in (\ref{3.5e2}) by $\widehat{\chi}(\alpha_{5})$. We can actually get rid of this factor since $f$ and $f'$ is supposed to have compact $t$ support. Therefore, we are reduced to estimate \begin{eqnarray}\label{3.5e3}\mathcal{S}&=&\sum_{n_{0}=n_{1}+n_{2}+n_{3}+m_{1}+\cdots+m_{\mu}}\int_{(T)}|\Phi^{3.5}|\cdot\big|\overline{f_{n_{0},\alpha_{0}}}\big|\times\\&\times&\big|((w'')^{\omega_{1}})_{n_{1},\alpha_{1}}\big|
\prod_{l=2}^{3}\big|(y^{l})_{n_{l},\alpha_{l}}\big|\cdot|\phi_{\alpha_{4}}|\prod_{i=1}^{\mu}\bigg|\frac{(u''')_{m_{i},\beta_{i}}}{m_{i}}\bigg|,\nonumber
\end{eqnarray} where the integral $(T)$ is taken over the set (\ref{hyperplanee}). Starting from this point we will no longer use the equivalence of (\ref{3.5e1}) and (\ref{3.5e2}), so we will assume here that each $\langle \alpha_{l}\rangle\ll 2^{(1+1.01s)d}$, since otherwise we may proceed as above (note that the bounds for $f$, $w''$ and $y^{l}$ are better than those for $f'$, $w'$ and $z^{l}$). For the same reason, we may assume $\langle \alpha_{0}\rangle+\langle \alpha_{1}\rangle\ll 2^{\frac{d}{900}}$ (otherwise we may gain $2^{cd}$, then control $\mathfrak{N}f'$ in $L^{2+}L^{2+}$, $\mathfrak{N}w''$ in $L^{6-}L^{6-}$, $\mathfrak{N}y^{l}$ in $L^{6}L^{6}$ and $\phi$ in $l^{1+}L^{1+}$ with $2+$ and $6-$ being $2+c$ and $6-c$ respectively).

To estimate (\ref{3.5e3}), we recall the bound (\ref{festt}) in the proof of Proposition \ref{easier}. Suppose that $n_{0}\sim n_{2}\sim 2^{d'}$ and $n_{1}\sim n_{3}\sim 2^{d''}$ (note that $|n_{0}-n_{2}|$ ans $|n_{1}+n_{3}|$ are small), then we have $\max\{d',d''\}=d+O(1)$, as well as \begin{equation}\label{boundphi}|\Phi^{3.5}|\lesssim 2^{-|d'-d''|}.\end{equation} Here, instead of fixing $d$, we will fix all of $d,d',d''$, then sum over $d'$ and $d''$. By (\ref{boundphi}), we may assume \begin{equation}\min\{d',d''\}\geq (1-1.01s)d,\nonumber\end{equation}so in particular $d'\sim d''\sim d$. Once we fix $\langle n_{3}\rangle\sim 2^{d''}$, we can invoke the $X_{8}$ norm of $y^{3}$ to write $y^{3}$ (now restricted to frequency $\sim 2^{d''}$) as a sum \begin{equation}\label{another22}y^{3}=\sum_{j}\gamma_{j}\pi_{k_{j}}y^{(j)},\,\,\,\,\,\,\,\,\sum_{j}\langle k_{j}\rangle^{s^{\frac{1}{2}}}|\gamma_{j}|\lesssim 1,\end{equation} such that $\|y^{(j)}\|_{L^{q}l^{2}}\lesssim 1$ for each $j$. See Section \ref{another}. We only need to consider a single $j$; namely we need to bound $\mathcal{S}$ provided $y^{3}=\langle k\rangle^{-s^{1/2}}\pi_{k}y''$, where $y''$ is some function verifying $\|y''\|_{L^{q}l^{2}}\lesssim 1$. Next, if $\langle k\rangle\gtrsim 2^{\frac{d}{90}}$, we will gain a power $2^{cs^{1/2}d}$ from the coefficient in $y^{3}$; we then fix $m_{i}$ and $\beta_{i}$. To estimate the resulting $\mathcal{S}_{sub}$, we can control $\mathfrak{N}f$ in $L^{6+}L^{6+}$, $\mathfrak{N}w''$ and $\mathfrak{N}y^{2}$ in $L^{6}L^{6}$ and $\phi$ in $l^{1+}L^{1+}$ with a loss of at most $2^{O(s)d}$ (where the $6+$ is $6+cs$ and $1+$ defined accordingly; also note that we have assumed $\langle\alpha_{0}\rangle\lesssim 2^{2d}$, as well as the bound for $f$ deduced from the $X_{6}'$ bound for $g$). We can then close the estimate if we can control $y''$ (and hence $\pi_{k}y''$) in $l^{2}L^{2}$. This can be achieved by inserting a $\chi(t)$ factor to every term in (\ref{another22}), which, while doing nothing to the equality and the $L^{q}l^{2}$ norms, allows us to control the $L^{2}l^{2}$ norm by the $L^{q}l^{2}$ norm. Thus here we also get the desired estimate.

We now assume $\langle k\rangle\ll 2^{\frac{d}{90}}$. We will fix $k$ and each $(m_{i},\beta_{i})$ to obtain some constants $K_{1}\ll 2^{1.2sd}$ and $K_{2}\ll 2^{\frac{d}{90}}$, and produce
\begin{eqnarray}2^{-|d'-d''|}\mathcal{S}_{sub} &=&2^{-|d'-d''|}\sum_{n_{0}=n_{1}+n_{2}+n_{3}+K_{1}}\int_{(T)}\big|\overline{f_{n_{0},\alpha_{0}}}\big|\times\nonumber\\
&\times&\big|((w'')^{\omega_{1}})_{n_{1},\alpha_{1}}\big|\cdot\big|(y^{2})_{n_{2},\alpha_{2}}\big|\cdot\big|(\pi_{k}y'')_{n_{3},\alpha_{3}}\big|\cdot|\phi_{\alpha_{4}}|,\nonumber
\end{eqnarray} where the integral $(T)$ is taken over the set
\begin{equation}\bigg\{(\alpha_{0},\cdots,\alpha_{4}):\alpha_{0}=\alpha_{14}+|n_{0}|n_{0}-\sum_{l=1}^{3}|n_{l}|n_{l}+K_{2}\bigg\},\end{equation} with all the restrictions made above taking effect. Now if $n_{0}-n_{2}\in\{0,K_{1}-k\}$, we can bound $\mathcal{S}_{sub}$ by
\begin{eqnarray}\mathcal{S}_{sub}&\lesssim &T^{0+}\sum_{n_{0}\sim 2^{d'}}\sum_{n_{1}\sim 2^{d''}}\big\|\widehat{f_{n_{0}}}\big\|_{L^{q}}\big\|(((w'')^{\omega_{1}})_{n_{1}})^{\wedge}\big\|_{L^{1}}\times\nonumber\\
&\times&\big\|((y^{2})_{n_{0}+c_{0}})^{\wedge}\big\|_{L^{1}}\big\|((y'')_{n_{1}+c_{1}})^{\wedge}\big\|_{L^{q}}\nonumber\\
&\lesssim &T^{0+}T_{d'}\big\|\langle n_{2}\rangle^{r}y^{2}\big\|_{l_{n_{2}\sim 2^{d'}}^{p}L^{1}}\cdot\big\|(w'')^{\omega_{1}}\big\|_{l_{n_{1}\sim 2^{d''}}^{2}L^{1}}\big\|y''\big\|_{l_{n_{1}\sim 2^{d''}}^{2}L^{q}}\nonumber\\
&\lesssim& T^{0+}T_{d'},\nonumber
\end{eqnarray}
using the bound (\ref{festt}) for $f$, the $X_{2}$ bound for $w''$ and $y^{2}$, and the $L^{q}l^{2}$ bound for $y''$, where $c_{j}$ are constants, small compared to $2^{d'}$ and $2^{d''}$, such that $n_{j}\sim n_{0}+c_{j}$ for $j\in\{0,1\}$. If we then sum and integrate over $m_{i}$ and $\beta_{i}$, then multiply by $2^{-|d'-d''|}$ and sum over $d'$ and $d''$, we will get a quantity bounded by $T^{0+}$.

Assume $n_{0}-n_{2}\not\in\{0,K_{1}-k\}$. Let $\lambda=n_{0}-n_{2}$, we can rewrite the expression for $\mathcal{S}_{sub}$ as
\begin{eqnarray}\mathcal{S}_{sub}&\lesssim&T^{0+}2^{-0.999sd+O(s)|d'-d''|}\sum_{n_{0},n_{1},\lambda}\int_{\mathbb{R}^{4}}A_{n_{0},\alpha_{0}}B_{n_{1},\alpha_{1}}C_{n_{0}-\lambda,\alpha_{2}}\times\nonumber\\
&\times &D_{\lambda-n_{1}+c_{1},\alpha_{3}'}\big\langle\alpha_{0}-\alpha_{1}-\alpha_{2}-\alpha_{3}'-\Xi'+c_{2}\big\rangle^{-1-s^{2}}\prod_{l=0}^{2}\,\mathrm{d}\alpha_{l}\cdot\mathrm{d}\alpha_{3}',\nonumber
\end{eqnarray} where $c_{j}\ll 2^{\frac{d}{90}}$ are constants, and the summation-integration is restricted to the subset where all the restrictions made above are satisfied by $(n_{0},\cdots,n_{3},\alpha_{0},\cdots,\alpha_{4})$ which is defined in terms of our new variables (as well as the intermediate variable $n_{3}'$) by
\begin{equation}n_{2}=n_{0}-\lambda,\,\,\,n_{3}=n_{3}'-k,\,\,\,\alpha_{3}=\alpha_{3}'-|n_{3}'|n_{3}'+|n_{3}|n_{3};\nonumber\end{equation}
\begin{equation}n_{3}'=\lambda-n_{1}-K_{1}+k,\,\,\,\alpha_{4}=\alpha_{0}-\alpha_{13}-|n_{0}|n_{0}+\sum_{l=1}^{3}|n_{l}|n_{l}-K_{2}.\nonumber\end{equation} We can check from the assumptions made above that no two of $(-n_{0},n_{1}.n_{2},n_{3}')$ add to zero. Moreover, the $\Xi'$ is defined by
\begin{equation}\Xi'=\Xi'(n_{0},n_{1},\lambda)=|n_{0}|n_{0}-|n_{1}|n_{1}-|n_{2}|n_{2}-|n_{3}'|n_{3}',\nonumber\end{equation} and the relevant functions are defined by
\begin{equation}A_{n_{0},\alpha_{0}}=\langle n_{0}\rangle^{-r}\big|f_{n_{0},\alpha_{0}}\big|,\,\,\,B_{n_{1},\alpha_{1}}=\langle n_{1}\rangle^{r}\big|((w'')^{\omega_{1}})_{n_{1},\alpha_{1}}\big|;\nonumber
\end{equation}
\begin{equation}
C_{n_{2},\alpha_{2}}=\langle n_{2}\rangle^{r}\big|(y^{2})_{n_{2},\alpha_{2}}\big|,\,\,\,\,D_{n_{3}',\alpha_{3}'}=\big|(y'')_{n_{3}',\alpha_{3}'}\big|
.\nonumber\end{equation} When restricted to appropriate subsets (for example, we must have $n_{0}\sim n_{2}\sim 2^{d'}$ and $n_{1}\sim n_{3}\sim 2^{d''}$), these functions will verify
\begin{equation}\label{ultimatebound}
\|A\|_{L^{1}l^{p'}}+\|B\|_{L^{1}l^{p}}+\|C\|_{l^{p}L^{1}}+\|D\|_{L^{1}l^{2}}\lesssim 2^{0.002sd}.\end{equation} In fact, due to the restrictions we made, we can bound all the variables by $2^{O(1)d}$; so when we replace $L^{q}$ norm by $L^{1}$ norm we lose (by H\"{o}lder) at most $2^{O(q-1)d}$. Thus the bound for $A$ follows from (\ref{festt}), and the bound for $D$ follows from our assumption about $y''$. The bound for $C$ follows from the $X_{2}$ bound for $y^{2}$, while for $B$ we simply estimate (note that $\langle \alpha_{1}\rangle\ll 2^{\frac{d}{900}}$)
\begin{eqnarray}
\|B\|_{L^{1}l^{p}}&\lesssim & 2^{O(1-q)d}\big\|\langle n_{1}\rangle^{r}\langle \alpha_{1}\rangle^{b}(w'')^{\omega_{1}}\big\|_{L^{2}l^{p}}\nonumber\\
&\lesssim &2^{O(1-q)d}2^{\frac{sd}{800}}\|\langle n_{1}\rangle^{s}\langle \alpha_{1}\rangle^{b}(w'')^{\omega_{1}}\big\|_{L^{p}l^{p}}\nonumber\\
&\lesssim &2^{0.002sd},\nonumber
\end{eqnarray} using the $X_{1}$ bound for $w''$. Note that by inserting $\chi(t)$ to $w''$, we may control the $l^{p}L^{p}$ norm by the $l^{p}L^{2}$ norm.

Now we need to estimate $\mathcal{S}_{sub}$ under the assumption of (\ref{ultimatebound}). First replace the bounds in (\ref{ultimatebound}) by $1$, so that we only need to bound the summation-integration part of $\mathcal{S}_{sub}$ by $2^{0.998sd}$. Fix $\alpha_{0}$ and $\alpha_{1}$ which are $\ll 2^{0.02d}$ (then integrate over them), we may assume $A$ and $B$ are functions of $n_{0}$ and $n_{1}$ only, and are bounded in $l^{p'}$ and $l^{p}$ respectively. We then bound (with $c_{j}\ll 2^{\frac{d}{90}}$ being constants)
\begin{eqnarray}\mathcal{S}_{sub}'&=&\int_{\mathbb{R}^{2}}\mathrm{d}\alpha_{2}\mathrm{d}\alpha_{3}'\cdot\sum_{n_{0},n_{1},\lambda}A_{n_{0}}B_{n_{1}}C_{n_{0}-\lambda,\alpha_{2}}\times\nonumber\\
&\times&D_{\lambda-n_{1}+c_{1},\alpha_{3}'}\big\langle\alpha_{2}+\alpha_{3}'+\Xi'+c_{3}\big\rangle^{-1-s^{2}}\nonumber\\
&\lesssim&\sum_{\rho\in\mathbb{Z}}\langle\rho\rangle^{-1-s^{2}}\int_{\mathbb{R}^{2}}\mathrm{d}\alpha_{2}\mathrm{d}\alpha_{3}'\sum_{(n_{0},n_{1},\lambda):\lfloor \Xi''\rfloor=\rho}A_{n_{0}}B_{n_{1}}C_{n_{0}-\lambda,\alpha_{2}}D_{\lambda-n_{1}+c_{1},\alpha_{3}'}\nonumber\\
\label{trilin}&\lesssim&\sup_{\rho}\int_{\mathbb{R}^{2}}\mathrm{d}\alpha_{2}\mathrm{d}\alpha_{3}'\cdot\bigg(\sum_{(n_{0},n_{1},\lambda):\lfloor \Xi''\rfloor=\rho}A_{n_{0}}^{2}C_{n_{0}-\lambda,\alpha_{2}}^{2}\bigg)^{\frac{1}{2}}\times\\
&\times&\bigg(\sum_{(n_{0},n_{1},\lambda):\lfloor \Xi''\rfloor=\rho}B_{n_{1}}^{4}\bigg)^{\frac{1}{4}}\bigg(\sum_{(n_{0},n_{1},\lambda):\lfloor \Xi''\rfloor=\rho}D_{\lambda-n_{1}+c_{1},\alpha_{3}'}^{4}\bigg)^{\frac{1}{4}},\nonumber
\end{eqnarray} where we write $\alpha_{2}+\alpha_{3}'+\Xi'+c_{3}=\Xi''$ for simplicity. Now for any positive function $E_{n_{1}}$ of $n_{1}$, when $\rho$ and $\alpha_{2},\alpha_{3}'$ are fixed, we may bound
\begin{equation}\sum_{(n_{0},n_{1},\lambda):\lfloor \Xi''\rfloor=\rho}E_{n_{1}}\lesssim\sum_{n_{1}}E_{n_{1}}\sum_{(n_{0},\lambda):\Xi'=c''}1
\lesssim2^{O(s^{4})d}\|E\|_{l^{1}}\nonumber\end{equation} (where $c'$ and $c''$ are constants depending on $\alpha_{2},\alpha_{3}'$ and $\rho$), thanks to part (iii) of Lemma \ref{divisor} (or actually, an argument similar to the proof of that part). The same inequality holds if we replace $n_{1}$ by $\lambda-n_{1}+c_{1}$ (which equals $n_{3}'$ plus a constant). Therefore we can bound the second factor in (\ref{trilin}) by $2^{O(s^{4})d}\|B\|_{l^{4}}\lesssim 2^{O(s^{4})d}$, and the third factor by $2^{O(s^{4})d}\|D_{\cdot,\alpha_{3}'}\|_{l^{4}}$. Ignoring the $2^{O(s^{4})d}$ factors, we thus bound (\ref{trilin}) by
\begin{eqnarray}
\mathcal{S}_{sub}'&\lesssim&\sup_{\rho}\int_{\mathbb{R}^{2}}\mathrm{d}\alpha_{2}\mathrm{d}\alpha_{3}'\cdot\big\|D_{\cdot,\alpha_{3}'}\big\|_{l^{4}}\cdot\bigg(\sum_{(n_{0},n_{1},\lambda):\lfloor \Xi''\rfloor=\rho}A_{n_{0}}^{2}C_{n_{0}-\lambda,\alpha_{2}}^{2}\bigg)^{\frac{1}{2}}\nonumber\\
&\lesssim &\sup_{\rho,\alpha_{3}'}\int_{\mathbb{R}}\bigg(\sum_{(n_{0},n_{1},\lambda):\lfloor \Xi''\rfloor=\rho}A_{n_{0}}^{2}C_{n_{0}-\lambda,\alpha_{2}}^{2}\bigg)^{\frac{1}{2}}\,\mathrm{d}\alpha_{2}\nonumber\\
&\lesssim &\sup_{\rho,\alpha_{3}'}\int_{\mathbb{R}}\,\mathrm{d}\alpha_{2}\cdot\sum_{(n_{0},n_{1},\lambda):\lfloor \Xi''\rfloor=\rho}A_{n_{0}}C_{n_{0}-\lambda,\alpha_{2}}.\nonumber
\end{eqnarray} Now we fix $\rho$ and $\alpha_{3}'$. Notice $n_{0}-\lambda=n_{2}$, and that
\begin{equation}A_{n_{0}}\leq F_{n_{2}}:=\bigg(\sum_{|m-n_{2}|\lesssim 2^{1.9sd}}A_{m}^{p'}\bigg)^{\frac{1}{p'}}\end{equation} which is because $|n_{0}-n_{2}|\lesssim 2^{1.9sd}$, we proceed to estimate
\begin{eqnarray}
\mathcal{S}_{sub}''&\lesssim &\int_{\mathbb{R}}\,\mathrm{d}\alpha_{2}\cdot\sum_{(n_{1},n_{2},\lambda):\lfloor \Xi''\rfloor=\rho}F_{n_{2}}C_{n_{2},\alpha_{2}}\nonumber\\
&\lesssim &2^{O(s^{4})d}\sum_{n_{2}}F_{n_{2}}\int_{\mathbb{R}}C_{n_{2},\alpha_{2}}\,\mathrm{d}\alpha_{2}\nonumber\\
&\lesssim & 2^{O(s^{4})d}\|F\|_{l^{p'}}\|C\|_{l^{p}L^{1}}.\nonumber
\end{eqnarray} Here we have again used the divisor estimate as above. Finally, notice that
\begin{equation}\|F\|_{l^{p'}}^{p'}=\sum_{n_{2}}\sum_{|m-n_{2}|\lesssim 2^{1.9sd}}A_{m}^{p'}\lesssim 2^{1.9sd}\|A\|_{l^{p'}}^{p'},\nonumber\end{equation} so we deduce that $\mathcal{S}_{sub}''\lesssim 2^{\frac{1.91sd}{p'}}\lesssim 2^{0.998sd}$, as desired.

It remains to consider the case where $\langle n_{l}\rangle\ll 2^{\frac{2d}{3}}$ for some $l$. Note that if $\langle m_{i}\rangle+\langle\beta_{i}\rangle+\langle\alpha_{4}\rangle\gtrsim 2^{\frac{d}{90}}$ for some $i$, the weight $|\Phi^{3.5}|\lesssim 2^{-cd}$ or the NR factor (as defined in (\ref{nrfac})) satisfies $|\Xi|\gtrsim 2^{(1+c)d}$, we will be done using the same arguments as before. This in particular includes the cases when (i) three of the $n_{l}$ are $\gtrsim 2^{\frac{3d}{4}}$ and the remaining one is $\ll 2^{\frac{2d}{3}}$; (ii) at least two of the $n_{l}$ are $\ll 2^{\frac{3d}{4}}$, and $\langle n_{2}\rangle+\langle n_{3}\rangle\gtrsim 2^{\frac{4d}{5}}$; (iii) both $n_{2}$ and $n_{3}$ are $\ll 2^{\frac{4d}{5}}$, and $n_{0}n_{1}<0$.

Now we assume that $n_{0}n_{1}>0$, and $\langle n_{2}\rangle+\langle n_{3}\rangle\ll 2^{\frac{4d}{5}}$. Let $n_{0}-n_{1}=k$ and $n_{2}+n_{3}=l$, so that $|k-l|\ll 2^{\frac{d}{90}}$. If $l\lesssim 2^{\frac{d}{80}}$, we must have $\langle n_{2}\rangle+\langle n_{3}\rangle\ll 2^{\frac{d}{70}}$ (or we gain from the $\Phi$ factor). These two variables being small means that we will be able to repeat the argument made before and gain $2^{cd}$ even if $|\Xi|$ is bounded below by $2^{0.99d}$ instead of $2^{(1+c)d}$. But when $|\Xi|\ll 2^{0.99d}$, it is clear that we must have $k=0$. If $l\gg 2^{\frac{d}{80}}$, we will have $k\sim l$, so that \begin{equation}\big||n_{0}|n_{0}-|n_{1}|n_{1}\big|\gtrsim 2^{d}|k|\gg 2^{\frac{4d}{5}}|l|\gg\big||n_{2}|n_{2}+|n_{3}|n_{3}\big|,\nonumber\end{equation} which implies $|\Xi|\gtrsim 2^{\frac{81d}{80}}$, contradicting our assumptions. Thus in any case we deduce that $n_{0}=n_{1}\sim 2^{d}$. Now we may use the expression (\ref{3.5e1}) for $\mathcal{S}$, but with $f'$ and $w'$ replaced with $f$ and $w''$ respectively (see the proof of Proposition \ref{easier}; note that we have made no restrictions for $\alpha_{0}$ or $\alpha_{1}$).

Next, suppose $\langle n_{2}\rangle+\langle n_{3}\rangle\sim 2^{d'}$, we may assume $d'<\frac{d}{10}$, otherwise we will gain a power $2^{cd}$ from the weight $\Phi^{3.5}$ (note that $n_{2}+n_{3}$ equals a linear combination of the $m$ variables since $n_{0}=n_{1}$). We will fix $d$ and $d'$ (then sum over them). If $\langle m_{i}\rangle\ll 2^{\frac{d'}{2}}$ for all $i$, then we gain a power $2^{cd'}$ from the weight $\Phi^{3.5}$; otherwise we have $\langle m_{i}\rangle\gtrsim 2^{\frac{d'}{2}}$ for some $i$, so we may extract a power $2^{cd'}$ from the $\frac{1}{m_{i}}$ factor (without affecting summability in $m_{i}$). In any case, we will be able to fix $m_{i}$ and $\beta_{i}$ and sum over them later, and the $\mathcal{S}_{sub}$ term can be bounded by
\begin{eqnarray}\mathcal{S}_{sub}&\lesssim &2^{-cd'}\sum_{n_{0},n_{2}}\int_{(T)}\big|f_{n_{0},\alpha_{0}}\big|\cdot\big|((w'')^{\omega_{1}})_{n_{0},\alpha_{1}}\big|\times\nonumber\\
&\times&\big|(z^{2})_{n_{2},\alpha_{2}}\big|\cdot\big|(z^{3})_{c_{1}-n_{2},\alpha_{3}}\big|\cdot\min\bigg\{T,\frac{1}{\langle \alpha_{4}\rangle}\bigg\}\nonumber\\
&\lesssim &2^{-cd'}T^{0+}\sum_{n_{0},n_{2}}\big\|\widehat{f_{n_{0}}}\big\|_{L^{q}}\big\|(((w'')^{\omega_{1}})_{n_{0}})^{\wedge}\big\|_{L^{1}}\prod_{l=2}^{3}\big\|\langle n_{l}\rangle^{-c}\widehat{(z^{l})_{n_{l}}}\big\|_{L^{1}}\nonumber\\
&\lesssim &2^{-cd'}T^{0+}\cdot 2^{rd}T_{d}\cdot 2^{-rd}\nonumber\\
&\lesssim& 2^{-cd'}T^{0+}T_{d}\nonumber.
\end{eqnarray} using the bound (\ref{festt}) for $f$ and the $X_{2}$ bound for $w''$, where $c_{j}$ are constants, $n_{3}=c_{1}-n_{2}$, and the integral $(T)$ is over the set
\begin{equation}\big\{(\alpha_{0},\cdots,\alpha_{4}:\alpha_{0}=\alpha_{14}-|n_{2}|n_{2}-|c_{1}-n_{2}|(c_{1}-n_{2})+c_{2})\big\}.\nonumber\end{equation} Now we can (sum over $m_{i}$ and integrate over $\beta_{i}$ and then) sum over $d$ and $d'$ to conclude that $\mathcal{S}$ is bounded by $T^{0+}$. This proves Proposition \ref{easy}.
\end{proof}
\section{The \emph{a priori} estimate IV: The remaining estimates}\label{mid3} In this section we will construct appropriate extensions of $u^{*}$, $v^{*}$ and $u$ so that the improved versions of (\ref{output3}) and (\ref{output4}) holds. Note that we have already constructed a function, denoted by $w^{(4)}$, that coincides with $w^{*}$ on $[-T,T]$, and verifies $\|w^{(4)}\|_{Y_{1}}\leq C_{0}e^{C_{0}A}$. We will fix this function in later discussions. In particular, we may (starting from this point) redefine the $\delta_{n}$ and $\Delta_{n}$ factors as in (\ref{factor0}) and (\ref{factor1}) by replacing $w^{*}$ with $w^{(4)}$ (instead of $w''$) and $u$ with $u'''$.
\subsection{The extension of $u$}\label{controllu} Fix a scale $K$ so that $K=C_{1.5}e^{C_{1.5}A}$ where $C_{1.5}$ is large enough depending on $C_{1}$, and the $C_{2}$ defined before is large enough depending on $C_{1.5}$. In order to construct a function $u^{(5)}$ that coincides with $u$ on $[-T,T]$ and satisfies\begin{equation}\label{loosebound}\|\langle\partial_{x}\rangle^{-s^{3}}u^{(5)}\|_{X_{2}}+\|\langle\partial_{x}\rangle^{-s^{3}}u^{(5)}\|_{X_{3}}+\|\langle\partial_{x}\rangle^{-s^{3}}u^{(5)}\|_{X_{4}}\leq C_{0}A,\end{equation} we only need to construct $\mathbb{P}_{>K}u^{(5)}$ and $\mathbb{P}_{\leq K}u^{(5)}$ separately.

To construct $\mathbb{P}_{>K}u^{(5)}$, simply note that $u''$ coincides with $u^{*}$ on $[-T,T]$, and we have $\|u''\|_{Y_{2}}\leq C_{1}e^{C_{1}A}$; thus if we define $(u^{(5)})_{n}=e^{\mathrm{i}\Delta_{n}}(u'')_{n}$, where $\Delta_{n}$ is redefined as above, then $\mathbb{P}_{>K}u^{(5)}$ will equal $\mathbb{P}_{>K}u$ on $[-T,T]$, and we have
\begin{equation}\label{bigucontrol}
\|\langle\partial_{x}\rangle^{-s^{4}}u^{(5)}\|_{X_{j}}\lesssim O_{C_{1}}(1)e^{C_{0}C_{1}A},
\end{equation} for $j\in\{2,3,4\}$, thanks to Proposition \ref{weaken}. Here note that the $s^{3}$ exponent in that proposition can actually be replaced by $s^{4}$ (which is clear from the proof), and the current $(\delta_{n},\Delta_{n})$ also verifies Proposition \ref{factt} (in the same way as the $(\delta_{n},\Delta_{n})$ defined in Section \ref{begin} does). Since we are restricting to high frequencies, the inequality (\ref{bigucontrol}) will easily imply
\begin{equation}\|\langle \partial_{x}\rangle^{-s^{3}}\mathbb{P}_{>K}u^{(5)}\|_{X_{j}}\leq A\nonumber\end{equation} for $j\in\{2,3,4\}$, which is what we need for $\mathbb{P}_{>K}u^{(5)}$.

Now let us construct $\mathbb{P}_{\leq K}u^{(5)}$. Recall that the function $u$ verifies the equation (\ref{smoothtrunc}), and the $Y_{2}$ norm of $\chi(t)e^{-tH\partial_{xx}}u(0)$ is clearly bounded by $C_{0}A$, we only need to prove
\begin{equation}\label{lowforu}\bigg\|\int_{0}^{t}e^{-(t-t')H\partial_{xx}}\mathbb{P}_{\neq 0}((S_{N}u(t'))^{2})\,\mathrm{d}t'\bigg\|_{(X^{-\frac{1}{s},\kappa})^{T}}\lesssim T^{0+},\end{equation} with the implicit constants bounded by $O_{C_{1.5}}(1)e^{C_{0}C_{1.5}A}$, where $X^{\sigma,\beta}$ is the standard space normed by $\|\langle n\rangle^{\sigma}\langle\xi\rangle^{\beta}\ast\|_{l^{2}L^{2}}$. Define the function $u^{(7)}$ by equations (\ref{possub}) and (\ref{negsub}), with the $u$ appearing on the right hand side replaced by $u'''$, and $v$ replaced by $\mathbb{P}_{\leq 0}v'''+w'''$ with $v'''$ defined by\footnote[1]{Note that the $\Delta_{n}$ here is different from the $\Delta_{n}$ defined in Section \ref{begin}. Later we will further modify the definition of $\Delta_{n}$, and this will be clearly stated at that time.} $(v''')_{n}=e^{\mathrm{i}\Delta_{n}}(v'')_{n}$ and $w'''$ similarly, so that $u^{(7)}$ coincides with $u$ on $[-T,T]$. We claim that
\begin{equation}\label{lowforu2}\big\|\mathcal{E}\big(\mathbf{1}_{[-T,T]}\mathbb{P}_{\neq 0}((S_{N}u^{(7)})^{2})\big)\big\|_{X^{-\frac{1}{s},\kappa}}\lesssim T^{0+}.\end{equation} This implies (\ref{lowforu}), since the two functions on the left hand side of (\ref{lowforu}) and (\ref{lowforu2}) coincide on $[-T,T]$.

Let $\mathcal{N}=\mathbb{P}_{\neq 0}((S_{N}u^{(7)})^{2})$, we will have
\begin{eqnarray}\label{expansofn}\mathcal{N}_{n_{0}}&=&\sum_{(\omega_{1},\omega_{2})\in\{-1,1\}^{2}}\sum_{\mu_{1},\mu_{2}}\frac{\omega_{1}^{\mu_{1}}\omega_{2}^{\mu_{2}}}{2^{\mu_{12}}\mu_{1}!\mu_{2}!}\times\\
&\times&\sum_{n_{1}+n_{2}+m_{1}+\cdots+m_{\mu_{12}}=n_{0}}\Psi\cdot\prod_{l=1}^{2}(z^{\omega_{l}})_{n_{l}}\prod_{i=1}^{\mu_{12}}\frac{(u''')_{m_{i}}}{m_{i}},\nonumber
\end{eqnarray} where $z=\mathbb{P}_{\leq 0}v'''+w'''$, $\Psi$ is the product of some $\psi$ factors and two characteristic functions $\mathbf{1}_{E_{1}}\mathbf{1}_{E_{2}}$, where\begin{equation}E_{1}=\big\{\omega_{1}(n_{1}+m_{1\mu_{1}})<0\big\},\,\,\,\,\,\,E_{2}=\big\{\omega_{2}(n_{2}+m_{\mu_{1}+1,\mu_{12}})<0\big\}.\nonumber\end{equation} Now, by the same argument as in the proof of Proposition \ref{intereqn} (note that $n_{0}\neq 0$), we can rewrite the right hand side of (\ref{expansofn}) as a sum of the same form, but either with $\Psi$ bounded by $1$ and $n_{1}+n_{2}\neq 0$, or with $\Psi$ bounded by $\frac{\langle m_{i}\rangle+\langle n_{0}\rangle}{\langle n_{1}\rangle+\langle m_{i}\rangle+\langle n_{0}\rangle}$ for some $i$.

To prove (\ref{lowforu2}), we will use the function $g$ and $f$ as in the previous sections, and fix the scale $d$ as before; we are then reduced to estimate (with $\mu=\mu_{12}$)\begin{equation}\mathcal{S}=\sum_{n_{0}=n_{1}+n_{2}+m_{1}+\cdots +m_{\mu}}\int_{(T)}\overline{f_{n_{0},\alpha_{0}}}\phi_{\alpha_{3}}\prod_{l=1}^{2}(z^{\omega_{l}})_{n_{l},\alpha_{l}}\cdot\prod_{i=1}^{\mu}\frac{(u''')_{m_{i},\beta_{i}}}{m_{i}}\nonumber,\end{equation} where $\phi$ is the Fourier transform of $\mathbf{1}_{[-T,T]}$, the integration $(T)$ is taken over the set
\begin{equation}\big\{(\alpha_{0},\cdots,\alpha_{3},\beta_{1},\cdots,\beta_{\mu}):\alpha_{0}=\alpha_{13}+\beta_{1\mu}+\Xi\big\},\nonumber\end{equation} where the NR factor \begin{equation}\Xi=|n_{0}|n_{0}-|n_{1}|n_{1}-|n_{2}|n_{2}-\sum_{i=1}^{\mu}|m_{i}|m_{i}.\nonumber\end{equation}We may assume that $\langle n_{0}\rangle$ and $\langle m_{i}\rangle$ are all $\ll 2^{\frac{d}{90}}$; otherwise, since we can gain some small power of $\langle m_{i}\rangle$ and any large power of $\langle n_{0}\rangle$ (because of the $\frac{-1}{s}$ index), we will be able to gain some power $2^{cd}$. Then we simply fix $(m_{i},\beta_{i})$ to produce $\mathcal{S}_{sub}$, then bound $f$ in $L^{2}L^{2}$, $\mathfrak{N}z$ in $L^{6}L^{6}$ and $\phi_{\alpha_{3}}$ in $l^{1+}L^{1+}$ with $2^{O(s)d}$ loss to conclude. Now, since $n_{0}$ and all $m_{i}$ are small, we have either $n_{1}+n_{2}\neq 0$ (which implies $|\Xi|\gtrsim 2^{d}$) or $|\Psi|\lesssim 2^{-cd}$ (so we can proceed as above). In this case at least one of the $\alpha$ or $\beta$ variables must be $\gtrsim 2^{d}$; since we will also have $\omega_{l}n_{l}<0$ and hence $z=w'''$ which is bounded in $Y_{1}$ by $C_{1}e^{C_{1}A}$, we will always gain a power of at least $2^{c(1-\kappa)d}$ from the corresponding factor, then proceed as before to estimate $\mathcal{S}_{sub}$ and then $\mathcal{S}$, with a loss of at most $2^{O(\epsilon)d}$. Finally, note that we always gain a power $T^{0+}$ which overwhelms any loss $O_{C_{1.5}}(1)e^{O_{C_{1.5}}(1)A}$, we have already proved (\ref{lowforu2}).

Next, note that $(u^{*})_{n}=e^{-\mathrm{i}\Delta_{n}}u_{n}$ on the interval $[-T,T]$, we have
\begin{equation}(\partial_{t}+H\partial_{xx})(u^{*})_{n}=e^{-\mathrm{i}\Delta_{n}}(\partial_{t}+H\partial_{xx})u_{n}-\mathrm{i}e^{-\mathrm{i}\Delta_{n}}(\delta_{n}u_{n}).\nonumber\end{equation} The first term on the right hand side can be bounded in $X^{-\frac{2}{s},\kappa-1}$ using Proposition \ref{weaken} and what we proved above, while the second term is easily bounded in the stronger space $X^{-10,0}$, by $O_{C_{1.5}}(1)e^{O_{C_{1.5}}(1)A}$. Therefore by the same argument, we can construct an extension of $\mathbb{P}_{\leq K}u^{*}$ that verifies (\ref{output3}).
\subsection{The extensions of $u^{*}$ and $v^{*}$}\label{controllv} Now, in order to construct appropriate extensions of $\mathbb{P}_{>K}u^{*}$ and $v^{*}$, we need the following
\begin{proposition}\label{recoverlemma}Let $\delta_{n}$ and $\Delta_{n}$ be redefined using (\ref{factor0}) and (\ref{factor1}), this time with $w^{*}$ replaced by $w^{(4)}$ and $u$ replaced by $u^{(5)}$, then the new factors will verify Proposition \ref{factt} with the constants being $C_{0}e^{C_{0}A}$ instead of $O_{C_{1}}(1)e^{C_{0}C_{1}A}$.

Now suppose $h$, $k$ and $h'$, $k'$ are four functions, supported in $|t|\lesssim 1$, that are related by $(h')_{n}=e^{\mathrm{i}\Delta_{n}}h_{n}$ and $(k')_{n}=e^{\mathrm{i}\Delta_{n}}k_{n}$. Assume that
\begin{equation}\label{fu}(h')_{n_{0}}=\sum_{\mu}C_{\mu}\sum_{n_{0}=n_{1}+m_{1}+\cdots+m_{\mu}}\Psi\cdot(k')_{n_{1}}\prod_{i=1}^{\mu}\frac{(u^{(5)})_{m_{i}}}{m_{i}},\end{equation} with $\Psi$ bounded, we will have $\|h\|_{Y_{2}}\lesssim C_{0}e^{C_{0}A}\|k\|_{Y_{2}}$. Moreover, if $\Psi$ is nonzero only when $\langle m_{i}\rangle\gtrsim K$ for some $i$ (again, the constant here may involve polynomial factors of $\mu$), then we have $\|h\|_{Y_{2}}\lesssim K^{0-}\|k\|_{Y_{2}}$.
\end{proposition}
\begin{proof}The estimates about $\delta_{n}$ and $\Delta_{n}$ are proved in the same way as in Proposition \ref{factt}; notice that all the relevant norms bounded by $O_{C_{1}}(1)e^{C_{0}C_{1}A}$ there are now bounded by $C_{0}e^{C_{0}A}$ in this updated version, thanks to the construction of $w^{(4)}$ in previous sections and the construction of $u^{(5)}$ above.

Now we need to bound $\|h\|_{X_{j}}$ for $j\in\{2,3,4,8\}$. By fixing and then summing over $\mu$, we may assume that
\begin{equation}\big|h_{n_{0},\alpha_{0}}\big|\leq C_{0}\sum_{n_{0}=n_{1}+m_{1}\cdots+m_{\mu}}\int_{(T)}\big|k_{n_{1},\alpha_{1}}\big|\cdot\big(\chi e^{\mathrm{i}(\Delta_{n_{1}}-\Delta_{n_{0}})}\big)^{\wedge}(\alpha_{2})\prod_{i=1}^{\mu}\bigg|\frac{(u^{(5)})_{m_{i},\beta_{i}}}{m_{i}}\bigg|,\nonumber\end{equation} where the integration is taken over the set
\begin{equation}\big\{(\alpha_{1},\alpha_{2},\beta_{1},\cdots,\beta_{\mu}):\alpha_{0}=\alpha_{1}+\alpha_{2}+\beta_{1\mu}+\Xi\big\},\nonumber\end{equation} and the NR factor is
\begin{equation}\Xi=|n_{0}|n_{0}-|n_{1}|n_{1}-\sum_{i=1}^{\mu}|m_{i}|m_{i}.\nonumber\end{equation} Throughout the proof we will only use the $X_{j'}$ norm for $\langle\partial_{x}\rangle^{-s^{3}}u^{(5)}$ for $j'\in\{2,3,4\}$, and it is important to notice that these norms are bounded by $C_{0}A$ instead of $C_{1}A$.

First assume $j=4$. We introduce the function $g$ with $\|g\|_{X_{4}'}\lesssim 1$, so that we only need to estimate $\mathcal{S}:=(g,h)$. This is a summation-integration we have seen many times before; to analyze it, we notice that either $\langle\Xi\rangle$, or one of $\langle\alpha_{l}\rangle$ (where $l\in\{1,2\}$) or $\langle\beta_{i}\rangle$, must be $\gtrsim \langle\alpha_{0}\rangle$.

Suppose $\langle\alpha_{0}\rangle\lesssim\langle\Xi\rangle$. Let the maximum of $\langle n_{0}\rangle$, $\langle n_{1}\rangle$ and all $\langle m_{i}\rangle$ be $\sim 2^{d}$ (and we fix $d$), then $\langle\alpha_{0}\rangle\lesssim 2^{2d}$. If among the variables $n_{0}$ and $m_{i}$, at least two are $\gtrsim 2^{(1-s^{2})d}$, then we will gain a net power $2^{c(1-\kappa)d}$ from the weights in the $X_{4}'$ bound for $g$, or from the $|m_{i}|^{-1}$ weights appearing in $\mathcal{S}$. Then we will be able to bound the $\big(\chi e^{\mathrm{i}(\Delta_{n_{1}}-\Delta_{n_{0}})}\big)^{\wedge}(\alpha_{2})$ factor using some inequality similar to (\ref{expcon}), fix the irrelevant $(m_{j},\beta_{j})$ variables to produce $\mathcal{S}_{sub}$, then estimate it by bounding $\mathfrak{N}g$ in $L^{2+}L^{2+}$, $\mathfrak{N}k$ and the two $\mathfrak{N}u^{(5)}$ factors in $L^{6}L^{6}$ and the $\big(\chi e^{\mathrm{i}(\Delta_{n_{1}}-\Delta_{n_{0}})}\big)^{\wedge}(\alpha_{2})$ factor in $l^{1+}l^{1+}$, where $2+$ is some $2+cs^{2}$, with a further loss of at most $2^{O(\epsilon)d}$. We then sum over the $(m_{j},\beta_{j})$ variables and sum over $d$ to conclude the estimate for $\mathcal{S}$. If instead only one of them can be $\gtrsim 2^{(1-s^{2})d}$ (again, assume $d$ is large enough), then this variable and $n_{1}$ must both be $\sim 2^{d}$. Let the maximum of all the remaining variables be $\sim 2^{d'}$ where $d'\leq (1-s^{2})d$ is also fixed, then we will have $|\alpha_{0}|\lesssim 2^{d+d'}$. Since we will be able to gain a power $2^{c(1-\kappa)(d+d')}$ from the weights, we can proceed in the same way as above.

Next, suppose $\langle\alpha_{0}\rangle\lesssim\langle\alpha_{2}\rangle$. By invoking (\ref{factt2}) we may get an estimate better than (\ref{expcon}) for the $\alpha_{2}$ factor, namely
\begin{equation}\label{betcon}\big\|\langle\alpha_{2}\rangle\big(\chi e^{\mathrm{i}(\Delta_{n_{1}}-\Delta_{n_{0}})}\big)^{\wedge}(\alpha_{2})\big\|_{L^{\sigma}}\lesssim C_{0}e^{C_{0}A}\sum_{i=1}^{\mu}\langle m_{i}\rangle^{s^{5}},\end{equation} for all $1\leq\sigma\leq\infty$; the $\lesssim$ here allows for a polynomial factor in $\mu$. Therefore, by losing a tiny power of some $m_{i}$, we may cancel the $\alpha_{0}$ weight in the $X_{4}'$ bound for $g$ and still bound the $\alpha_{2}$ factor in $L^{2}$, then fix $(m_{i},\beta_{i})$ and produce $\mathcal{S}_{sub}$, and estimate it by
\begin{eqnarray}
\mathcal{S}_{sub}&\lesssim&\sum_{n_{0}}\langle n_{0}\rangle^{-1}\big\|\langle n_{0}\rangle\langle\alpha_{0}\rangle^{-\kappa}g_{n_{0},\alpha_{0}}\big\|_{L_{\alpha_{0}}^{2}}\big\|k_{n_{0}+c_{1},\alpha_{1}}\big\|_{L_{\alpha_{1}}^{1}}\nonumber\\
&\lesssim &\big\|\langle n_{0}\rangle^{-1}\langle n_{0}\rangle\langle\alpha_{0}\rangle^{-\kappa}g_{n_{0},\alpha_{0}}\big\|_{l^{\frac{3}{2}}L^{2}}\cdot\| k_{n_{1},\alpha_{1}}\|_{l^{3}L^{1}}\lesssim 1 \nonumber,
\end{eqnarray} where $c_{j}$ are constants. If instead $\langle\alpha_{0}\rangle\lesssim\langle\alpha_{1}\rangle$, we can invoke the $\alpha_{1}$ weight in $X_{4}$ norm for $k$ to cancel the $\alpha_{0}$ weight, then notice that $\langle n_{1}\rangle\lesssim\langle n_{0}\rangle+\langle m_{i}\rangle$ for some $i$, then bound the $\alpha_{2}$ factor in $L^{1}$ and fix all the other $(m_{j},\beta_{j})$ to produce $\mathcal{S}_{sub}$. If $\langle n_{1}\rangle\lesssim\langle m_{i}\rangle$ we will estimate $\mathcal{S}_{sub}$ by
\begin{eqnarray}\mathcal{S}_{sub}&\lesssim&\sum_{n_{0}=n_{1}+m_{i}+c_{1}}\frac{\langle n_{1}\rangle}{\langle n_{0}\rangle\langle m_{i}\rangle}\big\|\langle n_{0}\rangle\langle\alpha_{0}\rangle^{-\kappa}g_{n_{0},\alpha_{0}}\big\|_{L_{\alpha_{0}}^{2}}\times\nonumber\\
&\times&\big\|\langle n_{1}\rangle^{-1}\langle\alpha_{1}\rangle^{\kappa}k_{n_{1},\alpha_{1}}\big\|_{L_{\alpha_{1}}^{2}}\big\|(u^{(5)})_{m_{i},\beta_{i}}\big\|_{L_{\beta_{i}}^{1}}\nonumber\\&\lesssim&\big\|\langle\alpha_{0}\rangle^{-\kappa}g\big\|_{l^{1}L^{2}}\big\|\langle n_{1}\rangle^{-1}\langle\alpha_{1}\rangle^{\kappa}k\big\|_{l^{\gamma}L^{2}}\cdot\|u^{(5)}\|_{l^{\gamma'}L^{1}}\lesssim1\nonumber,
\end{eqnarray} where $c_{j}$ are constants; note that $\|u^{(5)}\|_{l^{\gamma'}L^{1}}$ can be controlled by the $X_{2}$ norm of $\langle\partial_{x}\rangle^{-4s^{3}}u^{(5)}$ due to (\ref{hierarchy}). If $\langle n_{1}\rangle\lesssim\langle n_{0}\rangle$ we will instead estimate the $g$ factor above in $l^{\gamma'}L^{2}$, the $k$ factor in $l^{\gamma}L^{2}$, and the $u^{(5)}$ factor with weight $\langle m_{i}\rangle^{-1}$ in $l^{1}L^{1}$. Finally, if $\langle\alpha_{0}\rangle\lesssim\langle\beta_{i}\rangle$ for some $i$, we will cancel the $\langle\alpha_{0}\rangle$ weight by the $\langle\beta_{i}\rangle$ weight, then fix $(m_{j},\beta_{j})$ and again get $\mathcal{S}_{sub}$, which we estimate by
\begin{eqnarray}\mathcal{S}_{sub}&\lesssim &\sum_{n_{0}=n_{1}+m_{i}+c_{1}}\langle c_{1}\rangle^{-s}\langle n_{0}\rangle^{-s}\langle n_{1}\rangle^{-c(2-\gamma)}\big\|\langle n_{0}\rangle^{s}\langle \alpha_{0}\rangle^{-\kappa}g_{n_{0},\alpha_{0}}\big\|_{L_{\alpha_{0}}^{2}}\times\nonumber\\
&\times&\big\|\langle n_{1}\rangle^{c(2-\gamma)}k_{n_{1},\alpha_{1}}\big\|_{L_{\alpha_{1}}^{1}}\big\|\langle m_{i}\rangle^{-1}\langle\beta_{i}\rangle^{\kappa}(u^{(5)})_{m_{i},\beta_{i}}\big\|_{L_{\beta_{i}}^{2}}\nonumber\\
&\lesssim&\big\|\langle n_{0}\rangle^{s}\langle \alpha_{0}\rangle^{-\kappa}g\big\|_{l^{1}L^{2}}\|\langle n_{1}\rangle^{c(2-\gamma)}k\|_{l^{\gamma'}L^{1}}\cdot\big\|\langle\partial_{x}\rangle^{-4s^{3}}u^{(5)}\big\|_{X_{4}}\lesssim 1,\nonumber\end{eqnarray} where $c_{j}$ are constants, and again note that we can gain any small power of $c_{1}$, since $\pm c_{1}$ is the sum of all $m_{j}$ where $j\neq i$.

Next, let us assume $j\in\{2,3,8\}$. In this case we only use the $l^{1}L^{1}$ norm of $m_{i}^{-1}(u^{(5)})_{m_{i},\beta_{i}}$, so we will be free to lose any power $\langle m_{i}\rangle^{c}$ for small $c$. Therefore we may fix each $(m_{i},\beta_{i})$, invoke (\ref{betcon}) to fix $\alpha_{2}$ also (by an argument similar to the proof of Proposition \ref{weaken}), then reduce to bounding $\|z\|_{X_{j}}$ in terms of $\|k\|_{X_{j}}$, provided \begin{equation}\big|z_{n_{0},\alpha_{0}}\big|\leq\big|k_{n_{0}+c_{1},\alpha_{0}+|n_{0}+c_{1}|(n_{0}+c_{1})-|n_{0}|n_{0}+c_{2}}\big|.\nonumber\end{equation} But since the bound we get is allowed to grow like $\langle c_{1}\rangle^{s^{1/3}}$ (note that $-c_{1}$ is the sum of all $m_{i}$, and we are allowed to lose $\langle m_{i}\rangle^{c}$ for small $c$), this will be easy if we examine $X_{2}$, $X_{3}$ and $\mathcal{Y}$ separately (in particular, we will use the definition of the $\mathcal{Y}$ norm). The only thing we need to address is the $\langle n\rangle$ weights in the definition of $X_{2}$ and $X_{3}$, and the step of taking supremum when obtaining $X_{8}$ norm from $\mathcal{Y}$ norm; however, by a standard argument we can show that through these we will lose at most $\langle c_{1}\rangle^{O(s)}$ power, which is acceptable.

Finally, we may check that throughout the above proof, we only need to use the $X_{j}'$ norms of $\langle \partial_{x}\rangle^{-2s^{3}}u^{(5)}$ instead of $\langle \partial_{x}\rangle^{-s^{3}}u^{(5)}$; thus we will gain a power $K^{0+}$ if we make the restriction $m_{i}\gtrsim K$ for some $i$.
\end{proof}
To see how Proposition \ref{recoverlemma} allows us to construct extensions of $\mathbb{P}_{>K}u^{*}$ and $v^{*}$, we first note that $u^{*}$ is real valued, so we only need to construct an extension of $\mathbb{P}_{>+K}u^{*}$ (which is an abbreviation of $\mathbb{P}_{+}\mathbb{P}_{>K}u^{*}$). Now, in Proposition \ref{recoverlemma} we may choose $k$ to be an arbitrary extension of $v^{*}$ and $h$ to be some extension of $u^{*}$ (and choose $h'$ and $k'$ accordingly) so that (\ref{fu}) holds with appropriate coefficients (cf. (\ref{possub}) and (\ref{negsub})).

Exploiting the freedom in the choice of $k$, we will set $\mathbb{P}_{+}k=w^{(4)}$ and $\mathbb{P}_{\leq 0}k=\mathbb{P}_{\leq 0}v''$. The part coming from $\mathbb{P}_{+}k$ is bounded in $Y_{2}$ (before or after the $\mathbb{P}_{>+K}$ projection) by $C_{0}e^{C_{0}A}$ due to Proposition \ref{recoverlemma}, since we already have $\|w^{(4)}\|_{Y_{2}}\lesssim \|w^{(4)}\|_{Y_{1}}\leq C_{0}e^{C_{0}A}$. As for the part coming from $\mathbb{P}_{\leq 0}k$, we must have $n_{0}>K$ and $n_{1}\leq 0$ in (\ref{fu}), so the $\Psi$ factor will be nonzero only when $\langle m_{i}\rangle\gtrsim (\mu+2)^{-2}K$ for some $i$, thus we may again use Proposition \ref{recoverlemma} to bound this part in $Y_{2}$ by $O_{C_{1}}(1)e^{C_{0}C_{1}A}K^{0-}\leq 1$, since we have $\|v''\|_{Y_{2}}\leq C_{1}e^{C_{1}A}$. This completes the construction for the extension of $u^{*}$.

Now, to construct the extension of $v^{*}$, simply set the $k$ in Proposition \ref{recoverlemma} to be $u^{(4)}$ (which is the extension of $u^{*}$ we just constructed) and $h$ to be some extension of $v^{*}$ so that (\ref{fu}) holds with appropriate coefficients. Then this extension will do the job, since we already have $\|u^{(4)}\|_{Y_{2}}\leq C_{0}e^{C_{0}A}$. This finally completes the proof of Proposition \ref{uniformest}.
\section{The \emph{a priori} estimate V: Controlling the difference}\label{end}
The main purpose of this section is to provide necessary estimates for differences of two solutions to (\ref{smoothtrunc}). First we need to introduce some notations, including the definition of the metric space $\mathcal{BO}^{T}$, which will be used also in Section \ref{lwp}.
\subsection{Preparations}
\begin{definition}\label{bott}Suppose $\mathcal{Q}=(u'',v'',w'',u''')$ and $\mathcal{Q}'=(u^{\dagger\dagger},v^{\dagger\dagger},w^{\dagger\dagger},u^{\dagger\dagger\dagger})$ are two quadruples of functions defined on $\mathbb{R}\times\mathbb{T}$, we define their distance by
\begin{eqnarray}\mathfrak{D}_{\sigma}(\mathcal{Q},\mathcal{Q}')&=&\big\|\langle \partial_{x}\rangle^{-\sigma}(w''-w^{\dagger\dagger})\big\|_{Y_{1}}+\big\|\langle \partial_{x}\rangle^{-\sigma}(v''-v^{\dagger\dagger})\big\|_{Y_{2}}+\nonumber\\
&+&\big\|\langle \partial_{x}\rangle^{-\sigma}(u''-u^{\dagger\dagger})\big\|_{Y_{2}}+\big\|\langle\partial_{x}\rangle^{-s^{3}-\sigma}(u'''-u^{\dagger\dagger\dagger})\big\|_{X_{2}\cap X_{3}\cap X_{4}},\nonumber\end{eqnarray} for $\sigma\in\{0,s^{5}\}$. In particular, if $\sigma=\mathcal{Q}'=0$, we define the triple norm\begin{eqnarray}
|\!|\!|\mathcal{Q}|\!|\!|:=\mathfrak{D}_{0}(\mathcal{Q},0)&=&\big\|w''\big\|_{Y_{1}}+\big\|v''\big\|_{Y_{2}}+\big\|u''\big\|_{Y_{2}}+\big\|\langle\partial_{x}\rangle^{-s^{3}}u'''\big\|_{X_{2}\cap X_{3}\cap X_{4}}\nonumber.\end{eqnarray} Next, suppose $u$ and $u^{-}$ are functions defined on $I\times\mathbb{T}$ for some interval $I$, we will define\footnote[1]{Note that the definition depends on the choice of the origin in $\Delta_{n}(t)=\int^{t}\delta_{n}(t')\mathrm{d}t'$, but this will not affect the triple norm $|\!|\!|\cdot|\!|\!|$. This does affect estimates for differences, but we need them only when $I=[-T,T]$ or its translation, in which case the choice of origin is canonical.} the functions $(u^{*},v^{*},w^{*})$ corresponding to $u$, and $(u^{+},v^{+},w^{+})$ corresponding to $u^{-}$, as in Sections \ref{gaugetransform} and \ref{gaugetransform3}, and then define
\begin{equation}\mathfrak{D}_{\sigma}^{I}(u,u^{-})=\inf_{\mathcal{Q},\mathcal{Q}'}\mathfrak{D}_{\sigma}(\mathcal{Q},\mathcal{Q}'),\end{equation} where the infimum is taken over all quadruples $\mathcal{Q}$ and $\mathcal{Q}'$ that extends $(u^{*},v^{*},w^{*},u)$ and $(u^{+},v^{+},w^{+},u^{-})$ from $I\times\mathbb{T}$ to $\mathbb{R}\times\mathbb{T}$, respectively. We will also define $|\!|\!|u|\!|\!|_{I}=\mathfrak{D}_{0}(u,0)=\inf_{\mathcal{Q}}|\!|\!|\mathcal{Q}|\!|\!|$; these notations can be written in a more familiar way as
\begin{eqnarray}\mathfrak{D}_{\sigma}^{I}(u,u^{-})&=&\big\|\langle\partial_{x}\rangle^{-\sigma}(w^{*}-w^{+})\big\|_{Y_{1}^{I}}+\big\|\langle\partial_{x}\rangle^{-\sigma}(v^{*}-v^{+})\big\|_{Y_{2}^{I}}+\nonumber\\
&+&\big\|\langle\partial_{x}\rangle^{-\sigma}(u^{*}-u^{+})\big\|_{Y_{2}^{I}}+\big\|\langle\partial_{x}\rangle^{-s^{3}-\sigma}(u-u^{-})\big\|_{(X_{2}\cap X_{3}\cap X_{4})^{I}};\nonumber
\end{eqnarray}
\begin{eqnarray}|\!|\!|u|\!|\!|_{I}&=&\big\|w^{*}\big\|_{Y_{1}^{I}}+\big\|v^{*}\big\|_{Y_{2}^{I}}+\big\|u^{*}\big\|_{Y_{2}^{I}}+\big\|\langle\partial_{x}\rangle^{-s^{3}}u\big\|_{(X_{2}\cap X_{3}\cap X_{4})^{I}}.\nonumber
\end{eqnarray}
Now we can define the metric space
\begin{equation}\mathcal{BO}^{I}=\{u:|\!|\!|u|\!|\!|_{I}<\infty\},\end{equation} with the distance function given by $\mathfrak{D}_{0}^{I}$ (we will also use $\mathfrak{D}_{s^{5}}^{I}$, which is also well-defined on $\mathcal{BO}^{I}$). Finally, when $I=[-T,T]$, we may use $T$ in place of $I$ in sub- or superscripts, so this contains the definition of $\mathcal{BO}^{T}$.
\end{definition}
\begin{remark}\label{rmk00} If $u\in\mathcal{BO}^{T}$, we may define $uu_{x}=\frac{1}{2}\partial_{x}(\mathbb{P}_{\neq 0}u^{2})$ as a distribution on $[-T,T]$ through an argument similar to the one in Section \ref{mid3}. More precisely, we may uniquely define the function
\begin{equation}\label{evoo}h(t)=\int_{0}^{t}e^{-(t-t')H\partial_{xx}}\big(u(t')\partial_{x}u(t')\big)\,\mathrm{d}t'\end{equation} as an element of $(X^{-\frac{1}{s},\kappa})^{T}$.

In particular, we may define $u\in \mathcal{BO}^{T}$ to be a solution to (\ref{bo}) on $[-T,T]$, if $u$ verifies the integral version of (\ref{bo}) with the evolution term defined as in (\ref{evoo}). Clearly this definition is independent of the choice of origin, and $[-T,T]$ may be replaced by any interval $I$.

Moreover, since the arguments in Section \ref{mid3} allow for some room, the map sending $u$ to $h$ in (\ref{evoo}) is continuous with respect to the \emph{weak} distance function $\mathfrak{D}_{s^{5}}^{T}$ (or $\mathfrak{D}_{s^{5}}$ if we consider the map sending the quadruple $\mathcal{Q}$ to $h$). This fact will be important in the proof of Theorem \ref{main'}.
\end{remark}
\begin{proposition}\label{embed000} Suppose $u$ and $u^{-}$ are two functions defined on $I\times\mathbb{T}$, and choose corresponding extensions $\mathcal{Q}=(u'',v'',w'',u''')$ and $\mathcal{Q}'=(u^{\dagger\dagger},v^{\dagger\dagger},w^{\dagger\dagger},u^{\dagger\dagger\dagger})$. We then have\begin{equation}\label{embed01}\|u\|_{C_{t}^{0}(I\to Z_{1})}\lesssim|\!|\!|\mathcal{Q}|\!|\!|;\,\,\,\,\,\,\,\,\,\,\|u\|_{C_{t}^{0}(I\to Z_{1})}\lesssim|\!|\!|u|\!|\!|_{I}.\end{equation} Concerning differences, we only have the weaker estimate
\begin{equation}\label{embed03}\|\langle\partial_{x}\rangle^{-s^{5}}(u-u^{-})\|_{C_{t}^{0}(I\to Z_{1})}\lesssim O_{|\!|\!|\mathcal{Q}|\!|\!|,|\!|\!|\mathcal{Q}'|\!|\!|}(1)\cdot\mathfrak{D}_{s^{5}}(\mathcal{Q},\mathcal{Q}');\end{equation}
\begin{equation}\label{embed04}\|\langle\partial_{x}\rangle^{-\theta}(u-u^{-})\|_{C_{t}^{0}(I\to Z_{1})}\lesssim O_{\theta,|\!|\!|\mathcal{Q}|\!|\!|,|\!|\!|\mathcal{Q}'|\!|\!|}(1)\cdot\mathfrak{D}_{0}(\mathcal{Q},\mathcal{Q}'),\end{equation} for all $\theta>0$, where the constant may also depend on the upper bound of the length of $I$. Note that the $\theta$ in (\ref{embed04}) cannot be removed, thus $C_{t}^{0}([-T,T]\to Z_{1})$ cannot be embedded into $\mathcal{BO}^{T}$ as a \emph{subspace}.
\end{proposition} 
\begin{proof} We may assume $I=[-T,T]$ with $T\lesssim 1$. The inequalities in (\ref{embed01}) follow directly from definition (and the fact that $u(t)$ and $u''(t)$ have the same $Z_{1}$ norm for $t\in [-T,T]$); the proof of (\ref{embed03}) and (\ref{embed04}) are similar, so we only prove (\ref{embed03}). Assume $|\!|\!|\mathcal{Q}|\!|\!|+|\!|\!|\mathcal{Q}'|\!|\!|\lesssim 1$ and $\mathfrak{D}_{s^{5}}(\mathcal{Q},\mathcal{Q}')\leq \varepsilon$, we will define $\Delta_{n}$ and $\Delta_{n}^{-}$ corresponding to $\mathcal{Q}$ and $\mathcal{Q}'$ as in (\ref{factor0}) and (\ref{factor1}) using functions $(w'',u''')$ and $(w^{\dagger\dagger},u^{\dagger\dagger\dagger})$ respectively, then set $u'$ and $u^{\dagger}$ be extensions of $u$ and $u^{-}$, defined by $(u')_{n}=\chi(t)e^{\mathrm{i}\Delta_{n}}(u'')_{n}$ and similarly for $u^{\dagger}$. Since 
\begin{equation}\big\|\langle \partial_{x}\rangle^{-s^{5}}(u''-u^{\dagger\dagger})(t)\big\|_{Z_{1}}\lesssim\mathfrak{D}_{s^{5}}(\mathcal{Q},\mathcal{Q}')\lesssim\varepsilon,\end{equation}we only need to estimate the function $z$ defined by $z_{n}=(u'')_{n}(e^{\mathrm{i}\Delta_{n}}-e^{\mathrm{i}\Delta_{n}^{-}})$. Due to the bound $|\!|\!|\mathcal{Q}|\!|\!|\lesssim 1$ which implies the bound the the $Z_{1}$ norm of each $u''(t)$, we only need to prove 
\begin{equation}\big|\chi(t)(e^{\mathrm{i}\Delta_{n}}-e^{\mathrm{i}\Delta_{n}^{-}})(t)\big|\lesssim \varepsilon\langle n\rangle^{s^{5}}\end{equation} for each $n$ and $t$. Using the arguments in Lemma \ref{general}, it suffices to prove the bound for $\delta_{n}-\delta_{n}^{-}$, but if we use (\ref{factor1}), this will be clear from the strong bounds on $w''$ and $w^{\dagger\dagger}$, and the weak bound on their difference.
\end{proof}
\subsection{Statement and proof}
Now suppose $u$ is a smooth function solving (\ref{smoothtrunc}) on $[-T,T]$. The arguments in Sections \ref{begin}-\ref{mid2} actually give us a way to update a given quadruple $\mathcal{Q}=(u'',v'',w'',u''')$ extending $(u^{*},v^{*},w^{*},u)$ to a new quadruple $\mathcal{Q}'=(u^{(4)},v^{(4)},w^{(4)},u^{(5)})$, which remains to be an extension, and verifies better bounds. We will define $\mathfrak{I}$ to be the map from the set of extensions to itself, that sends $\mathcal{Q}$ to $\mathcal{Q}'$. Using the arguments from Sections \ref{begin}-\ref{mid2}, we can prove
\begin{proposition}\label{reform}Let $C_{1}$ be large enough, $C_{2}$ large enough depending on $C_{1}$, and $0<T\leq C_{2}^{-1}e^{C_{2}A}$. Suppose $u$ is a smooth function solving (\ref{smoothtrunc}) on $[-T,T]$, and $\mathcal{Q}$ is an extension satisfying
\begin{equation}\label{singleest}|\!|\!|\mathcal{Q}|\!|\!|\leq C_{1}e^{C_{1}A},\,\,\,\,\,\,\,\,\big\|\langle\partial_{x}\rangle^{-s^{3}}u'''\big\|_{X_{2}\cap X_{3}+\cap X_{4}}\leq C_{1}A,\end{equation} then the same estimate will hold if we replace $\mathcal{Q}$ by $\mathfrak{I}\mathcal{Q}$.
\end{proposition}
Now we can state the main proposition in this section, namely
\begin{proposition}\label{difference0}Let $C_{1}$, $C_{2}$ and $T$ be as in Proposition \ref{reform}. Suppose $u$ and $u^{-}$ are two smooth functions solving (\ref{smoothtrunc}) with truncation $S_{N}$ and $S_{M}$ respectively, where $1\ll N\leq M\leq\infty$, $\mathcal{Q}$ and $\mathcal{Q}'$ are two quadruples corresponding to $u$ and $u^{-}$ respectively, such that (\ref{singleest}) holds, and that
\begin{equation}\label{doubleest}\mathfrak{D}_{s^{5}}(\mathcal{Q},\mathcal{Q}')\leq B\end{equation} for some $B>0$, then we have 
\begin{equation}\label{doubleest2}\mathfrak{D}_{s^{5}}(\mathfrak{I}\mathcal{Q},\mathfrak{I}\mathcal{Q}')\leq\frac{B}{2}+O_{C_{1}}(1)e^{C_{0}C_{1}A}\big(\big\|\langle\partial_{x}\rangle^{-s^{5}}(u(0)-u^{-}(0))\big\|_{Z_{1}}+N^{0-}\big),\end{equation} where $C_{0}$ is any constant appearing in previous sections. In particular, we have 
\begin{equation}\label{differencestate}\mathfrak{D}_{s^{5}}^{T}(u,u^{-})\leq O_{C_{2},A}(1)\big(\big\|\langle\partial_{x}\rangle^{-s^{5}}(u(0)-u^{-}(0))\big\|_{Z_{1}}+N^{0-}\big),\end{equation}provided $\|u(0)\|_{Z_{1}}+\|u^{-}(0)\|_{Z_{1}}\leq A$ for some large $A$. Moreover, if $M=N$, we may replace the $\mathfrak{D}_{s^{5}}$ distance by the $\mathfrak{D}_{0}$ distance and remove the $N^{0-}$ term on the right hand side of (\ref{differencestate}).
\end{proposition}
\begin{proof} When we take differences in the case $M=N$ the right hand side will involve only factors like $u-u^{-}$ and not the ones like $\mathbb{P}_{\geq N}u$, thus we will not have an $N^{0-}$ term on the right hand side. Also, it is easy to see from the proof below that removing the $\langle n\rangle^{-s^{5}}$ weight will only make arguments easier. Thus we will focus on (\ref{differencestate}) now. By an iteration using Proposition \ref{reform}, we only need to prove (\ref{doubleest2}) assuming (\ref{singleest}) and (\ref{doubleest}).

Recall the functions $\delta_{n}$, $\delta_{n}^{-}$, $\Delta_{n}$, $\Delta_{n}^{-}$ and $y$, $y^{-}$ that come from the two quadruples $\mathcal{Q}$ and $\mathcal{Q}'$ in the same way as in Section \ref{theextensions}. The two functions $y$ and $y^{-}$ will verify two equations with the form of (\ref{global2}) separately. Clearly we may also assume all relevant functions are supported in $|t|\lesssim 1$. To bound the first part of $\mathfrak{D}_{s^{5}}(\mathfrak{I}\mathcal{Q},\mathfrak{I}\mathcal{Q}')$ requires to prove
\begin{equation}\label{bound010}\big\|\langle\partial_{x}\rangle^{-s^{5}}(y-y^{-})\big\|_{Y_{1}}\leq\frac{B}{10}+O_{C_{1}}(1)e^{C_{0}C_{1}A}(\theta+N^{0-}),\end{equation} where we denote $\|\langle \partial_{x}\rangle^{-s^{5}}(u(0)-u^{-}(0))\|_{Z_{1}}=\theta$ for simplicity.

By another bootstrap argument, we may assume (\ref{bound010}) holds with right hand side multiplied by $O_{C_{1}}(1)$. Recall the equations
\begin{equation}\label{eqn0001}y=\chi(t)e^{\mathrm{i}t\partial_{xx}}w(0)+\mathcal{E}\big(\mathbf{1}_{[-T,T]}\mathcal{N}^{2}(y,y)\big)+\sum_{j\in\{3,3.5,4,4.5\}}\mathcal{E}(\mathbf{1}_{[-T,T]}\mathcal{N}^{j});\end{equation}
\begin{equation}\label{eqn0002}y^-{}=\chi(t)e^{\mathrm{i}t\partial_{xx}}w^{-}(0)+\mathcal{E}\big(\mathbf{1}_{[-T,T]}\mathcal{N}^{2-}(y^{-},y^{-})\big)+\sum_{j\in\{3,3.5,4,4.5\}}\mathcal{E}(\mathbf{1}_{[-T,T]}\mathcal{N}^{j-}),\end{equation} where $\mathcal{N}^{j}$ and $\mathcal{N}^{j-}$ are suitable nonlinearities; to bound $y-y^{-}$, we will first bound
\begin{equation}\sum_{j\in\{0,3,3.5,4,4.5\}}\big\|\langle\partial_{x}\rangle^{-s^{5}}(\mathcal{M}^{j}-\mathcal{M}^{j-})\big\|_{Y_{1}},\nonumber\end{equation} where the definitions of $\mathcal{M}^{j}$ and $\mathcal{M}^{j-}$ are clear (the term $j=0$ corresponds to the linear term which can be bounded by $\theta+N^{0-}$, so we will omit this below).

Here it is important to note that \emph{all the bounds in the previous sections are proved directly using multilinear estimates}, thus they will automatically imply the corresponding estimates for differences. In fact, when we try to estimate $\mathcal{M}^{j}-\mathcal{M}^{j-}$ by introducing some $(g,f)$ and forming an $\mathcal{S}$ expression, there are a few possibilities:

(1) Suppose we take the difference $y-y^{-}$, or (for example) some $v''-v^{\dagger\dagger}$ directly. Then one of the $y$ or $v''$ factors appearing in the previous sections will be replaced by this difference. Note that if we estimate this difference in the weakened norm $\|\langle\partial_{x}\rangle^{-s^{5}}\cdot\|_{Y_{j}}$ (we use $X_{2}\cap X_{3}\cap X_{4}$ norm for $u'''-u^{\dagger\dagger\dagger}$, but the proof will be the same), we will get a bound $O_{C_{1}}(1)e^{C_{0}C_{1}A}(B+\theta+N^{0-})$ which is what we need; the loss coming from using this weaker norm can be recovered from the fact that we only need to estimate the weaker norm of $\mathcal{M}^{j}-\mathcal{M}^{j-}$. To be precise, for each multilinear estimate we prove in the previous sections, suppose the term we bound in the weaker norm (i.e. the norm involving $\langle \partial_{x}\rangle^{-s^{5}}$) corresponds to the variable $n_{l}$, then one of the followings must hold: (i) we can gain a power $2^{(0+)d}$ in the estimate, where the $0+$ is at least $cs^{2.5}$, and we also have $\langle n_{l}\rangle\lesssim 2^{d}$. In this case it will suffice to use this weaker norm in all the discussions before, so this part will be acceptable; (ii) we have $\langle n_{0}\rangle\gtrsim\langle n_{l}\rangle$ (for example, when $n_{0}=n_{l}$ and the other variables are small compared to them). In this case, since we only need to estimate the output $y-y^{-}$ in the weaker norm, we will gain a power $\langle n_{0}\rangle^{s^{5}}$ compared to the proof in the previous sections, which is enough to cancel the loss $\langle n_{l}\rangle^{s^{5}}$, thus this part is also acceptable; (iii) we have $\langle n_{0}\rangle\sim 2^{d}$ and $\langle n_{l}\rangle\sim 2^{d'}$, and the expression $\mathcal{S}$ involves the factor $2^{-|d-d'|}$ (this appears, for example, in various ``resonant'' cases in Section \ref{mid1} and Proposition \ref{easy}, and is characterized by the need to use (\ref{festt})). In this case we lose at most $2^{s^{5}|d-d'|}$ from the additional weights compared to the proof in the previous sections, which can be cancelled by the $2^{-|d-d'|}$ factor, so it will still be acceptable. To conclude, we can estimate this part of $y-y^{-}$ in the weaker norm as
\begin{equation}T^{0+}O_{C_{1}}(1)e^{C_{0}C_{1}A}(B+\theta+N^{0-}),\nonumber\end{equation} by repeating the arguments in the previous sections, with minor modifications illustrated above.

(2) Suppose we take the difference of the $\Phi$ weights. The difference will verify the same bounds as the weights themselves; moreover it is nonzero only when some $m$ or $n$ variable is $\gtrsim N$. Therefore we may replace one of the $y$ or $v''$ factors appearing in the previous sections by $\mathbb{P}_{\geq N}y$ or $\mathbb{P}_{\geq N}v''$. We then proceed as in case (1), estimating this particular factor in the weakened norm to gain a power $N^{0-}$, and bound the whole expression in the same way as in case (1).

(3) Suppose we take (for example) the difference $v'-v^{\dagger}$, where $(v')_{n}=e^{\mathrm{i}\Delta_{n}}(v'')_{n}$ and $(v^{\dagger})_{n}=e^{\mathrm{i}\Delta_{n}^{-}}(v^{\dagger\dagger})_{n}$; alternatively, suppose we take the difference
\begin{equation}e^{\mathrm{i}(\pm \Delta_{n_{0}}\pm\Delta_{n_{1}}\pm\cdots)}-e^{\mathrm{i}(\pm \Delta_{n_{0}}^{-}\pm\Delta_{n_{1}}^{-}\pm\cdots)}\nonumber.\end{equation} It turns out that whenever we need to estimate these factors, we will always gain (from these factors themselves, or from elsewhere) some power $2^{(0+)d}$ where the $0+$ is at least $cs^{2.5}$, and $2^{d}$ controls every relevant variable (for typical examples, see the estimate of $J_{(n)}(\alpha_{5})$ as defined in (\ref{J(n)}) in the proof of Proposition \ref{easy}, as well as the last part of Section \ref{mid1}). Here we may use Proposition \ref{general} to reduce the estimation of the difference of these exponential factors to the estimation of the differences $\delta_{n}-\delta_{n}^{-}$ themselves. Since we can bound functions like $w''-w^{\dagger\dagger}$ in the weaker norm by $O_{C_{1}}(1)e^{C_{0}C_{1}A}(B+\theta+N^{0-})$, we will be able to obtain estimates similar to the ones in Proposition \ref{factt}, but with the coefficient $C_{0}C_{1}e^{C_{0}C_{1}A}$ on the right hand side replaced by $O_{C_{1}}(1)e^{C_{0}C_{1}A}(B+\theta+N^{0-})$, with a loss of at most $\langle n\rangle^{O(s^{5})}$ which is dwarfed by the power we gain. Finally we may use the $T^{0+}$ gain coming from the evolution to cancel the $O_{C_{1}}(1)e^{C_{0}C_{1}A}$ factor, thus this part is also acceptable.

Next we need to control the difference of the $\mathcal{M}^{2}$ terms. We will follow the proof in Section \ref{mid1}, and the part of the proof where no second iteration is needed can be completed in the same way as above. As for the remaining part, what we do in Section \ref{mid1} is basically rewriting
\begin{equation}\mathcal{N}^{6}(y,y)=\sum_{j\in\{0,3,3.5,4,4.5\}}\mathcal{N}^{6}(y,\mathcal{M}^{j})+\mathcal{N}^{6}(y,\mathcal{E}(\mathbf{1}_{[-T,T]}\mathcal{N}^{2}(y,y)))\nonumber\end{equation} where $\mathcal{N}^{6}$ is the part of $\mathcal{N}^{2}$ uner consideration; we may also rewrite $\mathcal{N}^{6-}(y^{-},y^{-})$ in the same way. When we take difference, we may control the first term on the right hand side using the bound for $\mathcal{M}^{j}-\mathcal{M}^{j-}$ as in Proposition \ref{easiest} (actually we have a slightly weaker version, but this will suffice); as for the second term, since it is bounded in Section \ref{mid1} via multilinear estimates, we can again treat the difference in the same way as above. This completes the proof for the bound of $w^{*}-w^{+}$.

Next, recall that the other parts of $\mathfrak{I}\mathcal{Q}$ and $\mathfrak{I}\mathcal{Q}'$ such as $u^{(5)}$ and $u^{[5]}$, $u^{(4)}$, $u^{[4]}$, $v^{(4)}$ and $v^{[4]}$ are constructed in the same way as in Section \ref{mid3}, where the scale $K$ is taken to be $K=C_{1.5}e^{C_{1.5}A}$ with $C_{1.5}$ large enough depending on $C_{1}$, but small compared to $C_{2}$. Note that we may redefine $\Delta_{n}$ and $\Delta_{n}^{-}$ when necessary. Now to prove (\ref{doubleest}), we need to bound the differences such as $u^{(4)}-u^{[4]}$ in the weaker norm by $O_{C_{1}}(1)e^{C_{0}C_{1}A}(K^{0-}B+\theta+N^{0-})$. But this can again be achieved by combining the argument above with the proof in Section \ref{mid3}, if we notice two things:

(1) In the proof of Proposition \ref{recoverlemma}, we can always gain some power $\langle m_{i}\rangle^{cs^{2.5}}$ for each $m_{i}$, so we will be able to cover the loss coming from using only the weaker norm if we take the difference of the exponential factors (cf. (\ref{betcon})), of if we take $u^{(5)}-u^{[5]}$. For the same reason, if we lose a power $\langle n_{1}\rangle^{s^{5}}$ we will be able to recover it from the gain $\langle n_{0}\rangle^{s^{5}}$. 

(2) From the above we already know that the weaker norm of $w^{(4)}-w^{[4]}$ can be bounded by $O_{C_{1}}(1)e^{C_{0}C_{1}A}(T^{0+}B+\theta+N^{0-})$. We may then prove the same bound (possibly with some $O_{C_{1}}(1)$ factors) for $\mathbb{P}_{>K}(u^{(5)}-u^{[5]})$, $\mathbb{P}_{\leq K}(u^{(5)}-u^{[5]})$, $\mathbb{P}_{\leq K}(u^{(4)}-u^{[4]})$, $\mathbb{P}_{>K}(u^{(4)}-u^{[4]})$ and $v^{(4)}-v^{[4]}$ \emph{in that order}, in the same way as in Section \ref{mid3} (note $T^{-1}$ is assumed to be larger than any power of $K$).

Therefore we will be able to bound all the differences and thus complete the proof of Proposition \ref{difference0}.
\end{proof}
\section{Proof of the main results}\label{lwp}With Propositions \ref{uniformest} and \ref{difference0}, it is now easy to prove our main results. Since the argument in this section will be more or less standard, we may present only the most important steps. 
\subsection{Local well-posedness and stability}\label{llwp}
\begin{theorem}[Precise version of Theorem \ref{main}]\label{main'} There exists a constant $C$ such that, when we choose any $A>0$ and $0<T\leq C^{-1}e^{-CA}$, the followings will hold.

(1) Existence: for any $f\in\mathcal{V}$ with $\|f\|_{Z_{1}}\leq A$, there exists some $u\in \mathcal{BO}^{T}$ such that $|\!|\!|u|\!|\!|_{T}\leq Ce^{CA}$, and it verifies the equation (\ref{bo}), in the sense described in Remark \ref{rmk00}, with initial data $u(0)=f$.

(2) Continuity: let the solution described in part (1) be $u=\Phi f=(\Phi_{t}f)_{t}$. Suppose $\|f\|_{Z_{1}}\leq A$ and $\|g\|_{Z_{1}}\leq A$, then each $\varepsilon>0$, we have 
\begin{equation}\sup_{|t|\leq T}\big\|\langle \partial_{x}\rangle^{-s^{5}}(\Phi_{t}f-\Phi_{t}g)\big\|_{Z_{1}}+\mathfrak{D}_{s^{5}}^{T}(\Phi f,\Phi g)\leq O_{C,A}(1)\big\|\langle \partial_{x}\rangle^{-s^{5}}(f-g)\big\|_{Z_{1}};\nonumber\end{equation}
\begin{equation}\sup_{|t|\leq T}\big\|\langle \partial_{x}\rangle^{-\varepsilon}(\Phi_{t}f-\Phi_{t}g)\big\|_{Z_{1}}+\mathfrak{D}_{0}^{T}(\Phi f,\Phi g)\leq O_{\varepsilon,C,A}(1)\big\|f-g\big\|_{Z_{1}}.\nonumber\end{equation}

(3) Short-time stability: let $u=\Phi f$ as in part (2), and let $\Phi^{N}$ be the solution flow of (\ref{smoothtrunc}) and $u^{N}=\Phi^{N}\Pi_{N}f$. Then we have
\begin{equation}\lim_{N\to\infty}\bigg(\mathfrak{D}_{s^{5}}^{T}(u^{N},u)+\sup_{|t|\leq T}\big\|\langle \partial_{x}\rangle^{-s^{5}}\big(u^{N}(t)-u(t)\big)\big\|_{Z_{1}}\bigg)=0.\nonumber\end{equation}

(4) Uniqueness: for any \emph{other} time $T'$, suppose $u$ and $u^{-}$ are two elements of $\mathcal{BO}^{T'}$ with the same initial data, and they both solve (\ref{bo}), then we must have $u=u^{-}$ (on $[-T',T']$).

(5) Long-time existence: consider any $f\in Z_{1}$, and define the functions $u^{N}$ as in part (3). Suppose for some \emph{other} time $T'$ and some subsequence $\{N_{k}\}$ that
\begin{equation}\sup_{k}|\!|\!|u^{N_{k}}|\!|\!|_{T'}<\infty,\end{equation} then there exists a solution $u\in\mathcal{BO}^{T'}$ to (\ref{bo}) with initial data $f$.
\end{theorem}
\begin{proof}Suppose $f\in Z_{1}$ and $\|f\|_{Z_{1}}\leq A$, and let $0<T\leq C_{2}^{-1}e^{-C_{2}A}$ with constants as in Propositions \ref{reform} and \ref{difference0}. Consider $u^{N}$ as defined in (3); using Proposition \ref{uniformest}, we may choose for each $N$ some quadruple $\mathcal{Q}_{N}$ corresponding to $u^{N}$ that verifies (\ref{singleest}). We then define
\begin{equation}\mathcal{Q}^{N}=\mathfrak{I}^{N}\mathcal{Q}_{N},\end{equation}
then it will be clear from Propositions \ref{reform} and \ref{difference0} that
\begin{equation}|\!|\!|\mathcal{Q}^{N}|\!|\!|\leq C_{1}e^{C_{1}A};\end{equation}
\begin{equation}\lim_{N,M\to\infty}\mathfrak{D}_{s^{5}}(\mathcal{Q}^{M},\mathcal{Q}^{N})=0.\end{equation} By a simple completeness argument we can then find some $\mathcal{Q}$ so that $\mathfrak{D}_{s^{5}}(\mathcal{Q}^{N},\mathcal{Q})\to 0$ (in particular $\mathcal{Q}$ will have initial data $f$), and by an argument similar to the proof of Proposition \ref{initialboot} we deduce that $|\!|\!|\mathcal{Q}|\!|\!|\leq C_{1}e^{C_{1}A}$. By using Remark \ref{rmk00}, we can now pass to the limit and show that the quadruple $\mathcal{Q}$ gives a solution $u\in\mathcal{BO}^{T}$ of (\ref{bo}) on the interval $[-T,T]$. This proves existence.

Parts (2) and (3) will follow from basically the same argument. In fact, for each $(f,g)$, we may construct $\mathcal{Q}^{N}$ and $\mathcal{Q}^{N-}$ corresponding to $\Phi^{N}\Pi_{N}f$ and $\Phi^{N}\Pi_{N}g$ as above, so that they have uniformly bounded triple norm, and moreover \[\mathfrak{D}_{s^{5}}(\mathcal{Q}^{N},\mathcal{Q}^{N-})\lesssim\|\langle \partial_{x}\rangle^{-s^{5}}(f-g)\|_{Z_{1}}+N^{0-}.\] Using Proposition \ref{embed000} and passing to the limit, we obtain the result in (2). The result in (3) follows from comparing $\mathcal{Q}^{N}$ with $\mathcal{Q}$ and using Proposition \ref{embed000} also.

As for part (5), we will deduce it merely from the condition that $|\!|\!|u^{N_{k}}|\!|\!|_{T'}\leq A$ and \begin{equation}\label{diffff0}\big\|\langle\partial_{x}\rangle^{-s^{5}}(u^{N_{k}}-u)(0)\big\|_{Z_{1}}\to 0,\end{equation} which is clearly satisfied in our setting. Choose some $\tau$ small enough depending on $A$, then $\|u(0)\|_{Z_{1}}\leq C_{0}A$ implies we can solve (\ref{bo}) on $[-\tau,\tau]$, and from (Proposition \ref{difference0} and) what we just proved, we also have 
\begin{equation}\label{diffff1}\big\|\langle\partial_{x}\rangle^{-s^{5}}(u^{N_{k}}-u)(\pm\tau)\big\|_{Z_{1}}\to 0,\end{equation} and therefore
\begin{equation}\|u(\pm\tau)\|_{Z_{1}}\leq \limsup_{N\to\infty}\|u^{N_{k}}(\pm\tau)\|_{Z_{1}}\leq C_{0}A.\end{equation} This information will allow us to restart from time $\pm\tau$, and thus obtain a solution to (\ref{bo}) on $[-2\tau,2\tau]$. Repeating this, we will finally get a solution on $[-T',T']$, which we can prove to be in $\mathcal{BO}^{T'}$ using partitions of unity. This proves (conditional) global existence.

Finally we need to prove uniqueness. Let $u$ and $u^{-}$ be to solutions to (\ref{bo}) that both belong to $\mathcal{BO}^{T'}$ and have the same initial data. Let their strong norms be bounded by $A$, and choose $\tau$ small enough depending on $A$. To prove that $u=u^{-}$ on $[-\tau,\tau]$, we need to prove the following claim: if for quadruples $\mathcal{Q}$ and $\mathcal{Q}^{'}$ corresponding to $u$ and $u^{-}$ respectively, we have \begin{equation}\label{uniqueness}|\!|\!|\mathcal{Q}|\!|\!|+|\!|\!|\mathcal{Q}'|\!|\!|\leq A,\,\,\,\,\,\,\mathfrak{D}_{s^{5}}(\mathcal{Q},\mathcal{Q}')\leq K,\end{equation} then with $\mathcal{Q}$ replaced by $\mathfrak{I}\mathcal{Q}$ and $\mathcal{Q}'$ by $\mathfrak{I}\mathcal{Q}'$, the inequalities will hold with $A$ unchanged and $K$ replaced by $K/2$. Thus we need to repeat the whole argument from Section \ref{begin} to Section \ref{end} \emph{without} the smoothness assumption. Fortunately, since we have chosen $\tau\leq\tau(A)$, we do not need  the bootstrap argument (which requires \emph{a priori} smoothness) in bounding the evolution term; however, we do need this in Section \ref{begin} when we try to obtain a first bound for $\|y\|_{Y_{1}}$. 

This difficulty can be overcome as follows: first, we may check every part of Sections \ref{begin}, \ref{mid1} and \ref{mid2} that in order to bound $y$ in $Y_{1}$ using the evolution equation (\ref{global2}), it will suffice to bound $y$ in some weaker space $Y_{1}^{w}$ defined by (cf. Section \ref{mainspace})
\begin{equation}\|u\|_{Y_{1}^{w}}=\|u\|_{X_{1}^{w}}+\|u\|_{X_{2}^{w}}+\|u\|_{X_{4}^{w}}+\|u\|_{X_{5}^{w}}+\|u\|_{X_{7}^{w}}.\end{equation} Here to obtain the $X_{j}^{w}$ norm, we weaken the $X_{j}$ by decreasing the powers $b$ in (\ref{norm0001}), $\kappa$ in (\ref{norm0004}) and $\frac{1}{8}$ in (\ref{norm0007}) by $s^{5}$, and increasing the indices $1$ in (\ref{norm0003}) and $q$ in (\ref{norm0005}) by $s^{5}$. Notice that any power of $n$ and any norm $l^{p}$ remain unchanged. Therefore, we only need to show that the linear map $L$ defining $y$ from $w''$ (see Section \ref{theextensions}) is bounded from $Y_{1}$ to $Y_{1}^{w}$, since this combined with the proof fro Section \ref{begin} to Section \ref{mid2} will give us a stronger bound of $y$ in $Y_{1}$ and close the estimate (note that after the end of Section \ref{mid2}, all the arguments will not depend on smoothness, and we will be able to finish just as Sections \ref{mid3} and \ref{end}).

Now, suppose $\|u\|_{Y_{1}}\leq 1$, we can easily show that $\|Lu\|_{X_{2}^{w}}+\|Lu\|_{X_{5}^{w}}\lesssim 1$ using the decomposition
\begin{eqnarray}\label{decomposition001}Lu&=&u\cdot\mathbf{1}_{[-T,T]}(t)+\chi(t)\mathbf{1}_{[T,+\infty)}(t)e^{-(t-T)H\partial_{xx}}u(T)\\
&+&\chi(t)\mathbf{1}_{(-\infty,-T]}(t)e^{-(t+T)H\partial_{xx}}u(-T)\nonumber.\end{eqnarray} In fact the last two terms in (\ref{decomposition001}) is bounded in $X_{2}^{w}$ and $X_{5}^{w}$ because $u(\pm T)$ is bounded in $Z_{1}$, and the Fourier transform of $\chi(t)\mathbf{1}_{[T,+\infty)}(t)$ is in $L^{k}$ for $k>1$; the first term is bounded because convolution with the Fourier transform of $\mathbf{1}_{[-T,T]}$ (which decays like $\langle\xi\rangle^{-1}$ uniformly for $T\lesssim 1$) is bounded from $L_{\xi}^{k}$ to $L_{\xi}^{k'}$ for all $k'>k$. Now to bound $Lu$ in $X_{j}^{w}$ form $j\in\{1,4,7\}$, we only need to bound the operator 
\begin{equation}\widetilde{L}:f(t)\mapsto\mathbf{1}_{[-T,T]}(t)f(t)+\chi(t)\mathbf{1}_{[T,+\infty)}f(T)+\chi(t)\mathbf{1}_{(-\infty,-T]}f(-T)\end{equation}from $H_{t}^{h}$ to $H_{t}^{h-\theta}$ for any $\theta>0$. By direct computations we can bound $\widetilde{L}$ from $H^{1}$ to itself, thus (by interpolation) it suffices to bound $\widetilde{L}$ from $H^{1/2+\theta}$ to $H^{1/2-\theta}$. But this result is well-known for the first part of $\widetilde{L}$, and trivial (given the decay for the Fourier transform of $\chi(t)\mathbf{1}_{[T,+\infty)}(t)$) for the last two parts.
\end{proof}
\subsection{The Hamiltonian structure and global well-posedness} In this section we will denote any constant by $C$, since they no longer make any difference. We fix some large time $T$, and recall the energy functional 
\begin{equation}E_{N}[f]=\int_{\mathbb{T}}\frac{1}{2}|\partial_{x}^{1/2}f|^{2}-\frac{1}{6}(S_{N}f)^{3}
\end{equation}defined in Section \ref{gibbs}. If we introduce the symplectic form
\begin{equation}\omega(u,v)=\int_{\mathbb{T}}u\cdot (\partial_{x}^{-1}v)\nonumber\end{equation} in the (finite dimensional) space $\mathcal{V}_{N}$, then a simple computation shows that the Hamiltonian equation with respect to the symplectic form $\omega$ and the functional $E_{N}$ is (up to a sign depending on the convention) the truncated equation (\ref{smoothtrunc}). By Liouville's Theorem, the solution flow $\{\Phi_{t}^{N}\}_{t\in\mathbb{R}}$ will preserve the measure $\mathcal{L}_{N}$ which corresponds to the Lebesgue measure on $\mathbb{R}^{2N}$ (see Section \ref{gibbs}). Since this flow also preserves the $L^{2}$ norm as well as the Hamiltonian $E_{N}$, we see that
\begin{equation}\label{appr.inv}\nu_{N}^{\circ}(E)=\nu_{N}^{\circ}\big(\Phi_{t}^{N}(E)\big)\end{equation} for all time $t$ and all Borel set $E\subset \mathcal{V}_{N}$.

Next, for any $f\in\mathcal{V}$, consider the functions $u^{N}(t)=\Phi_{t}^{N}\Pi_{N}f$, which are the solutions to (\ref{smoothtrunc}) with initial data $u^{N}(0)=\Pi_{N}f$. Thus $f\mapsto u^{N}$ is a map from $\mathcal{V}$ to $\mathcal{BO}^{T}$ depending on $N$, therefore we may denote $|\!|\!|u^{N}|\!|\!|_{T}=J_{N}(f)$.

Choose a large positive integer $M$, a parameter $A$ depending on $M$, and define
\begin{equation}\Omega_{N,A}=\big\{g\in \mathcal{V}_{N}:\|g\|_{Z_{1}}>A\big\},\nonumber\end{equation} then we have that
\begin{equation}\nu_{N}^{\circ}(\Omega_{N,A})=\nu_{N}(\Pi_{N}^{-1}(\Omega_{N,A}))\leq\nu_{N}\big(\big\{f\in\mathcal{V}:\|f\|_{Z_{1}}>A\big\}\big)\leq Ce^{-C^{-1}A^{2}},\end{equation} where the last inequality follows from Proposition \ref{compat}, Cauchy-Schwartz, and the fact that $\|\theta_{N}\|_{L^{2}(\mathrm{d}\rho)}=O(1)$ (which is part of Proposition \ref{convergence2}). Therefore if we introduce
\begin{equation}\Omega_{N,M,A}=\bigcup_{j=-M}^{M}\big(\Phi_{\frac{jT}{M}}^{N}\big)^{-1}(\Omega_{N,A}),\nonumber\end{equation} we will have
\begin{equation}\label{partition}\nu_{N}^{\circ}(\Omega_{N,M,A})\leq CMe^{-C^{-1}A^{2}}.\end{equation} If we choose $A=A(M)=C'\sqrt{\log M}$ with some sufficiently large $C'$, then the inequality (\ref{partition}) will imply $\nu_{N}^{\circ}(\Omega_{N,M,A})\leq CM^{-3}$. Now if $g\not\in \Omega_{N,M,A}$, we must have
\begin{equation}\Phi_{\frac{jT}{M}}^{N}(g)\not\in\Omega_{N,A(M)},\nonumber\end{equation} for all $|j|\leq M$. By Proposition \ref{uniformest}, this implies
\begin{equation}\label{totalll}\max_{|j|\leq M}|\!|\!|(\Phi_{t}^{N}g)_{t}|\!|\!|_{[\frac{(j-1)T}{M},\frac{(j+1)T}{M}]}\leq Ce^{CC'\sqrt{\log M}},\end{equation} provided $T/M\leq C^{-1}e^{-CA(M)}$, which is clearly true when $M$ is large enough depending on $T$. Using partitions of unity, we easily see that (\ref{totalll}) implies
\begin{equation}|\!|\!|(\Phi_{t}^{N}g)_{t}|\!|\!|_{T}\leq CM^{C},\nonumber\end{equation} again when $M$ is large enough depending on $T$. Therefore we have proved\begin{equation}\nu_{N}\big(\big\{f\in\mathcal{V}:J_{N}(f)>CM^{C}\big\}\big)\leq CM^{-3}\end{equation} for all $M>M(T)$, and hence (recall Section \ref{gibbs} for the definition of $\theta_{N}$)
\begin{equation}\sup_{N}\int_{\mathcal{V}}\log(J_{N}(f)+2)\theta_{N}(f)\,\mathrm{d}\rho(f)<\infty.\end{equation} Since $\theta_{N}(f)$ converges to $\theta(f)$ almost surely after passing to a subsequence, we may use Fatou's Lemma to conclude that except for a set with zero $\rho$ measure, for each $f$ with $\theta(f)>0$, there exists a sequence $N_{k}\uparrow\infty$ so that $J_{N_{k}}(f)\leq C$ for some $C$. By part (5) of Theorem \ref{main'}, this would imply the existence of a solution $u\in \mathcal{BO}^{T}$ to (\ref{bo}) on $[-T,T]$ with initial data $f$. Finally, by Remark \ref{rmk4.3} we may choose a sequence of Gibbs measures $\{\theta^{R}\}$ so that for almost every $f\in\mathcal{V}$ we have at least one $\theta^{R}(f)>0$ ; then we take another countable intersection with respect to $T$, to arrive at the following
\begin{proposition}\label{finalprop}
For almost every $f\in\mathcal{V}$ with respect to the Wiener measure $\rho$, there exists a unique global solution $u$ to (\ref{bo}) with initial data $f$, such that $u\in\mathcal{BO}^{T}$ for each $T>0$.
\end{proposition}
\subsection{The global flow and invariance of Gibbs measure}
In this section we will restate and prove Theorem \ref{main2}.
\begin{theorem}[Restatement of Theorem \ref{main2}]\label{main2'} Let the Wiener measure $\rho$ be defined as in Section \ref{gibbs}. There exists a subset $\Sigma\subset \mathcal{V}$ such that $\rho(\mathcal{V}-\Sigma)=0$, and the following holds: for any $f\in\Sigma$ there exists a unique global solution $u$ to (\ref{bo}) with initial data $f$ such that $u\in\mathcal{BO}^{T}$ for all $T>0$. Moreover, let $u=\Phi f=(\Phi_{t}f)_{t}$, then these $\Phi_{t}$ form a measurable transformation group from $\Sigma$ to itself. Finally, suppose the Gibbs measure $\nu$ is defined as in Section \ref{gibbs} (using some cutoff function $\zeta$), then each $\Phi_{t}$ keeps $\nu$ invariant.
\end{theorem}
\begin{proof} We define $\Sigma$ to be the set of all $f\in \mathcal{V}$ such that there exists a solution $u$ to (\ref{bo}) with initial data $f$ that belongs to $\mathcal{BO}^{T}$ for all $T>0$. Proposition \ref{finalprop} guarantees that $\rho(\mathcal{V}-\Sigma)=0$; also the map $\Phi$ is well-defined on $\Sigma$, and each $\Phi_{t}$ maps $\Sigma$ to itself. Note that from part (4) of Theorem \ref{main'}, any two solutions to (\ref{bo}) that belong to $\mathcal{BO}^{T}$ and agree at one time must coincide, thus $u$ will be unique for each fixed $f\in\Sigma$. Now fix a Gibbs measure $\nu$; to prove the invariance of $\nu$, we only need to show that
\begin{equation}\label{oneside}\nu(\Phi_{t}(E))\geq\nu(E)\end{equation} for each Borel subset $E$ and each $|t|\leq 1$, since the rest can be done by iteration.

Define the set
\begin{equation}\Sigma_{A}=\big\{f\in\Sigma:\sup_{|t|\leq 2}\|\Phi_{t}f\|_{Z_{1}}\leq A\big\}\nonumber\end{equation} for each $A$, we then have $\Sigma=\cup_{A}\Sigma_{A}$, so we only need to prove (\ref{oneside}) assuming $E\subset \Sigma_{A}$ for some $A$. By iteration, it then suffices to prove (\ref{oneside}) under the assumption that $E\subset\{f:\|f\|_{Z_{1}}\leq A\}$ and $|t|\leq t(A)$. Next, we introduce on the set $\{f:\|f\|_{Z_{1}}\leq A\}$ the metric
\begin{equation}d(f,g)=\big\|\langle n\rangle^{-s^{6}+r}(f-g)\big\|_{l^{p}},\nonumber\end{equation}making it a complete separable metric space. By a well-known theorem in measure theory, the restriction of $\nu$ to this set is a finite Borel measure on this metric space, and thus is regular (meaning every Borel set can be approximated from inside by compact sets). Therefore we may further assume $E$ is compact with respect to the metric $d$. Recall the solution flow $\{\Phi_{t}^{N}\}$ for (\ref{smoothtrunc}); for each $N$ we have
\begin{equation}\nu_{N}\big(\big\{g:\Pi_{N}g=\Phi_{t}^{N}(\Pi_{N}h),\,\,h\in E\big\}\big)\geq \nu_{N}(E)\end{equation} by the invariance of $\nu_{N}^{\circ}$ under the flow $\Phi_{t}^{N}$. To prove (\ref{oneside}) it thus suffices to show\begin{equation}\limsup_{N\to\infty}\big\{g:\Pi_{N}g=\Phi_{t}^{N}(\Pi_{N}h),\,\,h\in E\big\}\subset\Phi_{t}(E),\end{equation} since we already know that the total variation of $\nu_{N}-\nu$ tends to zero.

Now suppose for some $g\in\mathcal{V}$ we have a subsequence $N_{k}\uparrow\infty$ and $h^{N_{k}}\in E$ such that $\Pi_{N_{k}}g=\Phi_{t}^{N_{k}}(\Pi_{N_{k}}h^{N_{k}})$ for each $k$. By compactness we may assume $h^{N_{k}}\to h$ with respect to the metric $d$ for some $h\in E$. Since every function involved here is bounded in $Z_{1}$ norm by $O_{A}(1)$, and we are assuming $|t|\leq t(A)$, we may use Propositions \ref{difference0} and \ref{main'}, as well as the limit
\begin{equation}\big\|\langle \partial_{x}\rangle^{-s^{5}}(h^{N_{k}}-h)\big\|_{Z_{1}}\lesssim d(h^{N_{k}},h)\to 0\end{equation}to conclude that
\begin{eqnarray}\big\|\langle\partial_{x}\rangle^{-s^{5}}(\Phi_{t}h-\Pi_{N_{k}}g)\big\|_{Z_{1}}&\leq&\big\|\langle\partial_{x}\rangle^{-s^{5}}(\Phi_{t}h-\Phi_{t}^{N_{k}}\Pi_{N_{k}}h)\big\|_{Z_{1}}\nonumber\\
&+&\big\|\langle\partial_{x}\rangle^{-s^{5}}(\Phi_{t}^{N_{k}}\Pi_{N_{k}}h-\Phi_{t}^{N_{k}}\Pi_{N_{k}}h^{N_{k}})\big\|_{Z_{1}}\to 0\nonumber.\end{eqnarray} This implies $g=\Phi_{t}h\in\Phi_{t}(E)$, so the proof is complete.
\end{proof}

\end{document}